\documentclass[12pt]{amsart}

\usepackage{graphicx}
\usepackage{subfig}
\usepackage{amssymb}
\usepackage{amsthm}
\usepackage{amsmath}
\usepackage{amscd}
\usepackage{mathtools}
\usepackage{color}
\usepackage{comment}
\usepackage{xfrac}

\usepackage{imakeidx}

\makeindex

\usepackage[usenames,dvipsnames]{xcolor}

\usepackage[colorlinks = true,
            linkcolor = red,
            urlcolor  = blue,
            citecolor = blue,
            anchorcolor = blue]{hyperref}
\numberwithin{equation}{section}

\newcommand{\lab}{\label}

\newcommand{\ben}{\begin{enumerate}}
\newcommand{\een}{\end{enumerate}}

\newcommand{\bea}{\begin{eqnarray}}
\newcommand{\ba}{\begin{array}}
\newcommand{\bean}{\begin{eqnarray*}}
\newcommand{\ea}{\end{array}}
\newcommand{\eea}{\end{eqnarray}}
\newcommand{\eean}{\end{eqnarray*}}
\newcommand{\beq}{\begin{equation}}
\newcommand{\eeq}{\end{equation}}
\newcommand{\bthm}{\begin{thm}}
\newcommand{\ethm}{\end{thm}}
\newcommand{\blem}{\begin{lem}}
\newcommand{\elem}{\end{lem}}
\newcommand{\bprop}{\begin{prop}}
\newcommand{\eprop}{\end{prop}}
\newcommand{\bcor}{\begin{cor}}
\newcommand{\ecor}{\end{cor}}
\newcommand{\bobs}{\begin{obs}}
\newcommand{\eobs}{\end{obs}}
\newcommand{\bdfn}{\begin{dfn}}
\newcommand{\edfn}{\end{dfn}}
\newcommand{\brem}{\begin{rem}}
\newcommand{\erem}{\end{rem}}
\newcommand{\bpf}{\begin{proof}}
\newcommand{\epf}{\end{proof}}
\newcommand{\bfact}{\begin{fact}}
\newcommand{\efact}{\end{fact}}
\newcommand{\bprob}{\begin{prob}}
\newcommand{\eprob}{\end{prob}}
\newcommand{\bcon}{\begin{con}}
\newcommand{\econ}{\end{con}}

\newcommand{\nl}{\newline}

\alph{enumii} \roman{enumiii}

\newtheorem{thm}{Theorem}[section]
\newtheorem{prop}[thm]{Proposition}
\newtheorem{lem}[thm]{Lemma}

\newtheorem{cor}[thm]{Corollary}
\newtheorem{dfn}[thm]{Definition}
\newtheorem{rem}[thm]{Remark}
\newtheorem{fact}[thm]{Fact}
\newtheorem{ex}[thm]{Example}
\newtheorem{obs}[thm]{Observation}

\newtheorem{prob}[thm]{Problem}
\newtheorem{con}[thm]{Conjecture}
             \def\cB{\mathcal B}       
\def\cH{\mathcal H}             \def\cF{\mathcal F}       \def\cG{\mathcal G}
\def\cU{\mathcal U}                    
\def\cL{{\mathcal L}}           \def\cM{\mathcal M}       
\def\cS{\mathcal S}             \def\cN{\mathcal N}       \def\cD{\mathcal D}

\def\ve{\varepsilon}
\def\tf{\tilde{f}}

\def\lra{\longrightarrow}

\def\endpf{\qed}

\def\N{{\mathbb N}}                  \def\R{{\mathbb R}}
\def\C{{\mathbb C}}                  \def\oc{\hat \C}

\newcommand{\1}{\ensuremath{\mathbbold{1}}} %\def\1{1\!\!1}
\def\and{{\rm \  \  and \  \ }}        \def\for{{\rm \  \  for \  \ }}
        
\def\tor{{\rm \  \  or \  \ }}         \def\all{{\rm all \  \ }}
\def\some{{\rm some \  \ }}
          
\def\Comp{\text{{\rm Comp}}}  \def\diam{\text{\rm {diam}}}
\def\dist{\text{{\rm dist}}}   
\def\Trans{{\rm Trans}}       \def\Rat{{\rm Rat}}
\def\Crit{{\rm Crit}}         \def\CF {{\rm C--F}}
\def\Sing{{\rm Sing}}        
      
\def\F{{\mathcal F}}          
\def\TNR{{\rm TNR}}     \def\FNR{{\rm FNR}} 
\def\NOSC-FNR{{\rm NOSC-FNR}}

\def\HIB{{\rm HIB }}  \def\Pre{{\rm Pre }}
\def\h{{\rm h}}           \def\Rat{{\rm Rat }}
           
\def\H{\text{{\rm H}}}     \def\HD{\text{{\rm HD}}}   
\def\BD{\text{{\rm BD}}}         \def\PD{\text{{\rm PD}}}
     \def\CP{\text{CP}} \def\PCV{{\rm PCV}} \def\DPCV{{\rm DPCV}}
\def\Int{\text{{\rm Int}}}        
         \def\P{\text{{\rm P}}}     \def\Id{\text{{\rm Id}}}
 
             \def\Ba{\mathcal B}

\def\a{\alpha}                \def\b{\beta}             \def\d{\delta}
\def\De{\Delta}               \def\e{\varepsilon}          
\def\g{\gamma}                \def\Ga{\Gamma}           \def\l{\lambda}
              \def\om{\omega}           \def\Om{\Omega}
\def\Sg{\Sigma}               \def\sg{\sigma}			\def\zt{\zeta}
               \def\th{\theta}           \def\vth{\vartheta}
\def\ka{\kappa}
               \def\mh{\tilde{m}_h}

\def\abs{\prec}
\def\bi{\bigcap}              \def\bu{\bigcup}
\def\({\bigl(}                \def\){\bigr)}
\def\lt{\left}                \def\rt{\right}

\def\ld{\ldots}               \def\bd{\partial}         \def\^{\tilde}

\def\es{\emptyset}            \def\sms{\setminus}
\def\sbt{\subset}             \def\spt{\supset}
\def\gek{\succeq}             \def\lek{\preceq}

\def\comp{\asymp}
           \def\downto{\searrow}
\def\sp{\medskip}             \def\fr{\noindent}        \def\nl{\newline}
\def\ov{\overline}            \def\un{\underline}
\def\ess{{\rm ess}}           
				\def\osc{{\rm osc}}
\def\om{\omega}

\def\supp{\text{{\rm supp}}}
\def\endpf{{\hfill $\square$}}
\def\Fa{\mathcal F}

\newcommand{\ep}{\varepsilon}

\newcommand{\al}{\alpha}
\newcommand{\ga}{\gamma}
\newcommand{\vr}{\varrho}

\def\bs{\backslash}
\def\sub{\subset}
\def\bus{\supset}
\DeclareSymbolFont{bbold}{U}{bbold}{m}{n}
\DeclareSymbolFontAlphabet{\mathbbold}{bbold}

\newcommand{\dl}{\delta}
\newcommand{\set}[1]{\ensuremath{ \left\{#1\right\} }} % set notation
\providecommand{\phantomsection}{}% In case hyperref is not loaded
\AtBeginDocument{\let\textlabel\label}% http://tex.stackexchange.com/q/9939/5764
\makeatletter
\newcommand{\mylabel}[2]{\raisebox{.7\normalbaselineskip}{\phantomsection}(#1)%
	\def\@currentlabel{#1}\textlabel{#2}}
\makeatother

\providecommand{\phantomsection}{}% In case hyperref is not loaded
\AtBeginDocument{\let\textlabel\label}% http://tex.stackexchange.com/q/9939/5764
\makeatletter
\newcommand{\myClaimlabel}[2]{\raisebox{.7\normalbaselineskip}{\phantomsection}#1%
	\def\@currentlabel{#1}\textlabel{#2}}
\makeatother

\usepackage{nameref}

\newcounter{mylabelcounter}

\makeatletter
\newcommand{\labelText}[2]{%
#1\refstepcounter{mylabelcounter}%
\immediate\write\@auxout{%
  \string\newlabel{#2}{{1}{\thepage}{{\unexpanded{#1}}}{mylabelcounter.\number\value{mylabelcounter}}{}}%
}%
}
\makeatother
\newcommand\blfootnote[1]{%
	\begingroup
	\renewcommand\thefootnote{}\footnote{#1}%
	\addtocounter{footnote}{-1}%
	\endgroup
}

\begin{document}
%***********************************************

\title[]
{\bf\large {\Large T}he Dynamics and Geometry \\  of  \\  Semi-Hyperbolic   
Rational Semigroups}
\date{\today}
%\date{}
% % author information %
\author[\sc Jason ATNIP]{\sc Jason ATNIP}
\author[\sc Hiroki SUMI]{\sc Hiroki SUMI}
%\address{\ni Hiroki Sumi 
%Department of Mathematics,
%Graduate School of Science,
%1-1 Machikaneyama,
%Toyonaka,
%Osaka, 560-0043, 
%Japan}
%\email{sumi@math.sci.osaka-u.ac.jp\newline\hspace*{0.3cm} 
%Web: http://www.math.sci.osaka-u.ac.jp/$\sim$sumi/}
%
\author[\sc Mariusz URBA\'NSKI]{\sc Mariusz URBA\'NSKI}
%\address{Mariusz Urba\'nski, Department of Mathematics,
% University of North Texas, Denton, TX 76203-1430, USA}
%\email{urbanski@unt.edu\newline \hspace*{0.3cm} Web:
%http://www.math.unt.edu/$\sim$urbanski}
%
% dedication %
%\dedicatory{}
%
% AMS information %

%Bowen parameter , topological pressure, structural stability} 
%\subjclass{Primary: 37F35; Secondary: 37F15.}

\begin{abstract}
	We study skew-product dynamics for a large class of finitely-generated semi--hyperbolic semigroups of rational maps acting on the Riemann sphere, which generalizes both the theory of iteration of a single rational map of a single complex variable (complex/holomorphic dynamics) and the theory of countable alphabet conformal iterated function systems (CIFSs).
	We construct the thermodynamic formalism for such dynamical systems and geometric potentials by developing the notion of nice families that extend to the case of our highly disconnected skew product phase space the powerful notion of nice sets due to Rivera--Letelier and Przytycki, and the allied earlier notion of $K(V)$ sets due to Denker and the last named author.
	We leverage out techniques to prove the existence and uniqueness of equilibrium states for a wide class of H\"older potentials, and concomitant statistical laws: central limit theorem, law of iterated logarithm, and
exponential decay of correlations. We devote lots of space and effort to control (non-recurrent) critical points which is a notoriously challenging task even for a single rational function; more generators add qualitatively new challenges.
Beyond dynamics, but still with dynamical methods, we advance the study of finer fractal geometrical properties
	of the intricate Julia sets associated to such systems and, in  particular, via equilibrium states, we perform a multifractal analysis of Lyapunov exponents. We use the Nice
	Open Set Condition (NOSC) introduced by the last two authors, and apply our new
	techniques to settle a long-standing problem in the theory of rational semigroups
	by proving that for our class of semigroups the Hausdorff dimension of each fiber
	Julia set is strictly smaller than the Hausdorff dimension of the global Julia set of
	the semigroup. In Appendix~A we provide corrected proofs of some two lemmas (not used in the current paper) from the article 
"Measures and Dimensions of Julia Sets of Semi--hyperbolic Rational Semigroups", 
Discrete \& Continuous Dynamical Systems, 30 (2011), 313--363, by H. Sumi and M. Urba\'nski.
\end{abstract}
\maketitle

\clearpage
\tableofcontents

\section{Introduction}

A \textbf{rational semigroup} \index{rational semigroup} $G$ is a semigroup generated by a family of non--constant rational maps $g:\oc \longrightarrow \oc$,\ where $\oc $ \index{$\oc$} denotes the Riemann sphere,\ with the semigroup operation being functional 
composition. For a rational semigroup $G$,\ we set 
$$
F(G):=\Big\{ z\in \oc: G \mbox{ is normal in a neighborhood of } z\Big\} 
$$ 
and 
$$
J(G):=\oc \setminus F(G).
$$ \index{Julia set!$J(G)$}\index{$J(G)$}
$F(G)$ is called the \textbf{Fatou set of $G$}\index{Fatou set! $F(G)$}\index{$F(G)$} and $J(G)$ is called the 
\textbf{Julia set of $G$}. If $G$ is generated by a family $\{ f_{i}\} _{i}$,\ 
 then we write $G=\langle f_{1},f_{2},\ldots\! \rangle .$ We say that $G$ is \textbf{finitely generated} \index{finitely generated} if there exists a finite collection $f_1,\dots,f_k$ such that $G=\langle f_{1},f_{2},\ldots f_k \rangle .$

The work on the dynamics of rational semigroups was initiated by 
Hinkkanen and Martin (\cite{HM}, \cite{HM2}), who were interested 
in the role of the dynamics of polynomial 
semigroups while studying various one--complex--dimensional moduli spaces 
for discrete groups, and by 
F. Ren's group (\cite{ZR}), who studied such semigroups 
from the perspective of random complex dynamics.  
%For further studies on the dynamics of rational semigroups, 
%see \cite{hiroki1, hiroki2, hiroki3, hiroki4, 
%hirokidc, hiroki5, StSu, SU}.
The theory of the dynamics of rational semigroups on $\oc $ 
has developed in many directions since the 1990s (\cite{HM, ZR, HM2,
%St1, St2, St3, SSS, 
sumihyp1, sumihyp2, 
hiroki1, hiroki2, hiroki3, hiroki4, hirokidc, sumikokyuroku, SU1, SdpbpI, SdpbpII, SdpbpIII, hiroki5, 
%StSu, 
hiroki6}). 

%The theory of the dynamics of rational semigroups is  
% deeply related to that of the fractal geometry. In fact, 
Since the Julia set $J(G)$ of a rational semigroup $G$ 
generated by finitely many elements $f_{1},\ldots\!\!,f_{u}$ has   
backward self--similarity, i.e., 
\begin{equation}
\label{eq:bss} 
J(G)=f_{1}^{-1}(J(G))\cup \cdots \cup f_{u}^{-1}(J(G))
\end{equation} 
 (see \cite{hiroki1}), 
%Hence, 
% the behavior of the (backward) dynamics of finitely generated 
% rational semigroups can be regarded 
% as that of the ``backward iterated function systems.''
%For example, the Sierpi\'{n}ski gasket can be regarded as 
%the Julia set of a rational semigroup.
%
it can be viewed as a significant generalization and extension of 
both, the theory of iteration of rational maps (see \cite{M}),  
and conformal iterated function systems (see \cite{mugdms}). 
For example, the Sierpi\'{n}ski gasket can be regarded as the Julia set 
of a rational semigroup. 
The theory of the dynamics of 
rational semigroups borrows and develops tools 
from both of these theories. It has also developed its own 
unique methods, notably the skew product approach 
(see \cite{hiroki1, hiroki2, hiroki3, hiroki4, SdpbpI, SdpbpII, SdpbpIII, hiroki5, SU1}, and \cite{SU2}). 
We remark that by (\ref{eq:bss}), 
the analysis of the Julia sets of rational semigroups somewhat
resembles 
 ``backward iterated functions systems'', however since each map 
$f_{j}$ is not in general injective (critical points), some 
qualitatively different extra effort is needed in the case of semigroups.

The theory of the dynamics of rational semigroups is somehow 
related to that of the random dynamics of rational maps. Especially, if one looks at separated fiber Julia sets $J_\om$. Similarities and differences of these two theories become particularly transparent in Section~\ref{section:fiber-global} where we compare the Hausdorff dimension of those fiber Julia sets $J_\om$ and the 
Hausdorff dimension of the global Julia set $J(G)$.

\sp In this paper, we investigate dynamics, ergodic theory, and geometry of $t$--conformal measures $m_t$ and equilibrium states $\mu_t$, equivalent to $m_t$, of the skew product map 
$$
\^f:\Sg_u\times\hat\C\lra \Sg_u\times\hat\C
$$ 
and geometric potentials on the Julia set $J(\^f)$, i.e. ones of the form 
$$
J(\^f)\ni\xi\longmapsto -t\log|\tf'(\xi)|\in\R 
$$
where $t\ge 0$ is a fixed parameter. Our starting point is the paper \cite{sush} and we go far beyond. We introduce in the current manuscript the class of totally non--recurrent (TNR) \CF\ balanced rational semigroups of finite type which we abbreviate as the class of finely non--recurrent rational semigroups (FNR). This suffices for dynamics and ergodic theory. When we deal with geometry of fiber and global Julia sets we impose in addition the Nice Open Set Condition of \cite{sush}. In the current manuscript this condition is studied at length in Section~\ref{NOSC} of our current manuscript. We would like to emphasize that the Nice Open Set Condition and Nice Sets (Families), discussed later, are totally independent concepts. In particular, the adjective ``Nice" was  independently introduced for both concepts many years ago. Although it may be a little bit confusing for some readers, we stick to the historical terminology to respect history and in order not confuse readers even more by inventing yet new names. We think that in our current manuscript this is the first time in the literature that both ``nice" concepts are used simultaneously.

We develop the ergodic theory, stochastic properties, and geometry of measures $m_t$ and $\mu_t$. 

Throughout the course of our exposition we introduce and define various subclasses of rational semigroups. These look like quite technical concepts but this is an actual, essentially indispensable, feature of (intricate) ergodic theory, dynamics, and geometry of rational semigroups. For the convenience of the reader we collect all these definitions and summarize relations between them in Appendix~\ref{DCRS}.

One of our most important key tools and concepts is that of nice sets. It was originally introduced in \cite{Riv07} and extensively used, among others in \cite{PrzRiv07}. It permits us to build sufficiently rich symbolic dynamics of the skew product map $\^f:\Sg_u\times\hat\C\to \Sg_u\times\hat\C$ to utilize extensively the results and methods of the countable alphabet theory of subshifts of finite type, also known as topological Markov chains. Most notably those developed in \cite{mugdms}, \cite{muisrael}, and \cite{lsy1} and \cite{lsy2}. As in \cite{sush} our main technical tool is that of ``holomorphic'' inverse branches; in quotations because they live on the Cartesian product $\Sg_u\times\oc$. We however define them more carefully and more precisely than in \cite{sush}, and we make a more refined use of them than in \cite{sush}. In particular, they are instrumental in forming nice families of $\tf$ and are members of the graph directed Markov system (in the sense of \cite{mugdms}) generated by these families. 

We deal with $t$--conformal measures $m_t$ and equilibrium states $\mu_t$ by bringing up and elaborating on the refined tool of Vitali relations due
to Federer (see \cite{federer}). This tool is needed basically because, unlike for true conformal maps, the ``holomorphic'' inverse branches of iterates of $\tf$, enormously distort the balls in the product space $\Sg_u\times\oc$. In fact, they are more like affine maps with two different contracting factors. We rely heavily here on deep results from \cite{federer}. Another tool, already employed in \cite{ncp2} and subsequent papers of the third named author, is the Martens method of producing $\sg$--finite invariant measures absolutely continuous with respect to a given quasi--invariant measure. It was considerably refined and generalized in \cite{sush}. We apply and develop this method to construct the $\tf$--invariant measures $\mu_t$. These are first constructed as merely $\sg$--finite measures and proven to be finite only much later after preparing the machinery of nice families and symbolic representation. 

Using the, already several times mentioned, symbolic dynamics we prove that the $\tf$--invariant measures $\mu_t$ are the unique equilibrium states of the potentials $-t\log|\tf'(\xi)|$ in a very classical sense. These are the only measures that maximize the free energy functions associated to these potentials (see below in this introduction for a more precise statement) and the supremum is equal to the topological pressure $\P(t)$. 

The methods we develop in this paper are sufficiently strong and refined to allow us to solve a long standing open problem in the theory of iteration of rational semigroups asking about the relation of the Hausdorff dimension of the fiberwise Julia sets $J_\om$ and the global Julia set $J(G)$. We show that for \CF \ balanced \TNR \ rational semigroups of finite type
the Hausdorff dimension of every fiberwise Julia set $J_\om$ is smaller than the Hausdorff dimension of the global Julia set $J(G)$. We prove a little bit more in this respect.

\sp Now we shall describe the main results of \cite{sush} and the current paper. The notation and concepts we use in this description are fairly standard for the theory of rational semigroups and thermodynamic formalism. These are carefully defined and introduced in Preliminaries, throughout the manuscript and in Appendix~\ref{DCRS}. We start our description with \cite{sush}.  Its main results are comprised in the following.

\bthm\label{Theorem A} 
Let $f=(f_{1},\ldots ,f_{u})\in \mbox{{\em Rat}}^{u}$ be a $u$--tuple of rational maps  for a positive integer $u$. 
Let $G=\langle f_{1},\ldots ,f_{u}\rangle $.
Suppose that 

\begin{enumerate}
\item[\mylabel{1}{Theorem A item 1}] There exists an element $g$ of $G$ such that $\deg (g)\geq 2$, 

\item[\mylabel{2}{Theorem A item 2}] Each element of {\em Aut}$(\oc)\cap G$ (if this is not empty) is loxodromic, 

\item[\mylabel{3}{Theorem A item 3}] $G$ is semi--hyperbolic, and 

\item[\mylabel{4}{Theorem A item 4}] $G$ satisfies the Nice Open Set Condition. 
\end{enumerate}

\fr Then, we have the following. 
\begin{itemize}
\item[(a)] $J(G)\cap \PCV(G)$ is nowhere dense in 
$J(G)$ and, for each $t\geq 0$, the function 
$$
z\longmapsto \P_{z}(t)\in\R
$$ 
is constant throughout a neighborhood of $J(G)\setminus \PCV(G)$ in $\oc$.
Denote this constant by $\P(t)$. 

\, \item[(b)] The function 
$$
[0,+\infty)\ni t\longmapsto \P(t)\in\R
$$ 
has a unique zero. This zero is denoted by $h=h_f$. 

\, \item[(c)] There exists a unique $|\tf '|^{h}$--conformal measure $\mh $ for the map 
$\tf :J(\tf)\longrightarrow J(\tf ).$ 

\, \item[(d)] Let $m_{h}:= \mh\circ p_{2}^{-1}.$ 
Then there exists a constant $C\geq	1$ such that 
$$C^{-1}\leq \frac{m_{h}(B_{s}(z,r))}{r^{h}}\leq C$$
for all $z\in J(G)$ and all $r\in (0,1].$

\, \item[(e)] 
$$
h_f=\HD(J(G))=\PD(J(G))=\BD(J(G)), 
$$
where $\HD ,\PD, \BD $ denote the 
Hausdorff dimension, packing dimension, and box dimension, respectively, 
with respect to the spherical distance in $\oc .$  

We denote this common value by $h_G$; it depends only on the semigroup $G$ and is independent of the set of generators (satisfying conditions (1)--(4) above) used to for the skew product map $\tf$. Moreover, for each $z\in J(G)\setminus \PCV(G)$, 
we have  
$$
h_G=T_{f}(z)=t_{0}(f)=S_{G}(z)=s_{0}(G). 
$$ 
\item[(f)]  For every $t\ge 0$ there exists a $|\tf'|^t$--conformal measure $m_t$ for the map $\tf:J(\tf)\to J(\tf)$.

\item [(g)] 
Let $\H^{h}$ and $\P^{h}$ be the $h$--dimensional Hausdorff measure and 
$h$--dimensional packing measure respectively. Then, 
all the measures 
$$
\H^{h}, \  \P^{h}, \  {\rm and, } \   m_{h}
$$ 
are mutually equivalent with Radon--Nikodym derivatives uniformly separated away from zero and infinity.

\, \item[(h)] $$0<\H^{h}(J(G)),\P^{h}(J(G))<\infty .$$ 

\, \item[(i)] There exists a unique Borel probability $\tf$--invariant measure 
$\tilde{\mu} _{h}$ on $J(\tf )$ which is absolutely continuous with respect to $\mh .$ The measure $\tilde{\mu} _{h}$ is metrically exact, hence ergodic,  and equivalent with $\mh .$   
\end{itemize}
\ethm 

\bdfn\label{d120220816}
Any rational semigroup $G$ with non-empty Fatou set $F(G)$ that satisfies conditions \eqref{Theorem A item 1}--\eqref{Theorem A item 3} from Theorem~\ref{Theorem A} (i.e. it is semi-hyperbolic and satisfies the Fundamental Assumption formulated below) is called *semi--hyperbolic.
\edfn 

\sp\fr We want to emphasize now, and will repeat it within the main body of our manuscript, that being semi--hyperbolic and *semi--hyperbolic does not depend on the choice of generators but on the semigroup alone.

\sp We will now describe the main results of the current paper. As we have already mentioned, we introduce in it the class of non--recurrent (TNR) \CF \ balanced rational semigroups of finite type, called in short the class of finely non--recurrent (FNR) rational semigroups. Loosely speaking these adjectives respectively mean that the closure of the postcritical set is disjoint from the critical set and the part of the postcritical set lying in the Fatou set is at a positive distance from the Julia set. In Section~\ref{s:Ex} we give some examples of totally non--recurrent (TNR) C--F balanced rational semigroups of finite type. Many of our main results need only some parts of these assumptions but we do not discern them here for the sake of ease of exposition. We refer the reader to actual theorems in the body of the paper for most adequate assumptions. 

We prove that there exists an open interval $\De_G^*\spt [0,h]$ for which, among others, the following theorems hold. 

\bprop\label{l1h35B}
If $G$ is a *semi--hyperbolic rational semigroup generated by a $u$--tuple of rational maps   $f=(f_{1},\ldots ,f_{u})\in \mbox{{\em Rat}}^u$, then the function $t\longmapsto\P(t)$, $t\ge 0$, has the following properties.

\sp\begin{itemize}
\item[(a)] For every $t\ge 0$ we have that $\P(t)\in(-\infty,+\infty)$ and $\P(0)\ge \log 2>0$.

\sp\item[(b)] The function $[0,+\infty)\longmapsto \P(t)$ is strictly decreasing and Lipschitz continuous. More precisely:

\sp\item[(c)] If $0\le s\le t<+\infty$, then 
$$
-\log\|\tf'\|_\infty(t-s)\le \P(t)-\P(s)\le - \a(t-s),
$$
where the constant $\a>0$ comes from the Exponential Shrinking Property (Theorem~\ref{t1h4}).
\sp\item[(d)] $\lim_{t\to+\infty}\P(t)=-\infty$.
\end{itemize}
\eprop

Making a substantial use of Marco Martens's method, which originated in \cite{martens} and was explored for example in \cite{sush}, we prove the following.

\bthm\label{t4h65B}
Let $G$ be a *semi--hyperbolic rational semigroup. If $t\in \De_G$ then there exists a unique, up to a multiplicative constant, Borel $\sg$--finite $\tf$--invariant measure $\mu_t$ on $J(\^f)$ which is absolutely continuous with respect to $m_t$. In addition, the measure $\mu_t$ is weakly metrically exact and equivalent to $m_t$, in particular it is ergodic. 
\ethm

In Section~\ref{sec:nicesets}, entitled Nice Sets (Families), we explore in detail one of the most important tools for us in the current paper. It is commonly referred to as nice sets or nice families. It has been introduced in \cite{Riv07}, and extensively used, among others in \cite{PrzRiv07}. We  adopt this concept to the setting of rational semigroups. 
	
We would like to emphasize again, repeating what was written shortly before, that the Nice Open Set Condition and Nice Sets (Families) are totally independent concepts. In particular, the adjective ``Nice" was  independently introduced for both concepts many years ago. Although it may be a little bit confusing for some readers, we stick to the historical terminology to respect history and in order not confuse readers even more by inventing yet new names. We think that in our current manuscript this is the first time in the literature that both ``nice" concepts are used simultaneously.

The absolutely first fact needed about nice sets and families is their existence. It is by no means obvious and we devote the whole Section~\ref{sec:nicesets} for this task. In the existing proofs for ordinary conformal systems, i.e. cyclic semigroups, the concept of connectivity of the phase space, usually $\C$ or $\oc$, plays a substantial role. In our present setting of the skew product map 
$$
\^f:\Sg_u\times\hat\C\lra \Sg_u\times\hat\C,
$$
the phase space is ``highly'' not connected. In order to overcome this difficulty we define the concept of connected families of arbitrary sets. These have sufficiently many properties of ordinary connected sets, for example one can speak of connected components of any family of sets, to allow for a proof of the existence of nice families. As a matter of fact, we do not even use the topological concept of connected subsets of the Riemann sphere $\oc$. Our main theorem of this section is the following. 

\bthm\label{t1nsii7B}
Let $G=\langle f_1,\ld,f_u\rangle$ be a \TNR \ semigroup. Fix $R\in(0,R_*(G))$. Fix also $\ka\in (1,2)$. Let
$$
\Crit_*(f)\sbt S\sbt J(G)\sms B_2(\PCV(G),8R)
$$ 
be a finite aperiodic set. Then for every $r\in(0,R]$ small enough there exists 
$$
\cU_S(\ka,r)=\{U_s(\ka,r)\}_{s\in S},
$$
a nice family of sets for $\tf$, associated to the set $S$, such that

\sp\begin{itemize}
\item[(A)] 
$$
\Sigma_u\times B_2(s,r)\sbt U_s(\ka,r) \sbt \Sigma_u\times B_2(s,\ka r)
$$
for each $s\in S$.
\item[(B)] If $a, b\in S$, $\rho\in \HIB(U_b(\ka,r))$, and $\tf_\rho^{-\|\rho\|}(U_b(\ka,r))\sbt U_a(\ka,r)$, then
$$
\lt|\lt(f_\rho^{-1}\rt)'(z)\rt|\le\frac14
$$
for all $z\in B_2(b,2R)\spt p_2\(U_b(\ka,r)\)$.
\end{itemize}
\ethm

The first consequence of this theorem, which is the gate to all of its other consequences, is that it gives rise to the existence of sufficiently rich  ``conformal--holomorphic'' maximal graph directed Markov system in the sense of \cite{mugdms}. More precisely, it gives the following.

\begin{thm}[for the notation and details see Section~\ref{sec:nicesets} and Theorem~\ref{t1nsii15}]\label{t1nsii15B}
If $G$ is a \TNR \ rational semigroup generated by a $u$--tuple of rational maps   $(f_1,\ld,f_u)\in \Rat^u$ and $\cU=\{U_s\}_{s\in S}$ is a nice family of sets for $\^f$, then the family 
$$
\cS_\cU:=\big\{\^f_\tau^{-||\tau||}:X_{t(\tau)}\lra X_{i(\tau)}\big\}_{\tau\in\cD_\cU}
$$
forms a graph directed system in the sense of \cite{mugdms}. Furthermore,
%satisfying the Open Set Condition. 
\sp\begin{itemize}
\item[(a)] The corresponding incidence matrix $A(\cU)$ is then determined by the condition that
$$
A_{\tau\om}(\cU)=1 
$$
if and only if $t(\tau)=i(\om)$. 

\sp\item[(b)] The limit set $J_\cU$ of the system $\cS_\cU$ is contained in $J(\^f)$ and contains $U\cap\Trans(\tf)$, where $\Trans(\tf)$\index{$\Trans(\tf)$} is the set of transitive points of $\tf:J(\tf)\lra J(\tf)$, i.e. the set of points $z\in J(\^f)$ such that the set $\{\^f^n(z):n\ge 0\}$ is dense in $J(\^f)$ .

\sp\item[(c)] The graph directed system $\cS_\cU$ is finitely primitive.
\end{itemize}

\sp\fr We denote by $\cD_\cU^\infty$ the symbol space $\(\cD_\cU\)_{A(\cU)}^\infty$ generated by the matrix $A(\cU)$; as in the case of $\Sg_u$ its elements (infinite sequences) start with coordinates labeled by the integer $1$. Likewise $\cD_\cU^*$ and $\cD_\cU^n$, $n\in\N$, abbreviate respectively $\(\cD_\cU\)_{A(\cU)}^*$ and $\(\cD_\cU\)_{A(\cU)}^n$, $n\in\N$.

\sp\fr In addition, we denote by $\phi_e$, $e\in \cD_\cU$, all the elements of $\cS_\cU$.
\end{thm}

\sp The next section, Section~\ref{sec:nearcriticalpoints}, entitled The Behavior of Absolutely Continuous Invariant Measures $\mu_t$ Near Critical Points,  is very technical and devoted to study the behavior of conformal measures $m_t$ and their invariant versions $\mu_t$ near critical points of the skew product map 
$$
\^f:\Sg_u\times\hat\C\lra \Sg_u\times\hat\C.
$$
Its main outcome is Proposition~\ref{p1cp3} which gives a quantitative strengthening of quasi--invariance of conformal measures $m_t$. This is the first and only place where the hypothesis of finite type of the semigroup $G$ is explicitly needed; it demands that the set of critical points of $\tf$ lying in the Julia set $J(\^f)$ of $\tf$ is finite.

\sp Section~\ref{section:smallpressure}, Small Pressure $\P_V^\Xi(t)$, is still technical. We prove that the (ordinary) topological pressure of the potentials $-t\log|\tf'|$, $t\in\De_G^*$, with respect to the dynamical system $\tf$ restricted to the compact $\tf$--invariant set of all points whose forward iterates avoid an open neighborhood of the set of critical points of $\tf$, is smaller than $\P(t)$. This inconspicuous looking fact is instrumental, one could even say, indispensable, in many further proofs. It intervenes for example in the proof of Lemma~\ref{l1sp2} which in plain words asserts that the measure $m_t$ of the ``tails'' of the maximal graph directed Markov system generated by a (sufficiently good) nice family, decays exponentially fast. This fact is in turn instrumental in Sections~\ref{section:stochasticlawssymbolspace} and  \ref{section:originalstochasticlaws}, making application of Young towers possible. This fact also makes the proof of Variational Principle in Section~\ref{section:VP}, so simple. It is also used in Section~\ref{section:fiber-global} to show that the Hausdorff dimension of fiber Julia sets $J_\om$ is smaller than the Hausdorff dimension of the global Julia set $J(G)$. 

\sp Section~\ref{sec:thermodynamic formalism}, Symbol Space Thermodynamic Formalism associated to Nice Families; Real Analyticity of the Original Pressure $\P(t)$, brings up full fledged fruits of the existence of nice families. It forms a symbolic representation (subshift of finite type with a countable infinite alphabet) of the map generated by a nice family and develops the thermodynamic formalism of the potentials $\zeta_{t,s}$ resulting from those of the form $-t\log|\tf'|$ and the ``first return time'' $\|\tau_1\|$. The first most transparent of its consequences, already possible to be stated and proved in this section, is the following. 

\begin{thm}[Theorem~\ref{t3.11}]\lab{t3.11B} 
If $G$ is a \FNR \ rational semigroup, generated by a $u$--tuple of rational maps $(f_1,\ld,f_u)\in\Rat^u$, then the topological pressure function 
$$
\P:\De_G^*\longrightarrow\R
$$ 
is real--analytic.
\end{thm}

\sp In Section~\ref{sec:invariantmeasures}, Invariant Measures: $\mu_t$ versus $\^\mu_t\circ\pi_\cU^{-1}$; Finiteness of $\mu_t$, we link the measures $\^m_t$ and $\^\mu_t$ of the previous section living on the symbol space with the conformal and invariant measures $m_t$ and $\mu_t$ living on the Julia set $J(\^f)$. This link is given by Lemma~\ref{p1sl4}. We translate here many results of the previous sections, expressed in the symbolic language, to the language of the actual map $\tf$. We eventually prove here, see Theorem~\ref{t1_2016_06_16}, that all the measures $\mu_t$, $t\in\De_G^*$, are finite, thus probability measures after normalization. 

\sp In Section~\ref{section:VP}, Variational Principle; The Invariant Measures $\mu_t$ are the Unique Equilibrium States, we prove a full version of the classical Variational Principle for potentials $-t\log|\tf'|$, $t\in\De_G^*$ with respect to the dynamical system $\tf:\Sg_u\times\hat\C\to \Sg_u\times\hat\C$ and we identify measures $\mu_t$ as the only equilibrium states. More precisely, we prove the following. 

\begin{thm}\label{tvp3B}%\label{TVP}%\label{t1vp3}
	If $G$ is a \FNR \ rational semigroup, and $t\in\De_G^*$, then 
	\begin{align*}
		&\sup\lt\{\h_\mu(\^f)-t\int_{J(\^f)}\log|\^f'|d\mu
		:\mu\in M(\^f) \rt\}= \\
		&\qquad\qquad\qquad\qquad\qquad
		=\sup\lt\{\h_\mu(\^f)-t\int_{J(\^f)}\log|\^f'|d\mu
		:\mu\in M_e(\^f) \rt\}
		=\P(t),
	\end{align*}
	and 
	$$
		\h_{\mu_t}(\^f)-t\int_{J(\^f)}\log|\^f'|d\mu_t=\P(t),
	$$
	while
	$$
		\h_\mu(\^f)-t\int_{J(\^f)}\log|\^f'|d\mu<\P(t)
	$$
	for every measure $\mu\in M(\^f)$ different from $\mu_t$. 
\end{thm}

\sp In Section~\ref{section:stochasticlawssymbolspace}, Stochastic Laws on the Symbol Space for the Shift Map Generated by Nice Families, making use of the link with symbol thermodynamic formalism of Section~\ref{sec:thermodynamic formalism}, we embed the symbol space $\cD_\cU^\infty$, along with the shift map acting on it, into an abstract Young tower (see \cite{lsy1} and \cite{lsy2}) as its first return map. We then prove the fundamental stochastic laws such as the Law of Iterated Logarithm, the Central Limit Theorem, and exponential decay of correlations, in such abstract setting. 

\sp In Section~\ref{section:originalstochasticlaws}, Stochastic Laws for the Dynamical System $(\^f:J(\^f)\to J(\^f),\mu_t)$, making use of the previous section, via the natural projection from the abstract Young tower to the Julia set $J(\^f)$, we prove in Theorem~\ref{t12015_03_13} the fundamental stochastic laws for dynamical systems $(\tf,\mu_t)$, $t\in\De_G^*$, such as the Law of Iterated Logarithm, the Central Limit Theorem, and exponential decay of correlations.

\sp Part 3 of our manuscript is devoted to study finer fractal and geometrical properties of the fiber Julia sets $J_\om$ and the global Julia set $J(G)$. Throughout this whole part we assume that $G$ is a \FNR \ rational semigroup satisfying the Nice Open Set Condition; we refer to such rational semigroups as \NOSC-FNR. We would like to mention that we prove at the end of Section~\ref{NOSC} that each \TNR \ rational semigroup satisfying the Nice Open Set Condition is of finite type. 

In Section~\ref{NOSC}, entitled Nice Open Set Condition, we formulate this condition and thoroughly study it at length preparing all the tools based on this condition that we need in further sections. We would like to note that our treatment somewhat differs from that of \cite{sush}. We would like to emphasize that the Nice Open Set Condition and Nice Sets (Families) are totally independent concepts. In particular, the adjective ``Nice" was  independently introduced for both of them many years ago. Although it may be a little bit confusing for some readers, we stick to the historical terminology to respect history and in order not confuse readers even more by inventing yet new names. We think that in our current manuscript this is the first time in the literature that both ``nice" concepts are used simultaneously.

In Section~\ref{section:MA}, entitled Hausdorff Dimension of Invariant Measures $\mu_t$ and Multifractal Analysis of Lyapunov Exponents we provide a full account of Hausdorff dimensions of level sets of Lyapunov exponents.
More precisely, for every point $(\om,z)\in J(\^f)$, we denote 
$$
	\ov{\chi}(\om,z):=\limsup_{n\to\infty}\frac{1}{n}\log|(\^f^n)'(\om,z)| 
	\quad\text{ and }\quad 
	\un{\chi}(\om,z):=\liminf_{n\to\infty}\frac{1}{n}\log|(\^f^n)'(\om,z)|
$$
and call them respectively the upper and lower Lyapunov exponents at the point $(\om,z)$. If $\un\chi(\om,z)=\ov\chi(\om,z)$, we denote the common value by $\chi(\om,z)$ and call it the Lyapunov exponent at $(\om,z)$. Given $\chi\geq 0$, we define 
$$
K(\chi):=\lt\{(\om,z)\in J(\^f): \un\chi(\om,z)=\ov\chi(\om,z)=\chi \rt\},
$$
which is actually the level set of the function $J(\^f)\ni (\om,z)\longmapsto \chi(\om,z)$ corresponding to its value $\chi$. For every $t\in\De_G^*$, let
$$
D_t(\^f):=\lt\{\chi_{\mu_{t,q}}:q\in[0,1]\rt\},
$$
where the measures $\mu_{t,q}$ are introduced in the formula \eqref{520190307} by means of the temperature function $T_t(q)$, and $\chi_{\mu_{t,q}}$ are the corresponding Lyapunov exponents. The main result of Section~\ref{section:MA} is the following.

\begin{thm}[Theorem~\ref{t1jsm15}]%\label{t1jsm15}
If $G$ is a \NOSC-FNR non--exceptional rational semigroup, then for every $t\in\De_G^*\bs\{h_f\}$, the set $D_t(\^f)$ is a non--degenerate interval with endpoints $\chi_{\mu_{h_f}}$ and $\chi_{\mu_t}$, and the function
	$$
		D_t(\^f)\ni \chi\longmapsto\HD(p_2(K(\chi)))\in[0,2]
	$$
	is real--analytic.
\end{thm}

Being a non--exceptional semigroup is a mild requirement meaning that either one of the conditions \eqref{p1jsm14 item a}--\eqref{p1jsm14 item g} from Proposition~\ref{p1jsm14} holds. It is for example satisfied if $J(\^f)$ contains some non--exceptional critical points, see Proposition~\ref{p1jsm22}, in fact if and only if each element of $G$ is a critically finite map with parabolic orbifold; see Theorem~\ref{t2esg2}. As a preparatory result to the above theorem, which is however also interesting on its own,   we prove in this section the following result.

\begin{thm}[Theorem~\ref{t1jsm3}]\label{t1jsm3B}
If $G$ is a \NOSC-FNR rational semigroup and $t\in\De_G^*$, then
	\begin{align*}
		\HD(\mu_t\circ p_2^{-1})=\frac{\h_{\mu_t}(\tilde f)}{\chi_{\mu_t}}=t+\frac{\P(t)}{\chi_{\mu_t}}.
	\end{align*}
\end{thm} 

\sp Throughout the whole manuscript, given an integer $q\ge 1$ and $u>0$ large enough, we mean by $\log^q(u)$ the $q$th iteration of the natural logarithm applied to $u$; for example:
$$
\log^1(u)=\log(u), \  \, \log^2(u)=\log(\log(u)), \  \  
\log^3(u)=\log\(\log(\log(u))\).
$$ 
In Section~\ref{PUZ}, which is entitled Measures $m_t\circ p_2^{-1}$ and $\mu_t\circ p_2^{-1}$ versus Hausdorff Measures $\H_{u^\ka}$ and  $\H_{u^\ka\exp\lt(c\sqrt{\log(1/u)\log^3(1/u)}\rt)}$, following the general scheme of \cite{PUZ} and \cite{U3}  (see also \cite{mugdms}), we establish singularity and absolute continuity relations between measures $\mu_t$ and $m_t$ with respect to generalized measures $\H_{u^\ka\exp\lt(c\sqrt{\log(1/u)\log^3(1/u)}\rt)}$, where the subscript is a gauge function. More precisely, denoting
$$
g_{t,c}(u):=u^{\HD(\mu_t\circ p_2^{-1})}\exp\lt(c\sqrt{\log(1/u)\log^3(1/u)}\rt),
$$
as the ultimate theorem of this section, we prove the following. 
\begin{cor}[Corollary~\ref{c1puz18}]\label{c1puz18B}
	Let $G$ be a \NOSC-FNR rational semigroup. Assume that $t\in\De_G^*$ and $\^\sg_t>0$. Then 
	\begin{enumerate}
\item[(a)] If $0\leq c\leq 2\^\sg_t\chi_{\mu_t}^{-1/2}$, then $\mu_t\circ p_2^{-1}$ and $\H_{g_{t,c}}$ on $J(G)$ are mutually singular. In particular, the measures $\mu_t\circ p_2^{-1}$ and $\H_{t^{\HD(\mu_t\circ p_2^{-1})}}$ are mutually singular.
		\item[(b)] If $c>2\^\sg_t\chi_{\mu_t}^{-1/2}$, then $\mu_t\circ p_2^{-1}$ is absolutely continuous with respect to $\H_{g_{t,c}}$ on $J(G)$. Moreover, $\H_{(\^\ell_c)_t}(E)=+\infty$ whenever $E\sub J(G)$ is a Borel set such that $\mu_t\circ p_2^{-1}(E)>0$.
	\end{enumerate}
\end{cor}

\fr This result is a corollary for Theorem~\ref{t1puz15} involving the upper and lower class functions commonly used in probability theory. 

In %the final section of our paper, i.e. 
Section~\ref{section:fiber-global}, $\HD(J(G))$ versus Hausdorff Dimension of Fiber Julia Sets $J_\om$, $\om\in\Sg_u$, we provide the solution to the long standing open problem in the theory of rational semigroups concerning the size of fiberwise Julia sets $J_\om$ versus the global Julia set $J(G)$. We show that for all \NOSC-FNR rational semigroups of finite type
the Hausdorff dimension of every fiberwise Julia set $J_\om$ is smaller than the Hausdorff dimension of the global Julia set $J(G)$. We also show that if $G$ is expanding then the supremum of Hausdorff dimensions of all fiberwise Julia set $J_\om$, $\om\in\Sg_u$ is smaller than the Hausdorff dimension of the global Julia set $J(G)$. In formal terms, we have the following. 

\begin{thm}%\label{t1fj5}
If $G$ is a \NOSC-FNR rational semigroup, then 
	$$
		\HD(J_\om)<h=\HD(J(G))
	$$
	for every $\om\in\Sg_u$. If in addition, $G$ is expanding, then 
	$$
		\sup\big\{\HD(J_\om):\om\in\Sg_u \big\}<h=\HD(J(G)).
	$$
\end{thm}

In the final section of our paper, i.e. Section~\ref{s:Ex}, Examples, we provide a large class of examples of semi–hyperbolic rational semigroups with the Nice Open Set Condition. Figure~\ref{intro_example1} shows three Julia sets of semi–hyperbolic rational semigroups for which our results apply. 

\begin{comment}
\begin{figure}
	\caption{The Julia set of
		$\langle f_{1}^{2},f_{2}^{2}\rangle$,
		where $f_{1}(z)=z^{2}-1, f_{2}(z)=z^{2}/4$.}
	\includegraphics[width=3cm,width=3cm]{dcjuliatest.eps}
	\label{intro_example1}
\end{figure}

content...
\end{comment}

\begin{figure*}[ht!]
	\subfloat[\label{genworkflow}]{%
		\includegraphics[ width=0.3\textwidth]{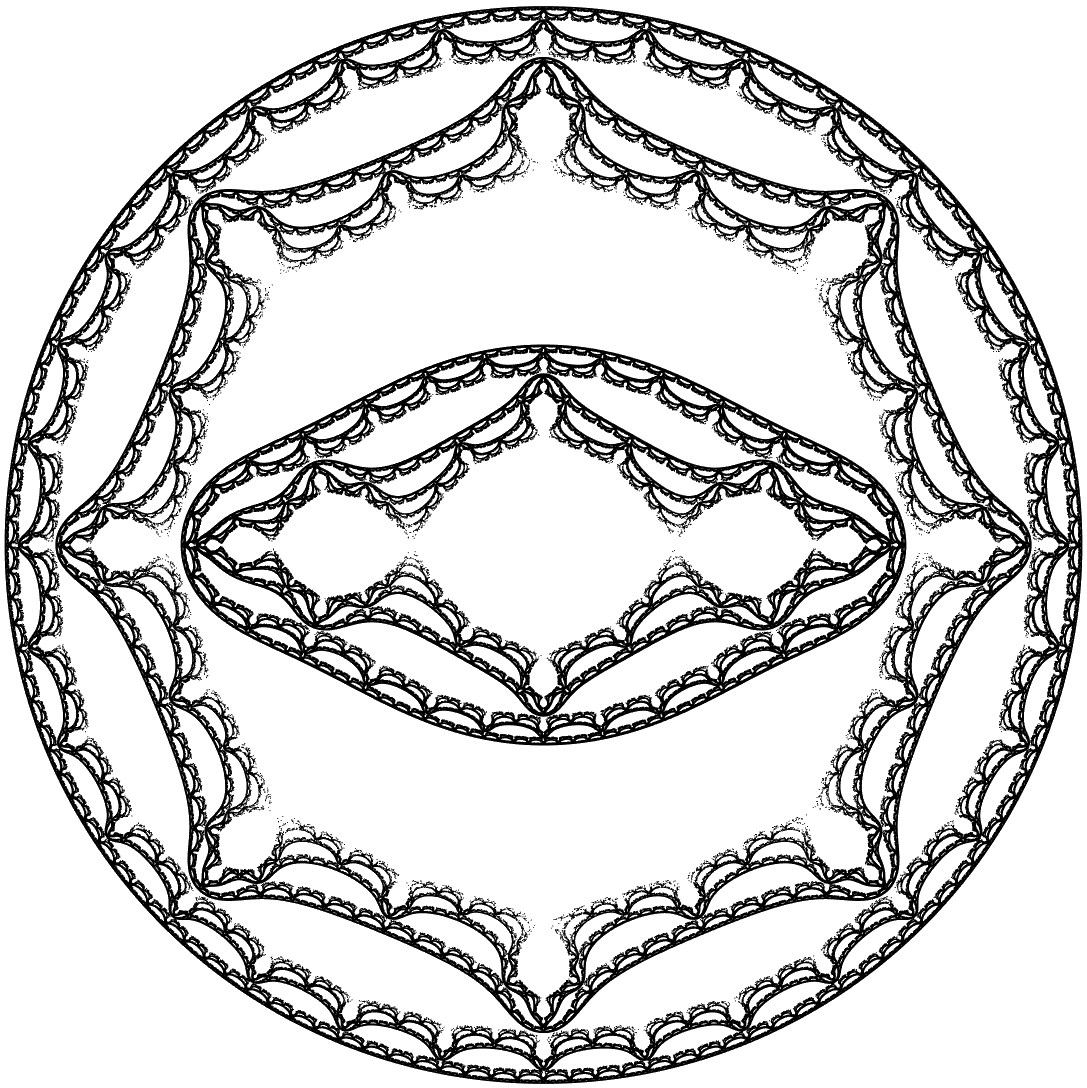}}
	\hspace{\fill}
	\subfloat[\label{pyramidprocess} ]{%
		\includegraphics[ width=0.3\textwidth]{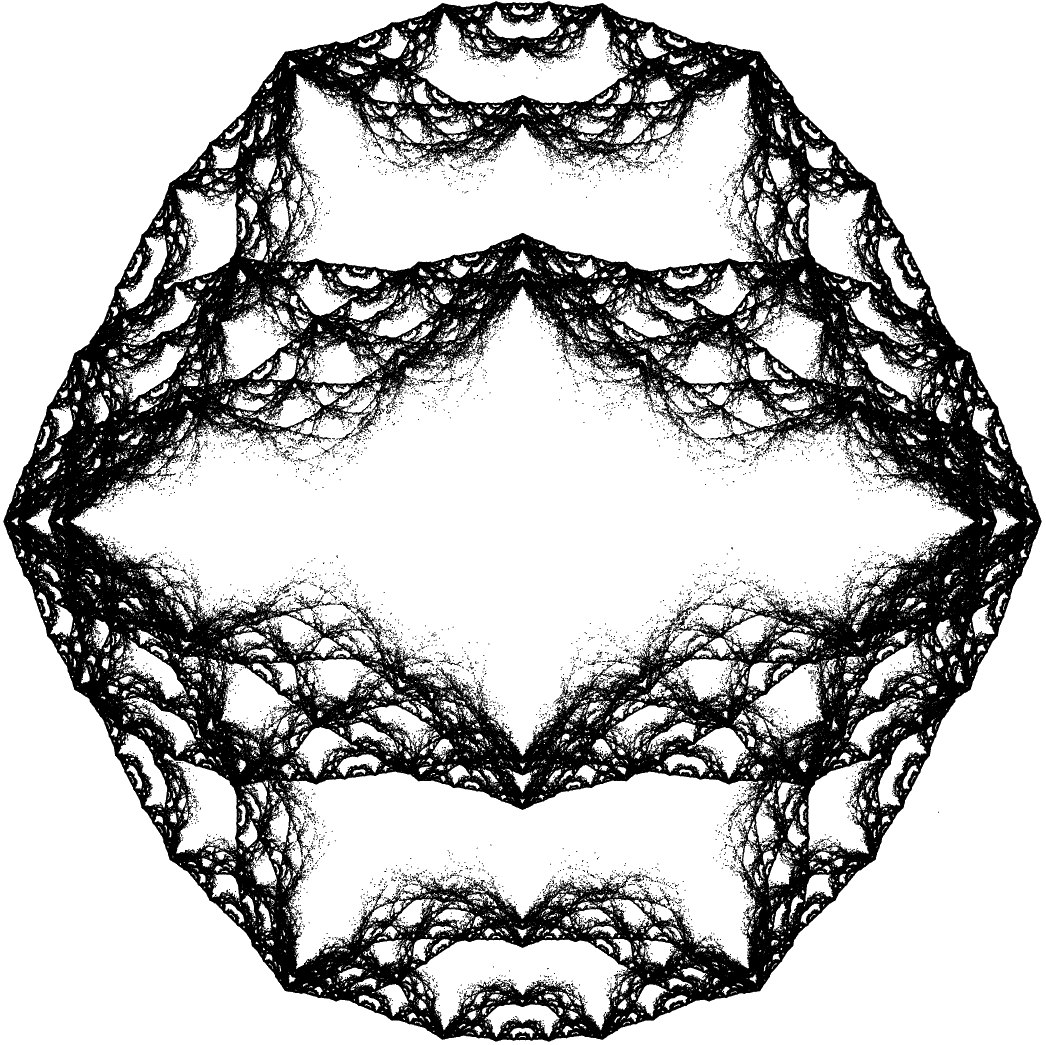}}
	\hspace{\fill}
	\subfloat[\label{mt-simtask}]{%
		\includegraphics[ width=0.3\textwidth]{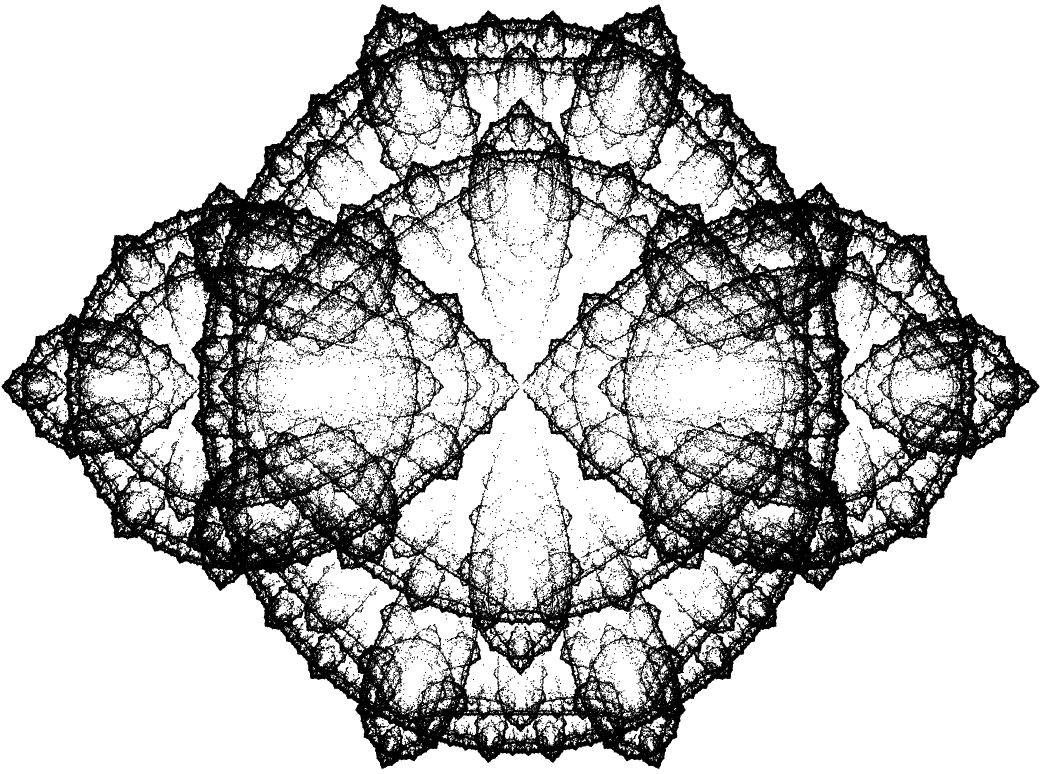}}\\
	\caption{Julia sets $J(G)$ for $G=\langle f_1, f_2 \rangle$ for (a) $f_1=\phi_1^2$, $f_2 =\phi_2^2$ where $\phi_1(z)=z^2-1$ and $\phi_2(z)=z^2/4$ ; (b) $f_1(z)=z^2-1$ and $f_2(z)=z^3/2$; (c) $f_1(z)=z^2-1$ and $f_2(z)=iz^4$.}
	\label{intro_example1}
\end{figure*}

\,

\section{General Preliminaries on Rational Semigroups}\label{s:Pre}
 
\sp Let $u\in \N=\{1,2,3,\ld\}$. Let Rat be the set of all rational maps on the Riemann sphere $\oc$. In this paper, an element of $\Rat^u$ \index{$\Rat^u$} is called a \textbf{$u$--tuple map} \index{$u$--tuple map}. Let $f=(f_{1},\ldots ,f_{u})\in \Rat^u$ be a $u$--tuple map and let 
$$
G=\langle f_{1},\ldots\! ,f_{u}\rangle 
$$ 
be the rational semigroup generated by $\{ f_{1},\ldots ,f_{u}\}.$ We then also say that $G$ is \textbf{generated} by the $u$--tuple map $f=(f_{1},\ldots ,f_{u})$. To be sure, this means that
$$
G=\Big\{f_{\om_n}\circ f_{\om_{n-1}}\circ\cdots \circ f_{\om_2}\circ f_{\om_1}: n\ge 1 \and (\om_1,\om_2,\ld,\om_n)\in \{ 1,\ldots\! ,u\}^n\Big\}.
$$
Let
$$
\Sigma _{u}:=\{ 1,\ldots\! ,u\} ^{\Bbb{N}}
$$\index{symbol space! $\Sg_u$}\index{$\Sg_u$}
be the space of all one--sided sequences of $u$--symbols endowed with the 
product (Tichonov) topology. $\Sg_u$ then becomes a compact metrizable space. There are plenty of metrics on $\Sg_u$ compatible with this topology. For our purposes in this paper we choose only one of them, one which is quite natural. Indeed,
fix $\vartheta\in(0,1)$ and denote by $|\cdot,\cdot|_\vartheta$ the metric on $\Sg_u$ defined by the following formula.
\beq\label{520200311}
|\om,\tau|_\vartheta:=\vartheta^{\min\{n\ge 1:\, \om_n\ne\tau_n\}-1},
\eeq 
with the usual convention that $\vartheta^\infty=0$. We then define the metric $\|\cdot,\cdot\|_\vartheta$ on $\Sg_u\times \C$ as equal to $\max\{|\cdot,\cdot|_\vartheta,|\cdot-\cdot |\}$; more precisely:
$$
\|(\om,z),(\tau,w)\|_\vartheta:=\max\big\{|\om,\tau|_\vartheta,|z-w|\big\}.
$$\index{symbol space!$\vert\vert\cdot,\cdot\vert\vert_\vartheta$} \index{$\vert\vert\cdot,\cdot\vert\vert_\vartheta$}
%It is a little bit awkward but even when dealing with the Riemann sphere $\oc$ we will usually work with Euclidean metric on $\C$ and then $|z-w|$ is the modulus of the complex number $z-w$. If we explicitly use spherical metric, we write $|z-w|_s$
We will frequently omit the subscript $\vartheta$ writing just $\|(\om,z),(\tau,w)\|$ for simplicity. We denote the ball with radius $r>0$ centered at some point $\om\in \Sg_u$ and generated by the metric $|\cdot,\cdot|_\vartheta$ as $B_1(\om,r)$. Likewise, we denote the ball with radius $r>0$ centered at some point $z\in\C$ and generated by the metric $|\cdot-\cdot|$ as $B_2(z,r)$. Then $B((\om,z),r)$ denotes the ball centered at $(\om,z)$ with radius $r$ defined with respect to the metric $\|\cdot,\cdot\|_\vartheta$. Obviously
\index{metric balls! Euclidean $B_2(z,r)$}\index{$B_2(z,r)$}
\index{metric balls! Spherical $B_s(z,r)$}\index{$B_s(z,r)$}
\index{metric balls!  $B_1(\om,r)$}\index{$B_1(\om,r)$}
\index{metric balls!  $B((\om,z),r)$}\index{$B((\om,z),r)$}
$$
B((\om,z),r)=B_1(\om,r)\times B_2(z,r). 
$$
Occasionally, but very rarely, we will use the spherical metric on $\oc$. If $z$ and $w$ are points in $\oc$ then the spherical distance between them is denoted by $|z-w|_s$. Spherical balls are denoted by $B_s(z,r)$, $z\in\oc$, $r>0$. 

\fr We will also need annuli. We recall their standard definition. Given $w\in\C$ and two radii $0<r\le R$, we set
$$
A(w;r,R)
:=B_2(w,R)\sms B_2(w,r)
 =\{z\in\C:r\le|z-w|<R\} .
$$
\index{$A(w;r,R)$}
In general, denote by $\diam(A)$ \index{diam(A)@$\diam(A)$} the diameter of a set $A$ with respect to the metric on the space the set $A$ is contained in. This is just the supremum of distances between points in $A$. Usually in this paper such metric space will be either $\C$, $\Sg_u$, or $\Sg_u\times \C$; most often this will be $\C$ and will the write $\diam_\C(A)$, $\diam_{\Sg_u}(A)$, and $\diam_{\Sg_u\times \C}(A)$ respectively. Also, for $A$, $B$, any two subsets of a metric space $(M,d)$ put
$$
\dist_M(A,B):=\inf\big\{d(a,b): a\in A, b\in B\big\}.
$$\index{dist(A,B)@$\dist(A,B)$}
Let $G$ be a rational semigroup and let $F$ be a subset of $\oc .$ 
We set 
$$
G(F):= \bigcup _{g\in G}g(F)
$$ 
and 
$$
G^{-1}(F):= \bigcup _{g\in G}g^{-1}(F).
$$ 
Moreover, we set 
$$
G^{\ast }:=G\cup \{\Id\},
$$
where $\Id$ denotes here the identity map on $\oc$. We then define analogously the sets $G^*(F)$ and $(G^*)^{-1}(F)$. In fact 
$$
G^*(F)=F\cup G(F)
\  \  \  {\rm and} \  \  \
(G^*)^{-1}(F)=F\cup G^{-1}(F).
$$

Let $\tf:\Sg_{u}\times \oc \longrightarrow \Sg_{u}\times \oc $ \index{skew product map! $\tf:\Sg_{u}\times \oc \longrightarrow \Sg_{u}\times \oc $}\index{$\tf:\Sg_{u}\times \oc \longrightarrow \Sg_{u}\times \oc $}
be the \textbf{skew product map}\index{skew product map} associated with $f=(f_{1},\ldots\! ,f_{u})$. It is given by the formula
$$
\tf(\om,z):=(\sg (\om ),f_{\om_{1}}(z)),
$$
where $(\om,z)\in \Sg _{u}\times \oc,\ \om=(\om_{1},\om_{2},\ldots\! ),$ and 
$\sg :\Sigma_{u}\longrightarrow \Sg _{u}$ \index{symbol space!$\sg :\Sigma_{u}\longrightarrow \Sg _{u}$}\index{$\sg :\Sigma_{u}\longrightarrow \Sg _{u}$} denotes the one--sided shift map, i.e.
$$
\sg\((\om_n)_{n=1}^\infty)\)=(\om_{n+1})_{n=1}^\infty.
$$
We denote by $p_{1}:\Sigma _{u}\times \oc \rightarrow \Sigma _{u}$ 
the projection onto $\Sigma _{u}$ and 
$p_{2}:\Sigma _{u}\times \oc \rightarrow \oc $ the 
projection onto $\oc $. That is, 
$$
p_{1}(\om ,z)=\om \  \  \text{ and } \  \ p_{2}(\om ,z)=z.
$$  \index{projection map!$p_{1}:\Sigma _{u}\times \oc \rightarrow \Sigma _{u}$}\index{projection map!$p_{2}:\Sigma _{u}\times \oc \rightarrow \oc $}\index{$p_{1}:\Sigma _{u}\times \oc \rightarrow \Sigma _{u}$}\index{$p_{2}:\Sigma _{u}\times \oc \rightarrow \oc $}
Under the canonical identification $p_{1}^{-1}\{ \om \} \cong \oc $, 
each fiber $p_{1}^{-1}\{ \om \} $ is a Riemann surface which 
is isomorphic to $\oc$ .

For $n\geq 0$ let
$$
\Sigma_{u}^n:=\{ 1,\ldots\! ,u\} ^{n}
\qquad \text{ and }\qquad
\Sg_u^*:= \bigcup _{n=0}^\infty \Sg_u^n%\{ 1,\ldots ,u\} ^{n} 
$$\index{symbol space! $\Sg_u^n$}\index{$\Sg_u^n$}\index{symbol space! $\Sg_u^*$}\index{$\Sg_u^*$}
respectively be the family of all words over the alphabet $\{1,2,\ld,u\}$ of length $n$ and the family of all finite words with the convention that $\{1,\ldots ,u\}^{0}$ is the singleton consisting of the empty word denoted in the sequel by $\es$. For every $\tau \in \Sg _{u}^{*}$,  
%$$
%{\hat\tau}=\tau|_{[\tau]-1} \  \text{ and } \  \tau_*=\tau_{|\tau|}.
%$$
we denote by $|\tau | $  \index{symbol space! $\vert\tau\vert$ }\index{$\vert\tau\vert$ } the only integer $n\ge 0$ such that $\tau \in
\Sg_u^n$. %\{ 1,\ldots ,u\} ^{n}.$  
For every $\tau \in \Sg _{u}$ we set $|\tau |=\infty .$ 
In addition, for every $\tau =(\tau _{1},\tau _{2},\ldots )\in \Sg
_{u}^{\ast }\cup \Sg _{u}$ and  
$n\in \N $ with $n\leq |\tau |$, 
we set 
$$
\tau |_{n}:= (\tau _{1},\tau _{2},\ldots ,\tau _{n}) \in \Sg_{u}^{\ast}.
$$ \index{symbol space! $\tau\vert_n$}\index{$\tau\vert_n$}
Furthermore, for every set $\Ga\sbt \Sg_{u}^{\ast }\cup \Sg _{u}$, we define
$$
\Ga|_n:=\{\om|_n:\om\in \Ga\}.
$$\index{symbol space! $\Ga\vert_n$}\index{$\Ga\vert_n$}
For every $\tau \in \Sg _{u}^{\ast }$, we denote
$$
\hat{\tau }=\tau |_{|\tau |-1},  
\  \  \  
%{\rm and}  \  \  \  
\tau_*:= \tau _{|\tau |},
$$\index{symbol space! $\hat\tau$}\index{symbol space! $\tau_*$}\index{$\hat\tau$}\index{ $\tau_*$}
and 
\beq\label{referee2}
[\tau ] := \{ \om \in \Sg _{u}: \om |_{|\tau |}=\tau \}. 
\eeq\index{symbol space!$[\tau]$}\index{$[\tau]$}
We call $[\tau]$ the cylinder generated by $\tau$.
Furthermore, for every $\om \in \Sg _{u}^{\ast }\cup \Sg _{u}$ and
all $a,b\in \N $ with $a<b\leq |\om |$, we set 
$$
\om _{a}^{b}:=(\om _{a},\ldots ,\om _{b})\in \Sg _{u}^{\ast }.
$$ \index{symbol space!$\om_a^b$} \index{$\om_a^b$} 
For any two words $\omega , \tau \in \Sigma _{u}^{\ast }$, we say that $\omega $ and $\tau $ are \textbf{comparable}\index{comparable words} if either 

\begin{itemize}

\sp\item[(1)] $|\tau |\leq |\omega |$ and $\om|_{|\tau|}=\tau$; equivalently $\omega \in [\tau ]$, 

\sp or

\sp\item[(2)] $|\omega |\leq |\tau|$ and $\tau|_{|\om|}=\om$; equivalently $\tau \in [\omega ]$.
\end{itemize}

\sp\fr We say that $\omega ,\tau $ are \textbf{incomparable} if they are not
comparable. By $\tau\om\in \Sg _{u}^{\ast }$ we denote the concatenation of the words $\tau$ and $\om$. We may now rewrite formula \eqref{520200311} in the form.
\beq\label{620200311}
|\om,\tau|_\vartheta=\vartheta^{|\om\wedge\tau|},
\eeq
where $\om\wedge\tau$\index{symbol space!$\om\wedge\tau$}\index{$\om\wedge\tau$} is the longest initial common block of both $\om$ and $\tau$. For each $\om =(\om _{1},\ldots ,\om _{n})\in \Sigma _{u}^{\ast }$, let 
$$
f_{\om }:= f_{\om _{n}}\circ \cdots \circ f_{\om _{1}}:\oc\lra\oc.
$$ \index{$f_\om$}
For each $n\in \N $ and $(\om ,z)\in 
\Sigma _{u}\times \oc $, we set 
$$
(\tf^{n})'(\om ,z):= f_\om'(z).
$$  \index{$(\tf^{n})'(\om ,z)$}
For each $\om\in \Sg_u$ we define 
$$
J_{\om }:=\Big\{ z\in \oc:\{ f_\om\}_{\om\in\Sg_u^*} \, \mbox{ is 
not normal in any neighborhood of } z\Big\}, 
$$\index{Julia set!$J_\om$}\index{$J_\om$}
and we then set 
$$
J(\tf):= \overline{\bu_{w\in \Sigma _{u}}\{ \om \} \times J_{\om } },
$$
\index{Julia set!$J(\tf)$}\index{$J(\tf)$}
where the closure is taken in the product space 
$\Sigma _{u}\times \oc .$ By its very definition, the set
$J(\tf)$ is compact. Naturally, $F(\tf)$ denotes the complement of $J(\tf)$ in $\Sigma _{u}\times \oc$.\index{Fatou set! $F(\tf)$}\index{$F(\tf)$}
Furthermore, setting
$$
E(G):= \lt\{ z\in \oc: \# \bigcup _{g\in G}g^{-1}(\{ z\} )<\infty \rt\},
$$ 
by Proposition~3.2 in \cite{hiroki1}, we have the following. 
 
\bprop\label{p120190313}
If $f=(f_{1},\ldots ,f_{u})\in \Rat^u$ is a $u$--tuple map and $G=\langle f_{1},\ldots\! ,f_{u}\rangle$ is the rational semigroup generated by $\{ f_{1},\ldots ,f_{u}\}$, then
 \sp\begin{enumerate} 
\item[\mylabel{a}{p120190313 item a}] $J(\tf)$ is completely invariant under $\tf$, meaning that $\tf^{-1}(J(\^f))=J(\^f)=\tf(J(\^f))$.
 
\sp\item[\mylabel{b}{p120190313 item b}] The map $\tf:J(\^f)\lra J(\^f)$ is open, meaning that it maps open sets onto open sets.
 
\sp\item[\mylabel{c}{p120190313 item c}] If
$$
\# J(G)\geq 3
\  \  \  {\rm and}  \  \  \
E(G)\subset F(G),
$$ 
then the skew product map $\tf:J(\^f)\lra J(\^f)$ is \textbf{topologically exact}\index{topologically exact} meaning that for every non--empty open set $U$ in $J(\^f)$ there exists an integer $n\ge 0$ such that 
$$
 \tf^n(U)=J(\^f).
$$ 
In particular, the map $\tf:J(\^f)\lra J(\^f)$ is topologically transitive. 
 
\sp\item[\mylabel{d}{p120190313 item d}] $J(\tf)$ is equal to the closure of the set of repelling periodic points of $\tf$ provided that $\sharp J(G)\geq 3$, where a periodic point $(\om ,z)$ of $\tf$  with period $n$ is said to be repelling if the absolute value of the multiplier $|(\tf^{n})'(\om ,z)|$ of $f_\om$ at $z$ is strictly larger than $1$.
  
%$|(\tf^{n})'(\om ,z)|>1.$      
\sp\item[\mylabel{e}{p120190313 item e}] Furthermore, 
$$
p_{2}(J(\tf))=J(G).
$$ 
\end{enumerate}
\eprop 
%We set 
%$$
%e_{j}:=\deg (f_{j}) \  \text{ and }  \  d:=\max_{j=1,\ldots\!
%  ,s}(2e_{j}-2).
%$$
%Throughout the paper, we assume the following.
%\begin{itemize}
%\item[(E1)] There exists an element 
%$g\in G$ with $\deg (g)\geq 2.$ Furthermore,\ for the semigroup $H:=\{
%h^{-1}\mid h\in \mbox{Aut} \oc \cap G\}$,\ we have 
%$J(G)\subset F(H).$ (If $H$ is empty, we put $F(H):=\oc .$)
%\item[(E2)]$\# $(Crit $(f)\cap J(\tf))<\infty $
%\item[(E3)]There is no super attracting cycle of $f$ in $J(\tf).$
%That is, if $(\om ,z)\in J(\tf)$ is a periodic point with period $n$,\ 
%then $(f^{n})'(\om ,z)\neq 0.$  
%\end{itemize}
%
%\ 
%
%\ni Any finitely generated rational semigroup 
%$G=\langle f_{1},\ldots\! ,f_{s}\rangle $ satisfying all the conditions 
%(E1)-(E3) will be called an E--semigroup of rational maps.
%
%\

\sp We now introduce further notation. 
A pair $(c,j)\in\oc\times\{1,2,\ld,u\}$ is called \textbf{critical}\index{critical pair} for the $u$--tuple map $f$ if the map $f_j$ is not $1$--to--$1$ on any open neighborhood of $c$ in $\oc$. If both $c$ and $f_j(c)$ belong to $\C$, and this will be almost always the case considered in our manuscript, this means that
$$
f_j'(c)=0.
$$
Therefore, abusing slightly notation, we will always indicate a critical point by writing that the derivative at this point is equal to $0$. The set of all critical pairs of $f$ will be denoted by $\CP(f)$. \index{critical pair! $\CP(f)$}\index{$\CP(f)$}
Let $\Crit(f)$ be the union 
$$
\Crit(f):=\bigcup _{j=1}^{u}\Crit (f_{j}),
$$ 
\index{critical set! $\Crit(f)$}\index{$\Crit(f)$}
where $\Crit (f_{j})$\index{critical set!$\Crit (f_{j})$}\index{$\Crit (f_{j})$} denotes the set of critical points of the map $f_j$. 
For every $c\in\Crit(f)$
put 
$$
c_+=\{f_j(c):(c,j)\in\CP(f)\}.
$$
The set $c_+$ is called \textbf{the set of critical values of $c$}. For any subset $A$ of $\Crit(f)$ put
$$
A_+=\{c_+:c\in A\}.
$$\index{critical set! $\Crit(f)_+$}\index{$\Crit(f)_+$}
For each $(c,j)\in \CP(f)$ let $q(c,j)$\index{$q(c,j)$} be the local order of $f_{j}$ at $c$.
We define the \textbf{direct postcritical set of $G$} to be
$$
G^*(\Crit(f)_+).
$$
We further consider the set
$$
\ov{G^*(\Crit(f)_+)}.
$$
Note that this set does not depend on the choice of generators, and it is in fact equal to the closure of the set
$$
\big\{g(c):g\in G,\,g'(c)=0\big\}.
$$
We denote it by $\PCV(G)$ and call it the \textbf{postcritical set of $G$} \footnote{Some authors call the direct postcritical set just postcritical and give no short name for $\PCV(G)$ just calling it the closure of the postcritical set. Our choice of notation and terminology follows a well established tradition and makes exposition simpler.}, i.e. 
$$
	\PCV(G):= \ov{G^*(\Crit(f)_+)} = \ov{\big\{g(c):g\in G,\,g'(c)=0\big\}}.
$$
\index{postcritical set! $\PCV(G)$}\index{$\PCV(G)$}
We also set
$$
\Crit_*(f):= \bu_{i=1}^u\big\{c\in\Crit(f_i): f_i(c)\in J(G)\big\}.
$$
\index{critical set! $\Crit_*(f)$}\index{$\Crit_*(f)$}
Of course
$$
\Crit_*(f)\sbt J(G)\cap\Crit(G),
$$
but equality need not to hold. Let 
$$
\Crit(\tf):=\big\{\xi\in\Sg_u\times\oc: \tf'(\xi)=0\big\}.
$$
\index{critical set! $\Crit(\tf)$}\index{$\Crit(\tf)$}
Of course
$$
\Crit(\tf)=\big\{(\om,c)\in\Sg_u\times\oc: f_{\om_1}'(c)=0\big\}.
$$
Put
$$
\Crit_*(\^f):=J(\^f)\cap\Crit(\^f).
$$
Note that
\beq\label{2nsii5}
\Crit_*(f)=p_2(\Crit_*(\^f)).
\eeq
We call 
$$
\bu_{n=1}^\infty \tf^n(\Crit(\tf))
$$
the \textbf{direct postcritical set of $\tf$}, and its closure, i.e. the set 
$$
\mbox{{\rm PCV}}(\tf):=\ov{\bu_{n=1}^\infty \tf^n(\Crit(\tf))},
$$
the \textbf{postcritical set of $\tf$}. \index{postcritical set! $\PCV(\tf)$}\index{$\PCV(\tf)$} We also define
$$
\PCV_*(\tf)
:=J(\^f)\cap\PCV(\tf)
 =\ov{\bu_{n=1}^\infty \tf^n(\Crit_*(\tf))}.
$$\index{postcritical set! $\PCV_*(\tf)$}\index{$\PCV_*(\tf)$}
Set 
$$
\Sing(\tf):=\bu_{n\ge 0}\tf^{-n}(\Crit(\tf))
$$
\index{$\Sing(\tf)$}
and
$$
\Sing(f):=\bu_{g\in G^{\ast }}g^{-1}(\Crit(f)).
$$
\index{$\Sing(f)$}
%{\bf(Jason: Please, type in here the pages Mc$\_$2 and Mc$\_$3 from Misc)}
Now we define and deal with holomorphic inverse branches of iterates of the skew product map $\^f:\Sg_u\times\hat\C \lra \Sg_u\times\hat\C$. We start with the following. 
\begin{lem}\label{l1mc2}
	For every integer $n\geq 1$, we have that 
	$$
		\^f^n(\Crit(\^f^n))=\Sg_u\times \lt(\bigcup_{\tau\in\Sg_u^n}f_\tau\(\Crit(f_\tau)\)\rt).
	$$
\end{lem}
\begin{proof}
	Assume that $(\om,z)\in\^f^n(\Crit(\^f^n))$. This means that there exists $(\ga,x)\in\Crit(\^f^n)$ such that 
	$$
		(\om,z)=\^f^n(\ga,x).
	$$
	Then, 
	$$
		z=f_{\ga\rvert_n}(x),
	$$
	and 
	$$
		f_{\ga\rvert_n}'(x)=(\^f^n)'(\ga,x)=0.
	$$
	So, $x\in\Crit(f_{\ga\rvert_n})$ and 
	$$
		z\in f_{\ga\rvert_n}(\Crit(f_{\ga\rvert_n}))\sub \bigcup_{\tau\in\Sg_u^n}f_\tau(\Crit(f_\tau)).
	$$
	Thus, 
	$$
		(\om,z)\in \Sg_u\times\lt(\bigcup_{\tau\in\Sg_u^n}f_\tau(\Crit(f_\tau))\rt),
	$$
	and the inclusion 
	$$
		\^f^n(\Crit(\^f^n))\sub\Sg_u\times \lt(\bigcup_{\tau\in\Sg_u^n}f_\tau(\Crit(f_\tau))\rt)
	$$
	is proven.
	
	So, assume now that 
	$$
		(\om,z)\in \Sg_u\times \lt(\bigcup_{\tau\in\Sg_u^n}f_\tau(\Crit(f_\tau))\rt).
	$$
	So, there exists $\tau\in\Sg_u^n$ such that $z\in f_\tau(\Crit(f_\tau))$. Hence, there exists $x\in\Crit(f_\tau)$ such that $z=f_\tau(x)$. So, $f_\tau'(x)=0$. Then, 
	$$
		\^f^n(\tau\om,x)=(\om,f_\tau(x))=(\om,z),
	$$
	and 
	$$
		(\^f^n)'(\tau\om,x)=f_\tau'(x)=0.
	$$
	So, $(\tau\om,x)\in\Crit(\^f^n)$, and, in consequence, $(\om,z)\in\^f^n(\Crit(\^f^n))$. We have thus proved the inclusion
	$$
		\Sg_u\times\lt(\bigcup_{\tau\in\Sg_u^n}f_\tau(\Crit(f_\tau))\rt)\sub\^f^n(\Crit(\^f^n)).
	$$
	Lemma~\ref{l1mc2} is thus proven.
\end{proof}
Denote by
$$
\DPCV(\^f):=\bigcup_{n=1}^\infty \^f^n(\Crit(\^f))=\Sg_u\times \lt(\bigcup_{\tau\in\Sg_u^n}f_\tau\(\Crit(f_\tau)\)\rt)
$$
the direct postcritical set of $\^f$ \index{direct postcritical set! $\DPCV(\^f)$}\index{$\DPCV(\^f)$} and by 
$$
	\DPCV(G):=G^*(\Crit(f)_+)=\bigcup_{\tau\in\Sg_u^*}f_\tau(\Crit(f_\tau))
$$
the direct postcritical set of $G$. \index{direct postcritical set! $\DPCV(G)$}\index{$\DPCV(G)$} Then, as an immediate consequence of Lemma~\ref{l1mc2}, we get the following. 
\begin{lem}\label{l1mc3}
	$\DPCV(\^f)=\Sg_u\times \DPCV(G)$.
\end{lem}
By applying the closures, we  thus obtain:
\begin{lem}\label{l2mc3}
	$\PCV(\^f)=\Sg_u\times\PCV(G)$.
\end{lem}

Fix $\tau \in \Sg _{u}^{\ast }$, $x\in \oc$, and $n\in \N $. Suppose that $x$ is not a critical point of $f_\tau$. If both $x$ and $f_\tau(x)$ belong to $\C$, this means that
$$
f_{\tau }'(x)\ne 0,
$$
or equivalently:
$$
(\tf^{|\tau|})'(\om,x)\ne 0
$$
for every $\om\in[\tau]$. Then if $V\sbt\oc$ is a sufficiently small (in the sense of diameter), non--empty, open, connected, simply connected set containing $f_{\tau}(x)$, more precisely if the connected component of $f_{\tau}^{-1}(V)$ does not contain critical points of $f_{\tau}$, then there exists a unique holomorphic inverse branch
$$
f_{\tau ,x}^{-1}:V\lra \oc
$$
mapping $f_{\tau}(x)$ to $x$.  Furthermore, we denote by 
\beq\label{320180620}
\tf_{\tau, x}^{-|\tau |}:\Sg_u\times V\lra [\tau]\times\oc 
\eeq
the map defined by the formula
\beq\label{420180620}
\tf _{\tau, x}^{-|\tau |}(\om ,y):=\(\tau \om, f_{\tau ,x}^{-1}(y)\),
\eeq\index{holomorphic inverse branch! $\tf_{\tau, x}^{-|\tau |}:\Sg_u\times V\lra [\tau]\times\oc $}\index{$\tf_{\tau, x}^{-|\tau |}:\Sg_u\times V\lra [\tau]\times\oc $} and we call it the \textbf{holomorphic inverse branch} \index{holomorphic inverse branch} of $\tf ^{|\tau|}$ defined on $\Sg_u\times V$ which maps $\Sg_u\times\{f_{\tau}(x)\}$ onto $[\tau]\times\{x\}$.

Now keep $\tau\in\Sg_u^*$, and  suppose that $V\sbt\oc$ is a non--empty, open, connected, simply connected set such that 
\beq\label{120180620}
V\cap f_\tau(\Crit(f^\tau))=\es.
\eeq
Equivalently
\beq\label{220180620}
(\Sg_u\times V)\cap \tf^{|\tau|}\([\tau]\times\Crit(f^\tau)\)=\es.
\eeq
Then, by the above, all holomorphic inverse branches of $f_{\tau}$ are well defined on $V$. More precisely, for every $x\in f_\tau^{-1}(V)$ there exists a unique holomorphic inverse branch
$$
f_{\tau ,x}^{-1}:V\lra \oc
$$
mapping $f_{\tau}(x)$ to $x$. The map $\tf_{\tau, x}^{-|\tau |}:\Sg_u\times V\lra [\tau]\times\oc$ has then the same meaning as in \eqref{320180620} and \eqref{420180620}.

Now fix an integer $n\ge 0$ and suppose that $V\sbt\oc$ is a non--empty, open, connected, simply connected set such that 
$$
V\cap \bu_{\tau\in\Sg_u^n} f_\tau\(\Crit(f_\tau)\)=\es.
$$
By Lemma~\ref{l1mc2}, this means that
$$
(\Sg_u\times V)\cap \tf^n(\Crit(\tf^n))=\es.
$$
Then, by the above, for every $\tau\in\Sg_u^n$, all holomorphic inverse branches of $f_{\tau}$ are well defined on $V$. More precisely, for every $x\in f_\tau^{-1}(V)$ there exists a unique holomorphic inverse branch
$$
f_{\tau ,x}^{-1}:V\lra \oc
$$
mapping $f_{\tau}(x)$ to $x$. As in the above, the map 
$$
\tf_{\tau, x}^{-n}:\Sg_u\times V\lra [\tau]\times\oc
$$ 
has then the same meaning as in \eqref{320180620} and \eqref{420180620}. We denote by $\HIB_n(V)$ \index{holomorphic inverse branch! $\HIB_n(V)$}\index{$\HIB_n(V)$} the family of all such inverse branches $\tf_{\tau, x}^{-|\tau |}:\Sg_u\times V\lra [\tau]\times\oc$. Fixing any point $\xi\in V$, this family is equal to
$$
\big\{\tf_{\tau, x}^{-|\tau |}:\Sg_u\times V\lra [\tau]\times\oc\big|\, 
\tau\in\Sg_u^n,\, x\in f_\tau^{-1}(\xi)\big\}. 
$$
For every $\rho\in\HIB_n(V)$ for the sake of naturality we write $\rho$ as 
\beq\label{720200303}
\^f_\rho^{-n}:\Sg_u\times V\lra[\^\rho]\times\oc,
\eeq
being given by the formula
\beq\label{820200303}
\^f_\rho^{-n}(\om,z)=(\^\rho\om,f_\rho^{-1}(z))\in [\^\rho]\times\oc,
\eeq
with a unique $\^\rho\in \Sg_u^n$ and a unique $f_\rho^{-1}:V\to\oc$ being a holomorphic inverse branch of $f_{\^\rho}$ defined on $V$. We set
\beq\label{920200303}
\|\rho\|:=|\^\rho|=n.
\eeq
Lastly, let $V\sbt\oc$ be a non--empty, open, connected, simply connected set such that 
$$
V\cap \PCV(G)=\es.
$$
By Lemma~\ref{l1mc3}, this means that
$$
(\Sg_u\times V)\cap \PCV(\tf)=\es.
$$
Then, by the above, all the families $\HIB_n(V)$ are well defined and we denote
$$
\HIB(V):=\bu_{n=0}^\infty\HIB_n(V).
$$\index{holomorphic inverse branch! $\HIB(V)$}\index{$\HIB(V)$}
Note that if $W$ is any non--empty, open, connected, simply connected subset of $V$, then
$$
\HIB(W)=\{\rho|_W:\rho\in\HIB(V)\}.
$$

We end these preliminaries by mentioning the classical and celebrated distortion theorems due to Koebe. We do so for the sake of completeness and since these will be an indispensable and truly powerful tool for us throughout the sequel.

\bthm [Koebe's ${\frac 1 4}$--Theorem] \label{kdt1/4}\index{Koebe's ${\frac 1 4}$--Theorem}
If $z\in {\mathbb C}$, $r>0$ and $H:B_2(z,r)\to
{\mathbb C}$  is an arbitrary univalent analytic function, then
$$
H(B_2(z,r))\spt B_2(H(z),4^{-1}|H'(z)|r).
$$
\ethm

\bthm [Koebe's Distortion Theorem, I]\label{kdt1E}\index{Koebe's Distortion Theorem! I} There exists a function  $k:[0,1)\to [1,\infty)$ such that for
any $z\in {\mathbb C}$, $ r >0$, $ t\in [0,1)$ and any univalent
analytic function $H:B_2(z,r)\to {\mathbb C}$ we have that
$$
\sup\big\{|H'(w)|:w\in B_2(z,tr)\big\} \le k(t) \inf\big\{|H'(w)|:w\in
B_2(z,tr)\big\}.
$$
We put $K=k(1/2)$.
\ethm

%\bthm\label{kdt1S} (Koebe's Distortion Theorem, I (spherical
%version))
%Given a number $s>0$ there exists a function \, $k_s:[0,1)\to
%[1,\infty)$ such \, that for\,  any $z\in \ov{\mathbb C}$, $ r >0$,
%$t\in [0,1)$ and any univalent analytic function $H:B_n(z,r)\to
%\ov{\mathbb C}$ such that the complement $\ov{\mathbb C}\sms
%H(B_n(z,r))$ contains a spherical ball of radius $s$ we have 
%$$
%\sup\{|H^*(w)|:w\in B_n(z,tr)\} \le k_s(t) \inf\{|H^*(w)|:w\in
%B_n(z,tr)\}.
%$$
%$B_n$  stands  in here for either Euclidean or spherical ball (if $z
%\neq \infty$).
%\ethm

The following is a straightforward consequence of these two
distortion theorems.

\blem\lab{lncp12.9.} Suppose that $D\sbt {\mathbb C}$ is an open
set, $z\in D$ and $H:D\to {\mathbb C}$ is an analytic map which has
an analytic inverse $H_z^{-1}$ defined on $B_2(H(z),2R)$ for some
$R>0$. Then for every $0\le r\le R$,
$$
B_2(z,K^{-1}r|H'(z)|^{-1})\sbt H_z^{-1}(B_2(H(z),r))\sbt
B_2(z,Kr|H'(z)|^{-1}).
$$
\elem

%\blem\lab{lncp12.9s.} Suppose that $D\sbt \ov{\mathbb C}$ is an open
%set, $z\in D$ and $H:D\to \ov{\mathbb C}$ is an analytic map which
%has an analytic inverse $H_z^{-1}$ defined on $B_s(H(z),2R)$ for
%some $R>0$ avoiding a spherical ball of some radius $s$. Then for
%every $0\le r\le R$
%$$
%B_s(z,k_s^{-1}(1/2)r|H^*(z)|^{-1}) \sbt H_z^{-1}(B_s(H(z),r)) \sbt
%B_s(z,k_s(1/2)r|H^*(z)|^{-1}).
%$$
%\elem

  We also use the following more geometric version of Koebe's
Distortion Theorem involving moduli of annuli.

\bthm[Koebe's Distortion Theorem, II]\label{kdt2E}\index{Koebe's Distortion Theorem! II} There exists a function $w:(0,+\infty)\to [1,\infty)$ such that
for any two open topological disks $Q_1\sbt Q_2\subset \C $ with
$\text{{\em Mod}}(Q_2\sms Q_1)\ge t$ and any univalent analytic function
$H:Q_2\to {\mathbb C}$ we have 
$$
\sup\{|H'(\xi)|:\xi\in Q_1\} \le w(t) \inf\{|H'(\xi)|:\xi\in Q_1\}.
$$
\ethm

Finally, by 
$$
A\lek B
$$\index{$A\lek B$}
we mean that there exists a constant $C\in(0,+\infty)$ independent of appropriate, always clearly indicated in the context, variables $A$ and/or $B$ such that
$$
A\le CB.
$$
Analogously, $A\gek B$.\index{$A\gek B$} Also,
$$
A\comp B
$$\index{$A\comp B$}
if $A\lek B$ and $B\lek A$.

\part{Ergodic Theory and Dynamics of Finitely Generated *Semi--Hyperbolic Rational Semigroups}

\sp\section{Basic Properties of \\ Semi--hyperbolic and *Semi--Hyperbolic \\ Rational Semigroups}

\sp\fr In this section we define semi--hyperbolic and *semi--hyperbolic rational semigroups, and we collect their dynamical properties, which will be needed in the sequel. 
\bdfn\label{d1h3}  \index{semi--hyperbolic}
A rational semigroup $G$ is called \textbf{semi--hyperbolic} if and only if 
there exist an $N\in \N $ and a $\delta >0$ such that 
for each $x\in J(G)$ and each $g\in G$, 
$$\deg \(g:V\rightarrow B_{s}(x,\delta )\)\leq N$$
for each connected component $V$ of $g^{-1}(B_{s}(x,\delta )).$ 
%In general, points enjoying this property are referred to as non--recurrent. 
\edfn

For the record we recall that a rational semigroup $G$ is called \textbf{hyperbolic} if and only if 
$$
\PCV(G)\sbt F(G).
%J(G)\sbt F(G)
$$\index{hyperbolic}
It is straightforward to see that each hyperbolic rational semigroup is semi--hyperbolic. 

The notion hyperbolicity is closely related to the concept of expanding rational semigroups and under mild technical assumptions expanding and hyperbolic semi--hyperbolic rational semigroups coincide. This class will be explicitly studied only in Section~\ref{section:fiber-global},
 {$\HD(J(G))$ versus Hausdorff Dimension of Fiber Julia Sets $J_\om$, $\om\in\Sg_u$}.
 
\sp From now on throughout the rest of the paper, we always assume the following.

\sp\noindent {\bf Fundamental Assumption:} \index{Fundamental Assumption} If $G$ is a rational semigroup, then the following three conditions are assumed to hold.

\sp\begin{itemize}
\item There exists an element $g$ of $G$ such that $\deg (g)\geq 2.$

\sp\item Each element of $g\in\rm{Aut}(\oc )\cap G$ is \textbf{loxodromic}\index{loxodromic}, meaning that the M\"obius transformation $g$ has two fixed points for which the modulus of the multiplier of the transformation (characteristic constant) is not equal to one. 

\sp\item $F(G)\ne\es$. By a M\"obius change of coordinates we therefore may and we do assume that 
$$
\infty\in F(G).
$$
\end{itemize} 

\sp Then,
$$
J(G)\sbt \C
\and
J(\tilde f)\sbt \Sg_u\times \C.
$$

The Fundamental Assumption is a mild standard hypothesis commonly assumed in the theory of rational semigroups. It is actually indispensable to get started with this theory. The interested reader is advised to consult the papers of the second name author of this manuscript cited in our references to learn more about the nature and status of this assumption. The fact that we may and  we do assume that $\infty\in F(G)$ allows us to work almost exclusively on the complex plane $\C$ rather than on $\oc$ and to avoid spherical metric and balls entirely. This is particularly convenient when we apply, which we do frequently, the various versions of Koebe's distortion theorems. If dealing with spherical distances and derivatives, these theorems take on somewhat cumbersome form with the need of checking ``annoying" hypotheses.

\sp
\bdfn\label{d220220816}\index{semi--hyperbolic!*semi--hyperbolic}
Any semi--hyperbolic rational semigroup satisfying the Fundamental Assumption is called \textbf{*semi--hyperbolic}.
\edfn

It is immediate from these definitions that both concepts of semi--hyperbolicity and *semi--hyperbolicity are independent of any set generators of the  rational semigroup considered.

\sp\fr The crucial tool indispensable in developing the theory of *semi--hyperbolic rational semigroups is given by the the
following semigroup version of Ma\~n\'e's Theorem proved in \cite{hiroki2}.

\bthm\label{t2.9h3} 
If $G=\langle f_{1},\ldots ,f_{u}\rangle$ is a finitely generated 
rational semigroup satisfying the Fundamental Assumption, then $G$ is *semi--hyperbolic if and only if all of the following conditions are satisfied.  

\begin{itemize}
\item[\mylabel{a}{t2.9h3 item a}] For each $z\in J(G)$ there exists a neighborhood $U$ of $z$ in $\C$ 
such that for any sequence $\{ g_{n}\} _{n=1}^{\infty }$ in $G$, 
any domain $V$ in $\oc $, and any point 
$\zeta \in U$, the sequence $\{ g_{n}\} _{n=1}^{\infty }$ does not converge 
to $\zeta $ locally uniformly on $V.$ 
\item[\mylabel{b}{t2.9h3 item b}] If $c\in \Crit_*(f)$, then
$$
\dist_\C\(c, G^*(c_+)\)>0.
$$ 
\end{itemize}
\ethm 

\fr Now the following characterization of *semi--hyperbolic semigroups, more in terms of the skew product map $\^f:\Sg_u\times\hat\C\lra \Sg_u\times\hat\C$, is immediate. 

\bthm\label{t2.9h3B} 
If $G=\langle f_{1},\ldots ,f_{u}\rangle$ is a finitely generated 
rational semigroup satisfying the Fundamental Assumption, then $G$ is *semi--hyperbolic if and only if all of the following conditions are satisfied.  

\begin{itemize}
\item[\mylabel{a}{t2.9h3B item a}] For each $z\in J(G)$ there exists a neighborhood $U$ of $z$ in $\C$ such that for any sequence $\{ g_{n}\} _{n=1}^{\infty }$ in $G$, any domain $V$ in $\oc$, and any point 
$\zeta \in U$, the sequence $\{ g_{n}\} _{n=1}^{\infty }$ does not converge 
to $\zeta $ locally uniformly on $V.$ 

\,

\item[\mylabel{b}{t2.9h3B item b}] For each $\xi\in \Crit_*(\tf)$ we have that
$$
\dist_\C\(\xi,\{\tf^n(\xi):n\ge 1\}\)>0.
$$
\end{itemize}
\ethm 

\sp\fr The second named author proved in \cite{hiroki2}, as Corollary~2.16, the following theorem which is very important for us in the current paper.

\bthm[Exponential Shrinking Property]\label{t1h4}\index{Exponential Shrinking Property}
Let $G=\langle f_{1},\ldots ,f_{u}\rangle $ be a *semi--hyperbolic finitely generated 
rational semigroup. 
%Assume that there exists an element of $G$ with degree at least two, that each element of {\em Aut}$(\oc )\cap G$ is  loxodromic, and that $F(G)\neq \emptyset .$ 
Then there exist 
$R>0$, $C>0$, and $\a>0$ such that if $x\in J(G)$, $\om \in \Sg
_{u}^{\ast }$, and $V$ is any connected component  
of $f_{\om }^{-1}(B_2(x,R))$, then 

\sp\begin{itemize}
\item[\mylabel{a}{t1h4 item a}] $V$ is simply connected 

\,

\fr and

\,

\item[\mylabel{b}{t1h4 item b}] 
$$
\diam_\C\le Ce^{-\a|\om |}.
$$
\end{itemize}
\ethm

\brem
Note that this property, i.e. Exponential Shrinking, does not in fact depend on the choice of finitely many generators. It is a property of finitely generated rational semigroups themselves. Only $\a$ (and $C$ too) may depend on the generators.
\erem

\sp\fr We have proved in \cite{sush} the following lemma. We provide here its simple independent proof for the sake of completeness and convenience of the reader.

\blem\label{l1h63}
If $G=\langle f_{1},\ldots ,f_{u}\rangle $ is a finitely generated 
*semi--hyperbolic rational semigroup, then 
\begin{itemize}
\item[\mylabel{a}{l1h63 item a}] $\ov{G^{\ast }(\Crit_*(f))}\cap J(G)$ is a nowhere dense subset of $J(G)$, 

\,

\fr and

\,

\item[\mylabel{b}{l1h63 item b}] $\PCV_*(\tf)$ is a nowhere dense subset of $J(\^f)$. 
\end{itemize}
\elem

\bpf
Proving item \eqref{l1h63 item a}, suppose for a contradiction the set $\ov{G^{\ast }(\Crit_*(f))}\cap J(G)$ is not nowhere dense in $J(G)$. Since 
$$
\ov{G^{\ast }(\Crit_*(f))}\cap J(G)
=\bu_{c\in \Crit_*(f)}J(G)\cap\ov{G^{\ast }(c_+)},
$$
since the set $\Crit_*(f)$ is finite, and since each constituent of the above union is closed, it follows that there exists $c\in \Crit_*(f)$ such that
$$
\Int_{J(G)}\(J(G)\cap\ov{G^{\ast }(c_+)}\)\ne\es.
$$
It then follows from item \eqref{p120190313 item c} of Proposition~\ref{p120190313} and forward $G$--invariance of the set $\ov{G^{\ast }(c_+)}$ that $\ov{G^{\ast }(c_+)}\spt J(G)$. In particular
$$
c\in \ov{G^{\ast }(c_+)},
$$
contrary to item \eqref{t2.9h3 item b} of Theorem~\ref{t2.9h3}. Item \eqref{l1h63 item a} of our lemma is thus proved.

\sp In order to prove its item \eqref{l1h63 item b} suppose for a contradiction that $\PCV_*(\tf)$ is not nowhere dense in $J(\tf)$.

This means that $\PCV_*(\tf)$ has non--empty interior in $J(\tf)$, and therefore, because of its forward invariance and topological exactness of the map $\tf:J(\tf)\to J(\tf)$, we conclude that 
$$
\PCV_*(\tf)=J(\tf).
$$ 
Hence, 
$$
J(G)=p_2(J(\tf))=p_2(\PCV_*(\tf))\sbt \ov{G^{\ast }(\Crit_*(f))}\cap J(G),
$$
contrary to, the already proved, item \eqref{l1h63 item a}. Item \eqref{l1h63 item b} is thus proved and, simultaneously, the whole Lemma~\ref{l1h63}.
\epf

\sp Following the common tradition, given a point
$(\tau,z)\in\Sg_u\times\oc$ we denote by $\om(\tau,z)$ the \textbf{$\om$--limit set}\index{$\om(\tau,z)$} of
$(\tau,z)$ with respect to the skew product map $\tf:\Sg_u\times\oc\to\Sg_u\times\oc$, i.e. the set of all accumulation points of the sequence $\(\tf^n(\tau,z)\)_{n=0}^\infty$. For every $(\tau,z)\in J(\tf)$ put
$$
\Crit(\tau,z):=\Crit(\tf)\cap\om(\tau,z).
%\  \text{ and } \
%\Crit(\tau,z)_+=p_2(\Crit(\tf)\cap\om(\tau,z))_+.
$$ \index{critical set!$\Crit(\tau,z)$}\index{$\Crit(\tau,z)$}
In \cite{sush} we proved the following whose proof we reproduce here for the sake of completeness and its importance for the further development of the present manuscript. It is convenient for the sake of this proof to introduce the following concept. If $c\in \Crit(G)$, then by 
$$
\om_G(c_+)
$$\index{$\om_G(c_+)$}
we denote the\textbf{ set of all accumulation points of the set $G^*(c_+)$}. Of course
$$
\ov{G^{\ast }(c_+)}=G^{\ast }(c_+)\cup \om_G(c_+).
$$

\blem\label{lu5.7h13}
Let $G=\langle f_{1},\ldots ,f_{u}\rangle$ be a finitely generated 
*semi--hyperbolic rational semigroup. If $(\tau,z)\in J(\tf)$, then
$$
p_2(\om(\tau,z))\not\sbt\ov{G^{\ast }\(p_2(\Crit(\tau,z))_+\)}.
$$
\elem

\bpf Suppose on the contrary that
\beq\label{u2.30h14}
p_2(\om(\tau,z))\sbt\ov{G^{\ast }\(p_2(\Crit(\tau,z))_+\)}.
\eeq
Consequently, $\Crit(\tau,z)\ne\es$. Let $(\tau^1,c_1)\in\Crit(\tau,z)$. This
means that $(\tau^1,c_1)\in \om(\tau,z)$, and it follows from (\ref{u2.30h14})
that there exists $(\tau^2,c_2)\in\Crit(\tau,z)$ such that

\,

\centerline{either $c_1\in\om_G({c_2}_+)$ \  \  or \  \ $c_1=g_1(c_2)$} 
\,

\fr for some $g_1\in G$ of the form $f_\om$ with $f_{\om_1}'(c_2)=0$.
Iterating this procedure we obtain an infinite sequence $((\tau^j,c_j))_{j=1}^\infty$
of points in $\Crit(\tau,z)$ such that for every $j\ge 1$

\,

\centerline{ either $c_j\in\om_G({c_{j+1}}_+)$ \  \ 
or \  \  $c_j=g_j(c_{j+1})$}

\,

\fr for some $g_j\in G$ of the form $f_\rho$ with $f_{\rho_1}'(c_{j+1})=0$.
Consider an arbitrary block 
$$
c_k,c_{k+1},\ld,c_l
$$ 
such that 
$$
c_j=g_j(c_{j+1})
$$ 
for every $k\le j\le l-1$, and suppose that 
$$
l-(k-1)\ge \#(\Crit(f)\cap J(G)).
$$
Then there are two indices $k\le a<b\le l$ such that 
$$
c_a=c_b.
$$
Hence, 
$$
g_a\circ g_{a+1}\circ\ld\circ
g_{b-1}(c_b)=c_a=c_b
\ \ \  {\rm and} \  \  \ (g_a\circ g_{a+1}\circ\ld\circ g_{b-1})'(c_b)=0.
$$
This however contradicts our assumption that the Julia set, $J(G)$, of $G$ contains no superstable fixed points. In consequence, the length of the block $c_k,c_{k+1},\ld,c_l$ is
bounded above by $\#(\Crit(f)\cap J(G))$. Therefore, there exists an infinite sequence
$(j_n)_{n=1}^\infty$ such that 
$$
c_{j_n}\in\om_G({c_{j_n+1}}_+)
$$ 
for all $n\ge 1$. In conclusion, there would exist at least one point $c\in \Crit_*(f)$ such that
$$
c\in\om_G(c_+),
$$
contrary to Theorem~\ref{t2.9h3} \eqref{t2.9h3 item b}. 
%Since,
%by our construction, $c_{j_n+1}\in\ov{G({c_{j_{n+1}}}_+)}$, we thus conclude that
%$c_{j_n}\in\om_G({c_{j_{n+1}}}_+)$, meaning that $c_{j_n}<c_{j_{n+1}}$.
This finishes the proof. 
\epf

The following proposition, also proved in \cite{sush}, will be used many times in the current paper. We therefore provide its proof here too.

\sp
\bprop\label{pu6.1h23}
If $G=\langle f_{1},\ldots ,f_{u}\rangle $ is a finitely generated 
*semi--hyperbolic rational semigroup, then for each point $(\tau,z)\in J(\tf)\sms \Sing(\tf)$, there exist a number $\eta(\tau,z)>0$, an increasing 
sequence $(n_j)_{j=1}^\infty$ of positive integers, and a point 
$$
(\hat\tau,\hat z)\in
\om(\tau,z)\sms p_2^{-1}\(\ov{G^{\ast }(p_2(\Crit(\tau,z))_+)}\)
$$
with the following two properties.

\sp\begin{itemize}
\item[\mylabel{a}{pu6.1h23 item a}] $\lim_{j\to\infty}\tf^{n_j}(\tau,z)=(\hat\tau,\hat z)$.
 
\item[\mylabel{b}{pu6.1h23 item b}] The connected component of the set $f_{\tau|_{n_j}}^{-1}\Big(B_2\(f_{\tau|_{n_j}}(\tau,z),\eta(\tau,z)\)\Big)$ containing $z$ contains no critical points of the map $f_{\tau|_{n_j}}:\oc\lra\oc$. 
\end{itemize}
\eprop

\bpf
In view of Lemma~\ref{lu5.7h13} there exists a point $(\hat\tau,\hat z)\in
\om(\tau,z)$ such that 
$$
\hat z\notin\ov{G^{\ast }(p_2(\Crit(\tau,z))_+)}.
$$
Let
$$
\eta:={1\over 2}\dist_\C\(\hat z, \ov{G^{\ast }(p_2(\Crit(\tau,z))_+)} \).
$$
Then there exists an infinite increasing sequence $(n_j)_{j=1}^\infty$ of positive integers
such that
\beq\label{u2.39h23}
\lim_{j\to\infty}\tf^{n_j}(\tau,z)=(\hat\tau,\hat z)
\eeq
and
\beq\label{u2.40h23}
f_{\tau|_{n_j}}(z)\notin B_2\( \ov{G^{\ast }(p_2(\Crit(\tau,z))_+)},\eta \)
\eeq
for all $j\ge 1$. We claim that there exists $\eta(\tau,z)>0$ such that for all
$j\ge 1$ large enough
$$
\Comp\(z,f_{\tau|_{n_j}},\eta(\tau,z)\)\cap \Crit(f_{\tau|_{n_j}})=\es.
$$
Indeed, otherwise we can find an increasing subsequence $(j_i)_{i=1}^\infty$ and a
decreasing to zero sequence of positive numbers $\eta_i<\eta$ such that
$$
\Comp\(z,f_{\tau|_{n_{j_i}}},\eta_i\)\cap \Crit(f_{\tau|_{n_{j_i}}})\ne\es.
$$
Let $\^c_i\in\Comp\(z,f_{\tau|_{n_{j_i}}},\eta_i\)\cap \Crit(f_{\tau|_{n_{j_i}}})$. Then
there exist $0\le p_i\le n_{j_i}-1$ and
\beq\label{1h23}
c_i\in\Crit(f_{\tau_{p_{i}+1}})
\eeq
such that $c_i=f_{\tau |_{p_i}}(\^c_i)$. Since $\lim_{i\to\infty}\eta_i=0$, it follows
from Theorem~\ref{t1h4} that $\lim_{i\to\infty}\^c_i=z$. Since $(\tau,z)
\notin \bu_{n\ge 0}\tf^{-n}(\Crit(\tf))$, this implies that $\lim_{i\to\infty}p_i=+\infty$.
But then, making use of Theorem~\ref{t1h4} again and of the formula
$(\sg^{p_i}(\tau),c_i)=f^{p_i}(\tau,\^c_i)$, we conclude that the set of accumulation
points of the sequence $$((\sg^{p_i}(\tau),c_i))_{i=1}^\infty $$ is contained in $\om(\tau,z)$.
Fix $(\tau^\infty,c)$ to be one of these accumulation points. Since $\Crit(\tf)$ is closed,
we conclude that
\beq\label{1h24}
(\tau^\infty,c)\in \Crit(\tau,z).
\eeq
Since that set $\Crit(f)$ is finite, passing to a subsequence, we may assume without
loss of generality that $(c_i)_{i=1}^\infty$ is a constant sequence, so equal to $c$.
Since $c=f_{\tau |_{p_i}}(\^c_i)$, we get
$$
\lt|f_{\tau|_{n_{j_i}}}(z)-f_{\tau|_{p_i+1}^{n_{j_i}}}(c)\rt|
=\lt|f_{\tau|_{n_{j_i}}}(z)-f_{\tau|_{n_{j_i}}}(\^c_i)\rt|
<\eta_i
<\eta.
$$
But, looking at (\ref{1h23}) and (\ref{1h24}), we conclude that $f_{\tau|_{p_i+1}^{n_{j_i}}}(c)
\in G^{\ast }(\Crit(\tau,z)_+)$. We have thus arrived at a contradiction with (\ref{u2.40h23}),
and the proof is finished. 
\epf

\sp

\section{The Conformal and Invariant Measures $m_t$ and $\mu_t$ for $\tilde f:J(\tilde f)\lra J(\tilde f)$}

\subsection{Conformal and Invariant Measures (and Topological Pressure) for $\tilde f:J(\tilde f)\lra J(\tilde f)$: Preliminaries}

\

In this section we rigorously recall, generalize, and extend, the thermodynamic concepts introduced in \cite{sush}. Some of them have already been introduced and explored in {\cite{hiroki3}, \cite{hiroki4}, and \cite{SU2}.

\bdfn
Let $G=\langle f_{1},\ldots ,f_{u}\rangle$ and let $t\geq 0.$  
For all $\xi\in J(\^f)\sms \PCV_*(\tf)$ define
$$
\P_\xi(t):= \limsup _{n\rightarrow\infty}
\frac{1}{n}\log\sum_{x\in \tf^{-n}(\xi)}\big|\(\tf^n\)'(x)\big|^{-t}\in [-\infty,+\infty],
$$\index{topological pressure!$P_\xi(t)$} \index{$P_\xi(t)$} 
where, if $x=(\om,a)$, then, we recall, $\big|\(\tf^n\)'(x)\big|$ denotes the norm of the derivative of $f_{\om_n}\circ\dots\circ f_{\om_1}$ at the point $a$.
\edfn
%As a matter of fact this definition is different than the one, Definition~4.1, given in \cite{sush}. Our definition now is by means of the skew product map $\tilde f:J(\tilde f)\lra J(\tilde f)$ while the one in \cite{sush} utilized the semigroup $G$ itself. It is however immediate to see that both definitions coincide.

The following proposition was proved in \cite{sush} as Lemma~7.3 and will be of crucial importance in the sequel. We provide its short and simple proof here for the sake of completeness and for the convenience of the reader. Furthermore, we expose the proof in terms of the skew product map $\tilde f:J(\tilde f)\lra J(\tilde f)$ rather than in terms of the semigroup $G$ itself, as it was done in \cite{sush}.

\bprop\label{p3_2016_06_21}
If $G$ is a *semi--hyperbolic rational semigroup generated by a $u$--tuple map $f=(f_{1},\ldots ,f_{u})\in \mbox{\em Rat}^u$, then for each $t\geq 0$, the function 
$$
J(\^f)\setminus \PCV_*(\tf)\ni\xi\longmapsto \P_{\xi}(t)\in [-\infty,+\infty]
$$ 
is constant. Denote this constant by $\P(t)$\index{topological pressure!$P(t)$}\index{$P(t)$} and call it the \textbf{topological pressure of $t$} \index{topological pressure}. 
\eprop

\bpf 
Because of Lemma~\ref{l2mc3} for every $\xi\in J(\^f)\sms \PCV_*(\tf)$ there exists $r_\xi>0$ such that
$$
B_2(p_2(\xi),2r_\xi)\cap \PCV(G)=\es.
$$
It then directly follows from Koebe's Distortion Theorem that the function 
$$
\g\longmapsto\P_\g(t)
$$ 
is constant on $J(\^f)\cap \(\Sg_u\times B_2(p_2(\xi),r_\xi)\)$. Now, fix $\xi_1,\xi_2\in J(\^f)\setminus \PCV_*(\tf)$. By Proposition~\ref{p120190313} \eqref{p120190313 item c} there exists an integer $q\ge 1$ such that 
$$
\tf^q\(J(\^f)\cap \(\Sg_u\times B_2(p_2(\xi_1),r_{\xi_1}))\)\cap \(J(\^f)\cap \(\Sg_u\times B_2(p_2(\xi_2),r_{\xi_2})\)\)\ne\es.
$$
Fix then a point $\rho\in J(\^f)\cap \(\Sg_u\times B_2(p_2(\xi_1),r_{\xi_1}))$ such that 
$$
\tf^q(\rho)\in J(\^f)\cap \(\Sg_u\times B_2(p_2(\xi_2),r_{\xi_2})\).
$$
Then 
$$
\sum_{x\in \tf^{-(n+q)}(\tf^q(\rho))}\big|\(\tf^{n+q}\)'(x)\big|^{-t}
\ge |(\tf^q)'(\rho)|^{-t}\sum_{y\in \tf^{-n}(\rho)}\big|\(\tf^n\)'(y)\big|^{-t}.
$$ 
Therefore, $\P_{\tf^q(\rho)}(t)\ge \P_\rho(t)$. Hence, $\P_{\xi_2}(t)\ge
\P_{\xi_1}(t)$. Exchanging the roles
of $\xi_1$ and $\xi_2$, we get $\P_{\xi_1}(t)\ge \P_{\xi_2}(t)$, and we are done. \epf

\sp\fr We want to emphasize that topological pressure does depend on the choice of generators and not on the semigroup alone. Now, we shall prove the following basic properties of the function 
$$
[0,+\infty)\ni t\longmapsto\P(t)\in[-\infty,+\infty].
$$

\bprop\label{l1h35}
If $G$ is a finitely generated *semi--hyperbolic rational semigroup generated by a $u$--tuple map $f=(f_{1},\ldots ,f_{u})\in \mbox{{\em Rat}}^u$, then the function $t\longmapsto\P(t)$, $t\ge 0$, has the following properties.

\sp\begin{itemize}
\item[\mylabel{a}{l1h35 item a}] For every $t\ge 0$ we have that $\P(t)\in(-\infty,+\infty)$ and $\P(0)\ge \log 2>0$.

\sp\item[\mylabel{b}{l1h35 item b}] The function $[0,+\infty)\longmapsto \P(t)$ is strictly decreasing and Lipschitz continuous. More precisely:

\sp\item[\mylabel{c}{l1h35 item c}] If $0\le s\le t<+\infty$, then 
$$
-\log\|\tf'\|_\infty(t-s)\le \P(t)-\P(s)\le - \a(t-s),
$$
where the constant $\a>0$ comes from the Exponential Shrinking Property (Theorem~\ref{t1h4}).
\sp\item[\mylabel{d}{l1h35 item d}] $\lim_{t\to+\infty}\P(t)=-\infty$.
\end{itemize}
\eprop

\bpf 
The inequality 
\beq\label{5_2016_06_21}
\log 2\le \P(0)<+\infty
\eeq
 follows immediately from degree considerations. Assuming that $\P(s)\in\R$, the left--hand side of item \eqref{l1h35 item c} follows immediately, while its right--hand side follows from the Exponential Shrinking Property (Theorem~\ref{t1h4}) in conjunction with Koebe's Distortion Theorem. Having this and \eqref{5_2016_06_21} we deduce that items \eqref{l1h35 item a} and \eqref{l1h35 item c} hold. Item \eqref{l1h35 item b} is a direct consequence of \eqref{l1h35 item c} while item \eqref{l1h35 item d} is now an immediate consequence of the right--hand side of \eqref{l1h35 item c}.
\epf

\sp\fr It follows from items \eqref{l1h35 item a}, \eqref{l1h35 item b}, and \eqref{l1h35 item d} of this proposition that there exists a unique $t\in[0,+\infty)$, we denote it by $h_{f}$, \index{$h_f$} such that 
\beq\label{6_2016_06_21}
\P(h_{f})=0.
\eeq

We now recall the concept of $t$--conformal measures for the skew product map $\tf:J(\tilde f)\lra J(\tilde f)$. 
\bdfn A Borel probability measure $m$ on $J(\tilde f)$ is called \textbf{$t$--conformal}\index{conformal measure!$t$--conformal measure} for $\tf:J(\tilde f)\to J(\tilde f)$ if and only if
$$
m(\tf(A))=e^{\P(t)}\int_A|\tf'|^t\,dm
$$
whenever $A\sbt J(\tilde f)$ is a Borel set such that the restricted map $\tf|_A$ is 1--to--1. 
\edfn
In order to discern it from ``truly'' conformal measures, i.e. $h_f$--conformal measures (for which the factor $e^{\P(t)}$ disappears in the above displayed formula), any $e^{\P(t)}|\tf'|^{t}$--conformal measure will also be frequently called a \textbf{generalized $t$--conformal measure}.\index{conformal measure! generalized conformal measure} We have proved in \cite{sush} the following fact of crucial importance for us in the current paper. We provide its proof here for the sake of completeness, convenience of the reader, its crucial importance, and because the construction of conformal measures given in this proof (coming from \cite{sush}) will be needed later in the paper.

\bprop\label{p1h39}
If $G$ is a finitely generated *semi--hyperbolic rational semigroup generated by a $u$--tuple map $f=(f_{1},\ldots ,f_{u})\in \mbox{{\em Rat}}^u$, then 
for every $t\ge 0$ there exists an $e^{\P(t)}|\tf'|^t$--conformal measure $m_t$ \index{conformal measure!$m_t$}\index{$m_t$} for the skew product map $\tf:J(\tf)\lra J(\tf)$.
\eprop

\bpf The construction of generalized $t$--conformal measures starts as follows. Fix $\xi\in J(\^f)\sms \PCV(\tf)$.
Observe that the critical parameter for the series
$$
SP_t(\xi,s):=\sum_{n=1}^\infty e^{-sn}\sum_{x\in \tf^{-n}(\xi)}\big|\(\tf^n\)'(x)\big|^{-t}, \  \ s\ge 0,
$$\index{topological pressure!$SP_t(\xi,s)$}\index{$SP_t(\xi,s)$}
is equal to the topological pressure $\P(t)$, i.e. $SP_t(\xi,s)=+\infty$ if $s<\P(t)$ and
$SP_t(\xi,s)<+\infty$ if $s>\P(t)$. In the latter case define a Borel probability measure  $\nu_{t,s}$ on $J(\^f)$ by the following formula:
\beq\label{1_2016_06_20}
\nu_{t,s}(A):=\frac1{SP_t(\xi,s)}\sum_{n=1}^\infty e^{-sn}\!\!\!\sum_{x\in A\cap\tf^{-n}(\xi)}\big|\tf'(x)\big|^{-t}.
\eeq
Now we want to take weak* limits as $s\downto \P(t)$. For this we need the concept of the Perron--Frobenius operator. Given a function $g:J(\tf)\lra\C$
let $L_tg:J(\tf)\lra\C$ be defined by the following formula:
\beq\label{120190320}
L_tg(y)=\sum_{x\in \tf^{-1}(y)}|\tf'(x)|^{-t}g(x).
\eeq
\index{Perron--Frobenius (transfer) operator!$L_t$}\index{$L_t$}
$L_tg(y)$ is finite if and only if $y\notin\Crit(\tf)$. Otherwise $L_tg(y)$ is declared to be $\infty$. Although $L_t$ is not actually an operator since it does not preserves the class of finite valued functions, it is still referred to as the Perron--Frobenius operator associated to the parameter $t\ge 0$. Iterating formula \eqref{120190320}, we get for all integers $n\ge 1$ that
$$
L_t^ng(y)=\sum_{x\in \tf^{-n}(y)}|(\tf^n)'(x)|^{-t} g(x).
%L_t^ng(y)=\sum_{x\in \tf^{-n}(y)}|(\tf^n)'(x)|^{-t}\phi(x).
$$

For every $\sg$--finite Borel measure $m$ on $J(\tf)$ let the $\sg$--finite Borel measure $L_t^{*n}m$ be given by the formula
$$
L_t^{*n}m(A)=m(L_t^n\1_A), 
$$\index{Perron--Frobenius (transfer) operator!$L_t^{*n}m$}\index{$L_t^{*n}m$}
where $A\sbt J(\tf)$ is a Borel set and $m(g):=\int gdm$. Notice that if $(\tau,\xi)\in J(\tf)\setminus \DPCV(\^f)$, then for all Borel sets $A\sbt J(\tf)$ we have
\begin{align*}
L_t^{*n}\d_{(\tau,\xi)}(A)
=\d_{(\tau,\xi)}(L_t^n\1_A)
=L_t^n\1_A(\tau,\xi)
=\sum_{|\om |=n}\sum_{x\in A\cap f_{\om }^{-1}(\xi)}|f_{\om }'(x)|^{-t}
\le L_t^n\1(\xi)
<\infty.
\end{align*}
In particular, 
$$
L_t^{*n}\d_{(\tau,\xi)}(J(\tf))\le L_t^n\1(\xi)<+\infty.
$$
Observing that the formula \eqref{1_2016_06_20} can be now expressed in the form
\beq\label{1h37}
\nu_{t,s}=SP_t(\xi,s)^{-1}\sum_{n=1}^\infty e^{-sn} L_t^{*n}\d_{(\tau,\xi)},
\eeq
a direct straightforward calculation shows that
\beq\label{220190320}
e^{-s}L_t^*\nu_{t,s}
=SP_t(\xi,s)^{-1}\sum_{n=1}^\infty e^{-s(n+1)}L_t^{*(n+1)}\d_{(\tau,\xi)}
=\nu_{t,s}-SP_t(\xi,s)^{-1}\(e^{-s}L_t^*\d_{(\tau,\xi)}\).
\eeq
If $SP_t\(\xi,\P(t)\)=+\infty)$, then $\lim_{s\downto\P(t)}S_s(\xi)=+\infty$ and it follows from the formula \eqref{220190320} that
any weak limit $m_t$ of $\nu_{t,s}$ when $s\downto\P(t)$ is the required generalized $t$--conformal measure. If $SP_t\(\xi,\P(t)\)<+\infty$ then we apply the usual modifications involving slowly varying functions whose details can be found for example in \cite{dusullivan}.
\epf

We have proved in \cite{sush} the following. We will need it in the sequel and we provide its short computational proof taken from \cite{sush}.

\blem\label{l1_2016_06_20}
Let $G$ be a finitely generated *semi--hyperbolic rational semigroup generated by a $u$--tuple map $f=(f_{1},\ldots ,f_{u})\in \mbox{{\em Rat}}^u$. Fix $t\ge 0$. If $s>\P(t)$, then
$$
\nu_{t,s}(A)=
e^{-s}\int_{\tf(A)}\big|(\tf|_A^{-1})'(x)\big|^td\nu_{t,s}(x)
    + \begin{cases} 0                  &\text{if }  \  A\cap \tf^{-1}(\xi)=\es \\
        e^{-s}SP_t(\xi,s)^{-1}(\xi)\big|\tf'(y)\big|^{-t} &\text{if }  \  A\cap \tf^{-1}(\xi)=\{y\}
      \end{cases}
$$
whenever $A\sbt J(\^f)$ is a Borel set such that the map $\tf|_A$ is 1--to--1. 
\elem

\bpf
We calculate:
$$
\begin{aligned}
\nu_{t,s}(A)
=& e^{-s}L_t^*\nu_{t,s}(\1_A) + S_s^{-1}(\xi)e^{-s}L_t^*\d_{(\tau,\xi)}(\1_A) \\
=& e^{-s}\int L_t(\1|_A)d\nu_{t,s} + e^{-s}S_s^{-1}(\xi)\int L_t(\1|_A)d\d_{(\tau,\xi)} \\
=& e^{-s}\int\sum_{y\in \tf^{-1}(x)}\big|\tf{'}\big|^{-t}\1_A(y)d\nu_{t,s}(x)
    + e^{-s}S_s^{-1}(\xi)L_t(\1|_A)(\tau,\xi) \\
=& e^{-s}\int_{\tf(A)}\big|(\tf|_A^{-1}){'}(x)\big|^td\^\nu_s(x)
    + \begin{cases} 0
 \  &\text{ if }  \  A\cap \tf^{-1}(\tau,\xi)=\es \\
        e^{-s}S_s^{-1}(\xi)\big|\tf{'}(y)\big|^{-t} &\text{ if }  \  A\cap \tf^{-1}(\tau,\xi)=\{y\}.
      \end{cases}
\end{aligned}
$$
\epf

\subsection{Ergodic Theory of Conformal and Invariant Measures for $\tilde f:J(\tilde f)\lra J(\tilde f)$}

\

\sp We begin this section with some abstract auxiliary facts about measures. 
Frequently, in order to denote that a Borel measure $\mu$ is
absolutely  continuous\index{absolutely continuous measure} with
respect to $\nu$, we  write $\mu \abs \nu $\index{absolutely continuous measure!$\abs$}\index{$\abs$}. However, we do
not use any special symbol to record the equivalence\index{equivalent
measures} of measures. We use some notations from \cite{Aa}. 
Let $(X,\mathcal{F}, \mu)$ be a $\sg$--finite measure space and let  
$T:X\rightarrow X$ be a
measurable almost everywhere defined transformation. 
$T$ is said to be \textbf{nonsingular}\index{nonsingular transformation} if and only if for any $A\in \mathcal{F}$,  
$$
\mu (T^{-1}(A))=0 \  \  \Leftrightarrow \  \  \mu (A)=0.
$$ 
The map $T$ is said to be \textbf{ergodic with respect to $\mu$}\index{ergodic measure}, or
$\mu$ is said to be ergodic with respect to $T$, if and only if either
$$
\mu(A)=0 \ \ {\rm or }  \  \  \mu(X \sms A)=0
$$  
whenever the measurable set $A$ is
$T$--invariant, meaning that  $T^{-1}(A)=A$. 
For a nonsingular transformation $T:X\rightarrow X$, 
the measure $\mu$  is
said to  be \textbf{conservative with respect to $T$} \index{conservative measure} or $T$ is said to be conservative with respect to $\mu$ if and only if for
every measurable set $A$ with $\mu(A)>0$,
$$
\mu\lt(\lt\{z\in A: \sum_{n=0}^\infty  \1_A\circ T^n(z)< +\infty\rt\}\rt)=0.
$$
Note that by \cite[Proposition 1.2.2]{Aa}, 
for a nonsingular transformation $T:X\rightarrow X$, 
$\mu $ is ergodic and conservative with respect to $T$ if and only if 
for any $A\in \mathcal{F}$ with $\mu (A)>0,$  
$$
\mu \lt(\lt\{ z\in X: \sum _{n=0}^{\infty }\1_{A}\circ T^{n}(z)<+\infty\rt\}\rt)=0.
$$ 
Finally, the measure $\mu$ is said to be
\textbf{$T$--invariant}\index{invariant measure}, or $T$ is said to preserve
the measure $\mu$ if and only if $\mu\circ T^{-1}=\mu$. 
Denote by $M(T)$ the collection of all $T$--invariant Borel probability measures on $X$.\index{invariant measure! $M(T)$}\index{$M(T)$}
It follows from Birkhoff's Ergodic Theorem that every finite ergodic
$T$--invariant measure $\mu$ is conservative;  for infinite measures
this is no longer true. Finally, two  ergodic invariant measures
defined on the same $\sigma$--algebra are  either singular or they
coincide up  to a multiplicative constant.

\

\bdfn\label{d1h65}
Suppose that $(X,\Fa,\nu)$ is a probability space and $T:X\to X$ is a measurable map
such that $T(A)\in \Fa$ whenever $A\in\Fa$. The map $T:X\to X$ is said to be \textbf{weakly metrically exact} \index{weakly metrically exact} provided that 
$$
\limsup_{n\to\infty}\mu(T^n(A))=1
$$
whenever $A\in\Fa$ and $\mu(A)>0$.
\edfn

\fr We need the following two facts about weak metrical exactness, 
the first being straightforward (see the argument in \cite[page 15]{Aa}),
the latter more involved (see\cite{MundayRoyUI} and \cite{KUbook}). 

\bfact\label{f2h65}
If a nonsingular measurable transformation $T:X\to X$ of a probability space $(X,\cF,\nu)$ is weakly metrically exact, then it is ergodic and conservative.
\efact

\bfact\label{f3h65}
A measure--preserving transformation $T:X\to X$ of a probability space $(X,\Fa,\mu)$ is
weakly metrically exact if and only if it is exact, which means that 
\begin{itemize}
\item[\mylabel{a}{f3h65 item a}]
$$
\lim_{n\to\infty}\mu(T^n(A))=1
$$ 
whenever $A\in\cF$ and $\mu(A)>0$, or equivalently, 

\sp\item[\mylabel{b}{f3h65 item b}] The $\sg$--algebra
$\bi_{n\ge 0}T^{-n}(\cF)$ consists of sets of measure $0$ and $1$ only. 
\end{itemize}

\sp\fr Note that if $T:X\rightarrow X$ is exact, then the
Rokhlin's natural extension $(\^T,\^X,\^\mu)$ of $(T,X,\mu)$ is K--mixing.
\efact

\sp\fr We now pass to our *semi--hyperbolic semigroup $G$ and investigate in detail generalized conformal measures.

\bdfn 
Let $G$ be a finitely generated *semi--hyperbolic rational semigroup generated by a $u$--tuple map $f=(f_{1},\ldots ,f_{u})\in \mbox{{\em Rat}}^u$. Then,
\begin{itemize}
\item A real number $t\ge 0$ is called a \textbf{parameter of continuity}\index{parameter of continuity} for the *semi--hyperbolic semigroup $G$ if and only if 
$$
m_t(\Crit(\tilde f))=0.
$$
\item The conformal measure $m_t$ is then called \textbf{continuous}\index{continuous measure}. 
\item The set of all parameters of continuity for $G$ will be denoted by $\De_G$\index{parameter of continuity! $\De_G$}\index{$\De_G$}.
\end{itemize}
\edfn

\fr As an immediate consequence of this definition and generalized conformality of measures $m_t$, we get the following.

\bobs\label{o1_2016_05_20}
If $G$ is a finitely generated *semi--hyperbolic semigroup generated by a $u$--tuple map $f=(f_{1},\ldots ,f_{u})\in \mbox{{\em Rat}}^u$ and $t\in\De_G$, then
$$
m_t(\Sing(\^f))=0.
$$
\eobs

\fr Given a point $(\tau,z)\in J(\tf)\sms \Sing(\^f)$ let
\beq\label{2h67A}
B_j(\tau,z)
:=\tf_{\tau|_{n_j},z}^{-n_j}\lt(\Sg_u\times B_2\Big(f_{\tau|_{n_j}}(z),
 {1\over 4}\eta(\tau,z)\Big)\rt) 
=[\tau|_{n_j}]\times f_{\tau|_{n_j},z}^{-1}\lt(B_2\Big(f_{\tau|_{n_j}}(z)
 {1\over 4}\eta(\tau,z)\Big)\rt) 
\eeq\index{$B_j(\tau,z)$}
and 
\beq\label{2h67C}
B_j^*(\tau,z):=B_j(\tau,z))\sms \Sing(\^f),
\eeq\index{$B_j^*(\tau,z)$}
where $\eta(\tau,z)>0$ and $n_j:=n_j(\tau,z)\ge 1$ are the integers produced in Proposition~\ref{pu6.1h23}. \
We now shall recall the concept of Vitali relations defined on the
page 151 of Federer's book \cite{federer}. Let $X$ be an arbitrary
set. By a \textbf{covering relation} on $X$ one means a subset of 
$$
\{(x,S):x\in S\sbt X\}.
$$
If $C$ is a covering relation on $X$ and $Z\sbt X$, one puts
$$
C(Z)=\{S\sbt X:(x,S)\in C \  \text{ for some } \ x\in Z\}.
$$
One then says that $C$ is \textbf{fine}\index{fine} at $x$ if 
$$
\inf\{\diam(S):(x,S)\in C\} = 0.
$$
If in addition $X$ is a metric space and a Borel measure $\mu$ is given
on $X$, then a covering relation $V$ on $X$ is called a \textbf{Vitali relation}\index{Vitali relation}
if 
\begin{itemize}
\item[\mylabel{a}{def: Vitali relation item a}] All elements of $V(X)$ are Borel sets,
\item[\mylabel{b}{def: Vitali relation item b}] $V$ is fine at each point of $X$,
\item[\mylabel{c}{def: Vitali relation item c}] If $C\sbt V$, $Z\sbt X$ and $C$ is fine at each point of
  $Z$, then there exists a countable disjoint subfamily $\Fa$ of $C(Z)$
  such that $\mu(Z\sms \cup\,\Fa)=0$.
\end{itemize}

\sp\fr Now, given $(\tau,z)\in J(\tf)\sms \Sing(\tf)$, let 
$$
\Ba_{(\tau,z)}=\lt\{\((\tau,z),B_j^*(\tau,z)\)\rt\}_{j=1}^\infty,
$$
where the sets $B_j(\tau,z)$ are defined by formula (\ref{2h67A}). Let
$$
\Ba=\bu_{(\tau,z)\in J(\tf)\sms \Sing(\tf)}\Ba_{(\tau,z)}
$$
and, following notation from Federer's book \cite{federer}, let
$$
\Ba_2:=\Ba\(J(\tf)\sms \Sing(\tf)\)=\{B_j^*(\tau,z):(\tau,z)\in
J(\tf)\sms \Sing(\tf),\, j\ge1\}. 
$$
We shall prove the following.

\blem\label{l1h71}
If $G$ is a finitely generated *semi--hyperbolic semigroup generated by a $u$--tuple map $f=(f_{1},\ldots ,f_{u})\in \mbox{{\em Rat}}^u$, %and $t\in\De_G$,
then for every $t\in\De_G$ the family $\Ba$ is a Vitali relation for the measure $m_t$ on the
set $J(\tf)\sms \Sing(\tf)$. 
\elem

{\sl Proof.} Fix $1<\ka<e$ (a different notation for 
$1<\tau<+\infty$ appearing in Theorem~2.8.17 from \cite{federer}). Fix $(\tau,z)\in J(\tf)\sms \Sing(\tf)$. Since
\beq\label{1h71}
\lim_{j\to\infty}n_j=+\infty,
\eeq
we have that
$$
\lim_{j\to\infty}\diam_{\Sg_u\times\C}(B_j^*(\tau,z))=0.
$$
This means that the relation $\Ba$ is fine at the point $(\tau,z)$. Aiming to apply 
Theorem~2.8.17 from \cite{federer}, we set
$$
\d(B_j^*(\om,x))=\exp(-n_j)
$$
for every $B_j^*(\om,x)\in \Ba_2$. With the notation 
from page 144 in \cite{federer} ($\Fa=\Ba_2$) we have
$$
\hat{B}_j^*(\tau,z)=\bu\Ba_j^*(\tau,z),
$$
where
$$
\Ba_j^*(\tau,z)                  
:=\big\{B:B\in\Ba_2,\, B\cap B_j^*(\tau,z)\ne\es,\,\d(B)\le\ka\d(B_j^*(\tau,z))\big\}.
$$
Fix a $B$ from the above family. Then there exists $(\om,x)\in J(\tf)\sms \Sing(\tf)$ and an integer $i\ge 1$ such that
$$
B=B_i^*(\om,x).
$$
Since $\d(B_i(\om,x))\le\ka\d(B_j(\tau,z))$, we have that $e^{-n_i}\le\ka e^{-n_j}$. Equivalently: $n_j\le n_i+\log\ka$. Since both $n_j$ and $n_i$ are integers and since $\log\ka<1$, this yields
\beq\label{1cmi1}
n_i\ge n_j.
\eeq
Since 
\beq\label{2cmi1}
B_i^*(\om,x)\cap B_j^*(\tau,z)\ne\es,
\eeq
we have that $[\om|_{n_i}]\cap[\tau|_{n_j}]\ne\es$. In conjunction with \eqref{1cmi1} this yields
$$
\om|_{n_j}=\tau|_{n_j}.
$$
Then
\beq\label{1cmi2}
\begin{aligned}
B&_i^*(\om,x)\cap B_j^*(\tau,z) = \\
&=\bigg(\Big([\tau|_{n_j}]\times f_{\tau|_{n_j},z}^{-1}
  \Big(B_2\Big(f_{\tau|_{n_j}}(z),{1\over 4}\eta(\tau,z)\Big)\Big)\Big)\sms \Sing(\tf)\bigg) \cap \\
&\  \  \cap \bigg(\Big([\tau|_{n_j}\sg^{n_j}(\om)|_{n_i-n_j}]\times f_{\tau|_{n_j},x}^{-1}\circ f^{-1}_{\sg^{n_j}(\om)|_{n_i-n_j},f_{\tau|_{n_j}}(x)}
  \Big(B_2\Big(f_{\om|_{n_i}}(x),\frac14\eta(\om,x)\Big)\Big)\Big)\sms \Sing(\tf)\bigg) \\
&\sbt \Bigg([\tau|_{n_j}]\times f_{\tau|_{n_j},z}^{-1}
  \bigg(B_2\Big(f_{\tau|_{n_j}}(z),{1\over 4}\eta(\tau,z)\Big)\cap \\
&\  \  \  \  \  \  \  \  \  \  \  \  \  \  \  \  \  \  \  \  \  \  \  \  \cap f^{-1}_{\sg^{n_j}(\om)|_{n_i-n_j},f_{\tau|_{n_j}}(x)}
  \Big(B_2\Big(f_{\om|_{n_i}}(x),\frac14\eta(\om,x)\Big)\Big)\bigg)\Bigg)\sms \Sing(\tf).
\end{aligned}
\eeq
Now, since $\hat{B_j^*}(\tau,z)$ is a separable metrizable topological space and $\Ba_j^*(\tau,z)$ is its cover consisting of open sets, it follows from Lindel\"of's Theorem that there are countably many points $(\om^{(s)},x_s)\in J(\tf)\sms \Sing(\tf)$, $s=1,2,\ld$, along with positive integers $i_s$, $s=1,2,\ld$, such that $B_{i_s}(\om^{(s)},x_s)\in \Ba_j^*(\tau,z) $ for every $s\ge 1$ and 
$$
\bu_{s=1}^\infty B_{i_s}^*(\om^{(s)},x_s)=\hat{B_j^*}(\tau,z).
$$
Define then the sets $F_s$, $s=1,2,\ld$, inductively as follows:
$$
F_1:=
\Big(B_2\Big(f_{\tau|_{n_j}}(z),{1\over 4}\eta(\tau,z)\Big)\cap
  f^{-1}_{\sg^{n_j}(\om^{(1)})|_{n_{i_1}-n_j},f_{\tau|_{n_j}}(x_1)}
  \Big(B_2\Big(f_{\om^{(1)}|_{n_{i_1}}}(x_1),\frac14\eta(\om^{(1)},x_1)\Big)\Big)\Big)\sms \Sing(\tf)
$$
and 
$$
\begin{aligned}
F_{s+1}&:= 
\Big(\Big(B_2\Big(f_{\tau|_{n_j}}(z),{1\over 4}\eta(\tau,z)\Big)\,\cap\,\\
	&\cap f^{-1}_{\sg^{n_j}(\om^{(s+1)})|_{n_{i_{s+1}}-n_j},f_{\tau|_{n_j}}(x_{s+1})}
  \Big(B_2\Big(f_{\om^{(s+1)}|_{n_{i_{s+1}}}}(x_{s+1}), 
  \frac14\eta(\om^{(s+1)},x_{s+1})\Big)\Big)\Big)\sms \Sing(\tf)\Big)\Big) \sms\\
&\quad
\sms (F_1\cup F_2\cup\ld\cup F_s).
\end{aligned}
$$
By virtue of \eqref{1cmi2} we have that
$$
\hat{B}_j^*(\tau,z)
\sbt \bu_{s=1}^\infty[\tau|_{n_j}]\times f_{\tau|_{n_j},z}^{-1}(F_s)
=\bu_{s=1}^\infty\tf_{\tau|_{n_j},z}^{-n_j}\(\Sg_u\times F_s\),
$$
and, by the very definition of the sets $F_s$, the sets forming this union are mutually disjoint. Using Koebe's Distortion Theorem, we therefore get
$$
\begin{aligned}
m_t(\hat{B}_j^*(\tau,z))
&\le  \sum_{s=1}^\infty m_t\Big(\tf_{\tau|_{n_j},z}^{-n_j}\(\Sg_u\times F_s\)\Big)
 \le K^te^{-\P(t)n_j}\big|f_{\tau|_{n_j}}'(z)\big|^{-t}\sum_{s=1}^\infty m_t(\Sg_u\times F_s)\\
&=K^te^{-\P(t)n_j}\big|f_{\tau|_{n_j}}'(z)\big|^{-t}m_t\lt(\Sg_u\times\bu_{s=1}^\infty F_s\rt) 
 \le K^te^{-\P(t)n_j}\big|f_{\tau|_{n_j}}'(z)\big|^{-t} \\
&\le K^{2t}\frac{m_t\Big([\tau|_{n_j}]\times f_{\tau|_{n_j},z}^{-1}
  \Big(B_2\Big(f_{\tau|_{n_j}}(z),{1\over 4}\eta(\tau,z)\Big)\Big)\Big)}
  {m_t\(\Sg_u\times B_2\(f_{\tau|_{n_j}}(z),{1\over 4}\eta(\tau,z)\)\)}\\
&=K^{2t}\frac{m_t(B_j^*(\tau,z))}
  {m_t\(\Sg_u\times B_2\(f_{\tau|_{n_j}}(z),{1\over 4}\eta(\tau,z)\)\)}\\
&\le K^{2t}M(\tau,z)^{-1}m_t(B_j^*(\tau,z)),
\end{aligned}
$$
where 
$$
M(\tau,z):=\inf\{m_t\circ p_2^{-1}(B_2(\xi,\eta(\tau,z)/4):\xi\in J(G)\}>0
$$
since $J(G)$ is a compact set and $m_t\circ p_2^{-1}$ is positive on non--empty open subsets of $J(G)$. Therefore,
$$
\varlimsup_{j\to\infty}\lt(\d\(B_j^*(\tau,z)\)
 +\frac{m_t\(\hat{B}_j^*(\tau,z)\)}{m_t\(B_j^*(\tau,z)\)}\rt)
\le \lim_{j\to\infty}(e^{-n_j}+K^{2t}M(\tau,z)^{-1})
=K^{2t}M(\tau,z)^{-1}
<+\infty.
$$
Thus, all the hypothesis of Theorem~2.8.17 in \cite{federer}, p. 151, are verified 
and the proof of our lemma is complete. \endpf

\sp\fr As an immediate consequence of this lemma and Theorem~2.9.11, p. 158 in \cite{federer} we get the following.

\bprop\label{p1h73}
Let $G$ be a finitely generated *semi--hyperbolic rational semigroup generated by a $u$--tuple map $f=(f_{1},\ldots ,f_{u})\in \mbox{{\em Rat}}^u$. If $t\in\De_G$, then for every Borel set $A\sbt J(\tf)\sms \Sing(\tf)$ with 
$$
A_t:=\lt\{(\tau,z)\in A:\lim_{j\to\infty}\frac{m_t(A\cap B_j(\tau,z))}{m_t(B_j(\tau,z))}=1\rt\},
$$
then $m_t(A_t)=m_t(A)$.
\eprop

\fr Now we are in position to prove the following.

\blem\label{l1h64}
Let $G$ be a finitely generated *semi--hyperbolic rational semigroup generated by a $u$--tuple map $f=(f_{1},\ldots ,f_{u})\in \mbox{{\em Rat}}^u$. If $t\in\De_G$, then every generalized $t$--conformal measure $\nu$ for the skew product map $\tf:J(\tf)\lra J(\tf)$ is equivalent to $m_t$.
\elem

{\sl Proof.} Fix an integer $v\ge 1$ and let 
$$
I_v:=\{(\tau,z)\in J(\tf)\sms \Sing(\^f):\eta(\tau,z)\ge 1/v\},
$$
where $\eta(\tau,z)>0$ is the number produced in Proposition~\ref{pu6.1h23}. 
We may assume without loss of generality that $\eta (\tau ,z)\leq 1.$ Let
also $(\hat\tau,\hat z)$ and $(n_j)_{j=1}^\infty$ be the objects produced in this
proposition. Fix $(\tau,z)\in I_v$. Disregarding finitely many values
of $j$, we may assume without loss of generality that
$$
|f_{\tau|_{n_j}}(z)-\hat z|<{1\over 8}\eta(\tau,z).
$$
Keep $B_j(\tau,z)$ and $B_j^*(\tau,z)$ respectively given by formulas \eqref{2h67A} and \eqref{2h67C}.
By Koebe's Distortion Theorem and Proposition~\ref{pu6.1h23} we get that,
$$
\begin{aligned}
\nu(B_j(\tau,z))
&=   \nu \lt(\tf_{\tau |_{n_{j}},z}^{-n_j}\lt(\Sg_u\times B_2\lt(f_{\tau|_{n_j}}(z), {1\over 4}\eta(\tau,z)\rt)\rt)\rt) \\
&\comp e^{-\P(t)n_j}|f_{\tau|_{n_j}}'(z)|^{-t}\nu\lt(\Sg_u\times B_2\lt(f_{\tau|_{n_j}}(z), {1\over 4}\eta(\tau,z)\rt)\rt) \\
&=e^{-\P(t)n_j}|f_{\tau|_{n_j}}'(z)|^{-t}\nu\circ p_2^{-1}\lt(B_2\lt(f_{\tau|_{n_j}}(z), {1\over 2}\eta(\tau,z)\rt)\rt).
\end{aligned}
$$
Hence, 
$$
\nu(B_j(\tau,z))
\lek e^{-\P(t)n_j}|f_{\tau|_{n_j}}'(z)|^{-t}
$$
and 
$$
\begin{aligned}
\nu(B_j(\tau,z))
&\gek e^{-\P(t)n_j}|f_{\tau|_{n_j}}'(z)|^{-t}
   \nu\circ p_2^{-1}\lt(B_2\lt(f_{\tau|_{n_j}}(z), {1\over 4v}\rt)\rt) \\
&\gek e^{-\P(t)n_j}|f_{\tau|_{n_j}}'(z)|^{-t}
\inf\big\{\nu\circ p_{2}^{-1}(B_2(w,1/(4v)))\mid w\in J(G)\big\}>0.
\end{aligned}
$$ 
In conclusion
\beq\label{1h67}
M_v^{-1}e^{-\P(t)n_j}|f_{\tau|_{n_j}}'(z)|^{-t}
\le \nu(B_j(\tau,z)), m_t(B_j(\tau,z))
\le M_ve^{-\P(t)n_j}|f_{\tau|_{n_j}}'(z)|^{-t},
\eeq
with some constant $M_v\in (0,+\infty)$ depending only on $v$. 
Now suppose that $F\sbt J(\^f)$ is an arbitrary Borel set contained with $m_t(F)>0$. Since $t\in \De_G$ and since $J(\^f)\sms\Sing(\tf)=\bu_{v=1}^\infty I_v$, there thus exists $v\ge 1$ such that
$$
m_t(F\cap I_v)>0.
$$
By regularity of $m_t$ there then exists a compact set $E\sbt F\cap I_v$ such that 
$$
m_t(E)>0.
$$
Fix also $\varepsilon\in (0,m_t(E)/2)$. By outer regularity of $\nu$ and compactness of $E$ there now exists $\d>0$ such that 
\beq\label{1h69}
\nu(B(E,\d))\le \nu(E)+\varepsilon.
\eeq
Of course the family 
$$
\big\{\((\tau,z),B_j^*(\tau,z)\):(\tau,z)\in E 
\  \  {\rm and } \  \ 
B_j^*(\tau,z)\sbt B(E,\d/2)\big\}
$$
is fine at each point of $E$. Therefore, because of Lemma~\ref{l1h71}, property \eqref{def: Vitali relation item c} of the definition of Vitali relations yields the existence of some countable set $Y\sbt E$ and some function $j:Y\to\N$ such that the countable family $\{B_{j(y)}(y)\}_{y\in Y}$ consists of mutually disjoint sets,
$$
B_{j(y)}^*(y)\sbt B(E,\d/2) 
$$
for every $y\in Y$, and also
$$
m_t\lt(E\sms \bu_{y\in Y}B_{j(y)}^*(y)\rt)=0.
$$
But then also
$$
m_t\lt(E\sms \bu_{y\in Y}B_{j(y)}(y)\rt)=0,
$$
and
$$
B_{j(y)}(y)\sbt B(E,\d) 
$$
for every $y\in Y$, and, as the set $J(\^f)\sms \Sing(\^f)$ is dense in $J(\^f)$ and all the sets $B_j^*(\tau,z)$ are open relative to $J(\^f)$, the countable family $\{B_{j(y)}(y)\}_{y\in Y}$ consists of mutually disjoint sets too. Therefore,
$$
\begin{aligned}
\nu(F)
&\ge \nu(E)
 \ge \nu(B(E,\d))-\varepsilon
 \ge \nu\lt(\bu_{y\in Y}B_{j(y)}(y)\rt)-\ve \\
&\ge M_v^{-1}\sum_{(\tau,z)\in Y}e^{-\P(t)n_j(\tau,z)}|f_{\tau|_{n_j(\tau,z)}}'(z)|^{-t}-\ve \\
&\ge M_v^{-2}\sum_{(\tau,z)\in Y}m_t\(B_{j(\tau,z)}(\tau,z)\)-\ep \\
&=   M_v^{-2}m_t\lt(\bu_{y\in Y}B_{j(y)}(y)\rt)-\ep \\
&\ge M_v^{-2}m_t(E)-\ep.
\end{aligned}
$$ 
Letting $\ve\downto 0$, we thus get
$$
\nu(F)\ge M_v^{-2}m_t(E)>0.
$$
This shows that
$$
m_t|_{J(\^f)\sms\Sing(\tf)} \abs  \nu|_{J(\^f)\sms\Sing(\tf)}.
$$
Since in addition $m_t(\Sing(\tf))=0$ (as $t\in \De_G$), we thus get that
\beq\label{1cmi5}
m_t\abs \nu.
\eeq
By symmetry we also have that
$$
m_t|_{J(\^f)\sms\Sing(\tf)} \comp  \nu|_{J(\^f)\sms\Sing(\tf)}.
$$
Thus, in order to complete the proof it suffices to show that 
$$
\nu(\Sing(\tf))=0.
$$ 
To do this, suppose on the contrary that $\nu(\Sing(\tf))>0$. Since $\tf'$ vanishes on $\Crit(\tf)$, the measure 
$$
\nu_0=(\nu(\Sing(\tf)))^{-1}\nu|_{\Sing(\tf)},
$$ 
is generalized $t$--conformal for $\tf:J(\^f)\lra J(\^f)$. But
then (\ref{1cmi5}) would be true with $\nu$ replaced by $\nu_0$. We would thus have 
$m_t(J(\^f)\sms\Sing(\tf))=0$. Since $m_t(\Sing(\tf))=0$,
we would get $m_t(J(\^f))=0$. This contradiction finishes the proof.
\endpf

\fr Now, we shall prove the following. 

\bprop\label{l2h73}
Let $G$ be a finitely generated *semi--hyperbolic rational semigroup generated by a $u$--tuple map $f=(f_{1},\ldots ,f_{u})\in \mbox{{\em Rat}}^u$. If $t\in\De_G$, then the measure $m_t$ is weakly metrically exact for the skew product map $\tf:J(\^f)\lra J(\^f)$. In particular, it is ergodic and conservative.
\eprop

{\sl Proof.} Fix a Borel set $F\sbt J(\^f)\sms \Sing(\tf)$ with $m_t(F)>0$. Proposition~\ref{p1h73} there exists at least one point $(\tau,z)\in F_t$. Our first goal is to show that
\beq\label{1h73}
\lim_{j\to\infty}{m_t\(\tf^{n_j}(F)\cap p_2^{-1}(B_2(f_{\tau|_{n_j}}(z),\eta/2))\) \over
     m_t\(p_2^{-1}\(B_2(f_{\tau|_{n_j}}(z),\eta/2))\)}
     =1,
\eeq
where, we recall $\eta=\eta(\tau,z)>0$ is the number produced in Proposition~\ref{pu6.1h23} and
$(n_j)_{j=1}^\infty$ is the corresponding sequence produced there. Indeed, suppose for the contrary
that
$$
\ka={1\over 2}\liminf_{j\to\infty}{m_t\(p_2^{-1}\(B_2(f_{\tau|_{n_j}}(z),\eta/2))\sms
 \tf^{n_j}(F)\) \over m_t\(p_2^{-1}(B_2(f_{\tau|_{n_j}}(z),\eta/2))\)}
    >0.
$$
Then, disregarding finitely many $j$'s, we may assume that
$$
{m_t\(p_2^{-1}(B_2(f_{\tau|_{n_j}}(z),\eta/2))\sms \tf^{n_j}(F)\) \over
    m_t\(p_2^{-1}(B_2(f_{\tau|_{n_j}}(z),\eta/2))\)}
\ge\ka>0
$$
for all $j\ge 1$ and some $\ka>0$. But
$$
\aligned
\tf_{\tau |_{n_{j}},z}^{-n_j}\lt(p_2^{-1}\(B_2(f_{\tau|_{n_j}}(z),\eta/2)\)\sms \tf^{n_j}(F)\rt)
&\sbt \lt([\tau|_{n_j}]\times B_2\lt(z,{1\over 2}K\eta\lt|f_{\tau|_{n_j}}'(z)\rt|^{-1}\rt)\rt)\sms F \\
&=    B_j(\tau,z)\sms F
\endaligned
$$
and
$$
\aligned
m_t\(\tf_{\tau |_{n_{j}},z}^{-n_j}\(p_2^{-1}(B_2(f_{\tau|_{n_j}}(z)&,\eta/2))\sms f^{n_j}(F)\)\) \ge \\
&\ge K^{-t}e^{-\P(t)n_j}\lt|f_{\tau|_{n_j}}'(z)\rt|^{-t}m_t\(p_2^{-1}(B_2(f_{\tau|_{n_j}}(z),\eta/2))\sms f^{n_j}(F)\) \\
&\ge \ka K^{-t}e^{-\P(t)n_j}\lt|f_{\tau|_{n_j}}'(z)\rt|^{-t}m_t\(p_2^{-1}(B_2(f_{\tau|_{n_j}}(z),\eta/2))\) \\
&= \ka K^{-h}e^{-\P(t)n_j}\lt|f_{\tau|_{n_j}}'(z)\rt|^{-t}m_t\circ p_2^{-1}\(B_2(f_{\tau|_{n_j}}(z),\eta/2)\) \\
&\ge \ka K^{-t}Q_{\eta/2}e^{-\P(t)n_j}\lt|f_{\tau|_{n_j}}'(z)\rt|^{-t},
\endaligned
$$
where 
$$
Q_{\eta/2}:=\inf\big\{m_t\circ p_{2}^{-1}(B_2(w,1/\eta/2))\mid w\in J(G)\big\}>0
$$
as $\supp(m_t\circ p_{2}^{-1})=J(G)$. 
Hence, making use of Koebe's Distortion Theorem and the generalized conformality of $m_t$, we obtain, we obtain
$$
\aligned
m_t(B_j(\tau,z)\sms F)
&\ge \ka K^{-t}Q_{\eta/2}e^{-\P(t)n_j}\lt|f_{\tau|_{n_j}}'(z)\rt|^{-t} \\
&\ge \ka K^{-2t}Q_{\eta/2}m_t(B_j(\tau,z)).
\endaligned
$$
Thus,
$$
{m_t(B_j(\tau,z)\sms F)\over m_t(B_j(\tau,z))}
\ge \ka K^{-2t}Q_{\eta/2}>0.
$$
Letting $j\to\infty$ this contradicts the fact that $(\tau,z)\in F_t$ and finishes the proof
of (\ref{1h73}). Now since $\tf:J(\^f)\to J(\^f)$ is topologically exact, there exists $q\ge 0$
such that $\tf^q(p_2^{-1}(B_2(w,\eta/2)))\spt J(\^f)$ for all $w\in J(G)$. It then easily
follows from (\ref{1h73}) and conformality of $m_t$ that
$$
\limsup_{k\to\infty} m_t(\tf^k(F))
\ge \limsup_{j\to\infty} m_t(\tf^{q+n_j}(F))
=1.
$$
Since $m_t(\Sing(\tf))=0$ (as $t\in\De_G$), we are therefore done. 
\endpf

\sp\fr As an immediate consequence of this proposition we get the following.

\bcor\label{c12015_01_30}
Assume that $G$ is a finitely generated *semi--hyperbolic rational semigroup generated by a $u$--tuple map $f=(f_{1},\ldots ,f_{u})\in \mbox{{\em Rat}}^u$. If $t\in\De_G$, then $m_t(\Trans(\tf))=1$, where, we recall, $\Trans(\tf)$ is the set of transitive points of $\tf:J(\tf)\lra J(\tf)$, i.e. the set of points $z\in J(\^f)$ such that the set $\{\^f^n(z):n\ge 0\}$ is dense in $J(\^f)$.
\ecor

\bcor\label{c1h77}
Let $G$ be a finitely generated *semi--hyperbolic rational semigroup generated by a $u$--tuple map $f=(f_{1},\ldots ,f_{u})\in \mbox{{\em Rat}}^u$. If $t\in\De_G$, then $m_t$ is the only generalized $t$--conformal measure on $J(\tf)$ for the map $\tf:J(\tf)\lra J(\tf)$. 
\ecor

{\sl Proof.} Let $\nu$ be an arbitrary generalized $t$--conformal measure on $J(\^f)$ for the 
map $\tf:J(\^f)\lra J(\^f)$. Since, by Lemma~\ref{l1h64} the measure $\nu$ is absolutely 
continuous with respect $m_t$, it follows from Theorem~2.9.7 in \cite{federer}, p. 155
and Lemma~\ref{l1h71} that for $m_t$--a.e. $(\tau,z)\in J(\^f)\sms \Sing(\tf)$,
$$
{d\nu\over dm_t}(\tf(\tau,z))
=  \lim_{j\to\infty}{\nu(B_j^*(\tf(\tau,z))) \over m_t(B_j^*(\tf(\tau,z)))}
$$
and 
$$
\lim_{j\rightarrow \infty }\frac{\nu (\tf(B_{j}^*(\tau ,z)))}
{m_t (\tf (B_{j}^*(\tau ,z)))}\\ 
= \lim _{j\rightarrow \infty }\frac{\int_{B_{j}^*(\tau ,z)}|\tf'|^{t}\ d\nu }
{\int _{B_{j}^*(\tau ,z)} |\tf'|^{t}\, dm_t} 
=\lim _{j\rightarrow \infty }\frac{\nu (B_{j}^*(\tau, z))}{m_t(B_{j}^*(\tau ,z))}
={d\nu\over dm_t}(\tau,z).
$$
But, a straightforward calculation shows that $\tf(B_{j}^*(\tau,z))$ is equal to the set 
$B_{j}^*(\tf(\tau,z))$ obtained with the redefined functions $\eta(\xi):=\eta(\tf(\xi))$ and the corresponding sequence $n_{j-1}(\xi)$. Therefore by the same token as above
$$
{d\nu\over dm_t}(\tf(\tau,z))
=\lim_{j\to\infty}\frac{\nu (\tf(B_{j}^*(\tau ,z)))}{m_t (\tf (B_{j}^*(\tau ,z)))} 
$$
for $m_t$--a.e. $(\tau,z)\in J(\^f)\sms \Sing(\tf)$. In consequence
$$
{d\nu\over dm_t}(\tf(\tau,z))= {d\nu\over dm_t}(\tau,z)
$$
for $m_t$--a.e. $(\tau,z)\in J(\^f)\sms \Sing(\tf)$.
Since, by Proposition~\ref{l2h73}, the measure $m_t$ is ergodic, it follows that the 
Radon--Nikodym derivative ${d\nu\over dm_t}$ is $m_t$--almost everywhere constant.
Thus $\nu=m_t$, and we are done. 
\endpf

\sp Now we pass to consider invariant measures absolutely continuous (and equivalent) to generalized conformal measures $m_t$. We first show their existence. Indeed, in order to prove the existence of a Borel probability $\tf$--invariant measure on 
$J(\^f)$ equivalent to $m_t$, $t\in\De_G$, we will use Marco Martens's method which originated in \cite{martens}.
This means that we shall first produce a $\sg$--finite $\tf$--invariant measure equivalent
to $m_t$ (this is the Martens method) and then we will prove this measure to be
finite. The heart of the Martens method is the following theorem which is a generalization
of Proposition~2.6 from \cite{martens} proved by us in \cite{sush} as Theorem~8.13. It is a generalization of Martens's original result in the sense that we neither
assume our probability space $(X,\mathcal{B},m )$ below to be a $\sg$--compact metric space,
nor we assume that our map is conservative. Instead, we merely assume that item \eqref{d:mmmap item 6}
in Definition~\ref{d:mmmap} holds. 

\bdfn
\label{d:mmmap}
Suppose $(X,\mathcal{B},m )$ is a probability space. 
Suppose $T:X\rightarrow X$ is a measurable mapping, such that 
$T(A)\in \mathcal{B}$ whenever $A\in \mathcal{B}$, and such that 
the measure $m$ is quasi--invariant with respect to $T$, 
meaning that $m\circ T^{-1}\prec m.$ Suppose further that 
there exists a countable family $\{ X_n\} _{n=0}^{\infty }$ 
of subsets of $X$ with the following properties.

\sp\begin{itemize}
\item[\mylabel{1}{d:mmmap item 1}]
For all $n\geq 0$, $X_{n}\in \mathcal{B}.$ 

\sp\item[\mylabel{2}{d:mmmap item 2}]
$m\(X\setminus \bigcup _{n=0}^{\infty }X_{n}\)=0.$ 

\sp\item[\mylabel{3}{d:mmmap item 3}] 
For all integers $m,n\geq 0$, there exists an integer $j\geq 0$ such that 
$m(X_{m}\cap T^{-j}(X_{n}))>0.$ 

\sp\item[\mylabel{4}{d:mmmap item 4}]
For all $j\geq 0$ there exists a $K_{j}\geq 1$ such that 
for all $A,B\in \mathcal{B}$ with $A,B\subset X_{j}$ and for all 
$n\geq 0$, 
$$
m(T^{-n}(A))m(B)\leq K_{j}m(A)m(T^{-n}(B)).
$$ 
\item[\mylabel{5}{d:mmmap item 5}] 
$\sum _{n=0}^{\infty }m(T^{-n}(X_{0}))=+\infty .$

\sp\item[\mylabel{6}{d:mmmap item 6}]
$$
\lim _{l\rightarrow \infty }m\lt(T\Big(\bigcup _{j=l}^{\infty }Y_{j}\Big)\rt)=0,
$$ 
where $Y_{j}:= X_{j}\setminus \bigcup _{i<j} X_{i}.$  
\end{itemize}

\sp\fr Then the map $T:X\lra X$ is called a \textbf{Martens map}\index{Martens map} and 
$\{ X_{j}\} _{j=0}^{\infty }$ is called a \textbf{Martens cover}\index{Martens cover}. 
\edfn 

\brem\label{r:mmmapf1}
Let us record the following observations.
\begin{itemize} 
\item[\mylabel{a}{r:mmmapf1 item a}] Of course, condition~\eqref{d:mmmap item 2} follows from the stronger hypothesis that $\bigcup_{n=0}^\infty X_n=X$. 

\,

\item[\mylabel{b}{r:mmmapf1 item b}] condition~\eqref{d:mmmap item 3} imposes that $m(X_n)>0$ for all $n\geq0$.  

\,

\item[\mylabel{c}{r:mmmapf1 item c}] If $T$ is conservative with respect to $\mu$, then condition~\eqref{d:mmmap item 5} is fulfilled.

\,

%\item[(4)] In conditions~(f-g), note that 
%          $\bigcup_{j=l}^\infty Y_j=\bigcup_{j=0}^\infty X_j\backslash\bigcup_{i<l}X_i\sbt X\backslash\bigcup_{i<l}X_i$.
\item[\mylabel{d}{r:mmmapf1 item d}] If the map $T:X\rightarrow X$ is finite--to--one, then condition~\eqref{d:mmmap item 6} is satisfied. For, if $T$ is finite--to--one, 
then $\bigcap _{l=1}^{\infty }T\(\bigcup _{j=l}^{\infty }Y_{j}\)=\emptyset .$ 
\end{itemize}
\erem

\bthm\label{t1h75}
If $(X,\mathcal{B},m)$ is a probability space and $T:X\longrightarrow X$ is a Martens map with a Martens cover $\{ X_{j}\} _{j=0}^{\infty }$, 
then

\sp\begin{itemize}
\item There exists a $\sigma$--finite $T$--invariant 
measure $\mu $ on $X$ equivalent to $m.$ 

\sp\item In addition, $0<\mu (X_{j})<+\infty $ for each $j\geq 0.$ 
\end{itemize}

\sp\fr The measure $\mu $ is constructed in the following way: 
Let $l_{B}:l^\infty\longrightarrow \R$ be a Banach limit. 
For each $A\in \mathcal{B}$, set 
$$
m_{n}(A):=\frac{\sum _{k=0}^{n}m(T^{-k}(A))}{\sum _{k=0}^{n}m(T^{-k}(X_{0}))}.
$$
If $A\in \mathcal{B}$ and $A\subset Y_{j}$ with some $j\geq 0$, 
then we obtain $(m_{n}(A))_{n=1}^{\infty }\in l^\infty.$ 
We set 
$$\
\mu (A):= l_{B}((m_{n}(A))_{n=1}^{\infty }).
$$ 
For a general measurable subset $A\subset X$, set 
$$\mu (A):=\sum _{j=0}^{\infty }\mu (A\cap Y_{j}).$$ 
In addition, if for a measurable subset $A\subset X$,  
the sequence $(m_{n}(A))_{n=1}^{\infty }$ is bounded, 
then we have the following formula: 
\begin{equation}
\label{eq:muequ}
\mu (A)
=l_{B}\((m_{n}(A))_{n=1}^{\infty }\)
-\lim _{l\rightarrow \infty }
l_{B}\lt(\lt(m_{n}\Big(A\cap \bigcup _{j=l}^{\infty }Y_{j}\Big)\rt)_{n=0}^{\infty}\rt).
\end{equation}
In particular, if $A\in\mathcal{B}$ is contained in a finite union 
of sets $X_j$, $j\geq0$, then
\[
\mu(A)=l_B\bigl((m_n(A))_{n=1}^\infty\bigr). 
\]
Furthermore, if the measure--preserving transformation $T:X\longrightarrow X$ is ergodic (equivalently with respect to  the measure $m$ or $\mu $), 
then the $T$--invariant measure $\mu $ is unique up to a multiplicative constant. 
\ethm

Now we are ready to prove the following.

\bthm\label{t4h65}
Let $G$ be a finitely generated *semi--hyperbolic rational semigroup. If $t\in \De_G$, then there exists a unique, up to a multiplicative constant, Borel $\sg$--finite $\tf$--invariant measure $\mu_t$\index{invariant measure! $\mu_t$}\index{$\mu_t$} on $J(\^f)$ which is absolutely continuous with respect to $m_t$. In addition, the measure $\mu_t$ is weakly metrically exact and equivalent to $m_t$, in particular it is ergodic.
\ethm

{\sl Proof.} Since the topological support of $m_t$ is equal to
the Julia set $J(\^f)$ and since, by Lemma~\ref{l1h63}, $\PCV(\tf)$ is a nowhere dense subset of $J(\^f)$, we have that
\beq\label{1_2016_05_16}
m_t(\PCV(\tf))<1.
\eeq
Since the set $\PCV(\tf)$ is forward invariant under $\tf$, it thus follows from ergodicity and conservativity of $m_t$ (see Proposition~\ref{l2h73}) that \beq\label{1_2016_05_16D}
m_t(\PCV(\tf))=0. 
\eeq
Now, because of Lemma~\ref{l2mc3}, for every point $(\om,z)\in J(\^f)\sms \PCV(\tf))$ there exists a radius $r_{(\om,z)}>0$ such that
$$
\Sg_u\times B_2\(z,r_{(\om,z)}\)
\sbt \Sg_u\times B_2\(z,2r_{(\om,z)}\)
\sbt J(\^f)\sms \PCV(\tf)).
$$ 
Since $J(\^f)\sms \PCV(\tf))$ is a separable metrizable space, Lindel\"of's Theorem yields the existence of a countable set $\big\{(\om^{(j)},z_j)\big\}_{j=0}^\infty\sbt
J(\^f)\sms \PCV(\tf))$ such that
\beq\label{2_2016_05_16}
\bu_{j=0}^\infty \Sg_u\times B_2\(z,r_{(\om^{(j)},z_j)}\)  
\spt J(\^f)\sms \PCV(\tf)).
\eeq
For every $j\ge 0$ set
$$
X_j:=\Sg_u\times B_2\(z,r_{(\om^{(j)},z_j)}\).
$$
Verifying the conditions of Definition~\ref{d:mmmap} (with $X:=J(\^f), T:=\tf, m:=m_t$), we note that $\tf $ is nonsingular because of generalized $t$--conformality of $m_t$ and since $t\in\De_G$. We immediately see that condition~\eqref{d:mmmap item 1} is
satisfied, that \eqref{d:mmmap item 2} holds because of \eqref{1_2016_05_16} and \eqref{2_2016_05_16}, and that \eqref{d:mmmap item 3} holds because of generalized $t$--conformality of $m_t$ and topological exactness of the map $\tf:J(\^f)\to J(\^f)$. Condition~\eqref{d:mmmap item 5} follows directly from ergodicity and conservativity of the measure $m_t$. Condition~\eqref{d:mmmap item 6} follows since $\tilde{f}:J(\^f)\rightarrow J(\^f)$ is finite--to--one (see Remark~\ref{r:mmmapf1}).

\sp\fr Let us prove condition~\eqref{d:mmmap item 4}. For every $j\ge 0$ put
$$
B_j:=B_2\(z,r_{(\om^{(j)},z_j)}\) 
\  \  \  {\rm and} \  \  \
2B_j:=B_2\(z,2r_{(\om^{(j)},z_j)}\).
$$
Because of Koebe's Distortion Theorem and generalized $t$--conformality of the measure $m_t$, we have
$$
\aligned
m_t\circ \tf^{-n}(A)
&=    m_t\lt(\bu_{\rho\in\HIB_n(2B_j)}\tf_\rho^{-n}(A)\rt)
 =   \sum_{\rho\in\HIB_n(2B_j)}m_t\(\tf_\rho^{-n}(A)\) \\
&\le \sum_{\rho\in\HIB_n(2B_j)}K^te^{-\P(t)n}\big|(\^f_\rho^{-n})'(\om^{(j)},z_j)\big|^tm_t(A) \\
&=   K^{2t}{m_t(A)\over m_t(B)}\sum_{\rho\in\HIB_n(2B_j)}
     K^{-t}e^{-\P(t)n}\big|(f_\rho^{-n})'(\om^{(j)},z_j)\big|^tm_t(B) \\
&\le K^{2t}{m_t(A)\over m_t(B)}\sum_{\rho\in\HIB_n(2B_j)}m_t\(\tf_\rho^{-n}(B)\) \\
&=   K^{2t}{m_t(A)\over m_t(B)}m_t\lt(\bu_{\rho\in\HIB_n(2B_j)}\tf_\rho^{-n}(B)\rt) \\
&=   K^{2t}m_t\circ \tf^{-n}(B){m_t(A)\over m_t(B)}.
\endaligned
$$
Hence,
$$
{m_t\circ \tf^{-n}(A)\over m_t\circ \tf^{-n}(B)}
\le K^{2t}{m_t(A)\over m_t(B)},
$$
and consequently, condition~\eqref{d:mmmap item 4} of Definition~\ref{d:mmmap}
is satisfied. Therefore, Theorem~\ref{t1h75} produces
a Borel $\sg$--finite $\tf$--invariant measure $\mu_t$ on $J(\^f)$, equivalent to $m_t$. The proof is complete.
\endpf

\

\part{Ergodic Theory and Dynamics of Totally and Finely Non--Recurrent Rational Semigroups}

\sp\section{Totally Non--Recurrent and Finely Non--Recurrent Rational Semigroups} 

In this short section we introduce various classes of rational semigroups with which we will deal in subsequent sections, and we establish their most basic properties.

\bdfn\label{d1nsii5}
We say that a finitely generated rational semigroup $G$ generated by a $u$--tuple map $f:=(f_1,\ld,f_u)\in \Rat^u$ and satisfying the Fundamental Assumption is \textbf{totally non--recurrent} (\TNR) if and only if \index{totally non--recurrent (TNR)}

\begin{itemize}

\sp\item[\mylabel{a}{d1nsii5 item a}] For each $z\in J(G)$ there exists a neighborhood $U$ of $z$ in $\oc$ (in fact in $\C$)
such that for any sequence $\{ g_{n}\} _{n=1}^{\infty }$ in $G$, 
any domain $V$ in $\oc $ and any point 
$\zeta \in U$, the sequence $\{ g_{n}\} _{n=1}^{\infty }$ does not converge 
to $\zeta $ locally uniformly on $V$ 

\,

\fr and

\sp\item[\mylabel{b}{d1nsii5 item b}] 
$$
\Crit_*(f)\cap \PCV(G)=\es. 
$$
\end{itemize}
\edfn

\brem\label{r120190917}
Note that if two $u$--tuple maps $f$ and $h$ generate the same rational semigroup $G$, then
$$
\Crit_*(f)\sbt G^{-1}(\Crit_*(h))
\  \  \  {\rm and} \  \  \
\Crit_*(h)\sbt G^{-1}(\Crit_*(f)).
$$
Since also the set $\PCV(G)$ is forward invariant under $G$, i.e. $G(\PCV(G))\sbt \PCV(G)$, it follows that item \eqref{d1nsii5 item b} of Definition~\ref{d1nsii5} holds for $f$ if and only if it holds for $h$. In conclusion, the concept of being totally non--recurrent is independent of the choice of generators and depends on the semigroup alone.
\erem

\fr The following observations are immediate consequences of Theorem~\ref{t2.9h3} and Theorem~\ref{t2.9h3B}.

\bobs\label{o1nsii6}
Every \TNR \ rational semigroup is a finitely generated *semi--hyperbolic semigroup.
\eobs

\bobs\label{o20190529}
Let $f=(f_{1},\ldots ,f_{u})\in \mbox{Rat}^{u}$ be a $u$--tuple map and let $G=\langle f_{1},\ldots ,f_{u}\rangle$. If $G$ is *semi--hyperbolic and $\Crit_*(f)$ has at most one element, then $G$ is a \TNR \ rational semigroup. 
\eobs

\fr Now we shall provide a characterization of TNR semigroups more in terms of the skew product map $\^f:\Sg_u\times\hat\C\lra \Sg_u\times\hat\C$. 

\bprop\label{p2nsii5}
A finitely generated rational semigroup $G$ generated by a $u$--tuple map $f:=(f_1,\ld,f_u)\in \Rat^u$ and satisfying the Fundamental Assumption is \TNR \ if and only if the following conditions are satisfied.

\sp\begin{itemize}
\item[\mylabel{a}{p2nsii5 item a}]  For each $z\in J(G)$ there exists a neighborhood $U$ of $z$ in $\oc $ 
such that for any sequence $\{g_{n}\} _{n=1}^{\infty }$ in $G$, 
any domain $V$ in $\oc $ and any point 
$\zeta \in U$, the sequence $\{ g_{n}\} _{n=1}^{\infty }$ does not converge 
to $\zeta $ locally uniformly on $V.$ 

\sp\item[\mylabel{b}{p2nsii5 item b}] 
$$
\Crit_*(\^f)\cap \PCV(\^f)=\es.
$$
\end{itemize}
\eprop

\bpf
We are to show that conditions \eqref{d1nsii5 item b} of Definition~\ref{d1nsii5} and \eqref{p2nsii5 item b} of Proposition~\ref{p2nsii5} are equivalent, and this follows immediately from Lemma~\ref{l2mc3}. The proof is complete.
\epf

\bdfn\label{d1_2016_06_19}
A rational semigroup $G$ is called \textbf{\textup{C--F} balanced} \index{C--F balanced} if and only if
$$
D(G):=\dist_\C\(J(G),\PCV(G)\cap F(G)\)>0.
$$\index{C--F balanced!$D(G)$}\index{$D(G)$}
\edfn

\brem\label{r220190917}
Of course the number $D(G)$ depends only on the semigroup $G$ and not on any set of generators. The same is therefore true of the concept of being {\rm C--F} balanced.
\erem

\sp\fr Assume that $G$ is finitely generated and let $f:=(f_1,\ld,f_u)\in \Rat^u$ be a $u$--tuple map generating $G$. Given a point $c\in \Crit_*(f)$ we define
$$
\Xi(c):=\Big\{\om\in \Sg_u^*:f_{\om_1}'(c)=0,\, f_\om(c)\in F(G) \and f_{\om|_k}(c)\in J(G)\ {\rm \text{for every} }\, 0\le k\le |\om|-1\Big\}.
$$\index{$\Xi(c)$}
The following observation follows from Theorem 1.36 in \cite{hiroki2}. 

\bobs\label{o1_2016_06_14}
If $G$ is a finitely generated *semi--hyperbolic rational semigroup generated by a $u$--tuple map $f:=(f_1,\ld,f_u)\in \Rat^u$, then $G$ is {\rm C--F} balanced if and only if
$$
\dist_\C\lt(J(G), F(G)\cap \bu_{c\in \Crit_*(f)}\{f_\om(c):\om\in\Xi(c)\}\rt)>0.
$$
\eobs

\bdfn\label{d2_2016_06_19}
A finitely generated rational semigroup $G$ generated by a $u$--tuple map $f:=(f_1,\ld,f_u)\in \Rat^u$ is called \textbf{\textup{J--F} balanced} \index{J--F balanced} if and only if
$$
\dist_\C\lt(J(G),\bu_{j=1}^u f_j(J(G))\cap F(G)\rt)>0.
$$
\edfn

\brem\label{r320190917}
It is easy to see that the concept of being {\rm J--F} balanced depends only on the semigroup $G$ alone and not on any set of its generators.
\erem

\fr Using again Theorem~1.36 in \cite{hiroki2}, we immediately have the following.

\bobs\label{o2_2016_06_14}
If $G$ is a {\rm J--F} balanced *semi--hyperbolic rational semigroup, then $G$ is {\rm C--F} balanced.
\eobs

\bdfn\label{d720220816}
A rational semigroup $G$ generated by a $u$--tuple map $f:=(f_1,\ld,f_u)\in \Rat^u$ is said to be of \textbf{finite type} \index{finite type (semigroup)} if the set $\Crit_*(\tf)$, i.e. the set of all critical points of $\tf$ lying in the Julia set $J(\^f)$, is finite.
\edfn

\brem\label{r420190917}
It is easy to see that the concept of being of finite type depends only on the semigroup $G$ alone and not on any set of its generators.
\erem

\brem\label{r120180618}
We would like to mention that the Nice Open Set Condition, the one assumed in \cite{sush}, treated at length in Section~\ref{NOSC}, and indispensable for the whole Part~3, entails, see Lemma~\ref{l320190325}, the semigroup $G$ to be of finite type.
\erem

Our ultimate definition for this part of our manuscript and being a standing assumption in the next part as well is the following. 

\bdfn\label{d920220816}
Any \CF\ balanced \TNR \ rational semigroup of finite type is called \textbf{finely non--recurrent} and is abbreviated as \FNR.  \index{finely non--recurrent (FNR)}
\edfn

\section{Nice Sets (Families)}\label{sec:nicesets}

In this Section~\ref{sec:nicesets}, we explore in detail one of the most important tools for us in the current manuscript. It is commonly referred to as nice sets or nice families. It has been introduced in \cite{Riv07} in the context of rational functions (cyclic semigroup), and extensively used, among others in \cite{PrzRiv07}. We adopt this concept to the setting of rational semigroups. We would like to emphasize that the Nice Open Set Condition and Nice Sets (Families) are totally independent concepts. In particular, the adjective ``Nice" was  independently introduced for both concepts many years ago. Although it may be a little bit confusing for some readers, we stick to the historical terminology to respect history and in order not confuse readers even more by inventing yet new names. We think that in our current manuscript this is the first time in the literature that both ``nice" concepts are used simultaneously.

The absolute first fact needed about nice sets and nice families is their existence. It is by no means obvious and we devote the whole current section for this task. In the existing proofs for ordinary conformal systems, i.e. cyclic semigroups, the concept of connectivity of the phase space, usually $\C$ or $\oc$, plays a substantial role. In our present setting of the skew product map 
$$
\^f:\Sg_u\times\hat\C\lra \Sg_u\times\hat\C,
$$
the phase space is ``highly'' not connected. In order to overcome this difficulty we define the concept of connected families of arbitrary sets. These have sufficiently many properties of ordinary connected sets, e.g. one can speak of connected components of any family of sets, to allow for a proof of the existence of nice families. As a matter of fact, we do not even use the topological concept of connected subsets of the Riemann sphere $\oc$.

\bdfn
	We say that a family $\cF$ of non--empty subsets of $\Sg_u\times\oc$ is a \textbf{connected family}\index{connected family} if for any two elements $D, H\in\cF$ there exist finitely many sets, call them $F_1, F_2,\ld,F_n$ such that 
	
	\sp\begin{itemize}
		\item[(a)] $F_1=D$ and $F_n=H$, and
		
		\sp\item[(b)] $F_j\cap F_{j+1}\ne\es$ for all $j=1, 2,\ld,n-1$.
	\end{itemize}
\edfn 

\sp\fr From now on throughout this section we assume that $f:=(f_1,\ld,f_u)\in \Rat^u$ and $G=\langle f_1,\ld,f_u\rangle$. We record the following immediate observations.

\bobs\label{o1nsii10}
If $\cF$ is a connected family of subsets of $\Sg_u\times\oc$, then so is the family
$$
\^f(\cF):=\{\^f(F):F\in\cF\}.
$$
\eobs

\bobs\label{o2nsii10}
If $\cF$ and $\cH$ are two connected families of subsets of $\Sg_u\times\oc$ and $F\cap H\ne\es$ for some sets $F\in\cF$ and $H\in\cH$, then the family $\cF\cup\cH$ is also connected.
\eobs

\fr This latter observation enables us to speak  about connected components of any given family $\cF$ of subsets of $\Sg_u\times\oc$. These are, by definition, maximal connected subfamilies of $\cF$. These components are mutually disjoint but even more, the unions of all their elements are mutually disjoint. In addition, the union of all connected components of $\cF$ is equal to $\cF$.

Let $V\sbt\C$ be a non--empty, open, connected, simply connected set such that 
$$
V\cap \PCV(G)=\es.
$$
Given an integer $n\ge 0$ we set
$$
\HIB_{\le n}(V):=\bu_{k=0}^n\HIB_k(V)=\{\rho\in \HIB(V):\|\rho\|\le n\}.
$$\index{holomorphic inverse branch! $\HIB_{\le n}(V)$}\index{$\HIB_{\le n}(V)$}
We also define
$$
\begin{aligned}
\Pre(V)  &:=\big\{\^f_\rho^{-\|\rho\|}(\Sg_u\times V):\rho\in \HIB(V)\big\}, \\
\Pre_n(V)&:=\big\{\^f_\rho^{-n}(\Sg_u\times V):\rho\in \HIB_n(V)\big\}, \\
\Pre_{\le n}(V)&:=\big\{\^f_\rho^{-\|\rho\|}(\Sg_u\times V):\rho\in \HIB_{\le n}(V)\big\}.
\end{aligned}
$$\index{$\Pre(V)$}\index{$\Pre_n(V)$}\index{$\Pre_{\le n}(V)$}
From now on throughout this section $G$ is assumed to be a \TNR \ rational semigroup  generated by a $u$--tuple map $(f_1,\ld,f_u)\in \Rat^u$. We can then take, and from now on fix, an arbitrary
$$
R_*(G)\in\bigg(0,\frac1{16}\min\Big\{\dist_\C\(\Crit_*(f),\PCV(G)\),\min\big\{|c_2-c_1|: c_1, c_2\in \Crit(G), \  c_1\ne c_2\big\}\Big\}\bigg)
$$\index{$R_*(G)$}
so small as required in the Exponential Shrinking Property, i.e. in Theorem~\ref{t1h4}. Because of this theorem~\ref{t1h4}, for every $\e>0$ there exists $\d_\e\in(0,\e/4)$ such that
\beq\label{1_2016_06_03B}
\diam_\C(W)<\e
\eeq
for every element $g\in G$, every $\xi\in J(G)$, and $W$, any connected component of $g^{-1}(B_2(\xi,\d_\e))$. We shall prove the following.

\sp

\blem\label{l1_2016_06_02} 
Let $G$ be a \TNR \ rational semigroup generated by a $u$--tuple map $(f_1\ld,f_u)\in \Rat^u$. Fix $R>0$. Fix any $r\in (0,\d_R/4)$. Fix a finite set $S\sbt J(G)\sms B_2(\PCV(G),8R)$. Suppose that $\tf_\a^{-\|\a\|}(\Sg_u\times B_2(a,r))$ and $\tf_\b^{-\|\b\|}(\Sg_u\times B_2(b,r))$ are two arbitrary elements of 
$$
\Pre(S,r):=\bu_{s\in S}\Pre(B_2(s,r))
$$ 
such that
\beq\label{1nsii9}
\tf_\a^{-\|\a\|}(\Sg_u\times B_2(a,4r)) \cap \tf_\b^{-\|\b\|}(\Sg_u\times B_2(b,4r))\ne\es.
\eeq
Assume without loss of generality that $|\^\b|\ge |\^\a|$, with the `` \,$\tilde{}$\,'' operation defined by formulas \eqref{720200303}--\eqref{920200303}. Then 
\beq\label{2nsii10}
\b=\a\g,
\eeq
where $\g$ represents the holomorphic branch of $\tf^{-(\|\b\|-\|\a\|)}$ defined on $\Sg_u\times B_2(b,4R)$, in particular belonging to $\HIB(B_2(b,4R))$, and sending any point $(\om,z)\in \Sg_u\times B_2(b,4R)$ to the point $\(\^{\b}|_{|\^\b|-|\^\a|}\om,f_{\^\a}\(f_\b^{-1}(z)\)\)$. Formula \eqref{2nsii10} more precisely means that
\beq\label{1_2016_06_03}
\tf_\g^{-\|\g\|}\(\Sg_u\times B_2(b,4r)\)\sbt \Sg_u\times B_2(a,2R)
\eeq
and 
\beq\label{2_2016_06_03}
\tf_\b^{-\|\b\|}\big|_{\Sg_u\times B_2(b,4r)}=\tf_\a^{-\|\a\|}\circ\tf_\g^{-\|\g\|}\big|_{\Sg_u\times B_2(b,4r)}.
\eeq
Furthermore,

\sp\begin{itemize}
\item[\mylabel{a}{l1_2016_06_02 item a}] If \eqref{1nsii9} holds (and \eqref{2nsii10} does then too), then we say that $\b$ is an extension of $\a$ or that $\a$ is a restriction of $\b$. 

\sp\item[\mylabel{b}{l1_2016_06_02 item b}] Formula \eqref{2nsii10}, i.e. \eqref{1_2016_06_03} and \eqref{2_2016_06_03} taken together, yields
\beq\label{1nsii9B}
\tf_\a^{-\|\a\|}(\Sg_u\times B_2(a,2R)) \spt \tf_\b^{-\|\b\|}(\Sg_u\times B_2(b,4r)).
\eeq
\end{itemize}

\elem

\bpf
All what we are to prove is that \eqref{1nsii9} entails \eqref{2nsii10}. So assume that \eqref{1nsii9} holds. Then $[\^\a]\cap [\^\b]\ne\es$ and (we assume that) $|\^\b|\ge |\^\a|$. We thus conclude that $\^\b$ is an extension of $\^\a$, which means that
$$
\^\b=\^\a\^\b\big|_{|\^\a|+1}^{|\^\b|-|\^\a|}.
$$
Formula \eqref{1nsii9} also gives
\beq\label{1nsii10C}
f_\a^{-1}(B_2(a,4r)) \cap f_\b^{-1}(B_2(b,4r))\ne\es.
\eeq
Therefore 
\beq\label{120180616}
f_\b^{-1}=f_{\a^*}^{-1}\circ f_\g^{-1},
\eeq
where $\a^*$ refers to the holomorphic branch of $f_{\^\a}^{-1}$ defined on $f_\g^{-1}(B_2(b,4r))$ and sending $f_\g^{-1}(b)$ to $f_\b^{-1}(b)$. More precisely, it defines the map 
$$
\Sg_u\times f_\g^{-1}(B_2(b,4r))
\ni (\om,z)\longmapsto \tf_{\a^*}^{-||\a||}(\om,z) 
=\Bigg(\^\a\om,f_\b^{-1}\lt(f_{\^\b\big|_{|\^\a|+1}^{|\^\b|-|\^\a|}} (z)\rt)\Bigg).
$$
By applying $f_{\^\a}$ to \eqref{1nsii10C}, we get
$$
B_2(a,4r)\cap f_\g^{-1}(B_2(b,4r))\ne\es.
$$
Due to our choice of $r$ we have that 
$$
\diam_\C\(f_\g^{-1}(B_2(b,4r))\)<R.
$$
Hence
$$
B_2(a,r)\cup f_\g^{-1}(B_2(b,4r))\sbt B_2(a,2R)\sbt B_2(a,4R).
$$
Therefore, $f_{\a^*}^{-1}$ extends uniquely to $B_2(a,4R)$ as a holomorphic branch of $f_{\^\a}^{-1}$; we can thus treat $\a^*$ as an element of $\HIB(B_2(a,4R))$. In view of \eqref{120180616} and \eqref{1nsii10C} this yields
$$
\begin{aligned}
f_\a^{-1}\(B_2(a,4R)\)\cap f_{\a^*}^{-1}\(B_2(a,4R)\)
&\spt f_\a^{-1}(B_2(a,4r))\cap f_{\a^*}^{-1}\(f_\g^{-1}(B_2(b,4r))\) \\
&=  f_\a^{-1}(B_2(a,4r))\cap f_{\a^*}^{-1}\circ f_\g^{-1} (B_2(b,4r)) \\
&=  f_\a^{-1}(B_2(a,4r))\cap f_{\b}^{-1}(B_2(b,4r))\ne\es. 
\end{aligned}
$$
Thus $\a^*=\a$ and the proof of our lemma is complete.  
\epf
\bdfn
We call a set $S\sbt\C$ \textbf{aperiodic}\index{aperiodic} if 
$$
S\cap G(S)=\es.
$$
\edfn
An immediate but important for us observation is this.

\bobs\label{o1_2016_06_07}
If $G=\langle f_1,\ld,f_u \rangle$ is a \TNR \ rational semigroup, then $\Crit_*(f)$ is aperiodic.
\eobs

\bdfn\label{d2nsii6}
Let $G=\langle f_1,\ld,f_u\rangle$ be a \TNR \ rational semigroup. Fix $R>0$. Let 
$$
\Crit_*(f)\sbt S\sbt J(G)\sms B_2(\PCV(G),128K^2R)
$$ 
be a finite aperiodic set. A collection $\cU=\{U_s\}_{s\in S}$ of open subsets of $\Sg_u\times\C$ is called a \textbf{nice family}\index{nice family} for $\^f$ if the following conditions are satisfied.

\begin{itemize}
\item[\mylabel{a}{d2nsii6 item a}] The sets $U_s$, $s\in S$, are mutually disjoint,
\item[\mylabel{b}{d2nsii6 item b}] $p_2(U_s)\sbt B_2(s,2R)$  for each $s\in S$,
\item[\mylabel{c}{d2nsii6 item c}] For any $a, b\in S$ and any $\rho\in \HIB(p_2(U_b))$ we have either that
$$
U_a\cap \^f_\rho^{-\|\rho\|}(U_b)=\es \tor \^f_\rho^{-\|\rho\|}(U_b)\sbt U_a.
$$
\end{itemize}
\edfn

\fr The main result of this section is the following.

\bthm\label{t1nsii7} 
Let $G$ be a finitely generated \TNR \ semigroup. Let  $f=(f_1\ld,f_u)\in \Rat^u$ be a $u$--tuple map generating $G$. Fix $R\in(0,R_*(G))$. Fix also $\ka\in (1,2)$. Let
$$
\Crit_*(f)\sbt S\sbt J(G)\sms B_2(\PCV(G),8R)
$$ 
be a finite aperiodic set. Then for every $r\in(0,R]$ small enough there exists 
$$
\cU_S(\ka,r)=\{U_s(\ka,r)\}_{s\in S},
$$
a nice family of sets for $\tf$, associated to the set $S$, such that

\sp\begin{itemize}
\item[\mylabel{A}{t1nsii7 item A}] 
$$
\Sigma_u\times B_2(s,r)\sbt U_s(\ka,r) \sbt \Sigma_u\times B_2(s,\ka r)
$$
for each $s\in S$.
\item[\mylabel{B}{t1nsii7 item B}] If $a, b\in S$, $\rho\in \HIB\(p_2(U_b(\ka,r))\)$, and $\tf_\rho^{-\|\rho\|}(U_b(\ka,r))\sbt U_a(\ka,r)$, then
$$
\lt|\lt(f_\rho^{-1}\rt)'(z)\rt|\le\frac14
$$
for all $z\in B_2(b,2R)\spt p_2\(U_b(\ka,r)\)$,
\end{itemize}
\ethm

\brem\label{r2nsii7}
We would like to note right away that with $r\in(0,R]$, condition \eqref{t1nsii7 item A} of Theorem~\ref{t1nsii7} alone entails conditions \eqref{d2nsii6 item a} and \eqref{d2nsii6 item b} of Definition~\ref{d2nsii6}.
\erem

\bpf
Fix $R\in(0,R_*(G)]$ arbitrary. Then, by Theorem~\ref{t1h4} and by Koebe's Distortion Theorem there exists an integer $N\ge 1$ so large that
\beq\label{1nsii8}
\lt|\lt(f_\rho^{-1}\rt)'(x)\rt|
\le\min\lt\{\frac14,\frac{\ka-1}{2\ka}\rt\}
\eeq
for all $s\in S$, all $\rho\in \HIB(B_2(s,4R))$ with $\|\rho\|\ge N$, and all $x\in B_2(s,4R_*(G))$. 
Now, because the set $S$ is aperiodic there exists $r\in(0,R]$ so small that if $a, b\in S$, $\Id\ne\rho\in\HIB(B_2(b,4R))$, and 
$$
f_\rho^{-1}(B_2(b,2r))\cap B_2(a,2r)\ne\es,
$$
then 
\beq\label{1_2016_06_06}
\diam_\C\(f_\rho^{-1}(B_2(b,r))\)<R \and \|\rho\|\ge N.
\eeq
For every integer $n\ge 0$ set
$$
\begin{aligned}
\Pre_{\le n}(S,r)
:&=\bu_{s\in S}\Pre_{\le n}(B_2(s,r)) \\
&=\Big\{\tf_\rho^{-\|\rho\|}\(\Sg_u\times B_2(s,r)\):s\in S,\, \rho\in \HIB(B_2(s,4R)),\ \|\rho\|\le n\Big\}.
\end{aligned}
$$
Furthermore, for every $s\in S$ denote by $\cN_n(s)$ the (unique) connected component of the set $\Pre_{\le n}(S,r)$ containing $\Sg_u\times B_2(s,r)$. Our first goal is to prove the following.

\blem\label{l1nsii9}
For every integer $n\ge 0$ and every $s\in S$, we have that
$$
\bu\cN_n(s)\sbt \Sg_u\times B_2(s,\ka r).
$$
\elem

\bpf
For $n=0$ this is immediate as $\cN_0(s)=\{\Sg_u\times B_2(s,r)\}$. Proceeding further by induction suppose that our lemma holds for all $s\in S$ and all $0\le j\le n-1$ with some $n\ge 1$. We are to prove it for the integer $n$. Towards this aim, given $s\in S$, let $\cM$ be an arbitrary connected component of the family
$$
\cN_n(s)\sms\{\Sg_u\times B_2(s',r)\}_{s'\in S} .
$$ 
From now on throughout this proof we let $k\ge 0$ to be the minimal length of an element $\rho\in \bu_{a\in S}\HIB(B_2(a,4R))$ such that 
$$
\tf_\rho^{-1}\(\Sg_u\times B_2(\xi,r)\)\in \cM
$$
for some $\xi\in S$. By the very definition of $k$ and $\cM$ we have that $k\ge 1$. Again by the definition of $k$ we have that
$$
\^f^k(\cM)\sbt \Pre_{\le n-k}(S,r).
$$
But 
$$
\Sg_u\times B_2(\xi,r)
=\^f^k\(\^f_\rho^{-\|\rho\|}(\Sg_u\times B_2(\xi,r))\)
\in \^f^k(\cM),
$$
and, by Observation~\ref{o1nsii10}, the family $\^f^k(\cM(s))$ is connected. Thus
$$
\^f^k(\cM)\sbt \cN_{n-k}(\xi).
$$
Hence, our inductive hypothesis yields
$$
\^f^k\Big(\bu\cM\Big)
=\bu\^f^k(\cM(s))
\sbt \Sg_u\times B_2(\xi,\ka r).
$$
Therefore,
\beq\label{1nsii12B}
\bu\cM
\sbt \^f^{-k}\(\Sg_u\times B_2(\xi,\ka r)\)
=\bu_{\a\in\HIB_k(B_2(\xi,4R))}\^f_\a^{-k}\(\Sg_u\times B_2(\xi,\ka r)\).
\eeq
So, if $\tf_\b^{-\|\b\|}(\Sg_u\times B_2(b,r))$ is an arbitrary element in $\cM$, then there must exist $\a\in\HIB_k(B_2(\xi,4R))$ (remembering that $\ka<2$) such that 
\beq\label{5_2016_06_03}
\tf_\b^{-\|\b\|}\(\Sg_u\times B_2(b,r)\)\cap \tf_\a^{-\|\a\|}\(\Sg_u\times B_2(\xi,2r)\)\ne\es.
\eeq
Hence, according to Lemma~\ref{l1_2016_06_02} (keeping in mind that $||\b||\ge k=\|\a\|$), 
\beq\label{3_2016_06_03}
\tf_\b^{-\|\b\|}\big|_{\Sg_u\times B_2(b,4r)}=\tf_\a^{-\|\a\|}\circ\tf_\g^{-\|\g\|}\big|_{\Sg_u\times B_2(b,4r)}
\eeq
with $\g$ as specified in this lemma. Likewise if $\om\in\HIB_k(B_2(\xi,4R))$ is another, different from $\a$, element such that
$$
\tf_\b^{-\|\b\|}\(\Sg_u\times B_2(b,r))\)\cap \tf_\om^{-\|\om\|}\(\Sg_u\times B_2(\xi,2r)\)\ne\es,
$$
then 
\beq\label{4_2016_06_03}
\tf_\b^{-|\\|\b\|}\big|_{\Sg_u\times B_2(b,4r)}=\tf_\om^{-\|\om\|}\circ \tf_\d^{-\|\d\|}\big|_{\Sg_u\times B_2(b,4r)}
\eeq
with $\d$ as specified in this Lemma~\ref{l1_2016_06_02}. It follows from \eqref{3_2016_06_03} and \eqref{4_2016_06_03} that $\^\a=\^\b|_k=\^\om$.
So, 
$$
f_\a^{-1}(B_2(\xi,4R))\cap f_\om^{-1}(B_2(\xi,4R))=\es
$$ 
as $\a\ne\om$. But this contradicts \eqref{3_2016_06_03} and \eqref{4_2016_06_03}, and shows that $\a$ is the only element of 
$\HIB_k(B_2(\xi,4R))$ such that \eqref{5_2016_06_03} holds. Hence,
\beq\label{6_2016_06_03}
\tf_\b^{-\|\b\|}(\Sg_u\times B_2(b,r))\sbt \tf_\a^{-\|\a\|}(\Sg_u\times B_2(\xi,\ka r)).
\eeq
Therefore, since the family $\big\{\tf_\rho^{-\|\rho\|}(\Sg_u\times B_2(\xi,\ka r)):\rho\in \HIB_k(B_2(\xi,4R))\}$ consists of mutually disjoint sets and since the family $\cM$ is connected, we conclude that there exists a unique $\tau\in \HIB_k(B_2(\xi,4R))$ such that
\beq\label{1nsii12}
\bu\cM
\sbt\^f_\tau^{-k}\(\Sg_u\times B_2(\xi,\ka r)\).
\eeq
Again, since $\cM$ is a connected component of $\cN_n(s)\sms\{\Sg_u\times B_2(s',r)\}_{s'\in S}$ which is a proper subset of $\cN_n(s)$ (as it misses $\Sg_u\times B_2(s,r)$), and since $\cN_n(s)$ is connected, there must exist $\tf_\a^{-\|\a\|}(\Sg_u\times B_2(a,r))\in\cM$ and $s'\in S$ such that
\beq\label{2nsii12}
\tf_\a^{-\|\a\|}\(\Sg_u\times B_2(a,r)\)\cap \(\Sg_u\times B_2(s',r)\)\ne\es.
\eeq
Looking at this and \eqref{1nsii12}, and remembering that $\ka<2$, we see that
\beq\label{1_2016_05_30}
\^f_\tau^{-k}\(\Sg_u\times B_2(\xi,2r)\)\cap \(\Sg_u\times B_2(s',r)\)\spt \^f_\tau^{-k}(\Sg_u\times B_2(\xi,\ka r)\)\cap B_2(s',r)\ne\es.
\eeq
Hence, $k=\|\tau\|\ge N$ by virtue of \eqref{1_2016_06_06}. Also, of course \eqref{1_2016_05_30} yields
\beq\label{2_2016_05_30}
f_\tau^{-1}(B_2(p_2(\xi),\ka r))\cap B_2(p_2(s'),r)\ne\es.
\eeq
Combining this and \eqref{1nsii8}, we get that 
\beq\label{3nsii12} 
%\begin{aligned}
\bu\cM
\sbt \Sg_u\times B_2\lt(s',r+\frac{\ka-1}{2\ka}(2\ka r)\rt) 
=\Sg_u\times B_2(s',\ka r).
%\end{aligned}
\eeq
Since $\cN_n(s)$ is a union of all such components $\cM$ along with some elements of $\{\Sg_u\times B_2(a,r)\}_{a\in S}$, and since $\ka>1$, we thus conclude that
$$
\bu\cN_n(s)\sbt \bu_{a\in S}\Sg_u\times B_2(a,\ka r).
$$
But since the sets $\Sg_u\times B_2(a,\ka r)$, $a\in S$, are mutually disjoint and since $\cN(s)$ is a connected family, there must exist $b\in S$ such that
$$
\bu\cN_n(s)\sbt \Sg_u\times B_2(b,\ka r).
$$
Since $\Sg_u\times B_2(s,r)\sbt\bu\cN_n(s)$, we must have $b=s$, and in conclusion
$$
\bu\cN_n(s)\sbt \Sg_u\times B_2(s,\ka r).
$$
The proof of Lemma~\ref{l1nsii9} is thus complete.
\epf 

\sp Now we can complete the proof of Theorem~\ref{t1nsii7}. We shall show that the family
$$
\cU_S:=\Big\{U_s=U_s(\ka,r):=\bu_{n=0}\bu\cN_n(s)\Big\}_{s\in S}
$$
has all the required properties, i.e. being a nice family of sets along with conditions \eqref{t1nsii7 item A} and \eqref{t1nsii7 item B} required in Theorem~\ref{t1nsii7}. 

Firstly, $U_s(\ka,r)$, $s\in S$, are all open subsets of $\Sg_u\times\oc$ since all elements of $\cN_n(s)$ are open. 

Secondly, for every $n\ge 0$, $\Sg_u\times B_2(s,r)\in \cN_n(s)$, so
\beq\label{2nsii13}
\Sg_u\times B_2(s,r)\sbt \bu\cN_n(s)\sbt U_s(\ka,r)
\eeq
for every $s\in S$. 

Thirdly, it directly follows from Lemma~\ref{l1nsii9} that
\beq\label{3nsii13}
U_s(\ka,r)\sbt \Sg_u\times B_2(s,\ka r)
\eeq
for all $s\in S$. Along with \eqref{2nsii13} this means that property \eqref{t1nsii7 item A} of Theorem~\ref{t1nsii7} holds. Therefore, by virtue of Remark~\ref{r2nsii7} the properties \eqref{d2nsii6 item a} and \eqref{d2nsii6 item b} of Definition~\ref{d2nsii6} (of a nice family of sets) also hold. In addition, condition \eqref{t1nsii7 item B} of Theorem~\ref{t1nsii7} follows immediately from \eqref{1nsii8}, the already proven property \eqref{d2nsii6 item b} of Definition~\ref{d2nsii6}, and the definition of $r$. We thus are only left to show that item \eqref{d2nsii6 item c} of Definition~\ref{d2nsii6} is satisfied. In order to prove this item assume that
\beq\label{1nsii13} 
\^f_\a^{-k}(U_b)\cap U_a\ne\es
\eeq
for some $a, b\in S$ and some $\a\in\HIB(p_2(U_b))$, where $k=\|\a\|$. Fix an arbitrary integer $n\ge 0$ such that 
\beq\label{4nsii13}
U_a\cap \^f_\a^{-k}\Big(\bu\cN_n(b)\Big)\ne\es.
\eeq
Now consider $\^f_\b^{-l}(\Sg_u\times B_2(\xi,r))$, $\xi\in S$, $l=\|\b\|\ge 1$, an arbitrary element of $\cN_n(b)$ such that
$$
U_a\cap \^f_\a^{-k}\(\^f_\b^{-l}\(\Sigma_u\times B_2(\xi,r)\)\)\ne\es.
$$ 
But since $\(\cN_n(a)\)_{n=0}^\infty$ is an ascending sequence of sets, there thus exists $p\ge k+l$ such that 
\beq\label{1_2016_06_07}
\bu\cN_p(a) \cap \^f_\a^{-k}\circ \^f_\b^{-l}\(\Sigma_u\times B_2(\xi,r)\)\ne\es.
\eeq
Denote by $\a\b$ the only element of $\HIB(B_2(b,4R))$ such that 
$$
\^f_{\a\b}^{-(k+l)}|_{\Sigma_u\times B_2(\xi,r)}
=\^f_\a^{-k}\circ \^f_\b^{-l}|_{\Sigma_u\times B_2(\xi,r)}.
$$
But then \eqref{1_2016_06_07} means
that $\^f_{\a\b}^{-(k+l)}\(\Sigma_u\times B_2(\xi,r)\)$ intersects at least one element of $\cN_p(a)$, and therefore $\^f_{\a\b}^{-(k+l)}\(\Sigma_u\times B_2(\xi,r)\)\in\cN_p(a)$. Consequently, 
$$
\^f_\a^{-k}\(\^f_\b^{-l}\(\Sigma_u\times B_2(\xi,r)\)\)
=\^f_{\a\b}^{-(k+l)}\(\Sigma_u\times B_2(\xi,r)\)
\sbt U_a.
$$
We thus have that
\beq\label{1nsii14}
\^f_\a^{-k}\Big(\bu\cN_n^*(a,b)\Big)\sbt U_a,
\eeq
where
$$
\cN_n^*(a,b):=\big\{\Ga\in \cN_n(b):U_a\cap \^f_\a^{-k}(\Ga)\ne\es\big\}.
$$
Now, if $\Ga_1\in \cN_n^*(a,b)$ and $\Ga_2\in \cN_n(b)\sms \cN_n^*(a,b)$, then on the one hand $\^f_\a^{-k}(\Ga_1)\sbt U_a$ and, on the other hand, $\^f_\a^{-k}(\Ga_2)\cap U_a=\es$. Hence, $\Ga_1\cap\Ga_2=\es$. Since however $\cN_n(b)$ is a connected family, this implies that either $\cN_n^*(a,b)=\es$ or $\cN_n(b)\sms \cN_n^*(a,b)=\es$. But by \eqref{4nsii13}, we have that $\cN_n^*(a,b)\ne\es$, yielding $\cN_n^*(a,b)=\cN_n(b)$. Along with \eqref{1nsii14}
this gives
\beq\label{2nsii14}
\^f_\a^{-k}\Big(\bu\cN_n(b)\Big)\sbt U_a.
\eeq
Since by \eqref{1nsii13}, formula \eqref{4nsii13} holds for all $n\ge 1$ sufficiently large, and since $\(\cN_n(b)\)_{n=0}^\infty$ is an ascending sequence of sets, formula \eqref{2nsii14} yields
$$
\^f_\a^{-k}(U_b)\sbt U_a.
$$
The proof of property \eqref{d2nsii6 item c} of Definition~\ref{d2nsii6} is complete. This simultaneously finishes the proof of Theorem~\ref{t1nsii7}. 
\epf

\sp\fr The most important consequence of having a nice family of sets is that it gives rise to a graph directed system in the sense of \cite{mugdms} which has a sufficient degree of conformality. Namely, let
$$
U:=\bu_{s\in S}U_s,
$$
and for every $s\in S$ and any integer $n\ge 0$ let
$$
\cD_n^*(s)
=\cD_n^*(G,s)
:=\big\{\tau\in\HIB_n(p_2(U_s)):\^f^k\(\^f_\tau^{-n}(U_s)\)\cap U=\es \for \all 0\le k\le n-1\big\},
$$\index{$\cD_n^*(s)$}
and let
$$
\cD_n(s)=
\cD_n(G,s)
:=\big\{\tau\in\HIB_n(p_2(U_s)):\^f_\tau^{-\|\tau\|}(U_s)\cap U\ne\es \and \tf\circ\^f_\tau^{-\|\tau\|}\in  \cD_{n-1}^*(s)\big\}.
$$\index{$\cD_n(s)$}
Finally let
$$
\cD_n=
\cD_n(G):=\bu_{s\in S}\cD_n(s) \and \cD_\cU=\cD_\cU(G):=\bu_{n=1}^\infty \cD_n.
$$\index{$\cD_n$}\index{$\cD_\cU$}
We will usually skip the indication of these sets on $G$ since we only very rarely deal with more than one rational semigroup at a time. This will be however the case for example in Section~\ref{section:MA}, Multifractal Analysis of Invariant Measures $\mu_t$, in the context of non--exceptional rational semigroups. 

We now note that for every element $\tau\in\cD_\cU$ there are a unique element $t(\tau)\in S$ such that $\tau\in \HIB\(p_2(B(t(\tau),4R))\)$ and a unique element $i(\tau)\in S$ such that $\^f_\tau^{-\|\tau\|}(U_{t(\tau)})\sbt U_{i(\tau)}$. \index{$t(\tau)$}\index{$i(\tau)$}
For every element $s\in S$ let
$$
X_s:=\ov U_s.
$$
As an immediate consequence of Definition~\ref{d2nsii6}, Theorem~\ref{t1nsii7} (particularly its part (B)) and topological exactness of the dynamical system $\tf:J(\^f)\lra J(\^f)$ (implying item \eqref{d2nsii6 item c} below), we get the following. 

\bthm\label{t1nsii15}
If $G$ is a \TNR \ rational semigroup generated by a $u$--tuple map $(f_1,\ld,f_u)\in \Rat^u$ and $\cU=\{U_s\}_{s\in S}$ is a nice family of sets for $\^f$ produced in Theorem~\ref{t1nsii7}, then the family 
$$
\cS_\cU
=\cS_\cU(G)
:=\big\{\^f_\tau^{-\|\tau\|}:X_{t(\tau)}\lra X_{i(\tau)}\big\}_{\tau\in\cD_\cU}
$$\index{$\cS_\cU$}\index{graph directed system (GDS)!$\cS_\cU$}\index{$X_{t(\tau)}$}\index{$X_{i(\tau)}$}\index{graph directed system (GDS)!$X_{t(\tau)}$}\index{graph directed system (GDS)!$X_{i(\tau)}$}
forms a \textbf{graph directed system} (GDS)\index{graph directed system (GDS)} in the sense of \cite{mugdms}. Furthermore,
%satisfying the Open Set Condition. 
\sp\begin{itemize}
\item[\mylabel{a}{t1nsii15 item a}] The corresponding incidence matrix $A(\cU)=A(G,\cU)$ is then determined by the condition that
$$
A_{\tau\om}(\cU)=1 
$$\index{$A_{\tau\om}(\cU)$}\index{graph directed system (GDS)!$A_{\tau\om}(\cU)$}
if and only if $t(\tau)=i(\om)$.

\sp\item[\mylabel{b}{t1nsii15 item b}] The limit set $J_\cU$\index{limit set! $J_\cU$}\index{graph directed system (GDS)!$J_\cU$} of the system $\cS_\cU$ is contained in $J(\^f)$ and contains $U\cap\Trans(\tf)$, where, we recall, $\Trans(\tf)$ is the set of transitive points of $\tf:J(\tf)\lra J(\tf)$, i.e. the set of points $z\in J(\^f)$ such that the set $\{\^f^n(z):n\ge 0\}$ is dense in $J(\^f)$.

\sp\item[\mylabel{c}{t1nsii15 item c}] The graph directed system $\cS_\cU$ is finitely primitive.
\end{itemize}

\sp\fr We denote by $\cD_\cU^\infty$\index{$\cD_\cU^\infty$}\index{graph directed system (GDS)!$\cD_\cU^\infty$} the symbol space $\(\cD_\cU\)_{A(\cU)}^\infty$ generated by the matrix $A(\cU)$; as in the case of $\Sg_u$ its elements (infinite sequences) start with coordinates labeled by the integer $1$. Likewise $\cD_\cU^*$ and $\cD_\cU^n$, $n\in\N$, abbreviate respectively $\(\cD_\cU\)_{A(\cU)}^*$ and $\(\cD_\cU\)_{A(\cU)}^n$, $n\in\N$. \index{$\cD_\cU^*$}\index{graph directed system (GDS)!$\cD_\cU^*$}\index{$\cD_\cU^n$}\index{graph directed system (GDS)!$\cD_\cU^n$}

\sp\fr In addition, we denote by $\phi_e$, $e\in \cD_\cU$, all the elements of $\cS_\cU$.\index{$\phi_e$}\index{graph directed system (GDS)!$\phi_e$}
\ethm

\section{The Behavior of the Absolutely Continuous Invariant Measures $\mu_t$ \\ Near Critical Points}\label{sec:nearcriticalpoints}

This section is very technical and devoted to study the behavior of conformal measures $m_t$ and their invariant versions $\mu_t$ near critical points of the skew product map 
$$
\^f:\Sg_u\times\hat\C\lra \Sg_u\times\hat\C.
$$
Its main outcome is Proposition~\ref{p1cp3} which gives a quantitative strengthening of quasi--invariance of conformal measures $m_t$. This is the first and only place where the hypothesis of finite type of the semigroup $G$ is explicitly needed; it demands that the set of critical points of $\tf$ lying in the Julia set $J(\^f)$ of $\tf$ is finite.

\sp Let $G$ be a finitely generated *semi--hyperbolic rational semigroup generated by $(f_1,\ld,f_u)\in\Rat^u$. 
Let $V\ne\es$ be an open subset of $\Sg_u\times\oc$ containing $\Crit_*(\^f)$. Define
$$
K(V)
:=\bi_{n=0}^\infty\^f^{-n}(J(\^f)\sms V)
 =\big\{z\in J(\^f): \^f^n(z)\notin V \,\, \forall n\ge 0\big\}\sbt \Sg_u\times\C.
$$\index{$K(V)$}
Of course $K(V)$ is a closed subset of $J(\^f)$,
\beq\label{1sp1}
\^f(K(V))\sbt K(V),
\eeq
and
\beq\label{2sp1}
K(V)\cap \Crit(\^f)=\es.
\eeq
So, we can consider the dynamical system $\^f|_{K(V)}:K(V)\lra K(V)$. Because of \eqref{2sp1} for every $t\in\R$ the potential $-t\log|\^f'|:K(V)\to\R$ is continuous, and because of Koebe's Distortion Theorem it is H\"older continuous. Because of Exponential Shrinking Property, Theorem~\ref{t1h4}, and Koebe's Distortion Theorem used again, we have the following.

\blem\label{l2sp1}
Let $G$ be a finitely generated *semi--hyperbolic rational semigroup generated by $(f_1,\ld,f_u)\in\Rat^u$. If $\, V\ne\es$ is an open subset of $\Sg_u\times\oc$ containing $\Crit_*(\^f)$, then 
$$
\^f|_{K(V)}:K(V)\lra K(V)
$$ 
is an infinitesimally expanding map. More precisely, there exists an integer $n\ge 1$ such that 
$$
\big|\(\tf^n\)'(\xi)\big|\ge 2
$$
for every $\xi\in K(V)$. In addition, there exists $R_2>0$ such that for every $\xi\in K(V)$ and every integer $k\ge 0$ there exists a unique continuous inverse branch $\tf_\xi^{-k}:B_2(\tf^k(\xi),2R_2)\to\Sg_u\times\C$ of $\tf^k$ sending $\tf^k(\xi)$ back to $\xi$.
\elem 

\bpf
Let $R>0$ come from the Exponential Shrinking Property (Theorem~\ref{t1h4}). Fix $R_1\in (0,R]$ so small that
\beq\label{2sp1.1}
B_2(\Crit_*(\tf),2R_1)\sbt V
\eeq
and
\beq\label{3sp1.1}
B_2(J(\^f),2R_1)\cap\(\PCV(\tf)\sms\PCV_*(\tf)\)=\es.
\eeq
Now fix an integer $q\ge 1$ so large that
\beq\label{1sp1.1}
\vartheta^q+Ce^{-\a q}<R_1,
\eeq
where both $C$ and $\a$ come from the Exponential Shrinking Property (Theorem~\ref{t1h4}). Now take $R_2\in(0,R_1/4)$ so small that
$$
\diam_{\Sg_u\times\C}([\tau]\times W)<R_1
$$
for every integer $0\le n\le q$, every $\tau\in \Sg_u^n$, every $z\in J(G)$ and any connected component $W$ of $f_\tau^{-1}(B_2(z,2R_2))$. In conjunction with the Exponential Shrinking Property (Theorem~\ref{t1h4}) and \eqref{1sp1.1} this gives that
$$
\diam_{\Sg_u\times\C}([\tau]\times W)<R_1
$$
for every integer $n\ge 0$, every $\tau\in \Sg_u^n$, every $z\in J(G)$ and any connected component $W$ of $f_\tau^{-1}(B_2(z,2R_2))$. Making also use of \eqref{2sp1.1}, \eqref{3sp1.1}, and the definition of $K(V)$, it thus follows that if $\xi=(\om,z)\in K(V)$, then 
$$
\([\om|_n]\times W(\xi,n)\)\cap\Crit(\tf^n)=\es
$$
for every integer $n\ge 0$ and $W(\xi,n)$, the connected component of $f_{\om|_n}^{-1}\(B_2(z,2R_2)\)$ containing $z$. Hence, the map 
$$
f_{\om|_n}|_{W(\xi,n)}:W(\xi,n)\lra B_2(z,2R_2)
$$ 
is a conformal homeomorphism. Thus Koebe's Distortion Theorem applies to give
$$
\big|\(\tf^n\)'(\xi)\big|
=\big|\(f_{\om|_n}\)'(z)\big|
\ge K^{-1}\frac{R_2}{Ce^{-\a n}}
=\frac{R_2}{KC}e^{\a n}
\ge 2
$$
for all $n\ge \frac1{\a}\log(2K/R_2)$. The proof is complete.
\epf

\sp\fr As an immediate consequence of this lemma, we get the following.

\bcor\label{l2sp1B}
Let $G$ be a finitely generated *semi--hyperbolic rational semigroup generated by $(f_1,\ld,f_u)\in\Rat^u$. If $\, V\ne\es$ is an open subset of $\Sg_u\times\oc$ containing $\Crit_*(\^f)$, then 
$$
\^f|_{K(V)}:K(V)\lra K(V)
$$ 
is a distance expanding map with respect to the metric in the sense of \cite{PUbook}. This precisely means that there exist $\eta>0$ and an integer $n\ge 1$ such that 
$$
\|\tf^n(w),\tf^n(z)\|_\vartheta\ge 2\|w,z\|_\vartheta
$$
whenever $w, z\in K(V)$ and $\|w,z\|_\vartheta\le \eta$.
\ecor

Since $G$ is TNR there exists $V$, an open neighborhood of $\Crit_*(\tf)$, such that 
$$
V\cap \PCV_*(\tf)=\es.
$$
Hence, 
$$
\PCV_*(\tf)\sbt K(V).
$$
Therefore, as an immediate consequence of Lemma~\ref{l2sp1}, we get the following.

\sp
\blem\label{l2sp1C}
Let $G$ be a  \CF\ balanced \TNR \ rational semigroup generated by $(f_1,\ld,f_u)\in\Rat^u$. Then there exist an integer $b\ge 1$, and constants $C>0$ and $\b>1$ such that 
$$
\big|\(\tf^b\)'(\xi)\big|\ge 2
$$
and
$$
\big|\(\tf^n\)'(\xi)\big|\ge C\b^n
$$
for every $\xi\in \PCV_*(\tf)$ and every integer $n\ge 0$. 
In addition, there exists $R_2>0$ such that for every $\xi\in \PCV_*(\tf)$ and every integer $n\ge 0$ there exists a unique continuous inverse branch $\tf_\xi^{-n}:B_2(\xi,2R_2)\lra\Sg_u\times\C$ of $\tf^n$, sending $\tf^n(\xi)$ back to $\xi$.
\elem 

\sp\fr For every point $\xi\in\Crit(\^f)$ we put
$$
\chi(\xi) :=\varliminf_{k\to\infty}{1\over k} \log \inf_{n\ge
1}\Big\{\big|\(\^f^k\)'(\^f^n(\xi))\big|\Big\} \  \text{ and } \
\chi_{\^f}:= \min \lt\{{\chi(\xi)\over q_\xi}: \xi\in \Crit_*(\^f)\rt\},
$$\index{$\chi(\xi)$}
where $q_\xi$\index{$q_\xi$} is the order of the critical point $p_2(\xi)$ with respect to the holomorphic map $f_{{p_1(\xi)}_1}$. Now we define 
\beq\label{1_2016_05_24}
\De_G^*:=\{t\ge 0:\P(t)>-\chi_{\^f}\, t\}. 
\eeq\index{$\De_G^*$}

\fr Although we will not really need for a long time to know that the set
$\De_G^*$ is nonempty, we remark already at the moment that this follows immediately from continuity of the pressure function $t\longmapsto\P(t)$ and is formally stated in Proposition~\ref{p1_2016_06_21} in a much stronger form which sheds some notable light on the structure of the set $\De_G^*$.
 
\fr The main result of this section is the following technical proposition. 

\bprop\label{p1cp3}
Let $G=\langle f_1, f_2,\ld, f_u \rangle$ be a \FNR \ rational semigroup. Then for every $b\in \De_G^*$ there exist $\eta>0$ and an integer $l\ge 1$ such that $(b-\eta,b+\eta)\sbt \De_G$, and if $t\in (b-\eta,b+\eta)\cap \De_G^*$, then 
\beq\label{2_2016_06_21}
m_t(\^f^{-1}(A))\le Cm_t^{1/l}(A)
\eeq
for every Borel sets $A\sbt J(\^f)$ with some constant $C>0$ independent of $t$.
\eprop

\bpf
For every $c\in\Crit_*(f)$ let 
$$
I(c):=\big\{i\in\{1,2\ld,u\}:c\in J_{i\om} \for \some \om\in\Sg_u \and f_i'(c)=0\big\}.
$$
Fix $i\in I(c)$. Let $q\ge 2$ be the order of the critical point $c$ with respect to the map $f_i$. Fix $\th>0$ so small that there exists a constant $Q\in[1,+\infty)$ such that if $f_{i,c}^{-1}(B_2(f_i(c),\th))$ denotes the connected component of $f_i^{-1}(B_2(f_i(c),\th))$ containing $c$, then 
\beq\label{1cp3}
Q^{-1}\le \frac{|f_i(z)-f_i(c)|}{|z-c|^q}\le Q,\
Q^{-1}\le \frac{|f_i'(z)|}{|z-c|^{q-1}}\le Q,\ Q^{-1}\le \frac{|f_i'(z)|}{|f_i(z)-f_i(c)|^{\frac{q-1}{q}}}\le Q
\eeq
for all $z\in f_{i,c}^{-1}(B_2(f_i(c),\th))$. For every set $A\sbt B_2(f_i(c),\th)$ we put
$$
f_{i,c}^{-1}(A):=f_i^{-1}(A)\cap f_{i,c}^{-1}(B_2(f_i(c),\th)).
$$
Furthermore, for every set $A\sbt \Sg_u\times B_2(f_i(c),\th)$, put
$$
\^f_{i,c}^{-1}(A)
:=\big\{(i\om,z)\in \Sg_u\times\oc:(\om,f_i(z))\in A \and z\in f_{i,c}^{-1}(B_2(f_i(c),\th))\big\}.
$$
Now we take $\d>0$ so small that $8Q^2\|\tf'\|_\infty\d\le \th$. We will in fact need $\d>0$ to be even smaller as specified later in the course of this proof. Fix  $\xi\in\tf^{-1}(\Crit(\tf))$ arbitrary. Fix also $\g>\P(t)$ arbitrary. We shall first prove that 
\beq\label{1_2016_06_16}
\nu_{t,\g}(\^f_{i,c}^{-1}(A))\lek \nu_{t,\g}^{1/l}(A)
\eeq
for every Borel set $A\sbt J(\^f)\cap p_2^{-1}(B_2(f_i(c),\d))$ and some integer $l\ge 1$ independent of $t$, $\g$, and $A$; the comparability constant of \eqref{1_2016_06_16} does not depend on them either. To that end, let
$$
\Ga_i(c):=\{\om\in\Sg_u:f_i(c)\in J_\om\}.
$$
Of course
$$
\Crit_*(\^f)
=\bu_{b\in\Crit_*(f)}\bu_{j\in I(b)}\{(j\om,b):\om\in \Ga_j(b)\}.
$$
Since $\Crit_*(\^f)$ is finite, each set $\Ga_j(b)$, $b\in\Crit_*(f)$, $j\in I(b)$, is finite. Therefore, there exists an integer $p\ge 1$ so large that the function $\Ga_i(c)\ni\om\longmapsto\om|_p\in \Sg_u^*$ is 1--to--1. Equivalently, the family
$$
\{[\om|_p]:\om\in\Ga_i(c)\}
$$
consists of mutually disjoint sets. From now on throughout this section we fix one such $p\ge 1$ arbitrary. Note that then
\beq\label{5_2016_06_20}
\(\Ga_i(c)\cap [\om|_p]\)\big|_n=\{\om|_n\}
\eeq
for all $n\ge p$. Clearly, in order to show that \eqref{1_2016_06_16} holds, it suffices to prove that
\beq\label{2_2016_06_16}
\nu_{t,\g}(\^f_{i,c}^{-1}(A))\lek \nu_{t,\g}^{1/l}(A)
\eeq
for every $\om\in\Ga_i(c)$ and for every Borel set $A\sbt J(\^f)\cap \([\om|_p]\times B_2(f_i(c),\d)\)$; in fact for $A\sbt J(\^f)\cap \([\om|_p]\times B_2(f_i(c),\d)\)\sms\{(\om,f_i(c))\}$ as $\nu_{t,\g}(\{i\om,c\})=0$. Fix such respective $\om$ and $A$ arbitrary. For every $n\ge 0$ and every $\om\in \Ga_i(c)$ let
$$
\l_n(\om):=\big|\(\^f^n\)'(\om,f_i(c))\big|
\  \and \
\hat\l_n(\om):=\max\big\{\l_k(\om):0\le k\le n\big\}.
$$
Because of Lemma~\ref{l2sp1C}, we have that
\beq\label{3_2016_06_16}
\l_n(\om)\le \hat\l_n(\om)\lek \l_n(\om).
\eeq

%\l_n &:=\max\{\l_n(\om):\om\in \Ga_i(c)\big\}, \\
%\hat\l_n &:=\max\{\l_k:0\le k\le n\big\} \\
%\^\l_n &:=\min\{\l_n(\om):\om\in \Ga_i(c)\big\}.
%\end{aligned}
%$$
\sp The structure of the proof consists of the following six steps \eqref{p1cp3 proof item a}, \eqref{p1cp3 proof item b}, \eqref{p1cp3 proof item c}, \eqref{p1cp3 proof item d}, \eqref{p1cp3 proof item e}, and \eqref{p1cp3 proof item f}. First, we will prove that

\sp\begin{itemize}
\item[\mylabel{a}{p1cp3 proof item a}]
$$
\nu_{t,\g}\Big([i\om|_n]\times B_2\(c,(Q\d\l_n^{-1}(\om))^{1/q}\)\Big)
\lek \l_n^{-t/q}(\om)e^{-\P(t)n}
$$
for all integers $n\ge 1$,

\sp\fr and

\sp
\item[\mylabel{b}{p1cp3 proof item b}]
$$
[i\tau]\times f_{i,c}^{-1}\Big(B_2\(f_i(c),\d\hat\l_n^{-1}(\om)\)\Big)\sbt F(\tf)
$$
for all integers $n\ge p+1$ and all $\tau\in\Sg_u^n\sms\{\om|_n\}$ such that $\tau|_p=\om|_p$.

\sp\item[\mylabel{c}{p1cp3 proof item c}] Having in mind the task of item \eqref{p1cp3 proof item d} ahead, fixing an integer $s\ge p+1$, we will partition the ball $B_2(f_i(c),\d)$ into suitable annuli and define the stopping time
$$
k:=\sup\Big\{n\ge 0:\nu_{t,\g}\Big(A\cap\([\om|_p]\times A(f_i(c);\d\hat\l_{sn}^{-1}(\om),\d)\)\Big)\le \hat\l_{sn}^{-t}(\om)e^{-\P(t)sn}\Big\}.
$$
Of course $0$ belongs to the set whose supremum is being taken above to define $k$ and we will show that $k$ is a finite number. Then combining this along with \eqref{p1cp3 proof item a} and \eqref{p1cp3 proof item b}, we will be able to prove that

\sp\item[\mylabel{d}{p1cp3 proof item d}] 
$$
\nu_{t,\g}(\^f_{i,c}^{-1}(A))\lek \l_{sk}^{-t}(\om)e^{-\P(t)sk}.
$$
\end{itemize}

\sp\fr Finally, we will prove the following two facts which will finish the proof of our proposition.

\sp\begin{itemize}

\item[\mylabel{e}{p1cp3 proof item e}] 
$$
\l_{sk}^{-t/q}(\om)e^{-\P(t)sk}\le\(\l_{sk}^{-t}(\om)e^{-\P(t)sk}\)^{1/l}
$$
for some $l>0$ and every $s\ge 1$ sufficiently large.

\sp\item[\mylabel{f}{p1cp3 proof item f}] 
$$
\l_{sk}^{-t}(\om)e^{-\P(t)sk}
\lek \nu_{t,\g}\Big(A\cap \([\om|_p]\times A(f_i(c);\d\hat\l_{s(k+1)}^{-1}(\om),\d)\)\Big)\, 
(\le \nu_{t,\g}(A)).
$$
\end{itemize}

\sp{\bf Proof of \eqref{p1cp3 proof item a}.} Fix an integer $n\ge 1$. Assume $\d>0$ to be so small that $8Q^2\d<R_2$, where $R_2>0$ comes from Lemma~\ref{l2sp1C}. It then follows from this lemma that for every integer $n\ge 0$ there exists a unique continuous inverse branch 
$$
\^f_{(\om,f_i(c))}^{-n}:B\(\tf^n(\om,f_i(c)),8Q^2\d\)\lra \Sg_u\times\C
$$
 sending 
$\tf^n(\om,f_i(c))$ back to $(\om,f_i(c))$. Also because of this same lemma there exists an integer $s\ge 1$ so large that
\beq\label{1_2016_06_23}
\l_{j+s}(\om)> \l_j(\om)
\eeq
for all integers $j\ge 0$. We furthermore require this integer $s\ge 1$ to be so large that if $\tau\in \Sg_u$ and $\rho\in \Sg_u^*$, then
\beq\label{2_2016_06_23}
\rho B_1(\tau,8Q^2\d\)\spt [\rho|_{j+s}],
\eeq
where $j=|\rho|$. Of course, it suffices to prove the inequality required in \eqref{p1cp3 proof item a} for all integers $n\ge s$. Due to our choice of $\xi$ and $\d$, the upper part of Lemma~\ref{l1_2016_06_20} is then applicable to the inverse branch 
$$
\^f_{(\om,f_i(c))}^{-(n-s)}:B\(\tf^{n-s}(\om,f_i(c)),8Q^2\d\)\lra \Sg_u\times\C.
$$
Applying also to it $\frac14$--Koebe's Distortion Theorem, making use of both \eqref{1_2016_06_23} and  \eqref{2_2016_06_23}, and then applying Koebe's Distortion Theorem, we get first
$$
\^f_{(\om,f_i(c))}^{-(n-s)}\(B\(\tf^{n-s}(\om,f_i(c)),4Q^2\d)\)\)
\spt [\om|_n]\times B_2\(f_i(c),Q^2\d\l_n^{-1}(\om)\),
$$
and then
\beq\label{3_2016_06_21}
\begin{aligned}  
\nu_{t,\g}\([\om|_n]\times B_2\(f_i(c),Q^2\d\l_n^{-1}(\om)\)\)
&\le \nu_{t,\g}\Big(\^f_{(\om,f_i(c))}^{-(n-s)}\(B\(\tf^{n-s}(\om,f_i(c)),4Q^2\d)\)\Big) \\
&\le K^t\l_{n-s}^{-t}(\om)e^{-\g n} \\
&\le K^t\|\tf'\|_\infty^{st}\l_n^{-t}(\om)e^{-\g n} \\
&\le K^t\|\tf'\|_\infty^{st}\l_n^{-t}(\om)e^{-\P(t)n}.
\end{aligned}
\eeq
Using this and \eqref{1cp3}, keeping in mind \eqref{1_2016_06_23}, and writing uniquely $n=n^*s+r$, $0\le r\le s-1$, we obtain
$$
\begin{aligned}
\nu_{t,\g}
\Big(&[(i\om)|_{n+1}]\times B_2\(c,(Q\d\l_n^{-1}(\om))^{1/q}\)\Big)=\\
&=    \sum_{j=n^*}^\infty \nu_{t,\g}\Big(\([(i\om)|_{js+r+1}]\sms[(i\om)|_{(j+1)s+r+1}]\)\times 
      A\(c;(Q\d\l_{(j+1)s+r}^{-1}(\om))^{1/q},(Q\d\l_{js+r}^{-1}(\om))^{1/q}\)\)\Big)\\
&=  \sum_{j=n^*}^\infty \nu_{t,\g}\Big([(i\om)|_{js+r+1}] \times 
      A\(c;(Q\d\l_{(j+1)s+r}^{-1}(\om))^{1/q},(Q\d\l_{js+r}^{-1}(\om))^{1/q}\)\)\Big)\\
&\le \sum_{j=n^*}^\infty \nu_{t,\g}\bigg([(i\om)|_{js+r+1}]\times 
      f_{i,c}^{-1}\Big(A\(f_i(c);\d\l_{(j+1)s+r}^{-1}(\om),Q^2\d\l_{js+r}^{-1}(\om)\)\Big)\bigg)\\
&\lek\sum_{j=n^*}^\infty\l_{js+r}^{(1-\frac1q)t}(\om)e^{-\g}
      \nu_{t,\g}\Big([\om|_{js+r}]\times 
      A\(f_i(c);\d\l_{(j+1)s+r}^{-1}(\om),Q^2\d\l_{js+r}^{-1}(\om)\)\Big)\\
&\le  e^{-\g}\sum_{j=n^*}^\infty\l_{js+r}^{(1-\frac1q)t}(\om)
      \nu_{t,\g}\Big([\om|_{js+r}]\times B_2\(f_i(c),Q^2\d\l_{js+r}^{-1}(\om)\)\Big)\\
&\lek \sum_{j=n^*}^\infty\l_{js+r}^{(1-\frac1q)t}(\om)\l_{js+r}^{-t}(\om)e^{-\g(js+r)} \\
&=   \sum_{j=n^*}^\infty\l_{js+r}^{-t/q}(\om)e^{-\g(js+r)} \\
&\le \sum_{j=n^*}^\infty\l_{js+r}^{-t/q}(\om)e^{-\P(t)(js+r)} \\
&\comp\l_n^{-t/q}(\om)e^{-\P(t)n}\lt(1+\sum_{j=n^*+1}^\infty
     \lt(\frac{\l_{js+r}(\om)}{\l_{n^*s+r}(\om)}\rt)^{-t/q}e^{-\P(t)s(j-n^*)}\rt).
\end{aligned}
$$
Hence, all we are left to do in order to prove \eqref{p1cp3 proof item a} is to show that the sum in the above parentheses is uniformly bounded above. Of course it is enough to do this for all $n\ge 1$ large enough. Since $b\in \De_G^*$, we have that 
$$
\chi(i\om,c)>-q_{(i\om,c)}\frac{\P(b)}b.
$$
Since, by Lemma~7.4 in \cite{sush}, the function $t\longmapsto\P(t)/t$ is continuous, there thus exist $\ka>0$, $\eta_1>0$, and some integer $s\ge 1$ so large that
$$
\frac1j\log\big|\(\^f^j\)'(\^f^n(i\om,c))\big|
\ge -\frac{q\P(t)}t+\ka
$$
for all $t\in(b-\eta_1,b+\eta_1)$, all $j\ge s$, and all $n\ge 1$, and where we have abbreviated $q_{(i\om,c)}$ by $q$. Thus for every $j\ge n^*+1$, we have
$$
\begin{aligned} 
\log\lt(\frac{\l_{js+r}(\om)}{\l_{n^*s+r}(\om)}\rt)
&=\sum_{p=n^*}^{j-1}\(\log\l_{s(p+1)+r}(\om)-\log\l_{sp+r}(\om)\) 
 =\sum_{p=n^*}^{j-1}\log\lt(\frac{\l_{s(p+1)+r}(\om)}{\l_{sp+r}(\om)}\rt)  \\
&=\sum_{p=n^*}^{j-1}\log\big|\(\^f^s\)'(\^f^{sp+r}(i\om,c)\)\big| \\
&\ge \lt(-\frac{q\P(t)}t+\ka\rt)s(j-n^*).
\end{aligned}
$$
Therefore,
$$
\begin{aligned}
\lt(\frac{\l_{js+r}(\om)}{\l_{n^*s+r}(\om)}\rt)^{-t/q}e^{-\P(t)s(j-n^*)} 
&\le \exp\lt(-(-q\P(t)+t\ka)\frac{s(j-n^*)}q-\P(t)s(j-n^*)\rt) \\
&=  \exp\lt(-\frac{\ka st}q(j-n^*)\rt).
\end{aligned}
$$
Hence, 
$$
\sum_{j=n^*+1}^\infty\lt(\frac{\l_{js+r}(\om)}{\l_{n^*s+r}(\om)}\rt)^{-t/q}e^{-\P(t)s(j-n^*)}
\le \sum_{j=1}^\infty\exp\lt(-\frac{\ka st}qj\rt)
<+\infty,
$$
and the proof of item \eqref{p1cp3 proof item a} is complete. 
\qed

\sp {\bf Proof of \eqref{p1cp3 proof item b}.}
In order to prove it, note that $f_\tau(f_i(c))\in F(G)$. Since $\tau\ne\om|_n$ we have that $|\om\wedge\tau|\le n-1$, 
and recall that $l:=|\om\wedge\tau|\ge p$. Then $f_{i\om|_l}(c)\in J(G)$, and by virtue of the formula
$$
J(G)=\cup_{\rho\in \Sg_u}J_\rho
$$ 
along with \eqref{5_2016_06_20}, we get that $f_{i\om|_l\tau_{l+1}}(c)\in F(G)$. Now we take $\d>0$ so small that 
$4\|\tf'\|_\infty\d<D(G)$, where $D(G)$ is the number coming from Definition~\ref{d1_2016_06_19} of C--F balanced semigroups. We then have that 
$$
B_2\(f_{i\om|_l\tau_{l+1}}(c),4\|\tf'\|_\infty\d\)\sbt F(G).
$$
Since $f_{i\om|_l}(c)\in\PCV(G)\cap J(G)$ and $G$ is a FNR semigroup, we will have both that $f_{i\om|_l}(c)\notin B_2(\Crit(G),4\|\tf'\|_\infty\d\)$ by taking $\d>0$ small enough, depending on the generators of $G$ only, and that $f_{\tau_{l+1}}|_W$ is 1--to--1, where $W$ is the unique connected component of $f_{\tau_{l+1}}^{-1}\Big(B_2\(f_{i\om|_l\tau_{l+1}}(c),4\|\tf'\|_\infty\d\)\Big)$ containing 
$f_{i\om|_l}(c)$. So,
$$
f_{\om|_l\tau_{l+1},f_i(c)}^{-1}:=f_{\om|_l,f_i(c)}^{-1}\circ \(f_{\tau_{l+1}}|_W\)^{-1}
:B_2\(f_{i\om|_l\tau_{l+1}}(c),4\|\tf'\|_\infty\d\)\lra\C
$$
is a unique holomorphic branch of $f_{\om|_l\tau_{l+1}}$ defined on 
$B_2\(f_{i\om|_l\tau_{l+1}}(c),4\|\tf'\|_\infty\d\)$ and sending $f_{\om|_l\tau_{l+1}}(c)$ to $f_i(c)$, 
$$
[\om|_l\tau_{l+1}]\times f_{\om|_l\tau_{l+1},f_i(c)}^{-1}\Big(B_2\(f_{i\om|_l\tau_{l+1}}(c),4\|\tf'\|_\infty\d\)\Big) \sbt F(\tf),
$$
and, by the $\frac14$ Koebe's Distortion Theorem, 
$$
\begin{aligned}
f_{\om|_l\tau_{l+1},f_i(c)}^{-1}
   \Big(B_2\(f_{i\om|_l\tau_{l+1}}(c),4\|\tf'\|_\infty\d\)\Big) 
&\spt B_2\Big(f_i(c),\l_l^{-1}(\om)|f_{\tau_{l+1}}'(f_{i\om|_l}(c))|^{-1}\|\tf'\|_ 
     \infty\d\)\Big)\\
&\spt B_2\Big(\(f_i(c),\d\l_l^{-1}(\om)\)\Big) \\
&\spt B_2\Big(f_i(c),\d\hat\l_n^{-1}(\om)\)\Big).
\end{aligned}
$$
Therefore,
$$
[i\om|_l\tau_{l+1}]\times f_{i,c}^{-1}\Big(B_2\(f_i(c),\d\hat\l_n^{-1}(\om)\)\Big)
\sbt F(\tf).
$$
Since also
$$
[i\tau]\times f_{i,c}^{-1}\Big(B_2\(f_i(c),\d\hat\l_n^{-1}(\om)\)\Big)
%[i\tau]\times B_2\(c,(\d\hat\l_n^{-1}(\om)))^{1/q}\)
\sbt [i\om|_l\tau_{l+1}]\times f_{i,c}^{-1}\Big(B_2\(f_i(c),\d\hat\l_n^{-1}(\om)\)\Big),
$$
item \eqref{p1cp3 proof item b} is therefore proved.
\qed

\sp{\bf Proof of \eqref{p1cp3 proof item c}.} Keep $s\ge p+1$ as determined in the above proof of item \eqref{p1cp3 proof item a}. Since
$$
\lim_{n\to\infty}\nu_{t,\g}\Big(A\cap \([\om|_p]\times A(f_i(c);\d\hat\l_{sn}^{-1}(\om),\d)\)\Big)
=\nu_{t,\g}\Big(A\cap \([\om|_p]\times B_2(f_i(c),\d)\)\Big)
=\nu_{t,\g}(A)
>0,
$$
in order to see that the stopping time $k$ is finite, it is sufficient to show that 
\beq\label{1cp7B}
\lim_{n\to\infty}\l_{sn}^{-t}(\om)e^{-\P(t)sn}=0.
\eeq
If $\P(t)\ge 0$, this is immediate as $\lim_{n\to\infty}\l_n(\om)=+\infty$. If $\P(t)<0$, then the inequality $t\chi_{(i\om,c)}>-q\P(t)$, holding since $t\in\De_G^*$, implies that 
$$
t\log\big|\(\^f^n\)'(\^f(i\om,c))\big|>-nq\P(t)
$$
for all $n\ge 1$ large enough. Hence
$$
\l_n^{-t}(\om)e^{-\P(t)n}<e^{\P(t)(q-1)n}.
$$
Since $q\ge 2$, formula \eqref{1cp7B} thus follows. The proof of item \eqref{p1cp3 proof item c} is complete.
\qed

\sp{\bf Proof of \eqref{p1cp3 proof item d}.} If $k=0$, the required inequality is trivial as its right--hand side is equal to $1$. So, we may, and we do, assume that $k\ge 1$. Then applying \eqref{p1cp3 proof item a}, \eqref{p1cp3 proof item b}, and \eqref{p1cp3 proof item c} we estimate as follows.
$$
\begin{aligned}
\nu_{t,\g}&(\^f_{i,c}^{-1}(A))= \\
&=\nu_{t,\g}\bigg(\^f_{i,c}^{-1}\Big(A\cap \([\om|_{sk}]\times B_2(f_i(c),\d\hat\l_{sk}^{-1}(\om))\)\Big) \cup \\
& \  \  \  \  \  \   \  \  \  \  \  \  \  \  \  \  \  \  \
 \cup \^f_{i,c}^{-1}\bigg(\bu_{\tau\in\([\om|_p]|_{sk}\cap(\Sg_u^{sk}\sms\{\om|_{sk}\})\)}\Big(A\cap \([\tau]\times B_2(f_i(c),\d\hat\l_{sk}^{-1}(\om))\)\Big)\bigg)\cup \\
& \  \  \  \  \  \  \  \  \  \  \  \  \  \  \  \  \  \  \  \  \  \  \  \  \  \  \  \
  \cup \^f_{i,c}^{-1}\Big(A\cap\([\om|_p]\times A(f_i(c);\d\hat\l_{sk}^{-1}(\om),\d)\)\Big)\bigg)\\
&=\nu_{t,\g}\bigg(\^f_{i,c}^{-1}\Big(A\cap \([\om|_{sk}]\times B_2(f_i(c),\d\hat\l_{sk}^{-1}(\om))\)\Big)\bigg) + \\
& \  \  \  \  \  \  \  \  \  \  \  \  \  \  \  \  \  \  \
+\sum_{\tau\in\([\om|_p]|_{sk}\cap(\Sg_u^{sk}\sms\{\om|_{sk}\})\)} \nu_{t,\g}\bigg(\^f_{i,c}^{-1}\Big(A\cap \([\tau]\times B_2(f_i(c),\d\hat\l_{sk}^{-1}(\om))\)\Big)\bigg) + \\
& \  \  \  \  \  \  \  \  \  \  \  \  \  \  \  \  \  \  \  \  \  \  \  \  \  \  \  \
  +\nu_{t,\g}\bigg(\^f_{i,c}^{-1}\Big(A\cap\([\om|_p]\times A(f_i(c);\d\hat\l_{sk}^{-1}(\om),\d)\)\Big)\bigg)\\
&=\nu_{t,\g}\bigg(\^f_{i,c}^{-1}\Big(A\cap \([\om|_{sk}]\times B_2(f_i(c),\d\hat\l_{sk}^{-1}(\om))\)\Big)\bigg) + \\
& \  \  \  \  \  \  \  \  \  \  \  \  \  \  \  \  \  \  \  \  \  \  \  \  \  \  \  \
  +\nu_{t,\g}\bigg(\^f_{i,c}^{-1}\Big(A\cap\([\om|_p]\times A(f_i(c);\d\hat \l_{sk}^{-1}(\om),\d)\)\Big)\bigg) \\
&\lek \nu_{t,\g}\bigg(\Big([i\om|_{sk}]\times B_2\(c,(Q\d\hat\l_{sk}^{-1}(\om))^{1/q}\)\Big)\bigg) + \\
& \  \  \  \  \  \  \  \  \  \  \  \  \  \  \  \  \  \  \  \  \  \  \  \  \  \  \  \
+ \(\d\hat\l_{sk}^{-1}(\om)\)^{\lt(\frac1q-1\rt)t}e^{-\g} 
\nu_{t,\g}\Big(A\cap\([\om|_p]\times A(f_i(c);\d\hat\l_{sk}^{-1}(\om),\d)\)\Big)\\
&\le \nu_{t,\g}\bigg(\Big([i\om|_{sk}]\times B_2\(c,(Q\d\l_{sk}^{-1}(\om))^{1/q}\)\Big)\bigg) + \\
& \  \  \  \  \  \  \  \  \  \  \  \  \  \  \  \  \  \  \  \  \  \  \  \  \  \  \  \
 + \(\d\hat\l_{sk}^{-1}(\om)\)^{\lt(\frac1q-1\rt)t}e^{-\P(t)}
    e^{-\P(t)sk}\hat\l_{sk}^{-t}(\om)\\
&\lek \l_{sk}^{-t/q}(\om)e^{-\P(t)sk} 
 + \hat\l_{sk}^{-t/q}(\om)e^{-\P(t)sk} \\
%&\comp \hat\l_{sk}^{-t/q}(\om)e^{-\P(t)sk} + \hat\l_{sk}^{-t/q}(\om)e^{-\P(t)sk} \\
&\comp \l_{sk}^{-t/q}(\om)e^{-\P(t)sk}.
\end{aligned}
$$
The proof of item \eqref{p1cp3 proof item d} is thus complete. 
\qed

\sp{\bf Proof of \eqref{p1cp3 proof item e}.} If $k=0$, then both sides of the required inequality are equal to $1$ and we are then done. So, we assume now that $k\ge 1$. Starting the proof we again recall that 
$$
\frac1n\log\big|\(\^f^n\)'(\om,f_i(c))\big|
\ge -\frac{q\P(t)}t+\ka
$$
for all $\om\in\Ga_i(c)$ and all $n\ge s$. Hence, if we take $l\ge q+1$ so large that
$$
\frac{\P(b)q(1-q)}{b(l-q)}<\ka/2.
$$
It then follows from the continuity of the function $t\longmapsto\P(t)/t$, holding because of Proposition~\ref{l1h35} \eqref{l1h35 item b}, that there exists $\eta\in(0,\eta_1]$ such that 
$$
\frac{\P(t)q(1-q)}{t(l-q)}<\ka
$$
for all $t\in(b-\eta,b+\eta)$. This in turn implies that for such parameters $t$ and all $n\ge s$, we have
$$
(l-q)\frac{\log\l_n(\om)}{n}
\ge \lt(-\frac{q\P(t)}t+\ka\rt)(l-q)
\ge \frac{q\P(t)}t(1-l).
$$
An elementary rearrangement then gives
$$
\l_{n}^{-t/q}(\om)e^{-\P(t)n}
\le \lt(\l_{n}^{-t}(\om)e^{-\P(t)n}\rt)^{1/l}.
$$
Hence, 
$$
\l_{sk}^{-t/q}(\om)e^{-\P(t)sk}
\le \lt(\l_{sk}^{-t}(\om)e^{-\P(t)sk}\rt)^{1/l}
$$
if $s\ge p$ is taken to be large enough. Thus item \eqref{p1cp3 proof item e} is proved.
\qed

\sp{\bf Proof of \eqref{p1cp3 proof item f}.} The finiteness of $k\ge 0$ and formula \eqref{3_2016_06_16} yield
$$
\begin{aligned}
\nu_{t,\g}\Big(A\cap\([\om|_p]\times A(f_i(c);\d\hat\l_{s(k+1)}^{-1}(\om),\d)\)\Big)
&\ge \hat\l_{s(k+1)}^{-t}(\om)e^{-\P(t)s(k+1)}  \\
&\gek \l_{s(k+1)}^{-t}(\om)e^{-\P(t)s(k+1)} \\
&\ge e^{-\P(t)s}\|\^f'\|_\infty^{-st}\l_{sk}^{-t}(\om)e^{-\P(t)sk}.
\end{aligned}
$$
The proof of item \eqref{p1cp3 proof item f} is thus complete.
\qed

\sp\fr In conclusion, the conjunction of \eqref{p1cp3 proof item d}, \eqref{p1cp3 proof item e}, and \eqref{p1cp3 proof item f}, completes the proof of \eqref{2_2016_06_16}. Hence, \eqref{1_2016_06_16} follows. Because the semigroup $G$ is of finite type, we thus have that
\beq\label{1_2016_06_21}
\nu_{t,\g}(\^f^{-1}(A))\le C \nu_{t,\g}^{1/l}(A)
\eeq
for every $\g>\P(t)$, every Borel set $A\sbt J(\^f)$ and some constant $C$ independent of $t$ and $\g$. Since $m_t$ is a weak* limit of the measures $\nu_{t,\g}$ when $\g\downto \P(t)$,    with the help of standard consideration involving outer and inner regularity of measures, the formula \eqref{2_2016_06_21} of Proposition~\ref{p1cp3} follows. Also, by the weak* convergence argument we conclude from \eqref{3_2016_06_21} that
$$  
m_t\([\om|_n]\times B_2\(f_i(c),Q^2\d\l_n^{-1}(\om)\)\)
\le K^t\l_n^{-t}(\om)e^{-\P(t)n}
$$
for every integer $n\ge p$. As $t\in \De_G^*$, this formula along with \eqref{1cp7B}, yield $m_t(\om,f_i(c))=0$. This in turn, in conjunction with formula \eqref{2_2016_06_21} entails $m_t(i\om,c)=0$. Hence, $t\in \De_G$, and the proof of Proposition~\ref{p1cp3} is complete.
\epf 

\sp\fr The following proposition enlightens the structure of the sets $\De_G$ and $\De_G^*$.

\bprop\label{p1_2016_06_21}
If $G$ is a \FNR \ rational semigroup generated by a $u$--tuple map $f=(f_1, f_2,\ld, f_u)\in\Rat^u$, then $\De_G^*$ is an open subset of $[0,+\infty)$ and there exists $\eta>0$ such that
$$
[0,h_f+\eta)\sbt \De_G^* \sbt \De_G.
$$
\eprop

\bpf
The inclusion $\De_G^* \sbt \De_G$ and openness of $\De_G^*$ in $[0,+\infty)$ both follow directly from Proposition~\ref{p1cp3}, while the inclusion $\De_G^*\spt [0,h_f+\eta)$ with some $\eta>0$ is an immediate consequence of the definition of $\De_G^*$ and continuity of the pressure function $[0,+\infty)\ni t\longmapsto\R$ proved in Proposition~\ref{l1h35} \eqref{l1h35 item b}. 
\epf

\section{Small Pressure $\P_V^\Xi(t)$}\label{section:smallpressure} 

\fr Let $G$ be a finitely generated *semi--hyperbolic rational semigroup generated by a $u$--tuple map $f=(f_1,\ld,f_u)\in\Rat^u$. Let $V\sbt\Sg_u\times \oc$ be a non--empty open neighborhood of $\Crit_*(\tf)$. Let $\Xi$ be a non--empty finite subset of $K(V)$. For every $t\ge 0$ we define

\beq\label{1_2016_06_24}
\P_V^\Xi(t):=\limsup_{n\to\infty}\,
\frac1n\,\log\sum_{\xi\in\(\tf|_{K(V)}\)^{-n}(\Xi)}\!\!\!\big|\(\tf^n\)'(\xi)\big|^{-t}.
\eeq\index{topological pressure! $\P_V^\Xi(t)$}\index{$\P_V^\Xi(t)$}

\fr The key technical result of this section is the following.

\blem\label{l1sp1}
Let $G$ be a finitely generated *semi--hyperbolic rational semigroup generated by a $u$--tuple map $(f_1,\ld,f_u)\in\Rat^u$. Let $t\in \De_G$. If $V$ is an open neighborhood of $\Crit_*(\^f)$ such that $V\cap J(\^f)\ne\es$, and $\Xi$ is a non--empty finite subset of $K(V)$, then
$$
\P_V^\Xi(t)<\P(t).
$$ 
\elem

\bpf
Note that the map $\tf|_{K(V)}:K(V)\to K(V)$ has no critical points and all its points of non--openness, i.e. points $\xi\in K(V)$ that have no local base of topology (consisting of open sets relative to $K(V)$) whose images under $\tf|_{K(V)}$ are also open relative to $K(V)$ , is contained in $\bd V$, the boundary of $V$ in $\Sg_u\times \oc$. Defining the sets
$$
E_n:=\(\tf|_{K(V)}\)^{-n}(\Xi),  \  \  n\ge 0,
$$
and applying to them Lemmas~8.2.6 and 8.2.7 of \cite{KUbook} with the function $\phi:=-t\log\big|\(\tf^n\)'(\xi)\big|$, we conclude that there exists $m_{V,t}^\Xi$, a Borel probability measure on $K(V)$, with the following properties:

\sp\begin{enumerate}
\item[\mylabel{a}{l1sp1 item a}] 
$$
m_{V,t}^\Xi(\tf(A))\ge e^{\P_V^\Xi(t)}\int_A|\tf'(\xi)\big|^t\,dm_{V,t}^\Xi
$$
for every Borel set $A\sbt K(V)$ such that the map $\tf|_A$ is 1--to--1,

\,

\fr and 

\,

\sp\item[\mylabel{b}{l1sp1 item b}] 
$$
m_{V,t}^\Xi(\tf(A))= e^{\P_V^\Xi(t)}\int_A|\tf'(\xi)\big|^t\,dm_{V,t}^\Xi
$$
if in addition $A\cap \bd V=\es$.
\end{enumerate}

\sp Now, seeking contradiction suppose that 
$$
\P_V^\Xi(t)=\P(t).
$$ 
Fix a point $\xi=(\tau,z)\in K(V)$ arbitrary. Applying formula \eqref{l1sp1 item a} above consecutively, we would obtain that 
\beq\label{2_2016_06_24}
m_{V,t}^\Xi(B_j^*(\tau,z))
\le K^t\exp\(-\P_V^\Xi(t)n_j(\xi)\)\big|\(\tf^{n_j(\xi)}\)'(\xi)\big|^{-t},
\eeq
where the sets $B_j^*(\tau,z)$ and integers $n_j(\xi)$, $j\ge 1$, are defined by formulas \eqref{2h67C} and \eqref{2h67A}. We also know that
\beq\label{19_2016_06_24}
m_t(B_j^*(\tau,z))
\ge C_t^{-1}\exp\(-\P(t)n_j(\xi)\)\big|\(\tf^{n_j(\xi)}\)'(\xi)\big|^{-t}
\eeq
with some constant $C_t\ge 1$. Therefore,
\beq\label{3_2016_06_24}
m_{V,t}^\Xi(B_j^*(\tau,z))
\le K^tC_tm_t(B_j^*(\tau,z)).
\eeq
Now, the same proof (with $m_t$ replaced by $m_{V,t}^\Xi$) as that of Lemma~\ref{l1h71}, gives us that the family $\Ba$ restricted to $K(V)$ is a Vitali relation for the measure $m_{V,t}^\Xi$ on the set $K(V)\sbt J(\tf)\sms \Sing(\tf)$. In conjunction with \eqref{3_2016_06_24} this yields the measure $m_{V,t}^\Xi$ to be absolutely continuous with respect to $m_t$, even more we see that  $\frac{dm_{V,t}^\Xi}{dm_t}\le K^tC_t$. But the set $K(V)$, and all its forward iterates under $\tf$, are disjoint from $V$. Hence, $K(V)\cap\Trans(\tf)=\es$, where, we recall, $\Trans(\tf)$ is the set of transitive points of $\tf:J(\tf)\lra J(\tf)$. By virtue of Corollary~\ref{c12015_01_30} this entails $m_t(K(V))=0$. Thus also $m_{V,t}^\Xi(K(V))=0$, and this contradiction finishes the proof.
\epf

\sp
\brem\label{r2_2016_06_24}
The proofs of Lemmas~8.2.6 and 8.2.7 from \cite{KUbook} are rather minor improvements of those from Section~3 in \cite{DUECM}. 
\erem

\sp From now on throughout this section we take $G$ to be a \FNR \ rational semigroup generated by a $u$--tuple map $(f_1,\ld,f_u)\in\Rat^u$. Fix $R\in(0,R_*(\tf))$. Let $\Crit_*(\tf)\sbt S\sbt J(\^f)\sms B_2(\PCV(\tf),8R)$ be an arbitrary finite aperiodic set. Let $\cU=\{U_s\}_{s\in S}$  be a nice family of sets, the existence of which is guaranteed by Theorem~\ref{t1nsii7}. Recall that
$$
U=\bu_{s\in S}U_s.
$$
We now also recall that for every $s\in S$ and every $n\ge 1$ we have denoted:
$$
\cD_n^*(s):=\big\{\tau\in\HIB_n(p_2(U_s)):\^f^k\(\^f_\tau^{-n}(U_s)\)\cap U=\es \for \all 0\le k\le n-1\big\}.
$$
We put
$$
\cD_n^*(S):=\bu_{s\in S}\cD_n^*(s).
$$\index{$\cD_n^*(S)$}
Recall that 
$$
X_s=\ov U_s
$$
for every $s\in S$. We shall prove the following.

\blem\label{l1sp2}
Let $G$ be a \FNR \ rational semigroup generated by $(f_1,\ld,f_u)\in\Rat^u$. If $t\in\De_G^*$ and $\cU_S=\{U_s\}_{s\in S}$ is a nice family of sets, produced in Theorem~\ref{t1nsii7}, then there exist $\g>0$ and a finite set $\Xi\sbt K(B_2(S,\g))$ such that for every $\ve>0$ there exists $C_\cU(\ve)>0$ such that for every integer $n\ge 1$ we have that
$$
m_t\Big(\bu_{s\in S}\bu_{\tau\in \cD_n^*(s)}\^f_\tau^{-n}(X_s)\Big)
\le C_\cU(\ve)\exp\((\P_{B_2(S,\g)}^\Xi(t)-\P(t)+\e)n\)
$$
with the number $\P(t)-\P_{B_2(S,\g)}^\Xi(t)$ being positive.
\elem

\bpf
Since periodic points of $\tf:J(\^f)\lra J(\^f)$ are dense in $J(\^f)$, for every $s\in S$ there exists $\xi_s$, a periodic point of $\^f$ belonging to $U_s\sms S$, whose forward orbit under $\^f$ is disjoint from $S$. Then there exists $\g>0$ so small that
\beq\label{1_2016_06_11}
B_2(S,\g)\sbt U
\eeq
and 
$$
B_2(S,\g)\cap \bu_{n=0}^\infty \^f^n(\{\xi_s:s\in S\})=\es.
$$
This latter formula is equivalent to saying that
$$
\Xi:=\{\xi_s:s\in S\}\sbt K(B_2(S,\g)),
$$
which in the conjunction with \eqref{1_2016_06_11}, gives that
$$
\bu\Big\{\^f_\tau^{-n}(\xi_s): s\in S,\, \tau\in \cD_n^*(s)\Big\}\sbt K(B_2(S,\g))
$$
for every integer $n\ge 0$.
It therefore follows from the definition of $\P_{B_2(S,\g)}^\Xi(t)$ that
$$
\frac1n\log\sum_{s\in S}\sum_{\tau\in \cD_n^*(s)}\Big|\Big(\^f_\tau^{-n}\Big)'(\xi_s)\Big|^t
\le \frac1n\log\sum_{\xi\in\(\tf|_{K(V)}\)^{-n}(\Xi)}\!\!\!\big|\(\tf^n\)'(\xi)\big|^{-t}
\le \P_{B_2(S,\g)}^\Xi(t)+\frac{\e}{2}
$$
for all $n\ge 1$ large enough. Invoking Koebe's Distortion Theorem, we further see that for all $n\ge 1$ large enough, say $n\ge N$, we have
$$
\frac1n\log\sum_{s\in S}
\sum_{\tau\in \cD_n^*(s)}\Big\|\Big(\^f_\tau^{-n}\Big)'\Big|_{X_s}\Big\|_\infty^t
\le \P_{B_2(S,\g)}^\Xi(t)+\e.
$$
Equivalently, for all $n\ge N$:
$$
\frac1n\log\sum_{s\in S}
\sum_{\tau\in \cD_n^*(s)}\Big\|\Big(\^f_\tau^{-n}\Big)'\Big|_{X_s}\Big\|_\infty^t
e^{-\P(t)n}
\le \P_{B_2(S,\g)}^\Xi(t)-\P(t)+\e,
$$
and note that the number $\P(t)-\P_{B_2(S,\g)}^\Xi(t)$ is positive because of Lemma~\ref{l1sp1}. But 
$$
m_t\(\^f_\tau^{-n}(X_s)\)
\le \Big\|\Big(\^f_\tau^{-n}\Big)'\Big|_{X_s}\Big\|_\infty^te^{-\P(t)n},
$$
whence 
$$
\sum_{s\in S}\sum_{\tau\in \cD_n^*(s)}m_t\(\^f_\tau^{-n}(X_s)\)
\le\exp\((\P_{B_2(S,\g)}^\Xi(t)-\P(t)+\e)n\).
$$
Therefore, our lemma follows with appropriate constant $C_\cU(\e)>0$ determined by the values of the above sum for $n$s ranging from $1$ up to $N-1$.
\epf

\sp\fr As an immediate consequence of this lemma and Proposition~\ref{p1cp3} we get the following. 

\bprop\label{p1sp4}
Let $G$ be a \FNR \ rational semigroup generated by a $u$--tuple map $(f_1,\ld,f_u)\in\Rat^u$. If $\cU=\{U_s\}_{s\in S}$ is a nice family of sets produced in Theorem~\ref{t1nsii7}, and 
$$
\cS_\cU=\big\{\^f_\tau^{-\|\tau\|}:X_{t(\tau)}\lra X_{i(\tau)} \big\}_{\tau\in\cD_\cU}
$$ 
is the corresponding graph directed system, then there exist $\g>0$ and a finite set $\Xi\sbt K(B_2(S,\g))$ such that for every $b\in\De_G^*$ there exists $\eta>0$ such that for every $t\in\De_G^*\cap(b-\eta,b+\eta)$ and every $\e>0$
$$
m_t\Bigg(\bu_{s\in S}\bu_{\tau\in \cD_n(s)}\^f_\tau^{-n}(X_s)\Bigg)
\le C_\e\exp\lt(\lt(\frac{\P_{B_2(S,\g)}^\Xi(t)-\P(t)+\e}l\rt)n\rt)
$$
for every integer $n\ge 1$ and some constant $C_\e\in(0,+\infty)$ depending on $\e$. Also,  $$
\P(t)-\P_{B_2(S,\g)}^\Xi(t)>0.
$$
\eprop

\section{Symbol Space Thermodynamic Formalism Associated to Nice Families;\nl Real Analyticity of the Original Pressure $\P(t)$}\label{sec:thermodynamic formalism}

This section brings up the full fledged fruits of the existence of nice families. It forms a symbolic representation (subshift of finite type with a countable infinite alphabet) of the map generated by a nice family and develops the thermodynamic formalism of the potentials $\zeta_{t,s}$ resulting from those of the form $-t\log|\tf'|$ and the ``first return time'' $\|\tau_1\|$. 

\sp Throughout this section $G$ is a \FNR \ rational semigroup generated by a $u$--tuple map $(f_1,\ld,f_u)\in\Rat^u$. Let
$$
\cU=\{U_s\}_{s\in S}
$$
be a nice family of sets coming from Theorem~\ref{t1nsii7}. Let $\cS_\cU$ be the corresponding graph directed system described in Theorem~\ref{t1nsii15}. Let 
$$
\pi_\cU:\cD_\cU^\infty\longrightarrow\Sg_u\times\oc
$$\index{projection map!$\pi_\cU:\cD_\cU^\infty\longrightarrow\Sg_u\times\oc$}\index{$\pi_\cU:\cD_\cU^\infty\longrightarrow\Sg_u\times\oc$}
be the canonical projection from $\cD_\cU^\infty$ to $\Sg_u\times\oc$, determined by the system $\cS_\cU$; see \cite{mugdms} for details. We recall (see also \cite{mugdms}) that $J_\cU$, the limit set of $\cS_\cU$, is defined as $\pi_\cU(\cD_\cU^\infty)$. 

\sp Given two real numbers $t, s\ge 0$ we define the function $\zeta_{t,s}:\cD_\cU^\infty\lra\R$ by the following formula:
$$
\zeta_{t,s}(\tau)
:=-t\log\big|(\tf^{\|\tau_1\|}\)'(\pi_\cU(\sg(\tau)))\big|-s\|\tau_1\|.
$$\index{$\zeta_{t,s}$}
The functions $\zeta_{t,s}$, $t, s\ge 0$, will be frequently referred to as \textbf{potentials}. Let $d_\vartheta$ be the metric defined on $\cD_\cU^\infty$ by the formula
$$
d_\vartheta,(\a,\b):=\vartheta^{\min\{n\ge 0:\a_{n+1}\ne\b_{n+1}\}}.
$$
Because of Property~\eqref{t1nsii7 item B} of Theorem~\ref{t1nsii7}, we immediately get the following.

\blem\label{l1tf1}
Let $G$ be a \FNR \ rational semigroup generated by a $u$--tuple map $(f_1,\ld,f_u)\in\Rat^u$. If $\cU=\{U_s\}_{s\in S}$ is a nice family of sets produced in Theorem~\ref{t1nsii7}, and $\cS_\cU=\big\{\^f_\tau^{-\|\tau\|}:X_{t(\tau)}\lra X_{i(\tau)} \big\}_{\tau\in\cD_\cU}$ is the corresponding graph directed system, then
the projection map 
$$
\pi_\cU:\cD_\cU^\infty\lra J_\cU
$$ 
is Lipschitz continuous if $\cD_\cU^\infty$ is endowed with the metric $d_\vartheta$.
\elem

\fr Applying Koebe's Distortion Theorem, this lemma yields the following.

\bprop\label{p1tf2}
Let $G$ be a \FNR \ rational semigroup generated by a $u$--tuple map $(f_1,\ld,f_u)\in\Rat^u$. If $\cU=\{U_s\}_{s\in S}$ is a nice family of sets produced in Theorem~\ref{t1nsii7}, and $\cS_\cU=\big\{\^f_\tau^{-\|\tau\|}:X_{t(\tau)}\lra X_{i(\tau)} \big\}_{\tau\in\cD_\cU}$ is the corresponding graph directed system, then
for all $t, s\ge 0$ the function 
$$
\zeta_{t,s}:\cD_\cU^\infty\lra\R
$$ 
is Lipschitz continuous. 
\eprop

\fr Let $\P(t,s)$\index{topological pressure! $\P(t,s)$}\index{$\P(t,s)$} be the topological pressure of the potential $\zeta_{t,s}:\cD_\cU^\infty\lra\R$ with respect to the shift map $\sg:\cD_\cU^\infty\lra\cD_\cU^\infty$. Given any function $g:\cD_\cU^\infty\lra\R$ and any integer $n\ge 1$, let
$$
S_n(g):=\sum_{j=0}^{n-1}g\circ\sg^j
$$\index{Birkhoff sum! $S_n(g)$}\index{$S_n(g)$}
be the \textbf{$n$th Birkhoff sum} of $g$ with respect to the dynamical system $\sg:\cD_\cU^\infty\lra\cD_\cU^\infty$. If
$$
\sum_{e\in \cD_\cU}\|\phi_e'\|_\infty^te^{-s\|e\|}<+\infty,
$$
then we call the function $\zeta_{t,s}$ \textbf{summable}; we then also call the parameter $(t,s)$ summable. Denote by $\Om(\cU)$\index{summable parameters!$\Om(\cU)$}\index{$\Om(\cU)$} the set of all summable parameters $(t,s)\in [0,+\infty)\times[0,+\infty)$. Invoking Proposition~\ref{p1tf2}, by virtue of Corollary~2.7.5 (a), (b), and then (c) in \cite{mugdms}, we respectively obtain the following two theorems. 

\bthm\label{t2tf2}
Let $G$ be a \FNR \ rational semigroup generated by a $u$--tuple map $(f_1,\ld,f_u)\in\Rat^u$. Let $\cU=\{U_s\}_{s\in S}$ be a nice family of set produced in Theorem~\ref{t1nsii7} and let $\cS_\cU=\big\{\^f_\tau^{-\|\tau\|}:X_{t(\tau)}\lra X_{i(\tau)} \big\}_{\tau\in\cD_\cU}$ the corresponding graph directed system.

If $(t,s)\in\Om(\cU)$, then there exists a unique Borel probability measure $\^m_{t,s}$\index{$\^m_{t,s}$} on 
$\cD_\cU^\infty$ such that
$$
\^m_{t,s}(eF)
=e^{-\P(t,s)}\int_F A_{e\om_1}(\cU)\big|\phi_e'(\pi_\cU(\om))\big|^te^{-s\|e\|}\,d\^m_{t,s}(\om)
$$
for every $e\in \cD_\cU$ and every Borel set $F\sbt \cD_\cU^\infty$. We then have a stronger property than the displayed formula above, namely
$$
\^m_{t,s}(\tau F)
=e^{-\P(t,s)|\tau|}\int_F A_{\tau_k\om_1}(\cU)\big|\phi_\tau'(\pi_\cU(\om))\big|^te^{-s|||\tau|||}\,d\^m_{t,s}(\om)
$$
for every $\tau\in \cD_\cU^*$ and every Borel set $F\sbt \cD_\cU^\infty$, where $k=|\tau|$, and 
$$
|||\tau|||:=\sum_{j=0}^k\|\tau_j\|. 
$$
\ethm

\bthm\label{t1tf3}
Let $G$ be a \FNR \ rational semigroup generated by a $u$--tuple map $(f_1,\ld,f_u)\in\Rat^u$. Let $\cU=\{U_s\}_{s\in S}$ be a nice family of sets produced in Theorem~\ref{t1nsii7}, and let $\cS_\cU=\big\{\^f_\tau^{-\|\tau\|}:X_{t(\tau)}\lra X_{i(\tau)} \big\}_{\tau\in\cD_\cU}$ be the corresponding graph directed system.

If $(t,s)\in\Om(\cU)$, then there exists a unique $\sg$--invariant Gibbs state $\^\mu_{t,s}$\index{$\^\mu_{t,s}$} for the potential $\zeta_{t,s}:\cD_\cU^\infty\lra\R$ with respect to the shift map $\sg:\cD_\cU^\infty\lra\cD_\cU^\infty$. This means that $\^\mu_{t,s}$ is $\sg$--invariant and at least one (equivalently all) of the following hold:

\sp\begin{itemize}
\item[\mylabel{a}{t1tf3 item a}]
$$
C^{-1}
\le \frac{\^\mu_{t,s}([\tau|_n])}
    {\big|\phi_{\tau|_n}'(\pi_\cU(\tau))\big|^te^{-s|||\tau|_n|||}}
\le C
$$
with some constant $C\ge 1$, all integers $n\ge 1$, and all $\tau\in\cD_\cU^\infty$.

\sp\item[\mylabel{b}{t1tf3 item b}]
$$
C^{-1}
\le \frac{\^m_{t,s}([\tau|_n])}
    {\big\|\phi_{\tau|_n}'\big\|_\infty^te^{-s|||\tau|_n|||}}
\le C
$$
with some constant $C\ge 1$, all integers $n\ge 1$, and all $\tau\in\cD_\cU^\infty$.

\sp\item[\mylabel{c}{t1tf3 item c}] $\^\mu_{t,s}$ is absolutely continuous with respect to $\^m_{t,s}$. 

\sp\item[\mylabel{d}{t1tf3 item d}] $\^\mu_{t,s}$ is equivalent to $\^m_{t,s}$.

\sp\item[\mylabel{e}{t1tf3 item e}] $\^\mu_{t,s}$ is equivalent to $\^m_{t,s}$ and the Radon--Nikodym derivative $\frac{d\^\mu_{t,s}}{d\^m_{t,s}}$ is a log bounded H\"older continuous function.
\end{itemize}

\sp\fr In addition, the measure $\^\mu_{t,s}$ is ergodic with respect to the shift map $\sg:\cD_\cU^\infty\lra\cD_\cU^\infty$. 
\ethm

Now we intend to establish some relationships between the measures $m_t$ and $\^m_{t,s}$. We first obtain a straightforward (by now) fact about the formers. Indeed, as an immediate consequence of Corollary~\ref{c12015_01_30}, of the fact that $\supp(m_t)=J(\^f)$, and of the fact that $J_\cU\spt (U\cap \Trans(\tf))$ (see Theorem~\ref{t1nsii15}), we get the following.

\blem\label{l2tf3}
Let $G$ be a \FNR \ rational semigroup generated by a $u$--tuple map $(f_1,\ld,f_u)\in\Rat^u$. Let $\cU=\{U_s\}_{s\in S}$ be a nice family of sets produced in Theorem~\ref{t1nsii7}, and let $\cS_\cU=\big\{\^f_\tau^{-\|\tau\|}:X_{t(\tau)}\to X_{i(\tau)} \big\}_{\tau\in\cD_\cU}$ be the corresponding graph directed system.
If $t\in\De_G$, then $m_t(J_\cU)>0$; furthermore
$$
m_t\(U_s\cap J_\cU\)>0  
$$
for all $s\in S$.
\elem

\fr The announced link between the measures $m_t$ and $\^m_{t,s}$ is given by the following.

\bprop\label{p1tf5}
Let $G$ be a \FNR \ rational semigroup generated by a $u$--tuple map $(f_1,\ld,f_u)\in\Rat^u$. Let $\cU=\{U_s\}_{s\in S}$ be a nice family of sets produced in Theorem~\ref{t1nsii7}, and let $\cS_\cU=\big\{\^f_\tau^{-||\tau||}:X_{t(\tau)}\lra X_{i(\tau)} \big\}_{\tau\in\cD_\cU}$ be the corresponding graph directed system.

If $t\in\De_G$, then $(t,\P(t))\in \Om(\cU)$ and 
\beq\label{520180604}
\P(t,\P(t))=0. 
\eeq
Furthermore,
\beq\label{1tf5}
\^m_{t,\P(t)}([\tau])
\comp m_t\(\phi_\tau(U_{t(\tau)})\)
\comp m_t\(\phi_\tau(X_{t(\tau)})\)
\eeq
for all $\tau\in\cD_\cU^*$.
\eprop

\bpf
Taking $\e>0$ small enough, it immediately follows from Proposition~\ref{p1sp4} that $(t,\P(t))\in \Om(\cU)$. By the very definition of both measures $m_t$ and $\^m_{t,\P(t)}$, along with Koebe's Distortion Theorem, we have that
$$
\begin{aligned}
m_t\(\phi_\tau(U_{t(\tau)})\)
&      =\int_{U_{t(\tau)}}e^{-\P(t)|||\tau|||}|\phi_\tau'(z)|^t\,dm_t(z)
 \comp e^{-\P(t)|||\tau|||}||\phi_\tau'||_\infty^tm_t(U_{t(\tau)}) \\
&\comp e^{-\P(t)|||\tau|||}||\phi_\tau'||_\infty^t \\
&=     e^{\P(t,P(t))|||\tau|||}
       \(e^{-\P(t,P(t))|||\tau|||}e^{-\P(t)|||\tau|||}||\phi_\tau'||_\infty^t\) \\
&\comp e^{\P(t,P(t))|||\tau|||}\^m_{t,\P(t)}([\tau]).
\end{aligned}
$$
Also, the same formula holds with $U$ replaced everywhere by $X$.
Therefore, if on the one hand $\P(t,\P(t))>0$, then by applying the standard covering argument with sets of the form $[\tau]$ and $U_{t(\tau)}$ with $|||\tau|||$ diverging to $+\infty$, we would conclude that $\^m_{t,\P(t)}\(\cD_\cU^\infty\)=0$, which is a contradiction. If on the other hand, $\P(t,\P(t))<0$, then by the same token $m_t(J_\cU)=0$, which contradicts Lemma~\ref{l2tf3}. Thus $\P(t,\P(t))=0$, and so formula \eqref{1tf5} is also proved.
\epf

\sp\fr From now on we denote
$$
\^m_t:=\^m_{t,\P(t)} \  \and  \  \^\mu_t:=\^\mu_{t,\P(t)}.
$$\index{$\^m_t$}\index{$\^\mu_t$}

\blem\label{l2tf5}
Let $G$ be a \FNR \ rational semigroup type generated by a $u$--tuple map $(f_1,\ld,f_u)\in\Rat^u$. Let $\cU=\{U_s\}_{s\in S}$ be a nice family of sets produced in Theorem~\ref{t1nsii7} and let $\cS_\cU=\big\{\^f_\tau^{-\|\tau\|}:X_{t(\tau)}\lra X_{i(\tau)} \big\}_{\tau\in\cD_\cU}$ the corresponding graph directed system.

If $b\in\De_G^*$, then there exists $\d>0$ such that 
$$
\De_G^*\cap \((b-\d,b+\d)\times(\P(b)-\d,\P(b)+\d)\)\sbt \Om(\cU).
$$
\elem

\bpf
Of course it suffices to show that 
$$
\De_G^*\cap \((b-\d,b]\times(\P(b)-\d,\P(b)]\)\sbt \Om(\cU)
$$
for some $\d>0$. Fix a finite set $S$ along with $\eta>0$, $\g>0$ and a finite set $\Xi\sub K(B_2(S,\g))$, all four of them coming from Proposition~\ref{p1sp4}. Fix $\e>0$. It follows from Proposition~\ref{p1sp4} and the generalized conformality of $m_t$ that, for $t\le b$ and $\P(b)-\d\le s\le \P(b)$ with $t\in\De_G^*$ and $\d>0$ to be determined in the course of the proof, we have 
$$
\begin{aligned}
       \sum_{\tau\in \cD_\cU^n}\|\phi_\tau'\|_\infty^te^{-s\|\tau\|}
&=     \sum_{\tau\in \cD_\cU^n}\|\phi_\tau'\|_\infty^te^{-\P(t)\|
          \tau\|}e^{(\P(t)-s)n}
 =     e^{(\P(t)-s)n}\sum_{\tau\in \cD_\cU^n}
         ||\phi_\tau'||_\infty^te^{-\P(t)||\tau||} \\
&\comp e^{(\P(t)-s)n}\sum_{\tau\in \cD_\cU^n}
       m_t\bigg(\bu_{\tau\in \cD_\cU^n}\phi_\tau\(X_{t(\tau)}\)\bigg) \\
&=     e^{(\P(t)-s)n}m_t\Big(\bu_{a\in S}\bu_{\tau\in\cD_n(a)}\phi_\tau(X_a)\Big) \\
&\le   C_\e e^{(\P(t)-s)n}\exp\lt(\lt(\frac{\P_{B_2(S,\g)}^\Xi(t)-\P(t)+\e}l\rt)n\rt) \\
&=  C_\e\exp\Big(\(\P(t)-s+l^{-1}(\P_{B_2(S,\g)}^\Xi(t)-\P(t)+\e)\)n\Big) \\
&\le C_\e \exp\Big(\(\P(t)-\P(b)+\d+l^{-1}(\P_{B_2(S,\g)}^\Xi(t)-\P(t)+\e)\)n\Big)
\end{aligned}
$$
for every integer $n\ge 1$.
Since the function $[0,+\infty)\ni t\longmapsto \P_{B_2(S,\g)}^\Xi(t)\in\R$ is continuous and since, by Proposition~\ref{l1h35} \eqref{l1h35 item b}, the function $[0,+\infty)\ni t\longmapsto \P(t)$ is also continuous, taking $\d>0$ sufficiently small we will have both
$$
|\P(t)-\P(b)|<\e/2 \  \and \  \lt|\P_{B_2(S,\g)}^\Xi(t)-\P_{B_2(S,\g)}^\Xi(b)\rt|<\e/2.
$$
Therefore, 
$$
\sum_{\tau\in \cD_\cU^n}\|\phi_\tau'\|_\infty^te^{-s\|\tau\|}
\leq \exp\Big(\(\e+\d+l^{-1}(\P_{B_2(S,\g)}^\Xi(t)-\P(t)+2\e)\)n\Big).
$$
Since, by Lemma~\ref{l1sp1}, $\P_{B_2(S,\g)}^\Xi(b)-\P(b)<0$, taking both $\d>0$ and $\e>0$ small enough, we will have $\e+\d+l^{-1}(\P_{B_2(S,\g)}^\Xi(t)-\P(t)+2\e)<-\b$ for some $\b>0$. Hence,
$$
\sum_{\tau\in \cD_\cU^n}\|\phi_\tau'\|_\infty^te^{-s\|\tau\|}
\leq e^{-\b n}.
$$
Thus 
$$
\sum_{n\ge 1}\sum_{\tau\in \cD_\cU^n}\|\phi_\tau'\|_\infty^te^{-s\|\tau\|}
\lek \sum_{n\ge 1}e^{-\b n}
<+\infty.
$$
The proof is complete.
\epf

As an immediate consequence of this lemma, formula \eqref{6_2016_06_21}, and Proposition~\ref{p1tf5}, we get the following.

\bcor\label{c1_2017_01_26}
Assume that $G$ is a \FNR \ rational semigroup generated by a $u$--tuple map $f=(f_1,\ld,f_u)\in\Rat^u$. If $\cU=\{U_s\}_{s\in S}$ is a nice family of sets, and $\cS_\cU=\big\{\^f_\tau^{-\|\tau\|}:X_{t(\tau)}\lra X_{i(\tau)} \big\}_{\tau\in\cD_\cU}$ is the corresponding graph directed system, then there exists $\d>0$ such that 
$$
(h_f-\d,h_f)\sbt\Om(\cU),
$$ 
and for all $t\in(h_f-\d,h_f)$, 
$$
0<\P(t,0)<+\infty.
$$
By analogy to the terminology of \cite{mugdms}, this would mean in its language that the system $\cS_\cU$ is strongly regular.
\ecor

\sp One of our main results, the last one in this section, is the following.

\bthm\label{t3.11}
If $G$ be a \FNR \ rational semigroup type generated by a $u$--tuple map $(f_1,\ld,f_u)\in\Rat^u$, then the topological pressure function 
$$
\P:\De_G^*\longrightarrow\R
$$ 
is real--analytic.
\ethm

\bpf 
Using Lemma~\ref{l2tf5}, the definition of $\De_G^*$, and applying Theorem~2.6.12 of \cite{mugdms}, we see that for each $\g\in\De_G^*$
there exists $\delta >0$ such that the function $\P$ is real--analytic
on $(\g-\delta,\g+\delta) \times (\P(\g)-\delta,\P(\g)+\delta)$ in both
variables $t$ and $s$.
In order to prove that $\P$ is real--analytic on $(\g-\delta,\g+\delta)$,
we thus may, and we will, employ the Implicit
Function Theorem to show that $\P$ is the unique real--analytic
function which satisfies $\P(t,\P(t))=0$ for all $t \in
(\g-\delta,\g+\delta)$. Since the equality $\P(t,\P(t))=0$ holds because of Proposition~\ref{p1tf5}, it is thus sufficient to prove that for
all $t \in (\g-\delta,\g+\delta)$ we have that
\begin{equation}\lab{1}
{\bd \P(t,s)\over\bd s} \Big|_{(t,\P(t))}<0.
\end{equation}
But by Proposition~2.6.13 of \cite{mugdms} we have 
\beq\lab{2}
{\bd\P(t,s)\over\bd s} \Big|_{(t,\P(t))} 
=-\int_{\cD_\cU^\infty}\|\om_1\|\,d\^\mu_t(\om)
<0,
\eeq
which completes the proof of Theorem~\ref{t3.11}.
\epf

\section{Invariant Measures: $\mu_t$ versus $\^\mu_t\circ\pi_\cU^{-1}$; Finiteness of $\mu_t$}\label{sec:invariantmeasures}

In this section we link the measures $\^m_t$ and $\^\mu_t$ of the previous section, living on the symbol space, with the conformal and invariant measures $m_t$ and $\mu_t$ living on the Julia set $J(\^f)$. This link is given by Lemma~\ref{p1sl4}. We translate here many results of the previous sections, expressed in the symbolic language, to the one of the actual map $\tf$. We eventually prove here, see Theorem~\ref{t1_2016_06_16}, that all of the measures $\mu_t$, $t\in\De_G^*$, are finite, and thus probability measures after normalization. 

\sp Throughout this section $G$ is again a \FNR \ rational semigroup generated by a $u$--tuple map $(f_1,\ld,f_u)\in\Rat^u$. As in previous sections let
$$
\cU=\{U_s\}_{s\in S}
$$
be a nice family of sets coming from Theorem~\ref{t1nsii7} and let $\cS_\cU$ be the corresponding graph directed system described in Theorem~\ref{t1nsii15}. Given $t\in\De_G^*$ let 
\beq\label{1sl1}
\hat\mu_t:=\^\mu_t\circ\pi_\cU^{-1}.
\eeq\index{$\hat\mu_t$}
We start with the following.

\blem\label{l1sl1}
Let $G$ be a \FNR \ rational semigroup generated by a $u$--tuple map $(f_1,\ld,f_u)\in\Rat^u$. Let $\cU=\{U_s\}_{s\in S}$ be a nice family of sets produced in Theorem~\ref{t1nsii7} and let
$$
\cS_\cU=\big\{\^f_\tau^{-\|\tau\|}:X_{t(\tau)}\lra X_{i(\tau)} \big\}_{\tau\in\cD_\cU},
$$
the corresponding graph directed system.

If $t\in\De_G^*$, then
$$
\hat\mu_t\Big(\bi_{n=1}^\infty\bu_{\tau\in \cD_\cU^n} \phi_\tau\(U_{t(\tau)}\)\Big)=1.
$$
\elem

\bpf
Since $U\cap J(\^f)$ is a non--empty open set relative to $J(\^f)$ and since $\Trans(\tf)$ is dense in $J(\^f)$, we have that $U\cap\Trans(\tf)\ne\es$. It then follows from Theorem~\ref{t1nsii15} \eqref{t1nsii15 item b} that there exists at least one point $\xi\in J_\cU\cap U$. Then $\xi=\pi_\cU(\tau)$ for some $\tau\in \cD_\cU^\infty$. Since $U$ is open, there thus exists an integer $n\ge 1$ such that $\pi_\cU([\tau|_n])\sbt U$. Equivalently, $[\tau|_n]\sbt\pi_\cU^{-1}(U)$. Hence,
\beq\label{1_2016_06_13}
\^\mu_t\circ\pi_\cU^{-1}(U)\ge \^\mu_t([\tau|_n])>0.
\eeq
But by item \eqref{d2nsii6 item c} of Definition~\ref{d2nsii6} and by the definition of nice sets, we have that
if $\om\in\pi_\cU^{-1}(U)$, $e\in\cD_\cU$ and $A_{e\om_1}(\cU)=1$, then
$$
\pi_\cU(e\om)
=\phi_e(\pi_\cU(\om))
\in \phi_e(U_{t(e)})
\sbt U_{i(e)}
\sbt U.
$$
Therefore, $e\om\in \pi_\cU^{-1}(U)$. Hence, 
$$
\sg^{-1}\(\pi_\cU^{-1}(U)\)
\sbt \pi_\cU^{-1}(U).
$$
Since the measure $\hat\mu_t$ is ergodic with respect to the shift map $\sg$, it follows from this and from \eqref{1_2016_06_13} that
\beq\label{2_2016_06_13}
\^\mu_t\(\pi_\cU^{-1}(U)\)=1.
\eeq
Now fix an arbitrary integer $n\ge 1$ and fix an $\om\in\sg^{-n}\(\pi_\cU^{-1}(U)\)$. This means that 
\beq\label{4_2016_06_13}
\pi_\cU(\sg^n(\om))\in U. 
\eeq
Then
\beq\label{3_2016_06_13}
\pi_\cU(\om)
=\phi_{\om|_n}\(\pi_\cU(\sg^n(\om))\)
\in \phi_{\om|_n}(X_{t(\om)})
=\phi_{\om|_n}(\ov U_{t(\om)}).
\eeq
So, since the sets $\ov U_s$, $s\in S$ are mutually disjoint, we conclude from \eqref{4_2016_06_13} and \eqref{3_2016_06_13} that 
$$
\pi_\cU(\sg^n(\om)) \in U_{t(\om)}.
$$ 
Combining this with the first part of \eqref{3_2016_06_13} we conclude that $\pi_\cU(\om)\in\phi_{\om|_n}( U_{t(\om)})$. Thus 
$$
\pi_\cU(\om)\in\bu_{\tau\in \cD_\cU^n} \phi_\tau\(U_{t(\tau)}\).
$$
Equivalently,
$$
\om\in \pi_\cU^{-1}\Big(\bu_{\tau\in \cD_\cU^n} \phi_\tau\(U_{t(\tau)}\)\Big).
$$
Thus we have proved that
\beq\label{5_2016_06_13}
\sg^{-n}\(\pi_\cU^{-1}(U)\) 
\sbt\pi_\cU^{-1}\Big(\bu_{\tau\in \cD_\cU^n} \phi_\tau\(U_{t(\tau)}\)\Big).
\eeq
Combining this along with \eqref{2_2016_06_13} and shift invariance of the measure $\mu_t$, we conclude that
$$
\hat\mu_t\Big(\bu_{\tau\in\cD_\cU^n}\phi_\tau\(U_{t(\tau)}\)\Big)
\ge \^\mu_t\(\sg^{-n}\(\pi_\cU^{-1}(U)\)\)
=   \^\mu_t\(\pi_\cU^{-1}(U)\)
=1.
$$
Thus,
$$
\hat\mu_t\Big(\bi_{n=1}^\infty\bu_{\tau\in \cD_\cU^n} \phi_\tau\(U_{t(\tau)}\)\Big)=1.
$$
and the proof is complete.
\epf

\sp Keep the setting of Lemma~\ref{l1sl1}. In particular $t\in\De_G^*$. 

\sp Let $J_\cU(\infty)$\index{$J_\cU(\infty)$} be the set of all those points $\xi\in J_\cU$ for which there exists infinitely many integers $n\ge 1$ such that $\tf^n(\xi)\in J_\cU$. 

\sp Let also $U(\infty)$\index{$U(\infty)$} be the set of all those points $\xi\in U$ for which there exists infinitely many integers $n\ge 1$ such that $\tf^n(\xi)\in U$. 

\sp\fr The Poincar\'e Recurrence Theorem asserts that 
$$
\mu_t(J_\cU(\infty))=\mu_t(J_\cU).
$$
For all points $\xi\in J_\cU(\infty)$ there is a well defined \textbf{first return time} to $J_\cU(\infty)$, equal also to the first entrance time to $J_\cU$, and defined as the least $n\ge1$ such that $\tf^n(\xi)\in J_\cU(\infty)$. Denote this $n$ by $N_{J_\cU}(\xi)$.\index{$N_{J_\cU}(\xi)$} The \textbf{first return map}
$\tf_{J_\cU}:J_\cU(\infty)\lra J_\cU(\infty)$\index{first return map!$\tf_{J_\cU}:J_\cU(\infty)\lra J_\cU(\infty)$}\index{$\tf_{J_\cU}:J_\cU(\infty)\lra J_\cU(\infty)$} is then defined as
$$
\tf_{J_\cU}(\xi)=\tf^{N_{J_\cU}(\xi)}(\xi).
$$
Now we shall prove the following. Of course in the above considerations $J_\cU$ could have been replaced by any measurable subset of $\Sg_u\times\C$.

\blem\label{l1sl5}
If $G$ is a \FNR \ rational semigroup generated by a $u$--tuple map $(f_1,\ld,f_u)\in\Rat^u$, then
$$
J_\cU^{\circ}:=\bi_{n=1}^\infty\bu_{\tau\in \cD_\cU^n} \phi_\tau\(U_{t(\tau)}\)
=U(\infty)
\sbt J_\cU\sbt \ov U
$$\index{$J_\cU^{\circ}$}
and
$$
\mu_t\(J_\cU^{\circ}\)
=\mu_t(U)
=\mu_t(\ov U)
=\mu_t(J_\cU)
>0.
$$
\elem

\bpf
Of course
$$
J_\cU^{\circ}\sbt U(\infty)\sbt J_\cU\sbt \ov U.
$$
We will now prove the inclusion opposite to the first one, i.e. $U(\infty)\sub J_\cU^{\circ}$ . For this end it suffices to show that if $\xi\in U(\infty)$ and $\tf^n(\xi)\in \ov U$ for some $n\ge 0$, then $\tf^n(\xi)\in U$. So assume that 
$$
\xi\in U(\infty)
\  \  \  {\rm and} \  \  
\tf^n(\xi)\in \ov U.
$$
Then, there exists an integer $q>n$ such that $\tf^q(\xi)\in U$. Hence, there exists $s\in S$ such that $\tf^q(\xi)\in U_s$. Let $\tf_\xi^{-(q-n)}$ be the unique continuous inverse branch of $\tf^{q-n}$ defined on $U_s$ and sending $\tf^q(\xi)$ back to $\tf^n(\xi)$. Then $\tf_\xi^{-(q-n)}=\phi_\om$ with some $\om\in\cD_\cU^*$. Since $\cU$ is a nice family, this implies that $\tf^n(\xi)\in \phi_\om(U_s)\sbt U_{i(\om)}\sbt U$. Thus
$$
\xi\in\bi_{n=1}^\infty\bu_{\tau\in \cD_\cU^n} \phi_\tau\(U_{t(\tau)}\).
$$
So, the first formula of our lemma is established. 

\sp Passing to the second formula of our lemma, note that its first two equality signs are now immediate. Note also that 
$$
0<\mu_t(U)
=\mu_t(U(\infty))
\le\mu_t(J_\cU)
\le \mu_t(\ov U).
$$
Thus, in order to conclude the proof, it suffices to show that
$$
\mu_t(\ov U\sms U)=0.
$$
But, since $U$ is the union of all members of some nice family, $U\cap\bu_{n=0}^\infty \tf^n(\ov U\sms U)=\es$. Hence, $(\ov U\sms U)\cap\Trans(\tf)=\es$, and therefore, $\mu_t(\ov U\sms U)=0$ by virtue of Corollary~\ref{c12015_01_30} and Theorem~\ref{t4h65}. The proof is complete. 
\epf

\sp\fr As an immediate consequence of this lemma, we get the following.

\sp\bcor\label{c120190815}
If $G$ is a \FNR \ rational semigroup generated by a $u$--tuple map $(f_1,\ld,f_u)\in\Rat^u$, then for every $\xi\in J_\cU^{\circ}$ the set
$$
\big\{n\ge 0: \tf^n(\xi)\in J_\cU^{\circ}\big\}
$$
is infinite.
\ecor 

\sp\fr Since, by virtue of Theorem~\ref{t4h65}, the measures $m_t$ and $\mu_t$ are equivalent, as an immediate consequence of Lemma~\ref{l1sl5} we get the following.

\sp\bcor\label{c2sl5}
If $G$ is a \FNR \ rational semigroup generated by a $u$--tuple map $(f_1,\ld,f_u)\in\Rat^u$, then
$$
m_t\(J_\cU^{\circ}\)
=m_t(U)
=m_t(\ov U)
=m_t(J_\cU)
>0.
$$
\ecor

\sp\fr Let
$$
\cD_\cU^\circ:=\pi_\cU^{-1}\(J_\cU^{\circ}\).
$$\index{$\cD_\cU^{\circ}$}
As an immediate consequence of Lemma~\ref{l1sl1} and Lemma~\ref{l1sl5} we get the following.

\bcor\label{c1sl3}
If $G$ is a \FNR \ rational semigroup generated by a $u$--tuple map $(f_1,\ld,f_u)\in\Rat^u$, then

\sp\begin{itemize}
\item[\mylabel{a}{c1sl3 item a}] $\pi_\cU^{-1}(\xi)$ is a singleton in $\cD_\cU^\circ$ for every $\xi\in J_\cU^{\circ}$,

\sp\item[\mylabel{b}{c1sl3 item b}] The map $\pi_\cU\big|_{\cD_\cU^\circ}:\cD_\cU^\circ\lra J_\cU^{\circ}$ is bijective,

\sp\item[\mylabel{c}{c1sl3 item c}] $\^\mu_t\(\cD_\cU^\circ\)=\hat\mu_t\(J_\cU^{\circ}\)=1$.

%\sp\item[(d)] $\mu_t\(J_\cU^{\circ}\)>0$.
\end{itemize}
\ecor

%\sp\fr Let $\^F_\cU:J_\cU^{\circ}\to J_\cU^{\circ}$ be the first return map from $J_\cU^{\circ}$ to $J_\cU^{\circ}$. 

\sp\fr The following is an immediate consequence of Corollary~\ref{c1sl3}.

\sp
\bcor\label{c2sl3}
If $G$ is a \FNR \ rational semigroup generated by a $u$--tuple map $(f_1,\ld,f_u)\in\Rat^u$, then

\sp\begin{itemize}
\item[\mylabel{a}{c2sl3 item a}] For every $\tau\in\cD_\cU^*$ we have that
$$
\^f_{J_\cU}\circ\phi_\tau\big|_{J_\cU^{\circ}\cap U_{t(\tau)}}=\Id\big|_{J_\cU^{\circ}\cap U_{t(\tau)}},
$$
\item[\mylabel{b}{c2sl3 item b}] The following diagram commutes 
$$
\begin{array}{rcl}
  \cD_\cU^\circ     &\;\;\begin{array}{c}  \sg  \\ [-0.2cm]
   \ \xrightarrow{\hspace*{2cm}} \\ [0.2cm]
  \end{array}\;\;
   &  \cD_\cU^\circ \\
   \pi_\cU \Big\downarrow  \;\;&  & \Big\downarrow \pi_\cU \\[0.2cm]
  J_\cU^{\circ}  &\;\;\begin{array}{c} \^f_{J_\cU} \\ [-0.2cm]
   \ \xrightarrow{\hspace*{2cm}} \\ [0.2cm]
  \end{array}\;\;  & J_\cU^{\circ}.
\end{array} 
$$
\end{itemize}
\ecor

\sp\fr Now we are in position to prove the following.

\sp
\blem\label{p1sl4}
If $G$ is a \FNR \ rational semigroup generated by a $u$--tuple map $(f_1,\ld,f_u)\in\Rat^u$, then
$$
\hat\mu_t=\frac{\mu_t}{\mu_t(J_\cU)}\bigg|_{J_\cU}.
$$
\elem

\bpf
Put 
$$
m_t^\circ:=\frac{m_t}{m_t(J_\cU)}\bigg|_{J_\cU}, \  \  \
\mu_t^\circ:=\frac{\mu_t}{\mu_t(J_\cU)}\bigg|_{J_\cU}
$$
and also
$$
\^m_t^\circ:=m_t^\circ\circ\pi_\cU\big|_{\cD_\cU^\circ}, \  \  \  
\^\mu_t^\circ:=\mu_t^\circ\circ\pi_\cU\big|_{\cD_\cU^\circ},
$$
where the latter two are  well defined since, by virtue of Corollary~\ref{c1sl3} \eqref{c1sl3 item b}, the map $\pi_\cU\big|_{\cD_\cU^\circ}$ is 1--to--1. It then follows from formula \eqref{1tf5} of Proposition~\ref{p1tf5} and from Corollary~\ref{c2sl5} that 
$$
\^m_t([\tau])
\comp m_t\(\phi_\tau(U_{t(\tau)})\)
%\comp m_t^\circ\(\phi_\tau(U_{t(\tau)})\)
=m_t(J_\cU)\,m_t^\circ\(J_\cU^{\circ}\cap \phi_\tau(U_{t(\tau)})\)
=m_t(J_\cU)\,\^m_t^\circ([\tau]),
$$
for every $\tau\in\cD_\cU^*$. Therefore, $\^m_t\comp \^m_t^\circ$. Hence
$$
\^\mu_t\comp \^\mu_t^\circ.
$$
Hence
\beq\label{1sl4}
\hat\mu_t=\^\mu_t\circ\pi_\cU^{-1}\comp\^\mu_t^\circ\circ\pi_\cU^{-1}\le \mu_t^\circ.
\eeq
Now, on the one hand, $\mu_t^\circ$ is $\tf_\cU$--invariant because of its definition and since $\mu_t$ is $\tf$--invariant. On the other hand, $\hat\mu_t=\^\mu_t\circ\pi_\cU^{-1}$ is $\tf_\cU$--invariant because of shift invariance of $\^\mu_t$ and because of item \eqref{c2sl3 item b} of Corollary~\ref{c2sl3}. Since in addition $\mu_t^\circ$ is ergodic with respect to $\tf_\cU$ (as $\mu_t$ is ergodic with respect to $\tf$), it follows from \eqref{1sl4} that $\hat\mu_t=\mu_t^\circ$. The proof is complete.
\epf

\sp\fr As an immediate consequence of this lemma, formula \eqref{1tf5} of Proposition~\ref{p1tf5}, item \eqref{t1tf3 item e} of Theorem~\ref{t1tf3}, and Proposition~\ref{p1sp4}, we get the following.

\bprop\label{p1sp4B}
Let $G$ be a \FNR \ rational semigroup generated by a $u$--tuple map $(f_1,\ld,f_u)\in\Rat^u$. If \ $\cU=\{U_s\}_{s\in S}$ is a nice family of sets produced in Theorem~\ref{t1nsii7} and $\cS_\cU=\big\{\^f_\tau^{-\|\tau\|}:X_{t(\tau)}\lra X_{i(\tau)} \big\}_{\tau\in\cD_\cU}$ is the corresponding graph directed system, then there exist $\g>0$ and a finite set $\Xi\sbt K(B_2(S,\g))$ such that for every $b\in\De_G^*$ there exists $\eta>0$ such that for every $t\in\De_G^*\cap(b-\eta,b+\eta)$ and every $\e>0$
$$
\mu_t\Big(\bu_{s\in S}\bu_{\tau\in \cD_n(s)}\^f_\tau^{-n}(X_s)\Big)
\le C_\e'\exp\lt(\lt(\frac{\P_{B_2(S,\g)}^\Xi(t)-\P(t)+\e}l\rt)n\rt)
$$
and
$$
\^m_t\Big(\bu_{\tau\in\cD_n}[\tau]\Big)\le C_\e'\exp\lt(\lt(\frac{\P_{B_2(S,\g)}^\Xi(t)-\P(t)+\e}l\rt)n\rt)
$$
for every integer $n\ge 1$ and some constant $C_\e'\in(0,+\infty)$ depending on $\e$. Also,  $$
\P(t)-\P_{B_2(S,\g)}^\Xi(t)>0.
$$
\eprop

\sp\fr In terms of return time, an equivalent reformulation of this proposition with regard to measure $\mu_t$ is the following.

\sp
\bprop\label{p1sp4C}
Let $G$ be a \FNR \ rational semigroup generated by a $u$--tuple map $(f_1,\ld,f_u)\in\Rat^u$. If \ $\cU=\{U_s\}_{s\in S}$ is a nice family of sets produced in Theorem~\ref{t1nsii7} and $\cS_\cU=\big\{\^f_\tau^{-\|\tau\|}:X_{t(\tau)}\lra X_{i(\tau)} \big\}_{\tau\in\cD_\cU}$ is the corresponding graph directed system, then there exist $\g>0$  and a finite set $\Xi\sbt K(B_2(S,\g))$ such that for every $b\in\De_G^*$ there exists $\eta>0$ such that for every $t\in\De_G^*\cap(b-\eta,b+\eta)$ and every $\e>0$
$$
\mu_t\Big(\big\{\xi\in J_\cU:N_{J_\cU}(\xi)=n \big\}\Big)
\le C_\e'\exp\lt(\lt(\frac{\P_{B_2(S,\g)}^\Xi(t)-\P(t)+\e}l\rt)n\rt)
$$
for every integer $n\ge 1$ and some constant $C_\e'\in(0,+\infty)$ depending on $\e$. Also,  $$
\P(t)-\P_{B_2(S,\g)}^\Xi(t)>0.
$$
\eprop

\sp\fr An immediate consequence of Proposition~\ref{p1sp4B} and Proposition~\ref{p1sp4C} along with item \eqref{t1tf3 item e} of Theorem~\ref{t1tf3}, is the following.

\sp
\bprop\label{p1sp4D}
Let $G$ be a \FNR \ rational semigroup generated by a $u$--tuple map $(f_1,\ld,f_u)\in\Rat^u$. If \ $\cU=\{U_s\}_{s\in S}$ is a nice family of sets produced in Theorem~\ref{t1nsii7}, and $\cS_\cU=\big\{\^f_\tau^{-\|\tau\|}:X_{t(\tau)}\lra X_{i(\tau)} \big\}_{\tau\in\cD_\cU}$ is the corresponding graph directed system, then there exist $\g>0$ and a finite set $\Xi\sbt K(B_2(S,\g))$ such that for every $b\in\De_G^*$ there exists $\eta>0$ such that for every $t\in\De_G^*\cap(b-\eta,b+\eta)$ and every $\e\in(0,\P(t)-\P_{B_2(S,\g)}^\Xi(t))$
$$ 
\begin{aligned}
\mu_t\Big(\big\{\xi\in J_\cU:N_{J_\cU}(\xi)\ge n \big\}\Big)
&\le C_\e''\exp\lt(\lt(\frac{\P_{B_2(S,\g)}^\Xi(t)-\P(t)+\e}l\rt)n\rt), \\
\^m_t\Big(\bu_{k\ge n}\bu_{\tau\in\cD_k}[\tau]\Big)
&\le C_\e''\exp\lt(\lt(\frac{\P_{B_2(S,\g)}^\Xi(t)-\P(t)+\e}l\rt)n\rt), \  \and \\
\^\mu_t\Big(\bu_{k\ge n}\bu_{\tau\in\cD_k}[\tau]\Big)
&\le C_\e''\exp\lt(\lt(\frac{\P_{B_2(S,\g)}^\Xi(t)-\P(t)+\e}l\rt)n\rt)
\end{aligned}
$$
for every integer $n\ge 1$ and some constant $C_\e''\in(0,+\infty)$ depending on $\e$. Also,  $$
\P(t)-\P_{B_2(S,\g)}^\Xi(t)>0.
$$
\eprop

\sp\fr An immediate consequence of this proposition is the following. 

\bcor\label{c1sp4E}
Let $G$ be a \FNR \ rational semigroup generated by a $u$--tuple map $(f_1,\ld,f_u)\in\Rat^u$. If \ $\cU=\{U_s\}_{s\in S}$ is a nice family of sets and $t\in\De_G^*$, then for every $p>0$ we have 
$$
\int_{J_\cU}N_{J_\cU}^p(\xi)\,d\mu_t(\xi)<+\infty.
$$
\ecor

Now we can also prove the following.

\blem\label{l120190819}
Let $G$ be a \FNR \ rational semigroup generated by a $u$--tuple map $(f_1,\ld,f_u)\in\Rat^u$. If \ $\cU=\{U_s\}_{s\in S}$ is a nice family of sets and $t\in\De_G^*$, then the function $\log\zeta_{1,0}:\cD_\cU^\infty\lra(-\infty,0)$ is integrable with respect to the measure $\^\mu_t$. Furthermore, all its positive moments are finite. More precisely: 
$$
0<\chi_{\^\mu_t}:=\int_{\cD_\cU^\infty}\log\big|(\tf^{\|\tau_1\|}\)'(\pi_\cU(\sg(\tau)))\big|\,d\^\mu_t(\tau)<+\infty
$$
and (note that the integrand of the above integral is, by Theorem~\ref{t1nsii7}, everywhere positive)
$$
\int_{\cD_\cU^\infty}\log^p\big|(\tf^{\|\tau_1\|}\)'(\pi_\cU(\sg(\tau)))\big|\,d\^\mu_t(\tau)<+\infty
$$
for every $p>0$.
\elem

\begin{proof}
Taking $\e>0$ small enough and applying Proposition~\ref{p1sp4D}, we get that
$$
\begin{aligned}
\int_{\cD_\cU^\infty}\log^p\big|(\tf^{\|\tau_1\|}\)'(\pi_\cU(\sg(\tau)))\big|\,d\^\mu_t(\tau)
&\le \sum_{e\in :\cD_\cU}\sup\lt(\log\big|(\tf^{\|e\|})'\big|\Big|_{X_{t(e)}}\rt)^p\^\mu_t([e])\\
&=\sum_{n=1}^\infty\sum_{e\in :\cD_n}\sup\lt(\log\big|(\tf^{\|e\|})'\big|\Big|_{X_{t(e)}}\rt)^p\^\mu_t([e]) \\ 
&\le \sum_{n=1}^\infty\sum_{e\in :\cD_n}\sup\Big(\log\(\|\tf'\|_\infty^n\)\Big)^p\^\mu_t([e]) \\
&\le \log\|\tf'\|_\infty^p\sum_{n=1}^\infty n^p\sum_{e\in :\cD_n}\^\mu_t([e])\\
&\le \log\|\tf'\|_\infty^p C_\e''\sum_{n=1}^\infty n^p\exp\lt(\lt(\frac{\P_{B_2(S,\g)}^\Xi(t)-\P(t)+\e}l\rt)n\rt)\\
&<+\infty.
\end{aligned}
$$
The inequality $\chi_{\^\mu_t}>0$ follows immediately from item \eqref{t1nsii7 item B} of Theorem~\ref{t1nsii7}.
\end{proof}

\sp\fr In turn, as an immediate consequence of the last two results, i.e. Corollary~\ref{c1sp4E} and Lemma~\ref{l120190819} along with Kac's Lemma, we get the following.

\sp
\bthm\label{t1_2016_06_16}  
Let $G$ be a \FNR \ rational semigroup generated by a $u$--tuple map $(f_1,\ld,f_u)\in\Rat^u$. If $t\in\De_G^*$, then the measure $\mu_t$ is finite. From now on we normalize this measure so that it becomes a probability measure. Furthermore, the function $J(\^f)\ni z\longmapsto \log|\tilde f'(z)|\in\R$ is integrable with respect to the measure $\mu_t$ and
$$
\chi_{\mu_t}:=\int_{J(\tilde f)}\log|\tilde f'|\,d\mu_t\in (0,+\infty). 
$$\index{Lyapunov exponents! $\chi_{\mu_t}$}\index{$\chi_{\mu_t}$}
This number, i.e. $\chi_{\mu_t}$, is commonly called the \textbf{Lyapunov exponent} of the dynamical system $(\tilde f,\mu_t)$.
\ethm

\fr Now, as a complement of Theorem~\ref{t3.11}, we can prove the following.

\bthm\label{t120180604}
Let $G$ be a \FNR \ rational semigroup generated by a $u$--tuple map $(f_1,\ld,f_u)\in\Rat^u$. If $t\in\De_G^*$, then
\beq\label{1_2018_01_12}
\P'(t)=-\chi_{\mu_t}<0,
\eeq
and 
\beq\label{2_2018_01_12}
\P''(t)=\sg_{\mu_t}^2\(-t\log|\tilde f'|\)\ge 0, 
\eeq
where 
$$
\sg_{\mu_t}^2\(-t\log|\tilde f'|\)
:=\lim_{n\to\infty}\frac1n\int_{J(\tilde f)}S_n^2\(-t\log|\tilde f'|+\chi_{\mu_t}\)\,d\mu_t,
$$
is commonly called the \textbf{asymptotic variance} of the function $-t\log|\tilde f'|$ with respect to the dynamical system $(\tilde f,\mu_t)$.
\ethm

\begin{proof}
Because of Lemma~\ref{p1sl4}, formula \eqref{1sl1}, the definition of the measure $\tilde\mu_t$, and Propositions~2.6.13 from \cite{mugdms} along with Kac's Lemma, by differentiating formula \eqref{520180604} of Proposition~\ref{p1tf5}, we get that
$$
\begin{aligned}
0
&={\bd\P(t,s)\over\bd t}\Big|_{(t,\P(t))}+{\bd\P(t,s)\over\bd s} \Big|_{(t,\P(t))}\P'(t)
=-\chi_{\tilde\mu_t}-\P'(t)\int_{\cD_\cU^\infty}\|\om_1\|d\^\mu_t(\om)\\
&=-\chi_{\hat\mu_t}-\P'(t)\int_{J_\cU}N_{J_\cU}\,d\hat\mu_t\\
&=-\chi_{\hat\mu_t}-\frac{\P'(t)}{\mu_t\(J_\cU\)}.
\end{aligned}
$$
Therefore, using the refined version of Kac's Lemma, we get
$$
\P'(t)=-\mu_t\(J_\cU\)\chi_{\tilde\mu_t}=-\chi_{\mu_t}<0,
$$
where the last inequality follows from Theorem~\ref{t1_2016_06_16}. 
Passing to the second derivative of $\P(t)$, by employing Propositions~2.6.14 from \cite{mugdms}, we calculate that
$$
\P''(t)=\sg_{\mu_t}^2(-t\log|\tilde f'|)\ge 0,
$$
which finishes the proof. 
\end{proof}  

\section{Variational Principle: \\ The Invariant Measures $\mu_t$ are  the Unique Equilibrium States}\label{section:VP}
Throughout this section we always assume that $G=\langle f_1, \dots, f_u \rangle$ is a \FNR \ rational semigroup. Our main goal in this section is to prove a variational principle for the potentials $-t\log|\^f'|$, $t\in\De_G^*$, and the dynamical system $\^f:J(\^f)\lra J(\^f)$, and to show that the measures $\mu_t$ are the only \textbf{equilibrium states}\index{equilibrium state} for these potentials. We start with the following technical auxiliary result. 
\blem\label{l1vp2}
Let $G=\langle f_1,\dots, f_u\rangle$ be a \FNR \ rational semigroup. If $0<r<s$, then there exists a finite set $\Xi\sub K\(B(\Crit_*(\tf),r)\)$ such that 
	$$
		\P\lt(\^f\rvert_{K\(B(\Crit_*(\tf),s)\)},-t\log|\^f'|\rt)\leq \P_{B\(\Crit_*(\^f),r\)}^\Xi(t)<\P(t)
	$$
	for every $t\in\De_G$. 
\elem
\begin{proof}
	For the ease of notation, for every $u>0$, put 
	\begin{align}
		B_u:=B(\Crit_*(\tf),u), 
		& \quad K_u:=K(B_u), \label{1vp2}\\
		\P_u(t):=P\lt(\^f\rvert_{K_u},-t\log|\^f'| \rt), &\text{ and } 
		\P_u^\Xi(t):=\P_{B_u}^\Xi(t). \label{2vp2}
	\end{align}
	Let $R_2>0$ be the constant produced in Lemma~\ref{l2sp1} for the set $V=B_r$. Let $\Xi\sub K_r$ be a finite $(R_2/4)$--spanning set for the set $K_r$. Let $n\geq 1$ be the integer produced in Lemma~\ref{l2sp1} for $V=B_r$. In view of this lemma and Corollary~\ref{l2sp1B}, we may assume $R_2>0$ to be so small that 
	\begin{align}\label{1vp2.1}
		\|\^f_\xi^{-nj}(x),\^f_\xi^{-nj}(y) \|_\vth\leq 2^{-j}\|x,y\|_\th
	\end{align}
	for all $\xi\in K_r$, all $j\geq 0$, and all $x,y\in B(\^f^j(\xi),2R_2)$, where $\^f_\xi^{-k}:B(\^f^{k}(\xi),2R_2)\lra\Sg_{u}\times\C$, $k\ge 0$, are the inverse branches of $\^f^{k}$ produced in Lemma~\ref{l2sp1}. Let $\ep\in(0,R_2/4)$ and fix an integer $q\geq 1$ so large that 
	\begin{align}\label{2vp2.1}
		2^{-q}R_2<\ep.
	\end{align}
	Fix an integer $k\geq 1$. Let $F_\ep\sub K_s$ be a maximal $(k,\ep)$--separated set for the dynamical system $\^f^n: K_s\to K_s$. Then for every $x\in F_\ep$ there exists an element $\hat x\in \Xi$ such that 
	$$
		\^f^{n(q+k)}(x)\in B(\hat x, R_2/4).
	$$    
	Let 
	$$
		\^x:=\^f_x^{-n(q+k)}(\hat x)
	$$
	where $\^f_x^{-n}:B\(\^f^{n(q+k)}(x),2R_2\)\lra\Sg_{u}\times\C$ is the inverse branch of $\^f^{kn}$ produced in Lemma~\ref{l2sp1} $\xi=\^f^{n(q+k)}(x)$. Then 
	$$
		\lt|\lt(\^f^{nk}\rt)'(\^x) \rt|\leq K^2 \lt|\lt(\^f^{nk}\rt)' (x)\rt|.
	$$
	Equivalently,
	$$
	\lt|\lt(\^f^{nk}\rt)'(x) \rt|^{-1}\leq K^2 \lt|\lt(\^f^{nk}\rt)' (\^x)\rt|^{-1}.
	$$	
	Therefore, 
	\begin{align}
		\P_\ep(k,t):&=\frac{1}{k}\log\sum_{x\in F_\ep}\lt|\lt(\^f^{nk}\rt)'(x)\rt|^{-t}
		\leq \frac{2\log K}{k}+\frac{1}{k}
		\log\sum_{x\in F_\ep} \lt|\lt(\^f^{nk}\rt)'(\^x)\rt|^{-t}
		\nonumber\\
		&\leq \frac{2\log K}{k}+\frac{qn}{k}\log\|\tf'\|_\infty +n\frac{1}{nk}\log\sum_{x\in F_\ep}\lt|\lt(\tf^{n(q+k)}\rt)'(\^x)\rt|^{-t}. \label{1vp3}
	\end{align}
	We now claim that the function 
	\begin{align}\label{2vp3}
		F_\ep\ni x\longmapsto\^x\in J(\^f)
	\end{align}
	is 1--to--1. Indeed, suppose that $x,y\in F_\ep$ and $\^x=\^y$. Then also $\hat x=\^f^{n(q+k)}(\^x)=\^f^{n(q+k)}(\^y)=\hat y$, and 
	$$
		x=\^f_z^{-n(q+k)}(\^f^{n(q+k)}(x)),  \  \  \ %\text{ and }
		y=\^f_z^{-n(q+k)}(\^fz^{n(q+k)}(y)),
	$$
where 	$z:=\tilde x=\tilde y$. Since also 
	\begin{align*}
		\|\^f^{n(q+k)}(x),\^f^{n(q+k)}(y) \|_{\vth} 
		\leq
		\|\^f^{n(q+k)}(x),\hat x \|_{\vth} + \|\hat y,\^f^{n(q+k)}(y) \|_{\vth}
		\leq 
		\frac{R_2}{4}+\frac{R_2}{4}=\frac{R_2}{2}, 
	\end{align*}
	we conclude from \eqref{1vp2.1} and \eqref{2vp2.1} that 
	\begin{align*}
		\|\^f^{nj}(x),\^f^{nj}(y) \|_\vth 
		&\leq 2^{j-(q+k)}\|\^f^{n(q+k)}(x),\^f^{n(q+k)}(y) \|_\vth 
		\\
		&\leq 2^{-q}\|\^f^{n(q+k)}(x),\^f^{n(q+k)}(y) \|_\vth \\
		&<2^{-q}R_2<\ep,
	\end{align*}
for all $j=0,1,\dots, k$. Since the set $F_\ep$ is $(k,\ep)$--separated with respect to the map $\^f\rvert_{K_r}:K_r\to K_r$, we thus conclude that $x=y$, and injectivity of the map from \eqref{2vp3} is established. Having this and using also Lemma~\ref{l1sp1}, we can continue \eqref{1vp3} as
	\begin{align*}
\P_\ep(t):=\varlimsup_{k\to\infty}\P_\ep(k,t)\leq n\varlimsup_{k\to\infty}\frac{1}{nk}\log\sum_{y\in\^f^{-n(q+k)}(\Xi)}\lt|\lt(\tf^{n(q+k)}\rt)'(y)\rt|^{-t}
		\leq n\P_{B_s}^\Xi(t)<n\P(t). 
	\end{align*}	
	Hence, 
	\begin{align*}
		\P_s(t)=\frac{1}{n}\P\lt(\^f\rvert_{K_s}^n,-t\log|(\tf^n)'|\rt)
		=\frac{1}{n}\varlimsup_{\ep\to 0} \P_\ep(t)
		\leq \P_{B_s}^\Xi(t)<\P(t).
	\end{align*}
\end{proof}

Recall that $M(\tf)$ denotes the set of all $\tf$--invariant Borel probability measures and for $\mu\in M(\tf)$ we let $h_\mu(\tf)$\index{$h_\mu(\tf)$}\index{entropy!$h_\mu(\tf)$} denote the Kolmogorov-Sinai entropy of $\tf$ with respect to the measure $\mu$.  
We shall now prove the following which is the main and only theorem of this section.

\begin{thm}\label{tvp3}\label{TVP}\label{t1vp3}
	If $G=\langle f_1,\dots, f_u\rangle$ is a \FNR \ rational semigroup and $t\in\De_G^*$, then the integrals $\int_{J(\^f)}\log|\^f'|d\mu$, $\mu\in M(\^f)$, are well defined and 
\beq\lab{520200316}
-t\int_{J(\^f)}\log|\^f'|d\mu>-\infty.
\eeq
	\begin{align*}
		&\sup\lt\{\h_\mu(\^f)-t\int_{J(\^f)}\log|\^f'|d\mu
		:\mu\in M(\^f) \rt\}= \\
		&\qquad\qquad\qquad\qquad\qquad
		=\sup\lt\{\h_\mu(\^f)-t\int_{J(\^f)}\log|\^f'|d\mu
		:\mu\in M_e(\^f) \rt\}
		=\P(t),
	\end{align*}
	and 
	$$
		\h_{\mu_t}(\^f)-t\int_{J(\^f)}\log|\^f'|d\mu_t=\P(t),	$$
	while
	$$
		\h_\mu(\^f)-t\int_{J(\^f)}\log|\^f'|d\mu<\P(t)
	$$
	for every measure $\mu\in M(\^f)$ different from $\mu_t$. 
\end{thm}
\begin{proof}
	First, note that since 
	$$
	\log|\tf'(\om,z)|\leq \log\|\tf'\|_\infty
	$$
for all $(\om,z)\in J(\^f)$, all the integrals $\int_{J(\^f)}\log|\^f'|d\mu$, $\mu\in M(\^f)$, are well defined and formula \eqref{520200316} follows.
Hence, all the sums 
$$
\h_\mu(\^f)-t\int_{J(\^f)}\log|\^f'|d\mu, \  \  \mu\in M(\^f),
$$
are well--defined. 
Using the definitions of the measures $\^\mu_t$ and $\hat\mu_t$, Corollary~\ref{c1sl3}, Corollary~\ref{c2sl3}, Lemma~\ref{p1sl4}, Abramov's Formula, and (the refined version of) Kac's Formula, we get 
	\begin{align*}
		\h_{\mu_t}(\^f)-\mu_t(t\log|\^f'|)
		&=\mu_t(J_\cU)\lt(\h_{\hat\mu_t}(\^f_{J_\cU})-\hat\mu_t\(t\log|\^f_{J_\cU}'|\)\rt)\\
		&%\quad
		=\mu_t(J_\cU)\lt(\h_{\^\mu_t}(\sg)+\^\mu_t(\zt_{t,0})\rt)
		\\
		&%\quad 
		=\mu_t(J_\cU)\lt(\h_{\^\mu_t}(\sg)+\^\mu_t(\zt_{t,\P(t)})+\P(t)\int_{\cD_\cU^\infty}\|\tau_1\| d\^\mu(\tau)\rt).
	\end{align*}
	Now using Proposition~\ref{p1tf5}, Theorem~2.2.9 of 
	\cite{mugdms} (for the potential $\zt_{t,\P(t)}$), Corollaries~\ref{c1sl3}, \ref{c2sl3} again, and the ordinary version of Kac's Lemma, along with Lemma~\ref{p1sl4}, we further get 
	\begin{align}
		\h_{\mu_t}(\^f)-\mu_t(t\log|\^f'|)
		&=\mu_t(J_\cU)\lt(\P(t,\P(t))+\P(t)\hat\mu_t(N_{J_\cU})\rt)
		\nonumber\\
		&%&\quad 
		=\mu_t(J_\cU)\P(t)\frac{1}{\mu_t(J_\cU)} \nonumber \\
		&=\P(t).\label{1vp5}
	\end{align}
	Now, let $\mu\in M_e(\^f)$ be arbitrary. Since $f(\PCV(\^f))\sbt \PCV(\^f))$, we have that either
	$$
		\mu(\PCV(\^f))=1 \  \  \text{ or } \  \  \mu(\PCV(\^f))=0.
	$$
	Consider first the case where 
	$$
		\mu(\PCV(\^f))=1.
	$$
	Since $\PCV(\^f)$ is a compact set, since $\^f(\PCV(\^f))\sub\PCV(\^f)$ again, and since $\PCV(\^f)\sub K_s$ (see \eqref{1vp2}) for every $s>0$ small enough, we conclude from Lemma~\ref{l1vp2} and the ordinary Variational Principle that 
	\begin{align}\label{2vp5}
		\h_{\mu_t}(\^f)-\mu_t(t\log|\^f'|)\leq \P\lt(\^f\rvert_{K_s},-t\log|\^f'|\rt)<\P(t).
	\end{align}
	Now, suppose in turn that
	$$
		\mu(\PCV(\^f))=0.
	$$
	Then, taking $R\in(0,R_*(\^f))$ sufficiently small, we will find a finite aperiodic set
	$$
		\Crit_*(\^f)\sub S\sub J(\^f)\bs B(\PCV(\^f),8R)
	$$
	such that 
	$$
		S\cap \supp(\mu)\neq\emptyset.
	$$
	Hence, if $\cU_S$ is a nice set produced in Theorem~\ref{t1nsii7}, then 
	\begin{align}\label{3vp5}
		\mu(U)>0.
	\end{align}
Having this, the same proof as that of Lemma~\ref{l1sl5}, gives that 
	$$
		\mu(J_{\cU_S}^\circ)>0.
	$$
	Let 
	\begin{align}\label{3vp6}
		\hat\mu:=\frac{\mu}{\mu(J_{\cU_S}^\circ)}\Big\rvert_{J_{\cU_S}^\circ}
	\end{align}
	be the corresponding conditional measure on $J_{\cU_S}^\circ$. Because of Corollary~\ref{c1sl3}, there exists a unique probability measure $\^\mu$ on $\cD_\cU^\infty$ such that 
	\begin{align}\label{2vp6}
		\hat\mu=\^\mu\circ\pi_{\cU_S}^{-1}
	\end{align}
	and $\pi_{\cU_S}:\cD_\cU^\infty\to J_{\cU_S}^\circ$ is a measure--theoretic isomorphism. Since  
$$
\hat\mu(N_{J_{\cU_S}})=\frac{1}{\mu(J_{\cU_S})},
$$
it follows from the very first two claims of Theorem~\ref{t1vp3} and the refined version of Kac's Lemma that $\^\mu(\zt_{t,0})$ and $\^\mu(\zt_{t,\P(t)})$ are both well--defined and larger than $-\infty$. Therefore, using Theorem~2.1.7 in \cite{mugdms}, Proposition~\ref{p1tf5}, and also Abramov's formula, we get that 
	\begin{align}
		\h_\mu(\^f)-\mu(t\log|\^f'|)
		&=\mu(J_{\cU_S})\lt(\h_{\hat\mu}(\^f\rvert_{J_{\cU_S}})-\hat\mu(t\log|\^f_{\cU_S}'|\rt) \label{1vp6}
=\mu(J_{\cU_S})\lt(\h_{\^\mu}(\sg)-\^\mu(\zt_{t,0})\rt)
		\nonumber\\
		&=\mu(J_{\cU_S})\lt(\h_{\^\mu}(\sg)-\^\mu(\zt_{t,\P(t)})+\P(t)\int_{\cD_\cU^\infty}\|\tau_1\|d\^\mu(\tau)\rt)
		\nonumber\\
		&\leq\mu(J_{\cU_S}) \P(t)\int_{\cD_\cU^\infty}\|\tau_1\|d\^\mu(\tau)
		=\mu(J_{\cU_S}) \P(t)\frac{1}{\mu(J_{\cU_S})}\\
		&=\P(t).
		\nonumber
	\end{align}
	Now, assume in addition that 
	$$
		\h_\mu(\^f)-\mu(t\log|\^f'|)=\P(t).	
	$$
	It then follows from the above formula that 
	$$
		\h_{\^\mu}(\sg)-\^\mu(\zt_{t,\P(t)})=0.
	$$
	Hence, invoking Proposition~\ref{p1tf5} and Theorem~2.2.9 of \cite{mugdms}, we get that $\^\mu=\^\mu_t$. Therefore, applying \eqref{2vp6} and formula \eqref{1sl1}, we get
	$$
		\hat\mu=\^\mu\circ\pi_{\cU_S}^{-1}=\^\mu_t\circ\pi_{\cU_S}^{-1}=\hat\mu_t.
	$$
	Finally, applying Lemma~\ref{p1sl4} and formula \eqref{3vp6}, we obtain
	$$
		\mu=\mu_t.
	$$	
	Along with formulas \eqref{1vp5}, \eqref{2vp5}, and \eqref{1vp6}, this completes the proof of Theorem~\ref{tvp3}.
\end{proof}

\section{Decay of Correlations, \\ Central Limit Theorems, the Law of Iterated Logarithm: 
\nl the Method of Lai--Sang Young Towers}\label{Young-Abstract}

\subsection{Stochastic Laws on the Symbol Space for the Shift Map Generated by Nice Families}\label{section:stochasticlawssymbolspace}

In this subsection, making use of the link with the symbolic thermodynamic formalism of Section~\ref{sec:thermodynamic formalism}, we embed the symbol space $\cD_\cU^\infty$, along with the shift map acting on it, into an abstract Young tower (see \cite{lsy1} and \cite{lsy2}) as its first return map, and we  prove the fundamental stochastic laws such as the Law of Iterated Logarithm, the Central Limit Theorem, and exponential decay of correlations, in such an abstract setting. 

\sp Let $T:X\to X$ be a measurable dynamical system preserving a
probability measure $\mu$ on $X$. We
say that a $\mu$--integrable function $g:X\to\R$ with $\int_X g\, d\mu=0$,
satisfies the \textbf{Central Limit Theorem}\index{Central Limit Theorem (CLT)} with respect to the measure $\mu$ if there exists $\sigma>0$ such that
$$
\frac{1}{\sqrt{n}}{\sum_{j=0}^{n-1}g\circ T^j} \xrightarrow[\ \, n\to\infty \ ]{}  
\mathcal N(0,\sigma)
$$ 
in distribution determined by $\mu$. $\mathcal N(0,\sigma)$\index{$\mathcal N(0,\sigma)$} is here the normal (Gaussian) distribution with $0$ mean and variance $\sigma$. More precisely, for every $t\in \R$,
$$
\lim_{n\to\infty}\mu\biggl(\lt\{x\in X: {1\over \sqrt{n}}S_ng (x)\le t\rt\}
= {1\over \sigma\sqrt{2\pi}}\int_{-\infty}^t\exp\(-u^2/2\sigma^2\)\, du. 
$$
We say the function $g$ satisfies the \textbf{Law of Iterated Logarithm}\index{Law of Iterated Logarithm (LIL)} if there exists a positive number $A_g$ such that
$$
\limsup_{n\to\infty}\frac{S_{n}g(x)}{\sqrt{n\log\log n}}=A_g. 
$$
for $\mu$--a.e. $x\in X$.

Another important stochastic feature of a dynamical system is the rate of decay of correlations it yields. Let $\psi_1$ and $\psi_2$ be real square $\mu$--integrable functions on $X$. For every positive integer $n$, the $n$th correlation \rm of the pair $\psi_1,\psi_2$ is the number
\beq\label{correlations}
C_n(\psi_1,\psi_2)
:=\int_X\psi_1\cdot(\psi_2\circ T^n)\,d\mu - \int_X
   \psi_1\,d\mu_\phi \int_X \psi_2\,d\mu,
\eeq
 provided the above integrals exist. Notice that, due to the
$T$--invariance of $\mu$, we can also write
$$
C_n(\psi_1,\psi_2)=\int_X\(\psi_1-\mu(\psi_1)\)\((\psi_2-\mu(\psi_2))\circ
T^n\)\, d\mu.
$$
Finally, we say that two functions $g,h:X\lra\R$ are homologous or, perhaps more adequately, \textbf{cohomologous},\index{cohomologous functions} in a class $\cG$ of real--valued functions defined on $X$ if and only if there exists a function $k\in \cG$ such that
\beq\label{120190910}
h-g=k-k\circ T.
\eeq

As an immediate consequence of results from \cite{MN} and \cite{SUZ}, both for the Law of Iterated Logarithm (LIL), Theorem~3.1 in \cite{gouezel1} (Central Limit Theorem, CLT), Theorem~1.3 in \cite{gouezel2} (exponential decay of correlations), (comp. \cite{gouezel1}, \cite{lsy1}, and \cite{lsy1}), and Proposition~\ref{p1sp4D}, we get the following.

\bthm\label{t1rm_2015_03_11A} 
Let $T:X\to X$ be a measurable dynamical system. Let $G$ be a \FNR \ rational semigroup generated by a $u$--tuple map $(f_1,\ld,f_u)\in\Rat^u$. Let $\cU$ be a nice family of sets for $\^f$. Assume the following.

\begin{itemize}
\item  $X$ contains $\cD_\cU^\infty$ as it measurable subset such that $\bu_{n\ge 0}T^n(\cD_\cU^\infty)=X$.

\item The first return map of $T$ from $\cD_\cU^\infty$ to $\cD_\cU^\infty$ is equal to the shift map $\sg:\cD_\cU^\infty\lra\cD_\cU^\infty$. 

\item The corresponding first return time is equal to $\|\om_1\|$ for every $\om\in\cD_\cU^\infty$. 
\end{itemize}

If $t\in\De_G^*$, then

\sp
\begin{itemize}
\item[\mylabel{a}{t1rm_2015_03_11A item a}] There exists a unique $T$--invariant probability measure $\nu_t$ on $X$ which conditioned on $\cD_\cU^\infty$ coincides with $\^\mu_t$. In addition, $\nu_t$ is ergodic with respect to the map $T:X\to X$.
\end{itemize}

\sp\fr For every function $g:X\to\R$, let $\hat g:\cD_\cU^\infty\to\R$ be defined by the following formula:
$$
\hat g(\om):=\sum_{j=0}^{\|\om_1\|-1}g(T^j(\om)).
$$
Assume that $\hat g\in L^2(\^\mu_t)$ and that $\hat g:\cD_\cU^\infty\to\R$ is H\"older continuous. Assume in addition that $\psi:X\to\R$ is a bounded measurable function. Then we have the following:

\sp\begin{itemize}
\item[\mylabel{b}{t1rm_2015_03_11A item b}]  
$$
|C_n(\psi,g)|
=\lt|\int_X(\psi\circ T^n) g \, d\nu_t-\int\psi\, d\nu_t\int_X g\, d\nu_t\rt|
=O(\theta^n)
$$
for some $\th\in (0,1)$.

\sp\item[\mylabel{c}{t1rm_2015_03_11A item c}] 
The Central Limit Theorem holds for the function $g:X\to\R$ with respect to the dynamical system $(T,\nu_t)$ provided that $g$ is not cohomologous to a constant in $L^2(\nu_t)$.

\sp\item[\mylabel{d}{t1rm_2015_03_11A item d}]
The Law of Iterated Logarithm holds for the function $g:X\to\R$ with respect to the dynamical system $(T,\nu_t)$ provided that $g$ is not cohomologous to a constant in $L^2(\nu_t)$.
\end{itemize}
\ethm

\

We shall now describe a canonical way, known as a \textbf{Young tower}\index{Young tower}  (see \cite{lsy1} and \cite{lsy2}), of embedding $\cD_\cU^\infty$ into a larger space $X$ and to construct an appropriate map $T:X\to X$ so that all the hypotheses of the above theorem are satisfied. Let 
$$
\hat\cD_\cU^\infty:=\{(\om,n)\in\cD_\cU^\infty\times\mathbb{N}\cup\{0\}:0\le n<\|\om_1\|\}
$$
where each point $\om\in\cD_\cU^\infty$ is identified with $(\om,0)\in\hat\cD_\cU^\infty$.  We refer to $\hat\cD_\cU^\infty$\index{Young tower!$\hat\cD_\cU^\infty$}\index{$\hat\cD_\cU^\infty$} as the \textbf{tower} induced by $\cD_\cU^\infty$. The map $T$ acts on $\hat\cD_\cU^\infty$ as follows.
\beq\label{1slpm_2015_03_10}
T(\om,n):=
\begin{cases}
(\om, n+1) \,  &{\rm if}~~ n+1<\|\om_1\|\\
(\sg(\om),0) \, &{\rm if}~~ n+1=\|\om_1\|
\end{cases}.
\eeq
As an immediate consequence of Theorem~\ref {t1rm_2015_03_11A} we get the following.

\bthm\label{t1rm_2015_03_11}
Let $G$ be a \FNR \ rational semigroup generated by a $u$--tuple map $f=(f_1,\ld,f_u)\in\Rat^u$. Let $\cU$ be a nice family of sets for $\^f$. Let $T:\hat\cD_\cU^\infty\to\hat\cD_\cU^\infty$ be defined by formula \eqref{1slpm_2015_03_10}. Fix $t\in \De^*(G)$. Then

\sp
\begin{itemize}
\item[\mylabel{a}{t1rm_2015_03_11 item a}] There exists a unique $T$--invariant probability measure $\nu_t$ on $\hat\cD_\cU^\infty$ which conditioned on $\cD_\cU^\infty$ coincides with $\^\mu_t$. 
\end{itemize}

\sp\fr For every function $g:\hat\cD_\cU^\infty\to\R$ let $\hat g:\cD_\cU^\infty\to\R$ be defined by the following formula:
$$
\hat g(\om):=\sum_{j=0}^{\|\om_1\|-1}g(T^j(\om)).
$$
Assume that $\hat g\in L^2(\^\mu_t)$ and that $\hat g:\cD_\cU^\infty\to\R$ is H\"older continuous. Assume in addition that $\psi:\hat\cD_\cU^\infty\to\R$ is a bounded measurable function. Then
\begin{itemize}

\sp\item[\mylabel{b}{t1rm_2015_03_11 item b}]  
$$
|C_n(\psi,g)|
=\lt|\int_{\hat\cD_\cU^\infty} (\psi\circ T^n) g \, d\nu_t-\int_{\hat\cD_\cU^\infty}\psi\, d\nu_t\int_{\hat\cD_\cU^\infty} g\, d\nu_t\rt|
=O(\theta^n)
$$
for some $\th\in (0,1)$.

\sp\item[\mylabel{c}{t1rm_2015_03_11 item c}] 
The Central Limit Theorem holds for the function $g:\hat\cD_\cU^\infty\to\R$ with respect to the dynamical system $(T,\nu_t)$ provided that $g$ is not cohomologous to a constant in $L^2(\nu_t)$.

\sp\item[\mylabel{d}{t1rm_2015_03_11 item d}]
The Law of Iterated Logarithm holds for the function $g:\hat\cD_\cU^\infty\to\R$ with respect to the dynamical system $(T,\nu_t)$ provided that $g$ is not cohomologous to a constant in $L^2(\nu_t)$.
\end{itemize}
\ethm

\sp

\subsection{Stochastic Laws for the Dynamical System $(\^f:J(\^f)\lra J(\^f),\mu_t)$}\label{section:originalstochasticlaws}

In this subsection, making use of the previous section, via the natural projection from the abstract Young tower to the Julia set $J(\^f)$, we prove in Theorem~\ref{t12015_03_13} the fundamental stochastic laws for dynamical systems $(\tf,\mu_t)$, $t\in\De_G^*$, such as the Law of Iterated Logarithm, the Central Limit Theorem, and exponential decay of correlations.

\sp So, we pass to the actual dynamics of $\^f$ on $J(\^f)$. Consider $H:\hat\cD_\cU^\infty\to\mathbb{C}$, the natural
projection from the tower $\hat\cD_\cU^\infty$ to the complex plane $\mathbb{C}$,
given by the formula
$$
H(\om,n)=\^f^n(\pi_\cU(\om)).
$$
Then
\begin{equation}\label{eq120110623}
H\circ T =\^f\circ H.
\end{equation}
Consequently, we immediately get the following.

\bprop\label{p1_2015_03_10}
Let $G$ be a \FNR \ rational semigroup generated by a $u$--tuple map $f=(f_1,\ld,f_u)\in\Rat^u$. Let $\cU$ be a nice family of sets for $\^f$. 

If $\nu_t$ is the measure produced in Theorem~\ref{t1rm_2015_03_11} \eqref{t1rm_2015_03_11 item a}, then the Borel probability measure 
$$
\nu_t\circ H^{-1}
$$
on $J(\^f)$ is $\^f$--invariant.
\eprop

\fr In order to proceed further we will need the following lemma from abstract ergodic theory. Its proof is standard and can be found in many textbooks on ergodic theory.

\blem\label{l320121210}
If $T:X\to X$ is a measurable map preserving a probability measure
$\mu$ and if $g\in L^2(\mu)$, then the following two statements are
equivalent.

\sp
\begin{itemize}
\item[\mylabel{a}{l320121210 item a}] The function $g$ is a \textbf{coboundary}\index{coboundary}, i.e. $g=u-u\circ T$ for
some $u\in L^2(\mu)$.

\sp\item[\mylabel{b}{l320121210 item b}] The sequence $(S_ng)_{n=1}^\infty$ is bounded in the Hilbert
  space $L^2(\mu)$.
\end{itemize}
\elem

\sp\fr We are now in position to prove the following.

\begin{thm}\label{t220110627}
Let $G$ be a \FNR \ rational semigroup generated by a $u$--tuple map $(f_1,\ld,f_u)\in\Rat^u$. Let $\cU$ be a nice family of sets for $\^f$. If $t\in\De_G^*$ and $\nu_t$ is the corresponding probability measure produced in Theorem~\ref{t1rm_2015_03_11} \eqref{t1rm_2015_03_11 item a}, then for the dynamical system $(\^f:J(\^f)\lra J(\^f),\nu_t\circ H^{-1})$ the following hold.
\begin{itemize}
\item[\mylabel{a}{t220110627 item a}] Fix $s\in(0,1]$ and a bounded function $g:J(\^f)\to\mathbb{R}$ which is H\"older continuous with the exponent $s$. Then for every bounded measurable function $\psi:J(\^f)\to\mathbb{R}$, we have that
$$
\bigg|\int_{J(\^f)}\psi\circ \^f^n \cdot g\, d(\nu_t\circ H^{-1})
-\int_{J(\^f)} g\, d(\nu_t\circ H^{-1})\int_{J(\^f)}\psi\, d(\nu_t\circ H^{-1})\bigg|
=O(\theta^n)
$$
for some $\theta\in(0,1)$ depending on $s$.

\sp\item[\mylabel{b}{t220110627 item b}] The Central Limit Theorem holds for every H\"older continuous function $g:J(\^f)\to \mathbb{R}$ that is not
cohomologous to a constant in $L^2(\nu_t\circ H^{-1})$, i.e. for which there is no square integrable function $\eta$ for which $g={\rm
const}+\eta\circ f-\eta$. More precisely, there exists $\sigma>0$ such that
$$
\frac{1}{\sqrt{n}}{\sum_{j=0}^{n-1}g\circ \^f^j} \xrightarrow[\ \, n\to\infty \ ]{}  \mathcal N(0,\sigma)
$$ 
in distribution with respect to the measure $\nu_t\circ H^{-1}$.

\sp\item[\mylabel{c}{t220110627 item c}] The Law of Iterated Logarithm holds for every H\"older continuous
function $g:J(\^f)\to \mathbb{R}$ that is not cohomologous to a constant
in $L^2(\nu_t\circ H^{-1})$. This, we recall, means that there exists a real positive
constant $A_g$ such that such that $\nu_t\circ H^{-1}$ almost everywhere 
$$
\limsup_{n\to\infty}\frac{S_{n}g-n\int g\,d(\nu_t\circ H^{-1})}{\sqrt{n\log\log n}}=A_g. 
$$
\end{itemize}
\end{thm}

{\sl Proof.} We aim to employ Theorem~\ref{t1rm_2015_03_11}.
Let $g:J(\^f)\to\mathbb{R}$ and $\psi:J(\^f)\to\mathbb{R}$ be as in the
hypotheses of our theorem. Define the functions 
$$
\tilde g:=g\circ H:\hat\cD_\cU^\infty\lra\mathbb{R} \  \  \text{ and } \ \
\tilde\psi:=\psi\circ H:\hat\cD_\cU^\infty\lra\mathbb{R}.
$$
In order to apply Theorem~\ref{t1rm_2015_03_11}, all what we need to do is to check the hypotheses of this theorem pertaining to the functions $\^\psi$ and $\^g$. For $\^\psi$ this is immediate: of course this function is measurable and bounded.
We shall prove the following.

%\sp\fr {\bf Claim~1:}
\sp\fr \textbf{Claim}~\myClaimlabel{\textbf{1}}{Claim1}\textbf{:} The function $\hat{\tilde g}:\cD_\cU^\infty\to\mathbb{R}$ is H\"older continuous. 

\sp\fr Indeed, consider two arbitrary points $\om, \tau\in\cD_\cU^\infty$ with $\om_1=\tau_1$. Put $\g:=\om_1=\tau_1$ and $l:=\|\g\|$. Put also $\b:=\om\wedge\tau$ and $n:=|\b|$. In particular, $\b=\g\sg(\b)$ and $\|\sg(\b)\|\ge n-1$. For every $0\le k\le l-1$, we have that 
\beq\label{320121208}
\begin{aligned}
|\^g(T^k(\tau,0))-\^g(T^k(\om,0))|
&= |\^g(\tau,k)-\^g(\om,k)| 
 = |g\circ H(\tau,k)-g\circ H(\om,k)|  \\
&=|g(\^f^k(\pi_\cU(\tau)))-g(\^f^k(\pi_\cU(\om)))| \\
&=|g(\^f^k(\pi_\cU(\tau)))-g(\^f^k(\pi_\cU(\om)))| \\
&\le H_g\|\^f^k(\pi_\cU(\tau)),f^k(\pi_\cU(\om))\|_\vartheta^s,
\end{aligned}
\eeq
where $H_g$ is the H\"older constant of $g$. Moreover,
$$
\^f^k(\pi_\cU(\om)),\^f^k(\pi_\cU(\tau))
\in\^f^k(\phi_{\g\sg(\b)}(W_{t(\g)}))
=\phi_{\sg^k(\g)\sg(\b)}(W_{t(\g)}),
$$
where the last two $\sg$s denote two different shift maps with obvious meanings. Therefore, because of Theorem~\ref{t1h4} (Exponential Shrinking Property), we have that 
$$
\begin{aligned}
\|\^f^k(\pi_\cU(\tau)),f^k(\pi_\cU(\om))\|_\vartheta
&\le \diam_{\Sg_u\times\C}\(\phi_{\sg^k(\g)\sg(\b)}(W_{t(\b)})\)
\\
&\le \max\lt\{\vartheta^{\|\sg^k(\g)\sg(\b)\|},\exp\(-\a(\|\sg^k(\g)\sg(\b)\|)\)\rt\}
\\
&=\exp\(-\iota(\|\sg^k(\g)\sg(\b)\|)\) 
\\
&=\exp\(-\iota\((l-k)+\|\sg(\b\|\)\)
\\
&\le e^{-1}\exp(-\iota(l-k)))e^{-\iota n},
\end{aligned}
$$
where $\a\in (0,+\infty)$ comes from Theorem~\ref{t1h4} and
$$
\iota:=\min\{\a, -\log\vartheta\}\in (0,+\infty).
$$
Hence,
$$
|\^g(T^k(\tau,0))-\^g(T^k(\om,0))|
\lek \exp(-s\iota(l-k)))e^{-s\iota n}.
$$
Finally,
$$
\begin{aligned}
|\hat{\tilde g}(\om)-\hat{\tilde g}(\tau)|
&=\lt|\sum_{k=0}^{l-1}\lt(\^g(T^k(\tau,0))-\^g(T^k(\om,0))\rt)\rt|
\le \sum_{k=0}^{l-1}\big|\^g(T^k(\tau,0))-\^g(T^k(\om,0))\big| \\
&\lek \sum_{k=0}^{l-1}\exp(-\iota(l-k)))e^{-s\iota n} \\
&\comp e^{-s\iota n}.
\end{aligned}
$$
So, the function $\hat{\tilde g}:\cD_\cU^\infty\to\mathbb{R}$ is H\"older continuous, and the proof of Claim~\ref{Claim1} is complete.

%\sp\fr {\bf Claim~2:}
\sp\fr \textbf{Claim}~\myClaimlabel{\textbf{2}}{Claim2}\textbf{:}  The function $\hat{\tilde g}:\cD_\cU^\infty\to\mathbb{R}$ is square integrable with respect to the measure $\^\mu_t$. 

\sp\fr Indeed, for every $\om\in \cD_\cU^\infty$ we have that $|\hat{\tilde g}(\om)|\le \|{\tilde g}\|_\infty\|\cdot\|\om_1||\le \|g\|_\infty\|\om_1\|$. Hence,
$$
|\hat{\tilde g}(\om)|^2\le \|g\|_\infty^2\|\om_1\|^2.
$$
Therefore, invoking Proposition~\ref{p1sp4D}, we can estimate as follows. 
$$
\begin{aligned}
\int_{\cD_\cU^\infty}|\hat{\tilde g}|^2\, d\^\mu_t
&=\sum_{n=1}^\infty\int_{\bu_{\|e\|=n}[e]}|\hat{\tilde g}|^2\, d\^\mu_t
\le \|g\|_\infty^2\sum_{n=1}^\infty n^2\^\mu_t\Big(\bu_{e\in\cD_n}[e]\Big) \\
&\lek \|g\|_\infty^2\sum_{n=1}^\infty n^2e^{-\eta n} \\
&<+\infty,
\end{aligned}
$$
with some constant $\eta>0$ resulting from Proposition~\ref{p1sp4}. The proof of Claim~\ref{Claim2} is complete. 
\qed

%\sp\fr{\bf Claim~3:}
\sp\fr \textbf{Claim}~\myClaimlabel{\textbf{3}}{Claim3}\textbf{:} The function $\tilde g:\hat\cD_\cU^\infty\to\mathbb{R}$ is not
cohomologous to a constant in $L^2(\nu_t)$.

\sp\fr Indeed, assume without loss of generality that $\nu_t(g)=0$. By
virtue of Lemma~\ref{l320121210} the fact that  $g:J(\^f)\to \mathbb{R}$ is not 
a coboundary in $L^2(\nu_t\circ H^{-1})$ equivalently means that the sequence
$\(S_n(g)\)_{n=1}^\infty$ is not uniformly bounded in
$L^2(\nu_t\circ H^{-1})$. But obviously $\|S_n(\tilde
g)\|_{L^2(\nu_t)} = \|S_n(g)\|_{L^2(\nu_t\circ H^{-1})}$. So, the sequence $\(S_n(\tilde
g)\)_{n=0}^\infty$ is not uniformly bounded in $L^2(\nu_t)$. Thus, by
Lemma~\ref{l320121210} again, $\^g$ is not a coboundary in $L^2(\nu_t)$. 

\sp\fr Having these two claims, all items, \eqref{t220110627 item a}, \eqref{t220110627 item b}, and \eqref{t220110627 item c}, now
follow immediately from Theorem~\ref{t1rm_2015_03_11} and formula \eqref{eq120110623}. The proof is
finished. 
\endpf

\sp\fr By making use Proposition~\ref{p1sl4}, as a fairly easy consequence of Theorem~\ref{t220110627}, we get the following main result of this section.

\begin{thm}\label{t12015_03_13}
Let $G$ be a \FNR \ rational semigroup generated by a $u$--tuple map $f=(f_1,\ld,f_u)\in\Rat^u$. If $t\in\De_G^*$, then for the dynamical system $(\^f:J(\^f)\lra J(\^f),\mu_t)$ the following hold.
\begin{itemize}
\item[\mylabel{a}{t12015_03_13 item a}] Fix $s\in(0,1]$ and a bounded function $g:J(\^f)\to
\mathbb{R}$ which is H\"older continuous with the exponent $\alpha$. Then for
every bounded measurable function $\psi:J(\^f)\to\mathbb{R}$, we have that
$$
\bigg|\int_{J(\^f)}\psi\circ \^f^n \cdot g\, d\mu_t
-\int_{J(\^f)} g\, d\mu_t\int_{J(\^f)}\psi\, d\mu_t\bigg|
=O(\theta^n)
$$
for some $\theta\in(0,1)$ depending on $s$.

\sp\item[\mylabel{b}{t12015_03_13 item b}] The Central Limit Theorem holds for every H\"older continuous
  function $g:J(\^f)\to \mathbb{R}$ that is not
  cohomologous to a constant in $L^2(\mu_t)$, i.e. for which there is no
  square integrable function $\eta$ for which $g={\rm
    const}+\eta\circ f-\eta$. More precisely, there exists $\sigma>0$ such that
$$
\frac{1}{\sqrt{n}}{\sum_{j=0}^{n-1}g\circ \^f^j} \xrightarrow[\ \, n\to\infty \ ]{}  \mathcal N(0,\sigma)
$$ 
in distribution with respect to the measure $\mu_t$.

\sp\item[\mylabel{c}{t12015_03_13 item c}] The Law of Iterated Logarithm holds for every H\"older continuous
function $g:J(\^f)\to \mathbb{R}$ that is not cohomologous to a constant
in $L^2(\mu_t)$. This, we recall, means that there exists a real positive
constant $A_g$ such that $\mu_t$ almost everywhere 
$$
\limsup_{n\to\infty}\frac{S_{n}g-n\int g\,d\mu_t}{\sqrt{n\log\log n}}=A_g. 
$$
\end{itemize}
\end{thm}

\begin{proof} 
Using in turn Proposition~\ref{p1sl4}, formula \eqref{1sl1} defining $=\hat\mu_t$, the fact that $H|_{\cD_\cU^\infty}=\pi_\cU$, and Theorem~\ref{t1rm_2015_03_11} \eqref{t1rm_2015_03_11 item a}, we get
\beq\label{1_2015_04_13}
(\mu_t)_{J_\cU}
=\hat\mu_t
=\(\^\mu_t\circ \pi_\cU^{-1}\)_{J_\cU}
=\(\^\mu_t\circ H^{-1}\)_{J_\cU}
=\(\(\nu_t\)_{\cD_\cU^\infty}\circ H^{-1}\)_{J_\cU}
\abs\abs\(\nu_t\circ H^{-1}\)_{J_\cU}.
\eeq
But, by virtue of Theorem~\ref{t1rm_2015_03_11}, the measure $\nu_t$ is ergodic with respect to the map $T:\hat\cD_\cU^\infty\to\hat\cD_\cU^\infty$, defined by formula \eqref{1slpm_2015_03_10}. So, the measure $\nu_t\circ H^{-1}$ is ergodic with respect to the map $\^f:J(\^f)\lra J(\^f)$. As $\mu_t$ is also ergodic with respect to this map, both measures $(\mu_t)_{J_\cU}$ and $\(\nu_t\circ H^{-1}\)_{J_\cU}$ are ergodic with respect to the first return map of $\^f$ from $J_\cU$ to $J_\cU$. Therefore, by virtue of \eqref{1_2015_04_13}, they coincide. Hence, $\mu_t$ and $\nu_t\circ H^{-1}$ are not singular with respect to each other. Since both are ergodic, these must be thus equal, i.e.
$\mu_t=\nu_t\circ H^{-1}$. Theorem~\ref{t12015_03_13} now directly follows from  Theorem~\ref{t220110627}. The proof is complete.
\end{proof}

\part{Geometry of Finely Non--Recurrent Rational Semigroups Satisfying the Nice Open Set Condition}

Part~3 of our paper is devoted to the study of the finer fractal and geometrical properties of the fiber Julia sets $J_\om$ and the global Julia set $J(G)$. 

\section{Nice Open Set Condition (for any Rational Semigroup)}\label{NOSC}
The notion of the Open Set Condition was introduced by Hutchinson in \cite{Hutchinson} in the context of finite alphabet iterated function systems consisting of contracting similarities. It was later adopted in \cite{mulms} to the setting of countable alphabet iterated function systems consisting of contracting conformal maps, and has commonly been used in the theory of iterated function systems since the work of Hutchinson. A version of the open set condition for expanding rational semigroups was introduced in \cite{hiroki3}. In the paper \cite{sush} we defined the Open Set Condition in the context of (loosely meaning) non--recurrent rational semigroups. Unlike \cite{hiroki3}, in \cite{sush} we had to allow critical points of $G$ to lie in the Julia set $J(G)$. This makes the whole situation more complex and demanding, and forced us in \cite{sush} to strengthen the Open Set Condition a little bit. We named such modified condition the Nice Open Set Condition. We would like to emphasize that the Nice Open Set Condition and Nice Sets (Families) are totally independent concepts. In particular, the adjective ``Nice" was  independently introduced for both of them many years ago. Although it may be a little bit confusing for some readers, we stick to the historical terminology to respect history and in order not to confuse readers even more by inventing yet new names. We think that in our current manuscript this is the first time in the literature that both ``nice" concepts are used simultaneously.

We proved in \cite{sush} several of its crucial properties and used it heavily therein. In the present section we reprove, with essentially new proofs, its properties that we will need and use throughout the end of the manuscript. We also prove some new results.

\sp Motivated by \cite{sush}, we adopt the following definition.

\bdfn\label{d1h2intro}
Let $f=(f_{1},\ldots ,f_{u})\in\Rat^u$ be a $u$--tuple map and
let $G=\langle f_{1},\ldots ,f_{u}\rangle .$ We say that
$G$ (or $f$) satisfies the\textbf{ Open Set Condition}\index{Open Set Condition} if there exists a non--empty
open subset $U$ of $\oc$ with the following two properties:
\begin{itemize}
\item[\mylabel{osc1}{osc1}] 
$$
f_1^{-1}(U)\cup f_2^{-1}(U)\cup\ld f_u^{-1}(U)\sbt U,
$$
\item[\mylabel{osc2}{osc2}] 
$$
f_i^{-1}(U)\cap f_j^{-1}(U)=\es
$$
whenever $i\ne j$.
\end{itemize}
Moreover, we say that $G$ (or $f$) satisfies the \textbf{Nice Open Set Condition} \index{Nice Open Set Condition} if, in addition, $\ov U\sbt\C$ and the following condition is satisfied.
\begin{itemize}
\item[\mylabel{osc3}{osc3}] $\exists(\a\in(0,1)) \,\, \forall(0< r\le 1) \,\, \forall(x\in\ov U)$
$$
l_2(U\cap B_{2}(x,r))\ge \a l_2(B_{2}(x,r)),
$$ 
where $l_{2}$\index{$l_2$} denotes $2$--dimensional Lebesgue measure on $\C$.
\end{itemize}
\edfn

\begin{rem}
\label{r:osc3}
Condition \eqref{osc3} is not needed if our semigroup $G$ is expanding (see \cite{hiroki3}
or note that our proofs would use only \eqref{osc1} and \eqref{osc2} under this assumption).
Condition \eqref{osc3} is
satisfied in the theory of conformal infinite iterated function systems (see \cite{mulms},
comp. \cite{mugdms}), where it follows from the open set condition and
the cone condition. Moreover, condition \eqref{osc3} holds for example if the boundary of $U$ is
smooth enough; piecewise smooth with no exterior cusps suffices.
Furthermore, \eqref{osc3} holds if $U$ is a John domain (see \cite{CJY}).
\end{rem}

As an immediate consequence of Definition~\ref{d1h2intro}, we get the following.

\blem\label{l120190325}
If a $u$--tuple map $f=(f_{1},\ldots ,f_{u})\in\Rat^u$ and
$G=\langle f_{1},\ldots ,f_{u}\rangle$ satisfy the Open Set Condition, and if $U\sbt\oc$ is an open set which witnesses this condition, then
$$
J(G)\sbt \ov U.
$$
\elem

\bpf
This is obvious if $\ov U=\oc$. So, suppose that 
\beq\label{520190325}
\ov U\ne \oc.
\eeq
It follows immediately from Definition~\ref{d1h2intro} that 
$$
G^*(\oc\sms \ov U)\sbt \oc\sms \ov U.
$$
Since, by \eqref{520190325}, $\oc\sms \ov U$ contains at least three points, it follows from Montel's Theorem that $G^*$ is normal on $\oc\sms \ov U$. This means that $\oc\sms \ov U\sbt F(G)$, whence $J(G)\sbt \ov U$ and we are done.
\epf

The following consequence of Definition~\ref{d1h2intro} is immediate.

\blem\label{l220190325}
If a $u$--tuple map $f=(f_{1},\ldots ,f_{u})\in\Rat^u$ and
$G=\langle f_{1},\ldots ,f_{u}\rangle$ satisfy the Open Set Condition, and if $\om, \tau\in \Sg_u^*$ are two incomparable words, then
$$
f_\om^{-1}(U)\cap f_\tau^{-1}(U)=\es.
$$
\elem

Further auxiliary, but necessary for us, consequences of Definition~\ref{d1h2intro}, proved in \cite{sush}, are these.

\blem\label{l120190328}
Let $f=(f_{1},\ldots ,f_{u})\in \Rat^u$ be a $u$--tuple map and
let $G=\langle f_{1},\ldots ,f_{u}\rangle$. Assume that $f$ and $G$ satisfy the Nice Open Set Condition with some open set $U\sbt\C$. Then, there exist 
$\b_1>0$, $\eta\in (0,1/8]$, and $\ka_1>0$ such that 
$$
l_{2}\(f_{j}^{-1}(U)\cap B_{s}(x,r)\)\geq \ka_1 r^{2}
$$
for each $j=1,\ldots,u$, for each $x\in f_{j}^{-1}(\overline{U})\sms \Crit(f_j)$, and for each 
$$
r\in\lt(0,\min\lt\{\b_1,\eta\, \dist_\C\(x,\Crit(f_j)\)\rt\}\rt).
$$
\elem

\bpf
Let $\dl>0$ and $\eta\in(0,1/8]$ be so small that for every $j\in\{1,2,\dots,u\}$, every $c\in\Crit(f_j)$, and every $z\in B_2(c,\dl)\bs\{c\}$ the map 
$$
	f_j\rvert_{B_s(z, 8\eta|z-c|)}: B_2(z,8\eta|z-c|)\lra \hat\C
$$
is 1--to--1. Let
$$
	\ga:=\min\big\{\b_1,\eta\, \dist(x,\Crit(f_j))\big\}.
$$
Let $\b_1\in(0,1]$ be so small that for every $z\in B_2(\ov U,16\b_1)\bs B_2(\Crit(G),\dl)$ and every $j\in\{1,\dots,u\}$, the map 
$$
f_j\rvert_{B_2(z,8\b_1)}: B_2(z, 8\b_1)\lra \hat\C
$$
is 1--to--1 and 
\begin{align}\label{1osc1}
	4\b_1\max\big\{\|f_j'\|_\infty : j\in \{1,\dots, u\}\big\}\leq 1.
\end{align}
Suppose now that 
\begin{align}\label{2osc1}
	x\in f^{-j}(\overline U)\bs\Crit(f_j).
\end{align}
Then, the inverse map 
$$
	f_{j,x}^{-1}:=\lt(f_j\rvert_{B_2(x,4\ga)}\rt)^{-1}: f_j(B_2(x,4\ga))\lra B_2(x,4\ga)
$$
is well defined, holomorphic, and, because of the ${1\over 4}$--Koebe's Distortion Theorem, 
$$
f_j(B_2(x,8\ga))\bus B_2(f_j(x), 2|f_j'(x)|\ga).
$$	
Hence, it follows from the standard Koebe's Distortion Theorem, applied this time to the map $f_j^{-1}$, together with condition \eqref{osc3} of Definition \ref{d1h2intro}, and \eqref{1osc1} that for every $r\in (0,\ga)$, we have that 
\begin{align}
	l_2 \lt( f_j^{-1}(U)\cap B_s(x,r)\rt)
	&\geq 
	l_2 \lt( f_{j,x}^{-1}\lt(U\cap B_s(f_j(x), K^{-1}|f_j'(x)|r) \rt) \rt)
	\nonumber\\
	&\geq
	K^{-2}|f_{j,x}'(x)|^{-2} l_2 \lt(U\cap B_2(f_j(x), K^{-1}|f_j'(x)|r) \rt)
	\nonumber\\
	&\geq
	K^{-2}\a|f_{j,x}'(x)|^{-2} l_2 \lt(B_2(f_j(x), K^{-1}|f_j'(x)|r) \rt)
	\nonumber\\
	&=\pi^2K^{-2}\a r^2. \label{3osc1}
\end{align}
So, the proof is completed by taking $\ka_1:=\pi^2K^{-2}\a$.

\epf
\blem\label{l:osclower}
Let $f=(f_{1},\ldots ,f_{u})\in (\mbox{Rat})^{u}$ be a $u$--tuple map and let $G=\langle f_{1},\ldots ,f_{u}\rangle$. Assume that $f$ and $G$ satisfy the Nice Open Set Condition with some open set $U\sbt\C$. Then there exist two constants $\b_2>0$ and $\ka_2>0$, such that
$$
l_{2}\(f_{j}^{-1}(U)\cap B_{s}(c,r)\)\geq \ka_2r^{2}
$$
for all $j\in \{ 1,\ldots ,u\}$, all points $c\in \Crit(f_{j})\cap f_{j}^{-1}(\overline{U})$, and all radii $r\in (0,\b_2)$.
\elem
\begin{proof} 
Denote the order of the critical point $c$ of $f_j$ by $q$. Let $\b_2>0$ be so small that 
$$
A_{j,c}^{-1}|z-c|^{q}\leq |f_j(z)-f_j(c)|\leq A_{j,c}|z-c|^{q}
$$
and 
$$
A_{j,c}^{-1}|z-c|^{q-1}\leq |f_j'(z)|\leq A_{j,c}|z-c|^{q-1}
$$
for some constant $A_{j,c}\geq 1$, depending on $j$ and $c$, and all $z\in B_2(c,\b_2)$. Then, 
$$
	f_j(B_2(c,r))\bus B_2(f_j(c), A_{j,c}^{-1}r^q)
$$
for all $r\in (0,\b_2)$. Since also 
$$
	f_j\lt(f_j^{-1}(U)\cap B_2(c,r)\rt)=U\cap f_j(B_2(c,r)), 
$$
using \eqref{osc3}, we thus get that 
\begin{align}
	l_2\lt( f_j\lt(f_j^{-1}(U)\cap B_2(c,r)\rt) \rt)
	&=
	l_2\lt(U\cap f_j(B_2(c,r)) \rt)
	\geq 
	l_2\lt(U\cap B_2(f_j(c), A^{-1}_{j,c}r^q) \rt)
	\nonumber \\
	&\geq 
	\a l_2\lt(B_2(f_j(c), A^{-1}_{j,c}r^q) \rt)  
	\nonumber \\
	&= 
	\pi\a A_{j,c}^{-2}r^{2q}.\label{1osc2}
\end{align}
But also, 
\begin{align*}
	l_2\lt(f_j\lt(f_j^{-1}(U)\cap B_2(c,r) \rt)\rt)
	&\leq 
	\int_{f_j^{-1}(U)\cap B_2(c,r)} |f_j'(z)|^2 \, dl_2(z)
	\\
	&\leq 
	A_{j,c}^2 \int_{f_j^{-1}(U)\cap B_2(c,r)} |z-c|^{2(q-1)}\, dl_2(z)
	\\
	&\leq 
	A_{j,c}^2 r^{2(q-1)}l_2\lt(f_j^{-1}(U\cap B_2(c,r))\rt ).
\end{align*}
Combining this with \eqref{1osc2}, we get 
$$
	l_2\lt(f_j^{-1}(U)\cap B_2(c,r) \rt) \geq \pi \a A_{j,c}^{-4}r^2.
$$
Since there are only finitely many (precisely $u$) maps $f_j$ and since each map $f_j$ has only finitely many critical points, we complete the proof by taking 
$$
	\ka_2:=\pi \a \min\lt\{A_{j,c}^{-4}: j\in \{1, \dots, u\}, c\in \Crit(f_j) \rt\}.
$$
\end{proof}

Combining Lemma~\ref{l120190328} and Lemma~\ref{l:osclower}, we obtain the following lemma.
\blem
\label{l:nosclc}
Let $G=\langle f_{1},\ldots ,f_{u}\rangle $ be a
rational semigroup satisfying the Nice Open Set Condition witnessed by some open set $U\sbt\C$. 
Then, for every $\b>0$ there exists $\ka_\b>0$ such that
for each $j=1,\ldots ,u$, for each $x\in f_{j}^{-1}(\overline{U})$, and for each radius $r\in (0,\b]$, we have that
$$
l_{2}\(f_{j}^{-1}(U)\cap B_s(x,r)\)\geq \ka_\b r^{2}.
$$
\elem

\bpf
As in the proofs of the two preceding lemmas, we again work with Euclidean balls and distances.

It suffices to deal with $r\in(0,\min\{\b_1,\b_2\})$. Fix $j\in\{1,\dots, u\}$. Because of Lemma~\ref{l120190328}, we need only consider the case when 
$$
	r\in\lt(\min\{\b_1, \eta\cdot\dist(x, \Crit(f_j))\}, \min\{\b_1,\b_2\}\rt),
$$
with $\eta\in(0,1/8]$ coming from Lemma~\ref{l120190328}. Consider two cases. Assume first that 
$$
	r\geq 2\dist(x,\Crit(f_j)).
$$
Then there exists a point $c\in\Crit(f_j)$ such that $|x-c|\leq r/2$. Hence, 
$$
	B_2(x,r)\bus B_2(c,r/2).
$$
It therefore follows from Lemma~\ref{l:osclower} that 
$$
	l_2\lt(B_2(x,r)\cap f_j^{-1}(U) \rt)
	\geq 
	l_2\lt(B_2(c,r/2)\cap f_j^{-1}(U) \rt)
	\geq 
	r\ka_2/4.
$$
Now assume that 
$$
	r<2\dist(x, \Crit(f_j)).
$$
Then, by Lemma~\ref{l120190328}, we get that 
\begin{align*}
	l_2\lt(B_2(x,r)\cap f_j^{-1}(U) \rt)
	&\geq 
	l_2\lt(B_2(x,\sfrac{1}{2}\min\{\b_1,\eta\cdot\dist(x,\Crit(f_j)) \})\cap f_j^{-1}(U) \rt)
	\\
	&\geq
	\ka_1\lt(\sfrac{1}{2}\min\{\b_1,\eta\cdot\dist(x,\Crit(f_j)) \} \rt)^2\\
	&\geq 
	\frac{\ka_1}{4}\lt(\frac{\eta r}{2}\rt)^2
	=
	\frac{\ka_1\eta^2}{16}r^2.
\end{align*}
We are done.
\epf

\sp For every family $\Fa\sbt\Sg_u^*$ let
$$
\hat\Fa=\{{\hat\tau}:\tau\in\F\} \  \  \  \text{ and } \  \  \ \Fa_*=\{\tau_*:\tau\in\F\}.
$$

\bdfn\label{d1h2a.1}
Let $G=\langle f_{1},\ldots, f_{u}\rangle $ be a rational semigroup
satisfying the Nice Open Set Condition witnessed by some open set $U$. %Suppose that $J(G)\subset \C. $
Fix a number $M>0$, a number $a>0$, and $V$, an open subset of $\Sg_u$. 
A family $\Fa\sbt\Sg_u^*$ is called \textbf{$(M,a,V)$--essential}\index{$(M,a,V)$--essential} for the
pair $(x,r)$ provided that the following conditions are satisfied.
\begin{itemize}
\item[\mylabel{ess0}{ess0}] For every $\tau \in \Fa$, 
$$
\dist_\C\(f_{\hat\tau}(x),f_{\tau_*}^{-1}(U)\)\le \frac12R_\tau.
$$
\item[\mylabel{ess1}{ess1}] For every $\tau\in \Fa$ there exist a number $R_\tau$ with $0<R_{\tau }<a$ and $f_{\hat{\tau},x}^{-1}:
B_2(f_{\hat{\tau}}(x),2R_\tau)\lra\C$, an analytic branch of $f_{\hat{\tau}}^{-1}$ sending $f_{\hat{\tau}}(x)$
to $x$, such that
$$
M^{-1}R_\tau\le |f_{\hat{\tau}}'(x)|r\le {1\over 4}R_\tau.
$$
\item[\mylabel{ess2}{ess2}] The family $\Fa$ consists of mutually incomparable words.

\,

\item[\mylabel{ess3}{ess3}] $\bu_{\tau\in\Fa}[\tau]=V$.
\end{itemize}

\,

If $V=\Sg_u$, the family $\Fa$ is simply called \textbf{$(M,a)$--essential} for the pair $(x,r)$.
\edfn

%\bdfn
%Let $f=(f_{1},\ldots ,f_{u})\in (\mbox{{\em Rat}})^{u}$ be a $u$--tuple map.
%
%
%\edfn

\,

\fr We shall prove the following.
\bprop\label{p1h2a.1}
Let $G=\langle f_{1},\ldots ,f_{u}\rangle $ be a rational semigroup satisfying the Nice Open Set Condition witnessed by some open set $U$.
Then, for every number $M>0$ and for every $a>0$ there exists an integer $\#_{(M,a)}\ge 1$ such that if $V$ is an open subset of $\Sg_u$, $x\in J(G)$, $r\in (0,1]$, and $\Fa\sbt\Sg_u^*$
is an $(M,a,V)$--essential family for $(x,r)$, then we have the following:
\begin{itemize}
\item[\mylabel{a}{p1h2a.1 item a}]
$$
B_2(x,r)
\sbt f_{{\hat\tau},x}^{-1}\(B_2(f_{\hat\tau}(x),R_\tau)\)
\sbt\bu_{\g\in\Fa}f_{\hat\g,x}^{-1}\(B_2(f_{\hat\g}(x),R_\g)\)
\sbt B_2(x,KMr)
$$
for all $\tau\in\Fa$,
\item[\mylabel{b}{p1h2a.1 item b}]
$$
J(\tf)\cap(V\times B_2(x,r))
\sbt \bu_{\tau\in \Fa}\^f_{{\hat\tau},x}^{-|{\hat\tau}|}\(p_2^{-1}(B_2(f_{\hat\tau}(x),R_\tau))\)
=    \bu_{\tau\in \Fa}[\tau]\times f_{{\hat\tau},x}^{-1}\(B_2(f_{\hat\tau}(x),R_\tau)\),
$$
%where $\^f_{{\hat\tau},x}^{-|{\hat\tau}|}(\om,z)=({\hat\tau}\om,f_{{\hat\tau},x}^{-1}(z))$.
\item[\mylabel{c}{p1h2a.1 item c}] $\#\Fa\le\#_{(M,a)}$.
\end{itemize}
\eprop
\noindent{\sl Proof.} Item \eqref{p1h2a.1 item a} follows immediately from Theorem~\ref{kdt1/4} (${1\over 4}$--Koebe's
Distortion Theorem), and Theorem~\ref{kdt1E}. The equality part in item \eqref{p1h2a.1 item b} is obvious. In order to prove the
inclusion take $(\om,z)\in J(\tf)\cap (V\times B_2(x,r))$. By item \eqref{ess3} of
Definition~\ref{d1h2a.1} there exists $\tau\in\Fa$ such that $\om\in[\tau]$.
But then, by the first in item \eqref{p1h2a.1 item a}, $(\om,z)\in
[\tau]\times f_{{\hat\tau},x}^{-1}\(B_2(f_{\hat\tau}(x),R_\tau)\)$, and item \eqref{p1h2a.1 item b} is entirely
proved. Let us deal with item \eqref{p1h2a.1 item c}.
%It follows from item \eqref{ess2} of Definition~\ref{d1h2a.1}
%that there exists a subfamily $\Fa_1\sbt\Fa$ such that $\#\Fa_1\ge\#\Fa/u$ and $\hat\Fa_1$
%consists of mutually incomparable words. Therefore, by the item \eqref{osc2} of Definition~\ref{d1h2}
%$\{f_{{\hat\tau},x}^{-1}\(B_2(f_{\hat\tau}(x),R_\tau)\cap U\)\}_{\tau\in\Fa_1}$ is a
%family of mutually disjoint sets.
%Let $0<c_{1}<1, 0<c_{2}<1$ be such that for any $y\in J(G)$ and for any $j,r$,
%if $y\in J(G)$ and $f_{j}(y)\in J(G)$, then
%$B_2(y,c_{2}r)\subset \mbox{Comp}(y,f_{j},c_{1}r)\subset B_2(y,r).$
By item \eqref{osc2} of Definition~\ref{d1h2intro},
$$ \{ f_{\hat{\tau},x}^{-1}((f_{\tau _{\ast }}|_{B_2(f_{\hat{\tau } }(x),R_{\tau })})^{-1}(U))\} _{\tau \in \mathcal{F}}$$
is a family of mutually disjoint sets.
Hence, using also \eqref{p1h2a.1 item a}, we get
\beq\label{1h2c.1}
%\Sg_{\tau\in\Fa_1}l_2\(f_{{\hat\tau},x}^{-1}\(B_2(f_{\hat\tau}(x),R_\tau)\cap U\)\}\)
\sum_{\tau \in \mathcal{F}}\,l_{2}
\(f_{\hat{\tau},x}^{-1}((f_{\tau _{\ast }}|_{B_2(f_{\hat{\tau }}(x),R_{\tau })})^{-1}(U))\)
\le l_2(B_2(x,KMr))
=C\pi(KM)^2r^2,
\eeq
where $C>0$ is a constant independent of $r,M$, and $a.$
%Let $\beta >0$ be a constant
Let $L_{a}:= \xi \min \{ (T/a)^{2},1\} $,
where $\xi $ and $T$ come from Lemma~\ref{l:nosclc}.
By Lemma~\ref{l:nosclc}, we obtain that
for each $j=1,\ldots,u$ , for each $y\in \overline{f_{j}^{-1}(U)}$, and
for each $0\le b\leq a$,
\beq
\label{eq:l2byb}
l_{2}(B_2(y,b)\cap f_{j}^{-1}(U))\geq \ka_ab^{2}.
\eeq
By \eqref{ess0} there exists a point $z\in f_{\tau _{\ast }}^{-1}(\ov U))$ such that $|f_{\hat{\tau }}(x)-z|\le R_\tau/2$. Then
$$
B_2\(f_{\hat{\tau }}(x),R_{\tau }\)
\spt B_2(z,R_\tau/2).
$$
It follows from this, Theorem~\ref{kdt1E}, (\ref{eq:l2byb}), and \eqref{ess1} that
%for all $\tau\in\Fa_1$ (in fact in $\Fa$),
for all $\tau \in \mathcal{F}$,
we have
$$
\aligned
l_{2}\left(f_{\hat{\tau},x}^{-1}((f_{\tau _{\ast }}|_{B_2(f_{\hat{\tau }}(x),R_{\tau })})^{-1}
(U))\right)
& \geq K^{-2}|f_{\hat{\tau }}'(x)|^{-2}
l_{2}((f_{\tau _{\ast }}|_{B_2(f_{\hat{\tau }}(x),R_{\tau })})^{-1}(U)) \\
& =K^{-2}|f_{\tau }'(x)|^{-2}l_{2}(B_2(f_{\hat{\tau }}(x),R_{\tau })\cap f_{\tau _{\ast }}^{-1}(U))\\
& \geq K^{-2}|f_{\tau }'(x)|^{-2}l_{2}\(B_2(z,R_{\tau }/2)\cap f_{\tau _{\ast }}^{-1}(U)\)\\
& \geq (2K)^{-2}|f_{\hat{\tau }}(x)|^{-2}\ka_{a}R_{\tau }^{2}\\
& \geq 4K^{-2}1\ka_{a}r^{2}.
\endaligned
$$
Combining this with (\ref{1h2c.1}) we get that 
$$
\#\Fa\le (4\ka_{a})^{-1}\pi CK^{4}M^2.
$$
We are done.
\endpf

\brem\label{r120190418}
Because of Lemma~\ref{l120190325}, if
$$
\dist_\C\(f_{\hat\tau}(x),f_{\tau_*}^{-1}(J(G))\)\le \frac12R_\tau,
$$
then condition \eqref{ess0} holds.
\erem

\sp\fr The following lemma sheds some additional light on the nature of the Nice Open Set Condition.

\blem\label{l320190325}
If $f=(f_{1},\ldots ,f_{u})\in \Rat^u$ is a $u$--tuple map such that the rational semigroup $G=\langle f_{1},\ldots ,f_{u}\rangle$ is \TNR \ and satisfies the Nice Open Set Condition, then the semigroup $G$ is of finite type, i.e. the set $\Crit_*(\tf)$ is also finite. In particular, if in addition $G$ is \CF \ balanced, then it is a \FNR \ rational semigroup.
\elem

\bpf
%{\color{red} (Jason*1osc4)}
Fix $c\in\Crit(f)$ and $i\in\{1,\dots, u\}$ such that $f_i(c)\in J(G)$ and $f_i'(c)=0$. Since $G$ is a \TNR \ rational semigroup, there exists $R\in(0,1)$ such that 
$$
	B_2(f_\om(f_i(c)),2R)\cap \PCV(G)=\emptyset
$$
for all $\om\in\Sg_u^*$ such that $f_\om(f_i(c))\in J(G)$. In particular, there exists a unique holomorphic branch
$$
	f_{\om,i}^{-1}:B_2(f_\om(f_i(c)), 2R)\lra \C
$$
of $f_\om^{-1}$ mapping $f_\om(f_i(c))$ to $f_i(c)$. Fix now an integer $n\geq 1$ and consider the set 
$$
	\Sg_n(c,i)=\lt\{\om\in\{1,2,\dots, u \}^n: f_\om(f_i(c))\in J(G) \rt\}.
$$
Let 
$$
	\ga_n:=\min\left\{1, \min\{K^{-1}R|f_\om'(f_i(c))|^{-1}: \om\in\Sg_n(c,i) \} \right\}.
$$
Then for every $\om\in\Sg_n(c,i)$, we have that 
$$
	Kr_n |f_\om'(f_i(c))|\leq R.
$$
So, using Koebe's Distortion Theorem, we get that 
\begin{align*}
	&f_\om^{-1}\lt(U\cap B_2(f_\om(f_i(c)), K^{-1}r_n |f_\om'(f_i(c))|)\rt)
	\\
	&\qquad\qquad\sub
	f_\om^{-1}\lt(B_2(f_\om(f_i(c)), K^{-1}r_n |f_\om'(f_i(c))|)\rt)	
	\sub
	B_2(f_i(c), \ga_n),
\end{align*}
and, invoking also condition \eqref{osc3} of Definition~\ref{d1h2intro}, Lemma~\ref{l120190325}, and the fact that $f_\om(f_i(c))\in J(G)$, we also get that 
\begin{align*}
	&l_2\lt(f_\om^{-1}\lt(U\cap B_2(f_\om(f_i(c)), K^{-1}r_n |f_\om'(f_i(c))|) \rt) \rt)
	\\
	&\qquad\geq 
	K^{-2}|f_\om'(f_i(c))|^{-2} l_2\lt(U\cap B_2(f_\om(f_i(c)), K^{-1}r_n |f_\om'(f_i(c))|)  \rt)
	\\
	&\qquad\geq 
	K^{-2}\al|f_\om'(f_i(c))|^{-2} l_2\lt( B_2(f_\om(f_i(c)), K^{-1}r_n |f_\om'(f_i(c))|)  \rt)
	\\
	&\qquad=
	\pi K^{-4}\al r_n^2
	=
	\al K^{-4}l_2(B_2(f_i(c), \ga_n)).
\end{align*}
Since by condition \eqref{osc2} of Definition~\ref{d1h2intro}, all of the sets 
$$
	f_\om^{-1}\lt(U\cap B_2(f_\om(f_i(c)), K^{-1}r_n |f_\om'(f_i(c))|)  \rt), \quad \om\in\Sg_n(c,i)
$$
are mutually disjoint, we conclude that 
\begin{align*}
	\#\Sg_n(c,i)\leq K^4\al^{-1}.
\end{align*}
Hence, the set 
$$
	\Crit_*(\tf)=\bigcup_{(c,i)\in \CP(f)}\lt\{(\om,c)\in\Crit_*(\^f):\om_1=i \rt\}
$$
is finite as it is a finite union of finite sets. The proof is complete.
\epf

\section{Hausdorff Dimension of Invariant Measures $\mu_t$ \\ and \\ Multifractal Analysis of Lyapunov Exponents}\label{section:MA}

Throughout this section, we assume that $f=\lt(f_1,\dots, f_u\rt)\in \Rat^u$ and $G=\langle f_1,\dots, f_u \rangle$ is a finely non--recurrent rational semigroup satisfying the Nice Open Set Condition, which we continue to abbreviate as \NOSC-FNR. \index{NOSC-FNR} Let $G$ be a \FNR \ rational semigroup. We fix $t\in\De_G^*$. Our goal in this section is to analyze fractal properties of Gibbs/equilibrium measure $\mu_t$ (and $m_t$) and to provide a full account of multifractal analysis of Lyapunov exponents. 

We first recall some notions and results from geometric measure theory which will be used in this and the following sections. Let $\mu$ be a Borel probability measure on a metric space $X$.
The \index{Hausdorff dimension!of a measure:  $\HD(\mu)$}\index{ $\HD(\mu)$} \textbf{Hausdorff dimension of $\mu$}, denoted $\HD(\mu)$, is defined to be
\[
\HD(\mu):=\inf\{\HD(Y):\mu(Y)=1\}=\sup\{\HD(Y):\mu(X\setminus Y)=0\}.
\]
Furthermore, the  \textbf{lower and  upper pointwise dimensions of $\mu$}\index{pointwise dimension!$\underline{d}_{\mu}(x)$ (lower)}\index{$\underline{d}_{\mu}(x)$}\index{pointwise dimension! $\overline{d}_{\mu}(x)$ (upper)}\index{$\overline{d}_{\mu}(x)$} at a point $x\in X$ are respectively defined by
\[
\underline{d}_{\mu}(x)=\liminf_{r\to0}\frac{\log\mu(B(x,r))}{\log r}
\  \  \  \  \mbox{and} \  \  \  \
\overline{d}_{\mu}(x)=\limsup_{r\to0}\frac{\log\mu(B(x,r))}{\log r}.
\]
Needless to say that if $X=\C$ is the complex plane, then we always use the standard Euclidean metric. Just for the record, if we used the spherical metric on $\C$ we would always get the same values for Hausdorff dimensions of sets and measures. Likewise for box dimensions of bounded sets.

The proof of the following theorem can be found for example in \cite{PUbook}, \cite{KUbook}, and \cite{MundayRoyUI}.

\begin{thm}\label{sem3cor3.5}
If $\mu$ is a Borel probability measure on a metric space $X$, then
\[
\HD(\mu)=%\mathrm{ess}\sup\liminf_{r\to0}\frac{\log\mu(B(x,r))}{\log r}.
\mathrm{ess}\sup(\underline{d}_{\mu}).
\]
In particular, if there exists some $\theta\geq0$ such that
%$\liminf_{r\to0}\frac{\log\mu(B(x,r))}{\log r}=\theta$ for $\mu$--a.e. $x\in X$,
$\underline{d}_{\mu}(x)=\theta$ for $\mu$--a.e. $x\in X$,
then $\HD(\mu)=\theta$.
\end{thm}

We first prove the following preparatory result which is also interesting on its own. 
\begin{thm}\label{t1jsm3}
	If $G$ is a \NOSC-FNR \ rational semigroup and $t\in\De_G^*$, then
	\begin{align*}
		\HD(\mu_t\circ p_2^{-1})=\frac{\h_{\mu_t}(\tilde f)}{\chi_{\mu_t}}=t+\frac{\P(t)}{\chi_{\mu_t}}.
	\end{align*}

\end{thm}
\begin{proof}
	By Birkhoff's Ergodic Theorem there exists a measurable set $\Ga\sbt J(\tilde f)$ such that $\mu_t(\Ga)=1$ and 
	\begin{align}\label{1jsm3}
		\lim_{n\to\infty}\frac{1}{n}\log|(\tilde f^n)'(w,z)|=\chi_{\mu_t}:=\int_{J(\tilde f)}\log|\tilde f'|\, d\mu_t.
	\end{align}
	Fix $(\om,z)\in\Ga\bs\Sing(\tilde f)$. Put $\eta:=\eta(\om,z)>0$ coming from Proposition~\ref{pu6.1h23}. Let $(n_j)_{j=1}^\infty$ be the increasing sequence of integers also coming from Proposition~\ref{pu6.1h23}. Let 
	\begin{align*}
		\tilde f_{\om\rvert_{n_j},z}^{-n_j}:p_2^{-1}\lt(B_2\lt(f_{\om\rvert_{n_j}}(z),\eta\rt)\rt)\lra\Sg_{u}\times\C
	\end{align*}
	be the holomorphic inverse branch of $\tilde f^{n_j}$ resulting also from Proposition~\ref{pu6.1h23} which sends $\tilde f^{n_j}(\om,z)$ to $(\om,z)$. By Koebe's Distortion Theorem we have 
	\begin{align*}
		B_2\lt(z,\frac{1}{2}K \lt|f_{\om\rvert_{n_j}}'(z)\rt|^{-1}\eta \rt)
		\spt f_{\om\rvert_{n_j},z}^{-1}\lt(B_2\lt(f_{\om\rvert_{n_j}}(z),\eta/2\rt)\rt).
	\end{align*}
	So, by conformality of $m_t$ and Koebe's Distortion Theorem again, putting 
	$$
	r_j:=\frac{1}{2}K \lt|f_{\om\rvert_{n_j}}'(z)\rt|^{-1}\eta,
	$$ 
	we get that
	\begin{align*}
		m_t\circ p_2^{-1}(B_2(z,r_j))&\geq m_t\lt(\^f_{\om\rvert_{n_j}, z}^{-n_j}\lt(p_2^{-1}\lt( B_2\lt(f_{\om\rvert_{n_j}}(z), \eta/2\rt) \rt)\rt)\rt)\\
		& =\int_{p_2^{-1}\lt(B_2\lt( f_\om\rvert_{n_j}(z),\eta/2\rt) \rt)} e^{-\P(t)n_j}\lt|\lt(\^f_{\om\rvert_{n_j}}^{-n_j} \rt)'\rt|^t \, dm_t \\
&\geq K^{-t}e^{-\P(t)n_j}\lt|\lt(\^f_{\om\rvert_{n_j}}^{n_j} \rt)'(\om,z) \rt|^{-t}m_t\circ p_2^{-1}\lt(B_2\lt(f_{\om\rvert_{n_j}}(z),\eta/2\rt) \rt)\\
		&\geq K^{-t}M_t(\eta/2)e^{-\P(t)n_j}\lt|\lt(\^f_{\om\rvert_{n_j}}^{n_j} \rt)'(\om,z) \rt|^{-t} \\
		&=2(K^2q)^{-t}(\eta/2)M_t(\eta/2)e^{-\P(t)n_j}r_j^t,
	\end{align*}
	where 
	\begin{align*}
		M_t(r):=\inf\{m_t\circ p_2^{-1}\lt(B_2(x,r)\rt) : x\in J(G) \}
	\end{align*}
	is positive since $\supp(m_t\circ p_2^{-1})=J(G)$. Therefore, using also \eqref{1jsm3}, we get 
	\begin{align}
		\liminf_{r\to 0} \frac{\log m_t\circ p_2^{-1}(B_2(z,r))}{\log r}
&\leq \liminf_{j\to\infty}\frac{\log m_t\circ p_2^{-1}(B_2(z,r_j))}{\log r_j}\nonumber\\
		&\leq \liminf_{j\to\infty}\frac{\log \lt(2(K^2\eta)^{-t}M_t(\eta/2)\rt)-\P(t)n_j+t\log r_j}{\log r_j}\nonumber\\
		&=t-\P(t)\lim_{j\to\infty}\frac{n_j}{\log r_j}\nonumber\\
		&=t+\frac{\P(t)}{\chi_{\mu_t}}.\label{1jsm4}
	\end{align}
Hence, invoking Theorem~\ref{sem3cor3.5}, we see that
	\begin{align}
		\HD\lt(m_t\circ p_2^{-1}\rt)\leq t+\frac{\P(t)}{\chi_{\mu_t}}.\label{4jsm7}
	\end{align}
For the opposite inequality fix a point $\xi\in J(\^f)\bs\PCV(\^f)$. Fix any number 
	\begin{align}\label{120190504}
		H\in \lt(0,\frac{\dist(p_2(\xi),\PCV(G))}{2(2K+1)H}\rt),
	\end{align}
which is positive because of Lemma~\ref{l2mc3}. For every $(\om,z)\in J(\^f)$ let $(n_k(\om,z))_{k=1}^\infty$ be the increasing sequence of all visits of $(\om,z)$ to $B(\xi,R)$, i.e. 
	\begin{align}\label{220190504}
		\^f^{n_k(\om,z)}(\om,z)\in B(\xi,H).
	\end{align}
	Fix an arbitrary measurable set $Z\sbt J(\^f)$ with $\mu_t(Z)=1$. By Birkhoff's Ergodic Theorem there exists a measurable set $Y\sbt Z$ such that $\mu_t(Y)=1$,  
	\begin{align}
		\lim_{k\to\infty}\frac{n_{k+1}(\om,z)}{n_k(\om,z)}=1,
		\quad\text{and}\quad
		\lim_{n\to\infty}\frac{1}{n}\log|(\^f^n)'(\om,z)|=\chi_{\mu_t}\label{3jsm7}
	\end{align}
	for all $(\om,z)\in Y$. 
	Fix $\ep>0$. By Egorov's Theorem there exists an integer $N_\ep\geq 1$ and a measurable set $Y_\ep\sbt Y$ such that $m_t(Y_\ep)\geq 1-\ep$, 
	\begin{align}
		\frac{n_{k+1}(\tau,y)}{n_k(\tau,y)}\leq 1+\ep,
		\quad\text{and}\quad
		\chi_{\mu_t}-\ep\leq \frac{1}{n}\log\lt|(\^f^n)'(\tau,y)\rt|\leq \chi_{\mu_t}+\ep \label{1jsm7.1}	
	\end{align}
	for all $(\tau,y)\in Y_\ep$ and all $n_k(\tau,y),n\geq N_\ep$. Let 
	$$
		m_{t,\ep}:=m_t\rvert_{Y_\ep}.
	$$
Fix $y\in p_2(Y_\ep)$ and $r>0$. Then, for every $\om\in p_1\(Y_\ep\cap p_2^{-1}(B_2(y,r))\)$ fix an element $x_\om\in B_2(y,r)$ such that 
\begin{align}\label{420190504}
(\om,x_\om)\in Y_\ep\cap p_2^{-1}(B_2(y,r)).	
\end{align}
Since for every $j\geq 0$
	$$
		\big|f_{\om\rvert_{n_j}}'(x_\om)\big|\leq \|f'\|_\infty^{n_j},
	$$
	if 
	$$
		r\leq\frac{1}{8}H\|f'\|_\infty^{-N_\ep},
	$$
	then there exists a largest integer $k:=k(\om;r)\geq 1$ such that 
	\begin{align}
		\frac{1}{4}H\big|f_{\om\rvert_{n_k}}'(x_\om)\big|^{-1}\geq 2r, \label{2jsm7.1}
	\end{align}
	and 
	\begin{align}
		n_k\geq N_\ep, \label{3jsm7.1}
	\end{align}
	where we have abbreviated 
$$
n_k:=n_{k(\om;r)}(\om,x_\om).
$$
Immediately we have
	\begin{align}
\frac{1}{4}H\big|f_{\om\rvert_{n_k+1}}'(x_\om)\big|^{-1}<2r.\label{4jsm7.1}
	\end{align}
Let
	$$
		\hat\cF(y;r):=\lt\{\om\rvert_{n_k+1}:\om\in p_1\(Y_\ep\cap p_2^{-1}(B_2(y,r))\)\rt\},
	$$
and for every $\om\in p_1\(Y_\ep\cap p_2^{-1}(B_2(y,r))$ let	
\beq\label{320190504}
R_{\om\rvert_{n_k+1}}:=16Kr\lt|f_{\om\rvert_{n_k}}'(x_\om)\rt|\le 2KH,
\eeq
with the convention that if $\om\rvert_{n_{k(\om,r)}+1}=\tau\rvert_{n_{k(\tau,r}+1}$, then we arbitrarily choose either $x_\om$ or $x_\tau$ with the inequality part following from \eqref{2jsm7.1}. Let $\cF(y,r)$ be the maximal subfamily of $\hat \cF(y;r)$ consisting of mutually incomparable words. We aim to check that this family satisfies the conditions of Definition~\ref{d1h2a.1} with appropriate numbers $M$, $a$, and an open set $V$. Indeed, it satisfies \eqref{ess2} by its very definition. It also satisfies \eqref{ess3} if we set 
	\begin{align}
		V:=\bigcup_{\vr\in \cF(y,r)}[\vr]\spt p_1\(Y_\ep\cap p_2^{-1}(B_2(y,r))\). \label{2jsm7}
	\end{align}
Now we aim to prove conditions \eqref{ess0} and \eqref{ess1}. It follows from \eqref{120190504}, \eqref{220190504}, the definition of $x_\om$, and \eqref{320190504} that
$$
\begin{aligned}
\dist\(f_{\om\rvert_{n_k}}(x_\om),\PCV(G)\)
&\ge \dist_\C\(p_2(\xi),\PCV(G)\)-\big|f_{\om\rvert_{n_k}}(x_\om)-p_2(\xi)\big|
\\
&\ge 4H-H\\
&=3H 
\ge 16r\lt|f_{\om\rvert_{n_k}}'(x_\om)\rt|.
\end{aligned}
$$
Therefore, the unique holomorphic branch
$$
f_{\om\rvert_{n_k},x_\om}^{-1}:B_2\Big(f_{\om\rvert_{n_k}}(x_\om),16r\lt|f_{\om\rvert_{n_k}}'(x_\om)\rt|\Big)\lra\C
$$
of $f_{\om\rvert_{n_k}}^{-1}$ sending $f_{\om\rvert_{n_k}}(x_\om)$ to $x_\om$ is well defined. It therefore follows from Koebe's $\frac14$--Distortion Theorem and \eqref{420190504} that
$$
f_{\om\rvert_{n_k},x_\om}^{-1}:B_2\Big(f_{\om\rvert_{n_k}}(x_\om),4r\lt|f_{\om\rvert_{n_k}}'(x_\om)\rt|\Big)
\spt B_2(x_\om,r)\ni y.
$$
Hence, 
\beq\label{520190504}
f_{\om\rvert_{n_k}}(y)
\in B_2\Big(f_{\om\rvert_{n_k}}(x_\om),4r\lt|f_{\om\rvert_{n_k}}'(x_\om)\rt|\Big)
\sbt B_2\lt(f_{\om\rvert_{n_k}}(x_\om),\frac{1}{4K}R_{\om\rvert_{n_k+1}}\rt).
\eeq
Since also
$$
f_{\om\rvert_{n_k}}(x_\om)
\in f_{\om_{n_{k+1}}}^{-1}\(f_{\om\rvert_{n_{k+1}}}(x_\om)\)
\sbt f_{\om_{n_{k+1}}}^{-1}(J(G)), 
$$
using Remark~\ref{r120190418}, we conclude that condition \eqref{ess0} holds. 

Now, passing to proving condition \eqref{ess1}, note that because of \eqref{520190504}, \eqref{220190504}, \eqref{320190504}, and the definition of $x_\om$, for every $z\in B_2\lt(f_{\om\rvert_{n_k}}(y),2R_{\om\rvert_{n_k+1}}\rt)$, we get that
$$
\begin{aligned}
|z-p_2(\xi)|
&< 2R_{\om\rvert_{n_k+1}}+
  \big|f_{\om\rvert_{n_k}}(y)-f_{\om\rvert_{n_k}}(x_\om)\big|
  +\big|f_{\om\rvert_{n_k}}(x_\om)-p_2(\xi)\big|
  +\frac{1}{4K}R_{\om\rvert_{n_k+1}} \\
&\le 2R_{\om\rvert_{n_k+1}}+\frac{1}{4K}R_{\om\rvert_{n_k+1}}+H \\
&\le 4KH+H+H \\
&=2(2K+1)H.
\end{aligned}
$$
Hence, using also \eqref{120190504}, we get that
$$
\begin{aligned}
\dist_\C(z,\PCV(G))
&\ge \dist_\C\(p_2(\xi),\PCV(G)\)-|z-p_2(\xi)|\\
&\ge \dist_\C\(p_2(\xi),\PCV(G)\)-2(2K+1)H \\
&>0.
\end{aligned}
$$
Therefore, the unique holomorphic branch
$$
f_{\om\rvert_{n_k},y}^{-1}:B_2\(f_{\om\rvert_{n_k}}(y),2R_{\om\rvert_{n_k+1}}\)
\lra\C
$$
of $f_{\om\rvert_{n_k}}^{-1}$ sending $f_{\om\rvert_{n_k}}(y)$ to $y$ is well defined. Also, by invoking \eqref{320190504}, we obtain
$$
(16K)^{-1}R_{\om\rvert_{n_k+1}}
\le
\lt|f_{\om\rvert_{n_k}}'(x_\om)\rt|4
\le \frac14R_{\om\rvert_{n_k+1}}
\le \frac12 KH
\le KH,
$$
which yields condition \eqref{ess1}. In conclusion, the conditions \eqref{ess0}--\eqref{ess3} of Definition~\ref{d1h2a.1} are satisfied with with $M=16K$ and $a=KH$.
For every $\vr\in\cF(y,r)$ fix an element $\^\vr\in\Sg_{u}$ such that $x_\vr=x_{\tilde\rho}$ and   
$$
\^\vr\rvert_{n_{k(\^\vr,x_\vr,r)}(\^\vr,x_\vr)+1}=\vr.
%\quad\text{and}\quad (\^\vr,x_\vr)\in Y_\ep\cap p_2^{-1}(B_2(y,r)).	
$$	
We claim that 
	\begin{align}\label{3jsm8}
		\bigcup_{\vr\in\cF(y,r)}[\vr]\times f_{\^\vr,x_\vr}^{-1}(B_2(f_{\^\vr}(x_\vr),R_\vr))\spt Y_\ep\cap p_2^{-1}(B_2(y,r)).
	\end{align}	
	Indeed, fix $(\b,x)\in Y_\ep\cap p_2^{-1}(B_2(y,r))$. It then follows from  \eqref{2jsm7} that there exists $\vr\in\cF(y,r)$ such that 
	\begin{align}
		\b\in[\vr].\label{4jsm8}
	\end{align} 
	Now, $\vr=\^\vr\rvert_{n_{k(\^\vr,x_\vr,r)}(\^\vr,x_\vr)+1}$. It follows from the Koebe'e $\frac14$--Distortion Theorem that 
	\begin{align*}
	 	f_{\^\vr,x_\vr}^{-1}(B_2(f_{\^\vr}(x_\vr),R_\vr))
	 	&=  f_{\^\vr\rvert_{n_{k(\^\vr,x_\vr,r)}(\^\vr,x_\vr)},x_\vr}^{-1}\lt(B_2\lt(f_{\^\vr\rvert_{n_{k(\^\vr,x_\vr,r)}(\^\vr,x_\vr)}}(x_\vr),16Kr\lt|f_{\^\vr\rvert_{\vr(\^\vr,x_\vr,r)}}'(x_\vr)\rt|^{-1}\rt)\rt)\\
	 	&\spt B_2(x_\vr,4r) \\
	 	&\spt B_2(x,r)\ni x.
	\end{align*}
	Along with \eqref{4jsm8}, this implies that
	$$
		(\b,x)\in\bigcup_{\vr\in \cF(y,r)}[\vr]\times f_{\^\vr,x_\vr}^{-1}(B_2(f_{\^\vr}(x_\vr),R_\vr)),
	$$
	and formula \eqref{3jsm8} is proved.
	
	It now follows from \eqref{3jsm8}, conformality of the measure $m_t$, Koebe's Distortion Theorem again, and \eqref{4jsm7.1}, that
	\begin{align}
		m_{t,\ep}&(p_s^{-1}(B_2(y,r)))=m_t(Y_\ep\cap p_2^{-1}(B_2(y,r)))
		\label{1jsm8}\le \\
		&\leq m_t\lt(\bigcup_{\vr\in \cF(y,r)}[\vr]\times f_{\^\vr,x_\vr}^{-1}(B_2(f_{\^\vr}(x_\vr),R_\vr))\rt)
		\nonumber\\
		&=\sum_{\vr\in\cF(y,r)}m_t\lt([\vr]\times f_{\^\vr,x_\vr}^{-1}(B_2(f_{\^\vr}(x_\vr),R_\vr))\rt)
		\nonumber\\
		&=\sum_{\vr\in\cF(y,r)} m_t(\Sg_u\times (B_2(f_{\^\vr}(x_\vr),R_\vr)))\int e^{-\P(t)(|\vr|-1)}|(f_{\^\vr,x_\vr}^{-1})'|^t dm_t 
		\nonumber\\
		&\leq K^t\sum_{\vr\in\cF(y,r)} m_t\(\Sg_u\times (B_2(f_{\^\vr}(x_\vr),R_\vr))\)\big|f'_{\^\vr\rvert_{n_{k(\^\vr,x_\vr,r)}(\^\vr,x_\vr)}}(x_\rho)\big|^{-t}\exp\lt(-\P(t)n_{k(\^\vr,x_\vr,r)}(\^\vr,x_\vr)\rt)
		\nonumber\\
		&\leq (8KH^{-1})^tr^t\sum_{\vr\in\cF(y,r)}\exp\lt(-\P(t)n_{k(\^\vr,x_\vr,r)}(\^\vr,x_\vr)\rt). \nonumber
	\end{align} 
	Now, denoting $\^n_k=n_{k(\^\vr,x_\vr,r)}(\^\vr,x_\vr)$, it follows from \eqref{1jsm7.1}--\eqref{4jsm7.1} that
	\begin{align}
		\^n_k\leq \frac{\log\Big|f_{\^\vr\rvert_{\^n_k}}'(x_\vr)\Big|}{\chi_{\mu_t}-\ep}\leq \frac{\log(R/8)-\log r}{\chi_{\mu_t}-\ep} \label{1jsm9}
	\end{align}
	and 
	\begin{align}
		\^n_k\geq(1+\ep)^{-1}\^n_{k+1}\geq (1+\ep)^{-1} \frac{\log\big|f_{\^\vr\rvert_{\^n_{k+1}}}'(x_\vr)\big|}{\chi_{\mu_t}+\ep}\geq (1+\ep)^{-1}\frac{\log(R/8)-\log r}{\chi_{\mu_t}+\ep}. \label{2jsm9}
	\end{align}
Making use of Proposition~\ref{p1h2a.1} \eqref{p1h2a.1 item c} and inserting \eqref{1jsm9} or \eqref{2jsm9} into \eqref{1jsm8}, respectively if $\P(t)\leq 0$ or if $\P(t)\geq 0$, we get that
	\beq
	\liminf_{r\to 0}\frac{\log(m_{t,\ep}(p_s^{-1}(B_2(y,r))))}{\log r}\geq 
	\begin{cases}
		t+\frac{\P(t)}{\chi_{\mu_t}-\ep} \,  &{\rm if}~~ \P(t)\leq 0\\
		t+\frac{\P(t)}{\chi_{\mu_t}+\ep}\frac{1}{1+\ep} \, &{\rm if}~~ \P(t)\geq 0
	\end{cases}.
	\eeq
	Hence, invoking Theorem~\ref{sem3cor3.5}, we get that
	\beq
	\HD(Z)\geq \HD(Y)\geq \HD(m_{t,\ep})\geq 
	\begin{cases}
		t+\frac{\P(t)}{\chi_{\mu_t}-\ep} \,  &{\rm if}~~ \P(t)\leq 0\\ 
		t+\frac{\P(t)}{\chi_{\mu_t}+\ep}\frac{1}{1+\ep} \, &{\rm if}~~ \P(t)\geq 0
	\end{cases}.
	\eeq
	In either case, letting $\ep\to 0$, we get that 
	$$
		\HD(Z)\geq t+\frac{\P(t)}{\chi_{\mu_t}}.
	$$
	Therefore, 
	$$
		\HD(\mu_t\circ p_2^{-1})=t+\frac{\P(t)}{\chi_{\mu_t}}.
	$$
	This proves the ``left--hand" side of Theorem~\ref{t1jsm3}. By making use of Theorem~\ref{tvp3}, i.e. the Variational Principle, we get
	$$
		t+\frac{\P(t)}{\chi_{\mu_t}}=t+\frac{\h_{\mu_t}-t\chi_{\mu_t}}{\chi_{\mu_t}}=\frac{\h_{\mu_t}}{\chi_{\mu_t}},
	$$
and Theorem~\ref{t1jsm3} is entirely proved.
\end{proof}
\begin{comment}
{\bf(Jason: Please, insert here the page JSM${_-}10.1$; dollar signs should be removed - these are only to make LaTeX ``happy'')}

content...
\end{comment}
As an immediate consequence of this theorem and the Variational Principle (Theorem~\ref{t1vp3}), we get the following.

\begin{cor}\label{c1jsm10.1}
	If $G$ is a \NOSC-FNR \ rational semigroup, then 
	$$
		\HD(\mu_h\circ p_2^{-1})=h=\HD(J(G)) \text{ and } \HD(\mu_t\circ p_2^{-1})<h
	$$
	if $t\neq h$.
\end{cor}
\begin{proof}
	The first part of this corollary immediately follows from Theorem~\ref{t1jsm3} by substituting $h$ for $t$ and remembering that $\P(h)=0$. For the second part, suppose that 
	$$
		\HD(\mu_t\circ p_2^{-1})=h
	$$	
	for some $t\in\De_G^*$. It then follows from Theorem~\ref{t1jsm3} that 
	$$
		\h_{\mu_t}(\^f)-\h\chi_{\mu_t}=0=\P(h).
	$$
	So, by virtue of Theorem~\ref{t1vp3}, applied with $t=h$, we conclude that $\mu_t=\mu_h$. The proof is complete.
\end{proof}
Now, passing to actual multifractal analysis, fix $t\geq 0$. Our first result is the following.

\begin{lem}\label{l1jsm10}
Let $G$ be a \NOSC-FNR \ rational semigroup. If $t\in\De_G^*$, then	for every $q\in[0,1]$ there exists a unique $T_t(q)\geq 0$ such that 
	$$
		\P(T_t(q)+qt)=q\P(t).
	$$
\end{lem}
\begin{proof}
	Because of Proposition~\ref{p1_2016_06_21}, also $qt\in\De_G^*$. So, by Theorem~\ref{t1vp3}, i.e. the Variational Principle, we have 
	$$
	\P\(qt)\geq \h_{\mu_t}-qt\chi_{\mu_t}
	\ge q\lt(h_{\mu_t}-t\chi_{\mu_t} \rt)=q\P(t). 
	$$
	So, employing Proposition~\ref{l1h35} \eqref{l1h35 item b} and \eqref{l1h35 item d}, the lemma immediately follows from the Intermediate Value Theorem.
\end{proof}
We denote 
\begin{equation}\label{520190307}
	\mu_{t,q}:=\mu_{T_t(q)+qt}
	\quad \text{ and }\quad 
	m_{t,q}:=m_{T_t(q)+qt}.
\end{equation}\index{$m_{t,q}$}\index{$\mu_{t,q}$}
We further denote
\begin{align}
	\a_t(q):=t+\frac{\P(t)}{\chi_{\mu_{t,q}}}.\label{1jsm10}
\end{align}\index{$\a_t(q)$}
Now, for every point $(\om,z)\in J(\^f)$, we denote 
$$
	\ov{\chi}(\om,z)=\limsup_{n\to\infty}\frac{1}{n}\log|(\^f^n)'(\om,z)| 
	\quad\text{ and }\quad 
	\un{\chi}(\om,z)=\liminf_{n\to\infty}\frac{1}{n}\log|(\^f^n)'(\om,z)|
$$
and call them respectively the \textbf{upper and lower Lyapunov exponents  at the point $(\om,z)$}.\index{Lyapunov exponents! $\ov{\chi}(\om,z)$ (upper)}\index{$\ov{\chi}(\om,z)$}\index{Lyapunov exponents! $\un{\chi}(\om,z)$ (lower)}\index{$\un{\chi}(\om,z)$} If $\un\chi(\om,z)=\ov\chi(\om,z)$, we denote the common value by $\chi(\om,z)$\index{$\chi(\om,z)$}\index{Lyapunov exponents!$\chi(\om,z)$} and call it the \textbf{Lyapunov exponent} at $(\om,z)$. Given $\chi\geq 0$, we denote 
$$
K(\chi):=\lt\{(\om,z)\in J(\^f): \un\chi(\om,z)=\ov\chi(\om,z)=\chi \rt\}.
$$\index{$K(\chi)$}
%$$
%	\a(\chi):=t+\frac{\P(t)}{\chi}
%	\quad\text{and}\quad
%	\chi(\a):=\frac{\P(t)}{\a-t}. 
%$$
We now shall prove the key statement of the main result of this section.
 
\begin{lem}\label{l1jsm11}
Let $G$ be a \NOSC-FNR \ rational semigroup.	If $t\in\De_G^*$, then for every $q\in [0,1]$, 
	$$
		\HD(p_2(K(\chi_{\mu_{t,q}})))=T_t(q)+q\a_t(q).
	$$
\end{lem}
\begin{proof}
	By Birkhoff's Ergodic Theorem there exists a measurable set $X\sbt J(\^f)$ such that $\mu_{t,q}(X)=1$ and 
	$$
		\chi(\om,z)=\int_{J(\^f)}\log|\^f'| d\mu_{t,q}=\chi_{\mu_{t,q}}	
	$$
	for all $(\om,z)\in X$. Hence, $X\sbt K(\chi_{\mu_{t,q}})$, and applying Theorem~\ref{t1jsm3} and Lemma~\ref{l1jsm10} along with \eqref{1jsm10}, we get
	\begin{align}
		\HD(p_2(K(\chi_{\mu_{t,q}})))
		&\geq \HD(p_2(X))
		\geq \HD(\mu_{t,q}\circ p_2^{-1})
		\nonumber\\
		&=T_t(q)+qt+\frac{\P(T_t(q)+qt)}{\chi_{\mu_{t,q}}}
		=T_t(q)+qt+\frac{q\P(t)}{\chi_{\mu_{t,q}}} \\
		&=T_t(q)+q\lt(t+\frac{\P(t)}{\chi_{\mu_{t,q}}}\rt)
		\nonumber\\
		&=T_t(q)+q\a_t(q).\label{1jsm11}
	\end{align}
	The proof of the opposite inequality is almost the same as the proof of the ``$\leq$" inequality of Theorem~\ref{t1jsm3}. However, since it does not formally follow from this theorem, we therefore present it here in full. 

	We start by fixing $(\om,z)\in K(\chi_{\mu_{t,q}})\bs\Sing(\^f)$. Put $\eta:=\eta(\om,z)>0$ coming from Proposition~\ref{pu6.1h23}. Let $(n_j)_{j=1}^\infty$ be an increasing sequence of integers also coming from this Proposition. Let 
	$$
		\^f_{\om\rvert_{n_j},z}^{-n_j}:p_2^{-1}(B_2(f_{\om\rvert_{n_j}}(z),\eta))\lra\Sg_u\times\C
	$$ 
	be the holomorphic inverse branch of $\^f^{n_j}$ resulting also from Proposition~\ref{pu6.1h23} which sends $\^f^{n_j}(\om,z)$ to $(\om,z)$. By Koebe's Distortion Theorem
	$$
		B_2\lt(z,\frac{1}{2}K\lt|f_{\om\rvert_{n_j}}'(z)\rt|^{-1}\eta\rt)\spt f_{\om\rvert_{n_j},z}^{-1}\lt(B_2\lt(f_{\om\rvert_{n_j}}(z),\frac{\eta}{2}\rt)\rt).
	$$
	So, by conformality of $m_t$ and Koebe's Distortion Theorem again, we get with 
	$$
		r_j=\frac12K\lt|f_{\om\rvert_{n_j}}'(z)\rt|^{-1}\eta
	$$
	that
	\begin{align*}
		m_{t,q}\circ p_2^{-1}&(B_2(z,r_j)) \ge \\
		&\geq m_t\lt(\^f_{\om\rvert_{n_j},z}^{-n_j}\lt(p_2^{-1}\lt(B_2\lt(f_{\om\rvert_{n_j}}(z),\frac{\eta}{2}\rt)\rt)\rt)\rt)\\
		&=\int_{p_2^{-1}\lt(B_2\lt( f_\om\rvert_{n_j}(z),\eta\rt) \rt)} e^{-qP(t)n_j}\lt|\lt(\^f_{\om\rvert_{n_j},z}^{-n_j}\rt)'\rt|^{-(T_t(q)+qt)} dm_{t,q} \\
		&\geq K^{-(T_t(q)+qt)}e^{-qP(t)n_j}\lt|\lt(\^f_{\om\rvert_{n_j}}^{n_j} \rt)'(\om,z) \rt|^{-(T_t(q)+qt)}m_{t,q}\circ p_2^{-1}\lt(B_2\lt(f_{\om\rvert_{n_j}}(z),\eta/2\rt) \rt)\\
		&\geq K^{-(T_t(q)+qt)}M_{t,q}(\eta/2)e^{-qP(t)n_j}\lt|\lt(\^f_{\om\rvert_{n_j}}^{n_j} \rt)'(\om,z) \rt|^{-(T_t(q)+qt)} \\
		&=2(K^2q)^{-(T_t(q)+qt)}M_{t,q}(\eta/2)e^{-q\P(t)n_j}r_j^{T_t(q)+qt},
	\end{align*}
	where 
	$$
		M_{t,q}(r):=\inf\lt\{m_{t,q}\circ p_2^{-1}\lt(B_2(x,r)\rt):x\in J(G)\rt\}
	$$
	is positive since $\supp(m_{t,q}\circ p_2^{-1})=J(G)$. Therefore, using also the fact that $(\om,z)\in K(\chi_{\mu_{t,q}})$, we get 
	\begin{align}
		\liminf_{r\to 0}& \frac{\log m_{t,q}\circ p_2^{-1}(B_2(z,r))}{\log r}
		\leq \\
&\le \liminf_{j\to\infty}\frac{\log m_{t,q}\circ p_2^{-1}(B_2(z,r_j))}{\log r_j}\nonumber\\
		&\leq \liminf_{j\to\infty}\frac{\log \lt(2(K^2\eta)^{-(T_t(q)+qt)}M_{t,q}(\eta/2)\rt)-q\P(t)n_j+(T_t(q)+qt)\log r_j}{\log r_j}\nonumber\\
		&=T_t(q)+qt-q\P(t)\lim_{j\to\infty}\frac{n_j}{\log r_j}=T_t(q)+qt+\frac{q\P(t)}{\chi_{\mu_{t,q}}}\label{1jsm4B}\\
		&=T_t(q)+q\lt(t+\frac{\P(t)}{\chi_{\mu_{t,q}}}\rt)\\
		&=T_t(q)+q\a_t(q).\nonumber
	\end{align}
	Therefore, 
	$$
		\HD(p_2(K(\chi_{\mu_{t,q}})))\leq T_t(q)+q\a_t(q).
	$$
Along with \eqref{1jsm11}, this completes the proof of Lemma~\ref{l1jsm11}.
\end{proof}

Let us prove the following auxiliary result.

\begin{lem}\label{l2jsm13}
Let $G$ be a \NOSC-FNR \ rational semigroup.	If $t\in\De_G^*$ and $q\in[0,1]$, then $T_t(q)+qt\in\De_G^*$.
\end{lem}
\begin{proof}
	Consider two cases. First assume that $t\leq h_f$. Then $\P(t)\geq 0$, whence $q\P(t)\geq 0$. Since also $\P(T_t(q)+qt)=q\P(t)$, and, by Proposition~\ref{l1h35} \eqref{l1h35 item b}, the function $s\longmapsto \P(s)$ is strictly decreasing, we have that $T_t(q)+qt\leq h_f$. Thus $T_t(q)+qt\in\De_G^*$ by Proposition~\ref{p1_2016_06_21}. Assume thus that $t\geq h_f$. Then $\P(t)\leq 0$, and hence $q\P(t)\geq \P(t)$. So, applying Proposition~\ref{l1h35} \eqref{l1h35 item b} again, we see that $T_t(q)+qt\leq t$. Thus, $T_t(q)+qt\in\De_G^*$ by Proposition~\ref{p1_2016_06_21} again.
\end{proof}

\fr Now, we will show that the function $T_t(q)$ depends in a real--analytic way on both $q$ and $t$.

\begin{lem}\label{l1jsm13}
If $G$ is a \NOSC-FNR \ rational semigroup, then the two functions 
	$$
		\De_G^*\times[0,1]\ni(t,q)\longmapsto T_t(q) \in\R
		\quad\text{ and }\quad
		\De_G^*\times[0,1]\ni(t,q)\longmapsto\chi_{\mu_{t,q}}\in\R
	$$
	are real--analytic with respect to both variables $t$ and $q$. 
\end{lem}
\begin{proof}
	Consider the following function of three variables
	$$
		[0,+\infty)\times\De_G^*\times[0,1]\ni(s,t,q)\longmapsto \P(s+qt)\in\R.
	$$
	Because of Lemma~\ref{l2jsm13} the derivative $\frac{\partial P}{\partial s}\rvert_{(T_t(q),t,q)}$ is well defined and by Theorem~\ref{t120180604},
	\begin{align}
		\frac{\partial \P}{\partial s}\Big\rvert_{T_t(q),t,q}=-\chi_{\mu_{t,q}}. \label{1jsm13}
	\end{align}
Since $\chi_{\mu_{t,q}}>0$, it therefore follows from Theorem~\ref{t3.11} and the Implicit Function Theorem that both functions 
	$$
		\De_G^*\ni t\longmapsto T_t(q)
		\quad\text{and}\quad
		[0,1]\ni q\longmapsto T_u(q), \quad (u\in\De_G^*),
	$$
	are real--analytic. Of course the function $(t,q)\longmapsto\chi_{\mu_{t,q}}$ is real--analytic immediately from Theorem~\ref{t3.11} and formula \eqref{1jsm13}.
\end{proof}
Keep $t\in\De_G^*$ fixed. Let 
$$
D_t(\^f):=\lt\{\chi_{\mu_{t,q}}:q\in[0,1]\rt\}.
$$\index{$D_t(\^f)$}
We call the parameter $t$ \textbf{exceptional}\index{exceptional!parameter} if $D_t(\^f)$ is a singleton. Otherwise, we call it \textbf{non--exceptional}. Since the function $[0,1]\ni q\longmapsto\chi_{\mu_{t,q}}\in(0,+\infty)$ is real--analytic, $T_t(0)=h_f$ and $T_t(1)=0$, it follows that $D_t(\^f)$ is a closed interval containing $\chi_{\mu_{h_f}}$ and $\chi_{\mu_t}$. It is clear from Lemma~\ref{l1jsm11} that if $t$ is non--exceptional, then the multifractal analysis is non--trivial. We shall explore this issue now in greater detail. We shall prove the following. 
\begin{prop} \label{p1jsm14}
If $G$ is a \NOSC-FNR \ rational semigroup, then the following are equivalent.
	\begin{enumerate}
		\item[\mylabel{a}{p1jsm14 item a}] $\De_G^*\bs\{h_f\}$ contains at least one exceptional parameter.
		
		\,
		\item[\mylabel{b}{p1jsm14 item b}] All elements of $\De_G^*\bs\{h_f\}$ are exceptional.		
		
		\,
		\item[\mylabel{c}{p1jsm14 item c}] The set $D(\^f):=\{\chi_{\mu_t}:t\in\De_G^*\}$ is a singleton.
				
		\,
		\item[\mylabel{d}{p1jsm14 item d}] There exist $t,s\in\De_G^*$ with $s\neq t$ such that $\chi_{\mu_s}=\chi_{\mu_t}$. 
				
		\,
		\item[\mylabel{e}{p1jsm14 item e}] The function $[0,+\infty)\ni t\longmapsto \P(t)\in\R$ is affine, i.e. there are $\a,\b\in\R$ such that $\P(t)=\a t+\b$. 
		
		\,
		\item[\mylabel{f}{p1jsm14 item f}] The set $\{\mu_t:t\in\De_G^*\}$ is a singleton.
		
		\,
		\item[\mylabel{g}{p1jsm14 item g}] There exist $t,s\in\De_G^*$ with $s\neq t$ such that $\mu_s=\mu_t$. 
				
	\end{enumerate}
\end{prop} 
\begin{proof}
	Of course \eqref{p1jsm14 item f}$\implies$\eqref{p1jsm14 item c}$\implies$\eqref{p1jsm14 item b} and \eqref{p1jsm14 item b}$\implies$\eqref{p1jsm14 item a}. As for every $t\in\De_G^*\bs\{h_f\}$,
	$$
		\Ga_t:=\{T_t(q)+qt: q\in[0,1]\}
	$$
	contains a (non--degenerate) interval between $t$ and $h_f$, so we have that \eqref{p1jsm14 item a}$\implies$\eqref{p1jsm14 item d}. In order to prove \eqref{p1jsm14 item d}$\implies$\eqref{p1jsm14 item e} assume without loss of generality that $s<t$. Then by \eqref{2_2018_01_12} and \eqref{1_2018_01_12}, we have that  
	$$
		\chi_{\mu_s}\geq \chi_{\mu_u}\geq \chi_{\mu_t}=\chi_{\mu_s}
	$$
	for every $u\in[s,t]$. 
	So, $\chi_{\mu_u}=\chi_{\mu_s}$. Applying \eqref{1_2018_01_12} again, it thus follows from Theorem~\ref{t3.11} that the function $\De_G^*\ni t\longmapsto \P'(t)\in\R$ is constant. This precisely means that the function $\De_G^*\ni t\longmapsto \P\(t)\in\R$ is affine and \eqref{p1jsm14 item e} is proved. Now, if \eqref{p1jsm14 item e} holds, then it follows from \eqref{1_2018_01_12} again that $D(\^f)$ is a singleton, meaning that \eqref{p1jsm14 item c} holds. We have thus proved that conditions \eqref{p1jsm14 item a}--\eqref{p1jsm14 item e} are mutually equivalent and \eqref{p1jsm14 item f} entails all of them. 
	
	Now, we shall prove the next implication, namely that conditions \eqref{p1jsm14 item a}-\eqref{p1jsm14 item e} yield \eqref{p1jsm14 item f}. Because of Theorem~\ref{t1vp3} and condition \eqref{p1jsm14 item e}, we have for every $t\in\De_G^*$ that 
	$$
		\h_{\mu_t}-t\chi=\a t+\b,
	$$ 
	where $\chi$ is the only element of $D(\^f)$, which is a singleton by \eqref{p1jsm14 item c}. Applying Theorem~\ref{t1jsm3}, we then get for every $t\in\De_G^*$ that 
	\begin{align}\label{1jsm15.1}
		\HD(\mu_t\circ p_2^{-1})=\frac{\h_{\mu_t}}{\chi}=t+\frac{\a}{\chi}t+\frac{\b}{\chi}=\lt(1+\frac{\a}{\chi}\rt)t+\frac{\b}{\chi}.
	\end{align}
	But, taking in \eqref{p1jsm14 item e}, $t=0$, we get by Theorem~\ref{t1vp3} that 
	\begin{align}\label{2jsm15.1}
		\b=\P(0)=\h_{\top}(\^f)>0,
	\end{align} 
	where $\h_\top(\tf)$ denotes the topological entropy, and taking $t=h$, we get 
	$$
		\a=-\frac{\b}{h}.
	$$
	Substituting this to \eqref{1jsm15.1}, we get 
	\begin{align}\label{3jsm15.1}
		\HD(\mu_t\circ p_2^{-1})=\lt(1-\frac{\b}{\chi h}\rt)t+\frac{\b}{\chi}.
	\end{align}
	Substituting here $t=h$ or applying Corollary~\ref{c1jsm10.1}, we get that 
	\begin{align}\label{4jsm15.1}
		\HD(\mu_h\circ p_2^{-1})=h.
	\end{align}
	Then from \eqref{1jsm15.1}, \eqref{2jsm15.1}, and from the ordinary Variational Principle for topological entropy, we obtain 
	\begin{align}\label{5jsm15.1}
		\chi h=\h_{\mu_h}\leq \h_{\top}(\^f)=\b.
	\end{align}
	Now, if $\chi h<\b$, then it would follow from \eqref{3jsm15.1} that 
	$$
		\frac{d}{dt}\lt(\HD(\mu_t\circ p_2^{-1})\rt)=1-\frac{\b}{\chi h}<0.
	$$
	This would imply that for all $t\in [0,h)$,
	$$
		\HD(\mu_t\circ p_2^{-1})>\HD(\mu_h\circ p_2^{-1})=h=\HD(J(G)).
	$$
	Along with \eqref{5jsm15.1}, this contradiction implies that 
	$$
		\chi h=\b.
	$$
	Substituting this to \eqref{3jsm15.1}, we get that 
	$$
		\HD(\mu_t\circ p_2^{-1})=\frac{\b}{\chi}=h
	$$
	for every $t\in\De_G^*$. So, by applying Corollary~\ref{c1jsm10.1}, we get that $\mu_t=\mu_h$ for every $t\in\De_G^*$, meaning that \eqref{p1jsm14 item f} holds. Thus all conditions \eqref{p1jsm14 item a}--\eqref{p1jsm14 item f} are equivalent. Of course \eqref{p1jsm14 item f} implies \eqref{p1jsm14 item g} and \eqref{p1jsm14 item g} implies \eqref{p1jsm14 item d}. The proof of Proposition~\ref{p1jsm14} is complete.
\end{proof}

If either one of the conditions \eqref{p1jsm14 item a}--\eqref{p1jsm14 item g} from Proposition~\ref{p1jsm14} holds then we call the semigroup $G$ and the skew product map $\^f:\Sg_u\times\hat\C\lra\Sg_u\times\hat\C$ \textbf{exceptional}\index{exceptional!semigroup}. Otherwise, we call it \textbf{non--exceptional}. As an immediate consequence of Lemma~\ref{l1jsm11}, Lemma~\ref{l1jsm13}, Theorem~\ref{t3.11}, Proposition~\ref{p1jsm14}, and the fact that 
\beq
	\frac{\partial}{\partial q}(T_t(q)+qt)=-\frac{\P(t)}{\chi_{\mu_{t,q}}} 
\begin{cases}
	<0 \,  &{\rm if}~~ t<h_f\\
	>0 \, &{\rm if}~~ t>h_f ,
\end{cases}
\eeq
for all $q\in[0,1]$, we get the following main result of this section. 

\begin{thm}\label{t1jsm15}
If $G$ is a non--exceptional \NOSC-FNR \ rational semigroup, then for every $t\in\De_G^*\bs\{h_f\}$, the set $D_t(\^f)$ is a non--degenerate interval with endpoints $\chi_{\mu_{h_f}}$ and $\chi_{\mu_t}$, and the function
	$$
		D_t(\^f)\ni \chi\longmapsto\HD(p_2(K(\chi)))\in[0,2]
	$$
	is real--analytic.
\end{thm}

\begin{comment}
 {\bf (Jason: Remove the rest of this section (Multifractal Analysis) and instead continue unt the page 26 (the last one) from my OneNotes)} 
content...
\end{comment}

The class of non--exceptional semigroups is huge. This will be fully evidenced from Theorem~\ref{t2esg2} and Proposition~\ref{p1jsm22}. The first step in this direction is the following proposition which is also a complement of Proposition~\ref{p1jsm14}. Recall that the (co)homology of two functions was defined in formula \eqref{120190910}. Denote by $\^N:\cD_\cU^\infty\to\N$\index{$\^N$} the function defined by 
$$
	\^N(\tau)=\|\tau_1\|.
$$
\begin{prop}\label{p1jsm16}
If $G$ is a non--exceptional \NOSC-FNR \ rational semigroup, then the following are equivalent.
	\begin{enumerate}
		\item[\mylabel{a}{p1jsm16 item a}] $G$ is exceptional.
		
		\,
		\item[\mylabel{b}{p1jsm16 item b}] The set $\{\^\mu_t:t\in\De_G \}$ is a singleton. 
		
		\,
		\item[\mylabel{c}{p1jsm16 item c}] There exist $s,t\in \De_G^*$ such that $s\neq t$ and $\^\mu_s=\^\mu_t$.
		
		\,
		\item[\mylabel{d}{p1jsm16 item d}] All functions $\zt_{t,\P(t)}$, $t\in\De_G^*$, are cohomologous to each other in the class of H\"older continuous real--valued bounded functions defined on $\cD_\cU^\infty$.
		
		\,
		\item[\mylabel{e}{p1jsm16 item e}]All functions $\zt_{t,\P(t)}$, $t\in\De_G^*$, are cohomologous to each other in the class of continuous real--valued bounded functions defined on $\cD_\cU^\infty$. 
		
		\,
		\item[\mylabel{f}{p1jsm16 item f}] There are $s,t\in\De_G^*$ such that $s\neq t$ and $\zt_{s,\P(s)}$ is cohomologous to $\zt_{t,\P(t)}$ in the class of H\"older continuous real--valued bounded continuous functions defined on $\cD_\cU^\infty$.
		
		\,
		\item[\mylabel{g}{p1jsm16 item g}] There are $s,t\in\De_G^*$ such that $s\neq t$ and $\zt_{s,\P(s)}$ is cohomologous to $\zt_{t,\P(t)}$ in the class of continuous real--valued bounded continuous functions defined on $\cD_\cU^\infty$.
		
		\,
		\item[\mylabel{h}{p1jsm16 item h}] There exists $\ga\in\R$ such that the function $\zt:=\zt_{1,0}:\cD_\cU^\infty\lra\R$ and $\ga \^N:\cD_\cU^\infty\lra\R$ are cohomologous in the class of H\"older continuous real--valued bounded functions defined on $\cD_\cU^\infty$.
		
		\,
		\item[\mylabel{i}{p1jsm16 item i}] There exists $\ga\in\R$ such that the function $\zt:=\zt_{1,0}:\cD_\cU^\infty\lra\R$ and $\ga \^N:\cD_\cU^\infty\lra\R$ are cohomologous in the class of continuous real--valued bounded functions defined on $\cD_\cU^\infty$.
	\end{enumerate}
\end{prop}
\begin{proof}
	The following implications are obvious:
	\begin{align*}
		\eqref{p1jsm16 item d}\implies \eqref{p1jsm16 item e}\implies \eqref{p1jsm16 item g}\implies \eqref{p1jsm16 item i},
		\qquad\eqref{p1jsm16 item d}\implies \eqref{p1jsm16 item f}\implies \eqref{p1jsm16 item h}\implies \eqref{p1jsm16 item i},
		\quad\text{and}\quad
		\eqref{p1jsm16 item b}\implies \eqref{p1jsm16 item c}.
	\end{align*}
	In order to complete the proof we shall establish the following implications: 
	\begin{align*}
		\eqref{p1jsm16 item i}\implies \eqref{p1jsm16 item a}\implies \eqref{p1jsm16 item b}\implies \eqref{p1jsm16 item d}
		\quad\text{ and }\quad
		\eqref{p1jsm16 item c}\implies \eqref{p1jsm16 item a}.
	\end{align*}
	Having this we will be done. First, assume that \eqref{p1jsm16 item i} holds. This means that there exists a bounded continuous function $u:\cD_\cU^\infty\to\R$ such that 
	$$
		\zt=\ga \^N+u-u\circ \sg.
	$$
	Then, given any $s,t\in\De_G^*$, we have 
	$$	
		\zt_{t,\P(t)}=(\ga t-\P\(t))\^N+tu-tu\circ \sg
	\  \  \  	{\rm  and }   \  \  \
		\zt_{s,\P(s)}=(\ga s-\P(s))\^N+su-tu\circ \sg.
	$$
	Hence, 
	$$
		\zt_{t,P(t)}=\zt_{t,\P(t)}+\(\ga (t-s)+\P(s)-\P(t)\)\^N+(t-s)u-(t-s)u\circ \sg.
	$$
	Thus, 
	$$
\P\lt(\sg,\zt_{t,\P(t)}+(\ga (t-s)+\P(s)-\P(t))\^N+(t-s)u-(t-s)u\circ \sg\rt)=\P\lt(\sg,\zt_{t,\P(t)}\rt)=0.
	$$
	Therefore, as $\P(\sg,\zt_{s,P(s)})=0$, we get
	\begin{align*}
		\(\ga (t-s)+\P(s)-&\P(t)\)\int_{\cD_\cU^\infty}\^N d\^\mu_s = \\
		&=\h_{\^\mu_t}(\sg)+\int_{\cD_\cU^\infty}\zt_{s,\P(s)} d\^\mu +(\ga (t-s)+\P(s)-\P(t))\int_{\cD_\cU^\infty}\^N d\^\mu_s\\
&\leq \P(\sg, \zt_{s,\P(s)}+(\ga(t-s)+\P(s)-\P(t))\^N)\\
&=0.
	\end{align*}
	Since 
	$$
		\int_{\cD_\cU^\infty}\^N\, d\^\mu_s\geq 0,
	$$
	we have 
	$$
		\ga(t-s)+\P(s)-\P(t)\leq 0.
	$$
	Exchanging the roles of $s$ and $t$, we also get 
	$$
		\ga(s-t)+\P(t)-\P(s)\leq 0.
	$$
	Hence, 
	$$
	\P(t)-\P(s)=\ga(t-s).
	$$
	This means that condition \eqref{p1jsm14 item e} of Proposition~\ref{p1jsm14} holds, i.e. the semigroup $G$ is exceptional. This in turn means that condition \eqref{p1jsm16 item a} of Proposition~\ref{p1jsm16} holds, so the implication \eqref{p1jsm16 item i}$\implies$\eqref{p1jsm16 item a} is established.
	
	Now, for proving the implication \eqref{p1jsm16 item a}$\implies$\eqref{p1jsm16 item b}, assume that \eqref{p1jsm16 item a} holds. Then, because of Proposition~\ref{p1jsm14} and Lemma~\ref{p1sl4}, the set $\{\mu_t: t\in\De_G^*\}$ is a singleton. 
	By formula \eqref{1sl1} this means that the set $\{\^\mu_t\circ\pi_\cU^{-1}: t\in\De_G^*\}$ is a singleton. Finally, because of Corollary~\ref{c1sl3}, this implies that the set $\{\^\mu_t: t\in \De_G\}$ is a singleton, meaning that \eqref{p1jsm16 item b} holds. 
	
	By Theorem~2.2.7 in \cite{mugdms}, condition \eqref{p1jsm16 item b} of our present proposition is equivalent (keeping in mind that $\P(\sg,\zt_{s,\P(s)})=\P(\sg,\zt_{t,\P(t)})=0$) to \eqref{p1jsm16 item d}, in particular, it entails \eqref{p1jsm16 item d}.
	
	For the second implication, assume that \eqref{p1jsm16 item c} holds. Then, by formula \eqref{1sl1}, we have that $\hat\mu_s=\hat\mu_t$. So, by Lemma~\ref{p1sl4}, the two measures $\mu_s$ and $\mu_t$ are mutually singular. Since, by Theorem~\ref{t4h65} both measures $\mu_s$ and $\mu_t$ are ergodic, they are equal. Because of Proposition~\ref{p1jsm14} (see its item \eqref{p1jsm14 item g}), the item \eqref{p1jsm16 item a} of Proposition~\ref{p1jsm16} is established, and thus we are done.
\end{proof}

We shall now prove that exceptional rational semigroups are exceptional indeed, i.e. we will almost fully classify all of them and we will show that they form a very small sub--collection of all rational semigroups. We need some preparation. Keep $G$, a \NOSC-FNR \ rational semigroup. Let
$$
\cU=\big\{U_s:s\in\Crit_*(\^f)\big\},
$$
be a nice family produced in Theorem~\ref{t1nsii7}. Define the \textbf{entrance time}\index{entrance time} to the set $J_{\cU}^\circ$ (defined in Lemma~\ref{l1sl5}),
	$$
		\hat N:J(\^f)\lra\N_0\cup\{\infty\}=\{0,1,2,\dots,\infty\},
	$$
	by declaring that $\hat N(\xi)$\index{$\hat N(\xi)$} is the least element $k$ of $\N_0\cup\{\infty\}$ such that 
	$$
		\^f^k(\xi)\in J_{\cU}^\circ.
	$$
	Further define
	$$
		J_\cU^{\circ,+}:=\big\{z\in J(\^f):\hat N(z)\in\N_0\big\}.
	$$\index{$J_\cU^{\circ,+}$}
	Now define the \textbf{entrance map}\index{entrance map} $\^f_+:J_\cU^{\circ,+}\lra J_\cU^{\circ}$\index{entrance map!$\^f_+:J_\cU^{\circ,+}\lra J_\cU^{\circ}$}\index{$\^f_+:J_\cU^{\circ,+}\lra J_\cU^{\circ}$} as 
	\begin{align}\label{1jsm22}
		\^f_+(\xi)=\^f^{\hat N(\xi)}(\xi).
	\end{align}
By its very definition the set $J_\cU^{\circ,+}$, is backward invariant, i.e. 
		\begin{align}\label{3jsm24}
			\bigcup_{k=0}^\infty \^f^{-k}(J_\cU^{\circ,+})\sub J_\cU^{\circ,+}
		\end{align}
		and 
		\begin{align}\label{1jsm24.1}
			\bigcup_{k=0}^\infty \^f^{-k}(\ov{J_\cU^{\circ,+}})\sub \ov{J_\cU^{\circ,+}}.
		\end{align}

%{\color{red} (Insert here the text from 1*ESG-2)}
%begin ESG-2

We first shall prove the following.
\begin{lem}\label{l1esg2}
If $G$ is an exceptional \NOSC-FNR \ rational semigroup, then there exist a constant $\ga\in\R$ and a Borel measurable function $u^+:J_\cU^{\circ,+}\lra\R$ such that 
	\begin{enumerate}
		\item[\mylabel{a}{l1esg2 item a}] $-\log|\^f'|=\ga+u^+-u^+\circ f$ everywhere on $J_\cU^{\circ, +}$,
		
		\,
		
		\item[\mylabel{b}{l1esg2 item b}] If $\xi\in\ov{J_\cU^{\circ,+}}\bs\PCV(\^f)$, then there exists $\Ga\sub\Sg_u\times\C$, an open neighborhood of $\xi$ in $\Sg_u\times\C$, such that the function $u^+\rvert_{\Ga\cap J_\cU^{\circ,+}}$ is bounded.
	\end{enumerate} 
\end{lem}
\begin{proof}
	Let
	$$
	\cU=\big\{U_s:s\in\Crit_*(\^f)\big\},
	$$
	a nice family produced in Theorem~\ref{t1nsii7}. 	
	Since $G$ is exceptional, by Proposition~\ref{p1jsm16}~\eqref{p1jsm16 item i} there exists a bounded continuous function $u:\cD_\cU^\infty\lra\R$ such that 
	\begin{align}\label{2jsm22}
	\zt=\ga\^N+u-u\circ\sg.
	\end{align}
	Since $\zt<0$ everywhere in $\cD_\cU^\infty$, the integral of $\zt$ against any $\^f$--invariant measure is negative. Since also the integral of $u-u\circ\sg$ against any (at least one suffices) such measure vanishes, and since $\^N\geq 1$ everywhere in $\cD_\cU^\infty$, we conclude that 
	\begin{align}\label{3jsm22}
	\ga<0.
	\end{align}
	Now we define the function $u^+:J_\cU^{\circ,+}\lra\R$ by 
	\begin{align}\label{2jsm23}
	u^+(\xi)=u\lt(\pi_\cU^{-1}\(\^f_+(\xi)\)\rt)-\log\big|\^f_+'(\xi)\big|-\ga\hat N(\xi),
	\end{align}
	where $\pi_\cU^{-1}\(\^f_+(\xi)\)$ is a singleton because of Corollary~\ref{c1sl3}. By virtue of Lemma~\ref{l1sl5}, 
	\begin{equation}\label{3jsm23}
	\^f(J_\cU^{\circ,+})\sub J_\cU^{\circ,+}.
	\end{equation}
	To prove \eqref{l1esg2 item a} let $\xi\in J_\cU^{\circ,+}$. We consider two cases. First assume that 
	$$
	\xi\not\in J_\cU^\circ.
	$$
	Then $\hat N(\xi)=0$ and, because of item \eqref{c2sl3 item b} of Corollary~\ref{c2sl3},
	$$
	\hat N(\^f(\xi))=\^N(\pi_\cU^{-1}(\xi))-1 \text{ and } \^f_+(\^f(\xi))=\^f\rvert_{J_\cU}(\xi).
	$$
	Inserting this into \eqref{2jsm23} and using \eqref{2jsm22}, along with item \eqref{c2sl3 item b} of Corollary~\ref{c2sl3}, 
	\begin{align*}
	u^+(\^f(\xi))&=u\(\pi_\cU^{-1}(\^f_+(\^f(\xi)))\)-\log|\^f_+'(\^f(\xi))|-\ga(\^N(\pi_\cU^{-1}(\xi))-1)
	\\
	&=u\(\pi_\cU^{-1}(\^f_{J_\cU}(\xi))\)-\log\lt|\lt(\^f^{\^N(\pi_\cU^{-1}(\xi))}\rt)'(\xi)\rt|+\log|\^f'(\xi)|-\ga\^N(\pi_\cU^{-1}(\xi))+\ga
	\\
	&=u\(\pi_\cU^{-1}(\^f_{J_\cU}(\xi))\)+u(\pi_\cU^{-1}(\xi))-u\(\sg(\pi_\cU^{-1}(\xi))\)+\log|\^f'(\xi)|+\ga
	\\
	&=u(\pi_\cU^{-1}(\xi))+\log|\^f'(\xi)|\\
	&=u^+(\pi_\cU^{-1}(\xi))+\log|\^f'(\xi)|+\ga,
	\end{align*}
	where the equality
	\begin{align}\label{1jsm24}
	u^+(\pi_\cU^{-1}(\xi))=u(\pi_\cU^{-1}(\xi))
	\end{align}
	follows from the fact that $\hat N(\xi)=0$ and formula \eqref{2jsm23}. The proof of \eqref{l1esg2 item a} is thus complete. 
	
	Now to prove \eqref{l1esg2 item b} we first note that if a point $x\in U\cap J_\cU^{\circ,+}$, then $x\in J_\cU^\circ$. It thus follows from \eqref{1jsm24} that 
		\begin{align}\label{2jsm24}
		u^+\rvert_{J_\cU^{\circ,+}\cap U}=u\rvert_{J_\cU^{\circ,+}\cap U}
		\end{align}
		is bounded. Because of \eqref{3jsm24} and \eqref{1jsm24.1}, and since the map $\^f:J(\^f)\lra J(\^f)$ is topologically exact, we see that for every $\xi\in\ov{J_\cU^{\circ,+}}\bs\PCV(\^f)$ there exists $j\geq 0$ and 
		$$
		x\in\lt(U\cap \^f^{-j}(\xi)\cap\ov{J_\cU^{\circ,+}}\rt)\bs\Crit(\^f^j).
		$$
		Since $\Crit(\^f)$ is a closed set, there this exists an open set $W\sub U$ such that 
		$$
		x\in W\sub\ov{W}\sub\hat\C\bs\Crit(\^f^j).
		$$
		Then the function 
		$$
		W\ni z\longmapsto\log|(\^f^j)'(z)|\in\R
		$$
		is bounded. Hence, by iterating \eqref{l1esg2 item a} (this is possible in view of \eqref{3jsm23}), which gives that
		$$
		u^+\circ f^j=\ga j+u^++\log|(\^f^j)'|,
		$$
		we get, with the use of \eqref{2jsm24}, that the set $u^+(\^f^j(W\cap J_\cU^{\circ,+}))$ is bounded. But
		$$
		\^f^j(W\cap J_\cU^{\circ,+})=\^f^j(W)\cap J_\cU^{\circ,+}
		$$
		since the set $J_\cU^{\circ,+}$ is, by \eqref{3jsm24}, backward $\^f$--invariant, and, by \eqref{3jsm23}, forward  $\^f$--invariant. In addition, the set $\^f^j(W)$ is open as $W\sub\Sg_u\times\C$ is open and the map $\^f:\Sg_u\times\hat\C\lra\Sg_u\times\hat\C$ is open. Finally, 
		$$
		\xi=\^f^j(x)\in\^f^j(W).
		$$	
		So, setting $\Ga:=\^f^j(W)$ finishes the proof of \eqref{l1esg2 item b}.
\end{proof}

As a fairly easy consequence of this lemma we get the following. 

\begin{thm}\label{t1esg2}
If $G$ is an exceptional \NOSC-FNR \ rational semigroup,
then for every non--empty set $D\sub\{1,2,\dots,u\}$ the corresponding rational semigroup $F=\langle f_j: j\in D \rangle$ is exceptional.
\end{thm}

\begin{proof}
	Since $J(F)$ is closed and backward invariant under $F$. i.e. 
	\begin{align*}
		\bigcup_{g\in F} g^{-1}(J(F))\sub J(F),
	\end{align*}
	there exists $R\in(0,R_*(G))$ such that 
	\begin{align*}
		\bigcup_{g\in F} g^{-1}\lt(B_2(\Crit_*(F),4R)\rt)
		\bigcap B_2(\Crit_*(f)\bs\Crit_*(F,4R))
		=\emptyset.
	\end{align*}
	Let $\cU$ be a nice family for the semigroup $G=\langle f_1,\dots f_u \rangle $  induced by the aperiodic set $S=\Crit_*(f)$ with some arbitrary number $\ka\in(1,2)$ and $r\in(0,R]$ according to Theorem~\ref{t1nsii7}. Then
	$$
		\cU_F:=\{U_s\}_{s\in\Crit_*(F)}
	$$	
	is a nice family for the semigroup $F=\langle f_j: j\in D\rangle$. Furthermore, 
	\begin{align*}
		\cD_\cU(F)
		&=
		\bigcup_{s\in\Crit_*(F)}\bigcup_{n=1}^\infty \set{\tau\in\cD_n(G,s):\^\tau\in D^n}
		\\
		&=
		\set{\tau\in\cD_\cU(G): \^\tau\in\Sg_u^* \text{ and } t(\tau), i(\tau)\in\Crit_*(F)}\sub\cD_\cU(G)
	\end{align*}
	and
	$$
		\cS_{\cU_F}(F)=\set{\^f_\tau^{-\|\tau\|}: X_{t(\tau)}\lra X_{i(\tau)} }_{\tau\in\cD_\cU(F)}\sub \cS_\cU(G).
	$$
	In addition, 
	$$
		A(F,\cU_F)=A(G,\cU)\rvert_{\cD_\cU(F)\times\^\cD_\cU(F)}
	$$	
	and
	$$
		\cD_{\cU_F}^\infty(F)\sub\cD_\cU^\infty(G).
	$$
	In particular, the subshift $\sg:\cD_{\cU_F}^\infty(F)\lra\cD_{\cU_F}^\infty(F)$ is a subsystem of the subshift $\sg:\cD_{\cU}^\infty(G)\lra\cD_{\cU}^\infty(G)$ and 
	\begin{align}\label{1esg3}
		\zt_{1,0}^{(F)}=\zt_{1,0}^{(G)}\rvert_{\cD_{\cU_F}^\infty(F)},
	\end{align}
	where $\zt_{1,0}^{(F)}$ and $\zt_{1,0}^{(G)}$ respectively denote the $\zt_{1,0}$ functions associated to the semigroups $F$ and $G$. Likewise, 
	\begin{align}\label{2esg3}
		\^N_F=\^N_G\rvert_{\cD_{\cU_F}^\infty(F)}.
	\end{align}
	Since $G$ is exceptional, it follows from item \eqref{p1jsm16 item i} of Proposition~\ref{p1jsm16} that with some $\ga\in\R$ the functions $\zt_{1,0}^{(G)}:\cD_\cU^\infty\lra\R$ and $\ga \^N_G:\cD_\cU^\infty(G)\lra\R$ are cohomologous in the class of continuous real--valued bounded functions defined on $\cD_\cU^\infty(G)$. We therefore conclude from \eqref{1esg3} and \eqref{2esg3} that the functions $\zt_{1,0}^{(F)}$ and $\ga \^N_F$ are cohomologous in the class of continuous real--valued bounded functions defined on $\cD_\cU^\infty(F)$.	
	The proof of Theorem~\ref{t1esg2} is thus complete by applying item \eqref{p1jsm16 item i} of Proposition~\ref{p1jsm16} again.
\end{proof}

%end ESG-2

In order to formulate and to prove our main result about exceptional rational semigroups we need the concept of parabolic orbifolds. The notion of an orbifold which we will utilize was introduced by William Thurston in \cite{Th1} and \cite{Th2}. It is very useful to study the dynamics of some rational functions. It was in particular used by Anna Zdunik in \cite{Zd1} to classify all exceptional (in her sense) rational functions of the Riemann sphere $\oc$. Our brief introduction to (parabolic) orbifolds closely follows hers from \cite{Zd1} and we will substantially rely on some results from \cite{Zd1} to prove our Theorem~\ref{t2esg2}. We consider only orbifolds homeomorphic to the Riemann sphere $\oc$. Such an \textbf{orbifold}\index{orbifold} is the sphere $\oc$ with a collection of distinguished, mutually distinct, points $x_1, x_2,\ld,x_k\in\oc$ and integers (including $+\infty$) $\nu(x_1), \nu(x_2),\ld,\nu(x_k)\ge 2$ ascribed to them. It is denoted by 
$$
\(\oc;x_1, x_2,\ld,x_k;\nu(x_1),\nu(x_2),\ld,\nu(x_k)\).
$$
Two orbifolds 
$$
\(\oc;x_1, x_2,\ld,x_k;\nu(x_1), \nu(x_2),\ld,\nu(x_k)\)
\  \  {\rm and} \  \
\(\oc;y_1, y_2,\ld,y_l;\nu(y_1), \nu(y_2),\ld,\nu(y_l)\)
$$ 
are considered equivalent if and only if $k=l$ and 
$$
\nu(x_1)=\nu(y_1),\nu(x_2)=\nu(y_2),\ld,\nu(x_k)=\nu(y_l).
$$
We then refer to an orbifold simply by listing some numbers 
$$
\nu_1, \nu_2,\ld,\nu_k\in\{2,3,\ld\}\cup\{+\infty\}.
$$
William Thurston introduced in \cite{Th1} and \cite{Th2} the notion of Euler characteristic of a (general) orbifold. In our context of the Riemann sphere $\oc$, it is given by the following formula
$$
\chi\(\oc;\nu_1, \nu_2,\ld,\nu_k\)=2-\sum_{j=1}^k\lt(1-\frac{1}{\nu_j}\rt).
$$
An orbifold $\(\oc;\nu_1, \nu_2,\ld,\nu_k\)$ is called \textbf{parabolic}\index{orbifold!parabolic orbifold} if and only if its Euler characteristic is equal to $0$. It is easy to list all parabolic orbifolds associated to the Riemann sphere $\oc$. These are
$$
\(\oc; 2, 2, 2, 2\), \, \(\oc; 3, 3, 3\),\, \(\oc; 2, 4, 4\),\, \(\oc; 2, 3,6\),\, \(\oc; 2, 2, +\infty\),\, \(\oc; +\infty, +\infty\).
$$
A rational function $f:\oc\lra\oc$ is called \textbf{critically finite}\index{critically finite rational function} if and only if the forward trajectory of each of its critical points is finite, in other words, if and only if each critical point of $f$ is either periodic or eventually periodic. There is a natural way of ascribing an orbifold to such a critically finite map $f$. The (finite) set of distinguished points is given by the direct postcritical set of $f$, i.e. the set
$$
\big\{f^k(c):c\in\Crit(f) \ {\rm and} \ k\ge 1\big\}.
$$
The numbers $\nu(f^k(c))$ are required to satisfy the relation that $\nu(f^{k+1}(c))$ is an integral multiple of $\nu(f^k(c))$. There is exactly one minimal (in an obvious sense) way of choosing these numbers. We would like to bring up (see \cite{Zd1}) the following fact.

\begin{fact}
We have the following.
\begin{enumerate}
\item The orbifold  $\(\oc; +\infty, +\infty\)$ corresponds, up to a conjugacy by a M\"obius map, to rational functions of the form
$$
\oc\ni z\longmapsto z^d\in\oc
$$
where $d$ is an integer with $|d|\ge 2$.

\, 

\item The orbifold  $\(\oc; 2, 2, +\infty\)$ corresponds, up to a conjugacy by a M\"obius map, to $\pm$Tchebyschev's polynomials.

\,

\item These two above are the only classes of critically finite rational  functions on the Riemann sphere $\oc$ yielding parabolic orbifolds and whose Julia sets are different (so nowhere dense) from $\oc$.
\end{enumerate}
\end{fact}

%{\color{red} (Insert here the text from 1*ESG-3)}
%begin ESG-3

Now, we shall prove the following main result about exceptional rational semigroups.
\begin{thm}\label{t2esg2}
If $G$ is an exceptional \NOSC-FNR \ rational semigroup, then each element of $G$ is a critically finite rational function with parabolic orbifold.
\end{thm}
\begin{proof}
	Fix an element of $G$. It is of the form 
	$$
		f_\tau:\hat\C\to\hat\C, \quad \tau\in\Sg_u^*.
	$$
	By our hypotheses and Theorem~\ref{t2esg2}, the rational semigroup $\langle g \rangle$ generated by $g$ is exceptional. Since $\langle g \rangle$ has exactly one generator, the corresponding skew product map $\^g$ is canonically identified with the map $g:\hat\C\to\hat\C$ itself. So, Lemma~\ref{l1esg2}, applied to the semigroup $\langle g \rangle$, tells us that there exist a constant $\ga\in\R$ and a Borel measurable function $u^+:J_\cU^{\circ,+}\to\R$ such that 
	\begin{align}\label{1esg4}
		-\log|g'|=\ga+u^+-u^+\circ g
	\end{align}
	everywhere on $J_\cU^{\circ,+}(\langle g \rangle)$. In addition, the function $u^+$ is given by the formula \eqref{2jsm23}. Let $\mu_0$ be the measure of maximal entropy for $g:J(g)\to J(g)$. By Poincar\'e's Recurrence Theorem and ergodicity of $\mu_0$, we have that 
	$$
		\mu_0(J_\cU^{\circ,+}(\langle g \rangle))=1.
	$$
	Since the function $u:\cD_\cU^\infty(\langle g \rangle)\to\R$ is measurable and bounded, the first summand 
	$$
		J_\cU^{\circ,+}\ni \xi\longmapsto u(\pi_\cU^{-1}(\^f_+(\xi)))\in\R
	$$
	in the formula \eqref{2jsm23} defining $u^+$, is bounded, thus belongs to $L^2(\mu_0)$. The other two terms $\xi\longmapsto-\ga \^N(\xi)$ and $\xi\longmapsto-\log|\^f_+'(\xi)|$ and belong to $L^2(\mu_0)$ respectively  because of Corollary~\ref{c1sp4E} and Lemma~\ref{l120190819}, both applied with $p=2$. 
	In conclusion, 
	\begin{align}\label{1esg5}
		u^+\in L^2(\mu_0).
	\end{align}
	Now we can integrate \eqref{1esg4} against $\mu_0$ to get that 
	$$
		\ga=-\int \log|g'| d\mu_0 = -\chi_{\mu_0}(g).
	$$
	Therefore, 
	\begin{align*}
		\frac{h_{\mu_0}(g)}{\chi_{\mu_0}(g)}\log|g'|
		&=
		\frac{h_{\mu_0}(g)}{\chi_{\mu_0}(g)}
		=
		\frac{\log\deg(g)}{\chi_{\mu_0}(g)}(\chi_{\mu_0}(g)+u^+\circ g- u^+)
		\\
		&=
		\log\deg(g)+\lt(\frac{\log\deg(g)}{\chi_{\mu_0}(g)}u^+\rt)\circ g - \frac{\log\deg(g)}{\chi_{\mu_0}(g)}u^+.
	\end{align*}
	Equivalently, 
	\begin{align*}
		\frac{\log\deg(g)}{\chi_{\mu_0}(g)}\log|g'|-\log\deg(g)
		=
		\lt(\frac{\log\deg(g)}{\chi_{\mu_0}(g)}u^+\rt)\circ g - \frac{\log\deg(g)}{\chi_{\mu_0}(g)}u^+,
	\end{align*}
	and
	$$
		\frac{\log\deg(g)}{\chi_{\mu_0}(g)}u^+\in L^2(\mu_0)
	$$
	by virtue of \eqref{1esg5}. Since the number $\al$ of \cite{Zd1} is equal to 
	$$
		\frac{\log\deg(g)}{\chi_{\mu_0}(g)},
	$$
	the equation (H) from page~634 of \cite{Zd1} is satisfied by our function $g$. It therefore follows from Corollary on page~637 (in Section~5) of \cite{Zd1}, the (last) Corollary of Section~7 on page~644 in \cite{Zd1}, and from Proposition~8 on page~645 (in Section~8) of \cite{Zd1} that $g$ is a critically finite rational function with parabolic orbifold. 	
	The proof is complete.
\end{proof}

%end ESG-3
It is evident from Theorem~\ref{t2esg2} that critical points of exceptional rational semigroups are very special. We shall now prove it without invoking the theory of orbifolds and quite deep results of \cite{Zd1}. We call a critical point $(\om,c)\in J(\^f)$ \textbf{exceptional}\index{exceptional!critical point} if for every integer $n\geq 1$
$$
	\^f^{-n}(\^f^n(\om,c))\sub \Crit(\^f^n)\cup\PCV(\^f).
$$
We shall prove the following. 
\begin{prop}\label{p1jsm22}
If $G$ is an exceptional \NOSC-FNR \ rational semigroup containing some non--exceptional critical points in $J(\^f)$, then $G$ is non--exceptional. 
\end{prop}
\begin{proof}
	Seeking contradiction, suppose that $G$ is exceptional. Keep
	$$
	\cU=\big\{U_s:s\in\Crit_*(\^f)\big\},
	$$
	a nice family produced in Theorem~\ref{t1nsii7}. 
	Now, passing to the last step of the proof of Proposition~\ref{p1jsm22}, let $c\in J(\^f)$ be a non--exceptional point of $\^f$. This means that there exists an integer $n\geq 1$ such that 
	$$
		\^f^{-n}(\^f^n(\om,c))\not\sub\Crit(\^f^n)\cup\PCV(\^f).
	$$
	This in turn means that there exists a point 
	\begin{align}\label{2jsm25}
		\xi\in J(\^f)\bs\lt(\Crit(\^f^n)\cup\PCV(\^f) \rt)
	\end{align}
	such that 
	\begin{align}\label{1jsm25}
		\^f^n(\xi)=c.
	\end{align}
	Since $c\in J(\^f)\cap U$, since $\Trans(\tf)$, the set of transitive points of $\^f$, is dense in $J(\^f)$, and since $\Trans(\tf)\cap U\sub J_\cU^\circ\sub J_\cU^{\circ, +}$, we conclude that $c\in\ov{J_\cU^{\circ,+}}$. Since $\tf$ is a \FNR \ map, we thus have that 
	\begin{align}\label{3jsm25}
		c\in \ov{J_\cU^{\circ,+}}\bs\PCV(\^f).
	\end{align}
	So, by \eqref{1jsm24.1}, \eqref{2jsm25}, and \eqref{1jsm25}, also 
	\begin{align}\label{4jsm25}
		\xi\in\ov{J_\cU^{\circ,+}}\bs\PCV(\^f).
	\end{align}
	Let then $\Ga_c$ and $\Ga_\xi$ be the open neighborhoods respectively of $c$ and $\xi$ produced in Lemma~\ref{l1esg2}. By \eqref{2jsm25}, \eqref{3jsm25}, and \eqref{1jsm24.1}
	$$
		w:=\^f^n(\xi)=\^f^n(c)\in J_\cU^{\circ,+}.
	$$
	So, there exists $(z_k)_{k=1}^\infty$, a sequence in $\^f^n(\Ga_c)\cap\^f^n(\Ga_\xi)\cap J_\cU^{\circ,+}$ such that 
	$$
		\lim_{k\to\infty} z_k=w.
	$$
	Then, there exist $(x_k)_{k=1}^\infty$ and $(y_k)_{k=1}^\infty$, two sequences respectively of points in $\Ga_c$ and $\Ga_\xi$, such that 
	$$
		\^f^n(x_k)=z_k=\^f^n(y_k)
	$$
	for all $k\geq 1$, 
	\begin{align}\label{1jsm26}
		\lim_{k\to\infty}x_k=c, \quad\text{ and }\quad \lim_{k\to\infty}y_k=\xi.
	\end{align}
	Then by \eqref{1jsm24.1}, $x_k,y_k\in J_\cU^{\circ,+}$ for all $k\geq 1$ and equation \eqref{l1esg2 item b} of Lemma~\ref{l1esg2}, iterated $n$ times, gives 
	$$
		\log|(\^f^n)'(x_k)|=u^+(z_k)-u^+(x_k)-\ga n
	$$
	and 
	$$
		\log|(\^f^n)'(y_k)|=u^+(z_k)-u^+(y_k)-\ga n
	$$
	for every $k\geq 1$. So, 
	$$
		\log|(\^f^n)'(x_k)|=\log|(\^f^n)'(y_k)|+u^+(y_k)-u^+(x_k)
	$$
	for every $k\geq 1$. But, the sequences $(u^+(y_k))_{k=1}^\infty$ and $(u^+(x_k))_{k=1}^\infty$ are both bounded because of item \eqref{l1esg2 item b} of Lemma~\ref{l1esg2} while, because of \eqref{2jsm25} and the right--hand side of \eqref{1jsm26}, there exists $q\ge 1$ such that the sequence
	$$
		\lt(\log|(\^f^n)'(y_k)|\rt)_{k=q}^\infty
	$$
	is bounded. Thus, the sequence 
	$$
		\lt(\log|(\^f^n)'(x_k)|\rt)_{k=q}^\infty
	$$ 
	is also bounded. This however is a contradiction since by the left--hand side of \eqref{1jsm26}, 
	$$
		\lim_{k\to\infty}\log|(\^f^n)'(x_k)|=-\infty.
	$$
	We are done. 
\end{proof}

\section{Measures $m_t\circ p_2^{-1}$ and $\mu_t\circ p_2^{-1}$ \\ versus \\ Hausdorff Measures $\H_{t^\ka}$ and  $\H_{t^\ka\exp\lt(c\sqrt{\log(1/t)\log^3(1/t)}\rt)}$}\label{PUZ}
In this section we will establish relationships between the Gibbs/equilibrium measures $\mu_t$ and generalized Hausdorff measures, for example corresponding to gauge functions of the form 
\begin{align*}
	(0,\ep)\ni u\longmapsto u^{\HD(\mu_t)}\exp\lt(s\sqrt{\log\frac{1}{u}\log^3\frac{1}{u}} \rt), \quad s>0,
\end{align*}
where $\log^3(x)=\log\(\log(\log(x))\)$ and, more generally, for every integer $q\ge 1$, $\log^q(x)$ is the $q$th iterate of the logarithm applied to $x$.
The general strategy is to work with the GDS $\cS_\cU$ of Theorem~\ref{t1nsii15} rather than the skew product map $\^ f: \Sg_u\times \hat{\C}\lra\Sg_u\times \hat{\C}$ itself and it is based on \cite{U3}, comp. also Section 4.8 of \cite{mugdms}. We start with an appropriate stochastic law which is an extension of the Law of the Iterated Logarithm. 

\bdfn
A monotone increasing function $\psi:[1,+\infty)\lra[0,+\infty)$ is said to belong to the \textbf{lower class}\index{lower class} if 
\begin{align*}
	\int_{1}^{\infty}\frac{\psi(u)}{u}\exp\lt(-\frac{1}{2}\psi^2(u)\rt) d\mu<+\infty
\end{align*} 
and to the \textbf{upper class}\index{upper class} if 
\begin{align*}
\int_{1}^{\infty}\frac{\psi(u)}{u}\exp\lt(-\frac{1}{2}\psi^2(u)\rt) d\mu=+\infty.
\end{align*} 
\edfn
Recall that given an integer $q\ge 1$ and $u>0$ large enough, we mean by $\log^q(u)$ the $q$th iteration of the logarithm applied to $u$; for example:
$$
\log^1(u)=\log(u), \  \, \log^2(u)=\log(\log(u)), \  \  
\log^3(u)=\log\(\log(\log(u))\).
$$
In what follows, in the proofs, we will need the following lemma providing suitable improvements of lower and upper functions. This lemma has been proved in \cite{DU_Proc_Gustrow} as Lemma~4.3 and was repeated, with proof, in \cite{U3} as Lemma~6.1. We provide here its formulation and a short proof for the sake of completeness and convenience of the reader. 

\begin{lem}\label{l6.1uifspuz}
Let $\eta, \chi>0$ and let $\rho:[(\chi+\eta)^{-1}, \infty)\lra\R_+$ belong to the upper (lower) class. Let 
$$
\th:[(\chi+\eta)^{-1}, \infty)\lra\R_+
$$
be a function such
that 
$$
\lim_{t\to\infty}\rho(t)\th(t)=0.
$$
Then there exists  respective upper and lower class functions  
$$
\rho_+:[1,\infty)\to \R_+ 
\  \  \  {\rm and} \  \  \ 
\rho_-:[1,\infty)\to \R_+
$$
such that

\begin{itemize}
\item[\mylabel{a}{l6.1uifspuz item a}] \  $\rho(t(\chi+\eta))+\th(t(\chi+\eta)) \le \rho_+(t)$,

\,  \fr{\rm and}  \,

\item[\mylabel{b}{l6.1uifspuz item b}] \ $\rho(t(\chi-\eta))-\th(t(\chi-\eta)) \ge \rho_-(t)$ 

\,
\end{itemize}
for all $t\ge 1$.
\end{lem}

\begin{proof}
Since $\lim_{t\to\infty}\rho(t)\th(t)=0$, there
exists a constant $M$ such that
$$
(\rho(t)+\th(t))^2\le \rho(t)^2+M.
$$
Let $\rho$ belong to the upper class. Then the function $t\longmapsto 
\rho(t/(\chi+\eta))$ also belongs to the upper class. Hence, we may
assume that $\chi+\eta=1$. Define
$$
\rho_+(t)^2:=\inf\big\{u(t)^2:u \text{ is non--decreasing and } u(t)\ge 
\rho(t)+\th(t)\big\}.
$$
Then $\rho_+(t)\ge\rho(t)+\th(t)$ for $t\ge 1$ and $\rho_+$ is
non--decreasing. Since $\rho_+(t)^2\le \rho(t)^2+M$, we also get
$$
\int_1^\infty{\rho_+(t)\over t}\exp\(-(1/2)\rho_+^2(t))dt
\ge \exp(-M/2)\int_1^\infty{\rho_+(t)\over t}\exp\(-(1/2)\rho^2(t))dt
=+\infty.
$$
The proof in the case of a function of the lower class is
similar. 
\end{proof}

As an immediate consequence of Theorem~5.2 and Lemma~5.3 of \cite{U3} (comp. Theorem 2.5.5 and Lemma 2.5.6 in \cite{mugdms}) along with Proposition \ref{p1tf2}, Corollary \ref{c1sp4E}, and the inequality
\begin{align*}
	\lt(t\log\lt|\lt(\^f^{||\tau_{1}||}\rt)'\(\pi_\cU(\sg(\tau))\)\rt| \rt)^\al  
	\lt(\lt(\^f^{||\tau_{1}||}\rt)'(\pi_\cU(\sg(\tau))) \rt)^{-t}
	\leq \lt|\lt(\^f^{\|\tau_{1}\|}\rt)'\(\pi_\cU(\sg(\tau))\) \rt|^{-t'}
\end{align*}
holding for all $t'<t$, $\al\ge 0$, and $\|\tau_{1}\|$ (depending on $t'$) large enough (the last two properties yielding finite moments required in Lemma 5.3 of \cite{U3} or Lemma 2.5.6 of \cite{mugdms}), we get the following.

\begin{thm}\label{t2puz3}
If $G$ is a \FNR \ rational semigroup and $\cU$ is a nice family of sets, then for all $t\in\De_G^*$ and all $\al,\b\in\R$
    \begin{align*}
    \^\mu_t &\lt(\lt\{\tau\in\cD_\cU^\infty:\sum_{j=0}^{n-1}(\a\zt+\b N_\zt)\circ\sg^j(\tau)-\^\mu_t(\a\zt+\b N_\zt)
    \rt.\rt.\\
    &\qquad\qquad\qquad\qquad\qquad\qquad\qquad
    >\sg_t(\a\zt+\b N_\zt)\psi(n)\sqrt{n} \text{ for infinitely many } n    \bigg\}   \bigg) \\
	& = \begin{cases}
			0 \, &{\rm if}~~ \psi:[1,+\infty)\to(0,+\infty) \text{ belongs to the lower class}\\
			1 \, &{\rm if}~~ \psi:[1,+\infty)\to(0,+\infty) \text{ belongs to the upper class},
		\end{cases}                             
	\end{align*}
	where $\zeta:\cD_\cU^\infty\lra\R$ is given by 
$$
\zeta(\tau):=\log\lt|\lt(\^f^{||\tau_{1}||}\rt)'\(\pi_\cU(\sg(\tau))\)\rt|,
$$
$N_\zt:=\|\tau_1\|$, and, with $k:=\a\zt+\b N_\zt$, the non--negative number
$$
\sg_t^2(k):=\int_{\cD_\cU^\infty}(k-\^\mu_t(k))^2\, d\^\mu_t +2\sum_{n=1}^\infty\int_{\cD_\cU^\infty}(k-\^\mu_t(k))(k\circ\sg^n-\^\mu_t(k))\,d\^\mu_t
$$
is assumed to be positive. 
\end{thm}	
Define $k^*:J(\^f)\to\R$ by the formula,
	$$
		k^*(\om,z):=\a\log|\^f'(\om,z)|+\b N_\zt(\om,z).
	$$ 
Further define:
	$$
	\sg_t^2(k^*)
	:=\int_{J(\^f)}\(k^*-\mu_t(k^*)\)^2\, d\mu_t +2\sum_{n=1}^\infty\int_{J(\^f)}\(k^*-\^\mu_t(k^*)\)\(k^*\circ\^f^n-\^\mu_t(k^*)\)\,d\mu_t
	$$	 
	and
	\begin{align}
		\chi_{\mu_t}:=\int_{J(\^f)}\log|\^f'| d\mu_t 
		\quad\text{and}\quad
		\chi_{\^\mu_t}:=\int_{\cD_\cU^\infty}\zt d\^\mu_t.
		\label{2puz3.1}
	\end{align}
	Then it follows from Lemma~\ref{p1sl4} and (a generalization of) Kac's Lemma that 
	\begin{align}
		\frac{\sg^2_t(k)}{\chi_{\^\mu_t}}=\frac{\sg_t^2(k^*)}{\chi_{\mu_t}}.\label{1puz3.1}
	\end{align}
	Put 
	$$
		\^\chi_t:=\chi_{\^\mu_t} 
		\quad\text{and}\quad
		\^\sg_t^2:=\sg_t^2\lt((\ka_t-t)\zt-P(t)\N_\zt(\xi)\rt).
	$$
The main technical result of this section is the following. 
\begin{lem}\label{l1puz3}
Let $G$ be a \NOSC-FNR \ rational semigroup. Assume that $t\in\De_G^*$ and $\^\sg_t^2>0$. If a slowly growing function $\psi$ belongs to the upper class, then 
		\begin{align}\label{1puz3}
			\limsup_{r\to 0}\frac{m_t\circ p_2^{-1}(B_2(x,r))}{r^{\HD(\mu_t\circ p_2^{-1})}\exp\lt(\^\sg_t\chi_{\mu_t}^{-1/2}\psi\(\log\frac{1}{r}\)\sqrt{\log\frac{1}{r}}\rt)}=+\infty
		\end{align}
		for $\mu_t\circ p_2^{-1}$--a.e. $x\in J(G)$. 
		
		If, on the other hand, $\psi$ belongs to the lower class, then for every $\ep>0$ there exists a measurable set $\^J_\ep\sub J(\^f)$ with $\mu_t(\^J_\ep)\geq 1-\ep$ and a constant $r_\ep\in(0,1]$ such that 
		\begin{align}\label{2puz3}
			\frac{m_t\lt(\^J_\ep\cap p_2^{-1}(B_2(x,r))\rt)}{r^{\HD(\mu_t\circ p_2^{-1})}\exp\lt(\^\sg_t\chi_{\mu_t}^{-1/2}\psi\(\log\frac{1}{r}\)\sqrt{\log\frac{1}{r}}\rt)}\leq \ep
		\end{align}
		for all $x\in p_2(\^J_\ep)$ and all $r\in(0,r_\ep]$.
\end{lem}
\begin{proof}
	We shall prove the first formula \eqref{1puz3}. Fix $x=p_2(\pi_\cU(\xi))$, where $\xi\in\cD_\cU^\infty$ (then  $\xi\in J(G)$) and for every integer $n\geq 1$ set
	$$
		r_n:=\diam\lt(p_2(\phi_{\xi\rvert_n}\(X_{t(\xi_n)})\)\rt).
	$$
	Then 
	$$
		B_2(x,r_n)\bus p_2\((\phi_{\xi\rvert_n}(X_{t(\xi_n)})|),
	$$
	and so, 
	\begin{align}
		m_t\circ p_2^{-1}(B_2(x,r_n))
		&\geq m_t\(\phi_{\xi\rvert_n}\(X_{t(\xi_n)}\)\)
		=\int_{X_{t(\xi_n)}}e^{-\P(t)N_\zt(\xi)}|\phi_{\xi\rvert_n}'|^t\,dm_t
		\nonumber\\
		&\geq Q^{-1}K^{-t}\exp\((-tS_n\zt-\P(n)N_\zt)(\xi)\)\label{1puz4}
	\end{align}
	where $K\geq 1$ is the Koebe's Distortion Constant and 
	$$
		Q:=\min\{m_t(X_s):s\in S \}>0.	
	$$ 
	Put 
	$$
		\ka_t:=\HD(\mu_t\circ p_2^{-1}).
	$$
	Using \eqref{1puz4} and denoting by $S_n$, $n\geq 1$, the Birkhoff sum corresponding to the dynamical system $\sg:\cD_\cU^\infty\to\cD_\cU^\infty$, we get 
	\begin{align}
		&\frac{m_t\circ p_2^{-1}(B_2(x,r_n))}{r_n^{\ka_t}\exp\lt(\^\sg_t\^\chi_t^{-1/2}\psi\lt(\log\frac{1}{r_n}\rt)\sqrt{\log\frac{1}{r_n}}\rt)}
		\geq 
		\frac{Q^{-1}K^{-t}\exp\((-tS_n\zt-\P(t)N_\zt)(\xi)\)}{r_n^{\ka_t}\exp\lt(\^\sg_t\^\chi_t^{-1/2}\psi\lt(\log\frac{1}{r_n}\rt)\sqrt{\log\frac{1}{r_n}}\rt)}
		\nonumber\\
		&\qquad\geq
		\frac{Q^{-1}K^{-t}\exp\((-tS_n\zt-\P(t)N_\zt)(\xi)\)}{K^{\ka_t}\exp\(-\ka_tS_n\zt(\xi) \)\exp\lt(\^\sg_t\^\chi_t^{-1/2}\psi\lt(\log\frac{1}{r_n}\rt)\sqrt{\log\frac{1}{r_n}}\rt)}
		\label{2puz4}\\
		&\qquad=
		\lt(QK^{t+\ka_t}\rt)^{-1}
		\exp\Big(S_n\((\ka_t-t)\zt-\P(t) N_\zt\)(\xi)\Big)
		\exp\lt(-\^\sg_t\^\chi_t^{-1/2}\psi\lt(\log\frac{1}{r_n}\rt)\sqrt{\log\frac{1}{r_n}}\rt)
		\nonumber\\
		&\qquad=
		\lt(QK^{t+\ka_t}\rt)^{-1}
		\exp\lt(S_n\((\ka_t-t)\zt-\P(t) N_\zt\)(\xi) -\^\sg_t\^\chi_t^{-1/2}\psi\lt(\log\frac{1}{r_n}\rt)\sqrt{\log\frac{1}{r_n}}\rt)
		\nonumber.
	\end{align}
	Now, we estimate $\log(1/r_n)$ from above. By definition, we have 
	\begin{align}\label{1puz5}
		\log\frac{1}{r_n}\leq \log K+\log D+S_n\zt(\xi).
	\end{align}
Now, Birkhoff's Ergodic Theorem asserts that for $\^\mu_t$--a.e. $\xi\in\cD_\cU^\infty$, 
	\begin{align}\label{4puz5}
		\lim_{n\to\infty}\frac{1}{n}S_n\zt(\xi)=\int_{\cD_\cU^\infty}\zt \, d\^\mu_t=\^\chi_t.
	\end{align}
	Hence, there exists a measurable set $Y_1\sub\cD_\cU^\infty$ such that $\^\mu_t(Y_1)=1$ and for every $\eta>0$ and every $\xi\in Y_1$ there exists $n_1'(\xi,\eta)\geq 1$ such that 
	$$
		S_n\zt(\xi)\leq \lt(\^\chi_t+\frac{1}{2}\eta\rt)n
	$$
	for all $n\geq n_1'(\xi,\eta)$. Combining this with \eqref{1puz5}, we see that there exists $n_1(\xi,\eta)\geq n_1'(\xi_\eta)$ such that 
	\begin{align}\label{3puz5}
		\log\frac{1}{r_n}\leq (\^\chi_t+\eta)n
	\end{align}
	for all $n\geq n_1(\xi,\eta)$. 
	In fact, in what follows, we will need a better estimate on $\log(1/r_n)$. Indeed, since the function $u\longmapsto3\sqrt{u\log^2u}$ belongs to the lower class, by \eqref{1puz5} and Theorem~\ref{t2puz3}, there exists a measurable set $Y_2\sub Y_1$ such that $\^\mu_t(Y_2)=1$ and 
	\begin{align}\label{2puz5}
		\log\frac{1}{r_n}\leq\^\chi_tn+3(\sg^2_t(\zt)+1)\sqrt{n\log^2n}
	\end{align}
	for every integer $n\geq 1$ large enough, say $n\geq n_2(\xi)\geq n_1(\xi,\eta)$. As a matter of fact, we just derived \eqref{2puz5} assuming that $\sg_t^2(\xi)>0$. If however $\sg_t^2(\xi)=0$, then by an appropriate theorem from \cite{mugdms} (or \cite{PUZ}) $\xi$ is cohomologous to a constant in the class of bounded H\"older continuous functions on $\cD_\cU^\infty$, and then the factor $3(\sg^2_t(\zt)+1)\sqrt{n\log^2n}$ can be replaced by a constant (so, much better) in \eqref{2puz5}.
	One would like to apply Theorem~\ref{t2puz3} to inequality \eqref{2puz4} to conclude the proof. However, neither \eqref{3puz5} nor even \eqref{2puz5} are sufficiently strong for such an argument. What we need is to make use of Lemma~\ref{l6.1uifspuz}. So, consider the function
	$$
\th(u)
:=\psi(u)\lt(\sqrt{1+\frac{3(\sg_t^2(\xi)+1)}{\^\chi_t}\sqrt{\frac{\log^2\lt((\^\chi_t+\eta)^{-1}u \rt)}{(\^\chi_t+\eta)^{-1}u}}}-1 \rt)+(\^\chi_t+\eta)^{\frac{1}{4}}u^{-\frac{1}{4}}.
	$$
	Since $\th>0$ and since $\psi$ is a slowly growing function, we have that
	$$
		\lim_{u\to\infty}\psi(u)\th(u)=0. 
	$$
	Therefore, by Lemma~\ref{l6.1uifspuz}, there exists $\psi_+$, a function in the upper class, such that
	\begin{align}\label{1puz6}
		\psi_+(u)\geq \psi((\^\chi_t+\eta)u)+\th((\^\chi_t+\eta)u).
	\end{align}
By Theorem~\ref{TVP}, i.e. the Variational Principle, we have that 
	\begin{align}
		\int_{\cD_\cU^\infty}\((\ka_t-t)\zt-\P(t)N_\zt\)d\^\mu_t
		&=(\ka_t-t)\^\chi_t-\P(t)\frac{1}{\mu_t(J_\cU)}
		=\frac{1}{\mu_t(J_\cU)}\((\ka_t-t)\chi_{\mu_t}-\P(t)\) \nonumber\\
		&=\frac{1}{\mu_t(J_\cU)}\lt(\lt(\frac{\h_{\mu_t}}{\chi_{\mu_t}}-t\rt)\chi_{\mu_t}-\P(t)\rt) 
		\nonumber\\
		&=\frac{1}{\mu_t(J_\cU)}\(\h_{\mu_t}-t\chi_{\mu_t}-\P(t)\)\nonumber\\
		&=0.\label{2puz6}
	\end{align}
	Hence, inserting this and \eqref{1puz6} to Theorem~\ref{t2puz3}, and using \eqref{3puz5}, \eqref{2puz5}, and the fact that $\psi$ is monotone increasing, we obtain for infinitely many $u$ that
	\begin{align*}
		S_n\((\ka_t-t)&\zt-\P(t)\)(\xi)-\^\sg_t\^\chi_t^{-1/2}\psi(\log 1/r_n)\sqrt{1/r_n}\ge \\
		&\geq \^\sg_t\sqrt{n}\psi_+(n)-\^\sg_t\^\chi_t^{-1/2}\psi\((\^\chi_t+\eta)n\)\^\chi_t^{1/2}\sqrt{n}\sqrt{1+\frac{3(\sg_t^2(\zt)+1)}{\^\chi_t}\sqrt{\frac{\log^2n}{n}}}
		\\
		&\geq \^\sg_t\sqrt{n}n^{-1/4}
		=\^\sg_t n^{1/4}.
	\end{align*} 
Inserting this to \eqref{2puz4}, we get 
\begin{equation}\label{1puz7}
	\begin{aligned}
		\limsup_{r\to 0}&\frac{m_t\circ p_2^{-1}(B_2(x,r))}{r^{\ka_t}\exp\lt(\^\sg_t\^\chi_t^{-1/2}\psi\(\log(1/r)\)\sqrt{\log(1/r)}\rt)}
		 \\
&\geq \limsup_{n\to \infty}\frac{m_t\circ p_2^{-1}(B_2(x,r_n))}{r_n^{\ka_t}\exp\lt(\^\sg_t\^\chi_t^{-1/2}\psi\(\log(1/r_n)\)\sqrt{\log(1/r_n)}\rt)}
		\\
&\geq \lt(QK^{t+\ka_t}\rt)^{-1}\limsup_{n\to\infty}\exp\lt(\^\sg_tn^{1/4}\rt)\\
&=+\infty.
	\end{aligned}
	\end{equation}
	Since, by Theorem~\ref{t4h65}, the measures $\mu_t\circ p_2^{-1}$ and $m_t\circ p_2^{-1}$ are equivalent, formula \eqref{1puz3} of Lemma~\ref{l1puz3} results from \eqref{1puz7} for $\mu_t\circ p_2^{-1}$--a.e. $x\in J(G)$.  

We now pass to the proof of \eqref{2puz3}. For ever $\xi\in\cD_\cU^\infty$ and every radius $r>0$ let $n=n(\xi,r)\geq 0$ be the first integer such that 
	\begin{align}\label{3puz7}
		\diam_\C\lt(p_2\lt(\phi_{\xi\rvert_{n+1}}\lt(X_{t(\xi_{n+1})}\rt)\rt)\rt)<r.
	\end{align}
	Of course 
	\begin{align}\label{5puz7}
		\lim_{r\to\infty}n(\xi,r)=+\infty,
	\end{align}
	and therefore there exists $r_1(\om)>0$ such that for all $r\in(0,r_1(\om))$, we have $n(\xi,r)\geq 2$; so
	\begin{align}\label{4puz7}
	\diam_\C\lt(p_2\lt(\phi_{\xi\rvert_{n}}\lt(X_{t(\xi_{n})}\rt)\rt)\rt)\geq r.
	\end{align}
Hence,
$$
r\lt|f_{\^\xi\rvert_n}'\(p_2(\pi_\cU(\xi))\)\rt| 
\le K\diam_{\Sg_u\times\C}\(X_{t(\xi_{n})}\)
\le 2KR,
$$
where $R>0$ comes from Definition \ref{d2nsii6}. So, fix $r\in(0,r_1(\om))$. It then follows from the definition \ref{d2nsii6} of nice families that 
\beq\label{120190511}
\dist_\C\(f_{\^\xi\rvert_n}\(p_2\(\pi_\cU(\sg^n(\xi)))\),\PCV(G)\)
\ge 64Kr\lt|f_{\^\xi\rvert_n}'\(p_2(\pi_\cU(\xi))\) \rt|
\eeq
and, consequently, the holomorphic inverse branch 
	$$
		f_{\^\xi\rvert_n}^{-1}:B_2\lt(f_{\^\xi\rvert_n}\(p_2\(\pi_\cU(\sg^n(\xi)))\),16Kr\lt|f_{\^\xi\rvert_n}'\(p_2(\pi_\cU(\xi))\) \rt| \rt) \lra \C
	$$
	sending $f_{\^\xi\rvert_n}\(p_2(\pi_\cU(\sg^n(\xi)))\)$ to $p_2\(\pi_\cU(\sg^n(\xi))\)$ is well defined, whence 
	\begin{align*}
		&m_t\lt([\^\xi\rvert_n] \times f_{\^\xi\rvert_n}^{-1}\lt(B_2\lt(f_{\^\xi\rvert_n}\(p_2(\pi_\cU(\sg^n(\xi)))\), 8r\lt|f_{\^\xi\rvert_n}'\(p_2(\pi_\cU(\xi))\) \rt| \rt)\rt)\rt)\\
		&\;\lek \lt|\phi_{\xi\rvert_n}'(\pi_\cU(\sg^n(\xi)))\rt|^t e^{-\P(t)n}
		%\\
		%&\quad
		 m_t\bigg(\Sg_u\times f_{\^\xi\rvert_n}^{-1}\lt(B_2\lt(f_{\^\xi\rvert_n}
\(p_2\(\pi_\cU(\sg^n(\xi))\)\rt), 8r\lt|f_{\^\xi\rvert_n}'(p_2\(\pi_\cU(\xi))\) \rt| \rt)\bigg)\\
		&\;\leq \exp(-tS_n\zt(\xi)-\P(t)n).
	\end{align*}
	Hence
	\begin{equation}\label{2puz7}
	\begin{aligned}
		&\frac
		{m_t\bigg([\^\xi\rvert_n] \times f_{\^\xi\rvert_n}^{-1}\lt(B_2\lt(f_{\^\xi\rvert_n}\lt(p_2\lt(\pi_\cU(\sg^n(\xi))\rt)\rt), 8r\lt|f_{\^\xi\rvert_n}'(p_2(\pi_\cU(\xi))) \rt| \rt)\rt)\bigg)}
		{r^{\ka_t}\exp\lt(\^\sg_t\^\chi_t^{-1/2}\psi(\log(1/r_n))\sqrt{\log(1/r_n)}\rt)}
		\\
		&\quad
		\lek\frac
		{\exp\lt(-tS_n\zt(\xi)-P(t)n\rt)}
		{\diam^{\ka_t}\lt(p_2\lt(\phi_{\xi\rvert_{n}}\lt(\chi_{t(\xi_{n})}\rt)\rt)\rt)\exp\lt(\^\sg_t\^\chi_t^{-1/2}\psi(\log(1/r))\sqrt{\log(1/r)}\rt)}
		\\
		&\quad
		\lek\frac
		{\exp\lt(-tS_n\zt(\xi)-P(t)n\rt)}
		{\big|\phi_{\xi\rvert_{n+1}'(\pi_\cU(\sg^{n+1}(\xi)))}\big|^{\ka_t}\exp\lt(\^\sg_t\^\chi_t^{-1/2}\psi(\log(1/r))\sqrt{\log(1/r)}\rt)} 
		\\
		&\quad
		=\exp\lt(\((\ka_t-t)S_n\zt-\P(t)N_\zt\)(\xi)-\^\sg_t\^\chi_t^{-1/2}\psi(\log(1/r))\sqrt{\log(1/r)}+\ka_t\zt(\sg^{n+1}(\xi))\rt).
	\end{aligned}
	\end{equation}
	Now, because \eqref{4puz5} there exists a measurable set $Y_2\sub\cD_\cU^\infty$ such that $\^\mu_t(Y_2)=1$ and for every $\eta>0$ and every $\xi\in Y_2$ there exists $n_2'(\xi,\eta)\geq 1$ such that 
	$$
		S_k\zt(\xi)\geq \lt(\^\chi_t-\frac{1}{2}\eta\rt)k
	$$	
	for all $k\geq n_2'(\xi,\eta)$. Combining this with \eqref{4puz7}, we see that 
	$$
		\log(1/r)\geq -\log D+S_n\zt(\xi)\geq -\log D+\lt(\^\chi_t-\frac{1}{2}\eta\rt)n,
	$$
	whenever $\xi\in Y_2$ and $n=n(\xi,r)\geq n_2'(\xi,\eta)$. Hence, there exists $n_2(\xi,\eta)\geq n_2'(\xi,\eta)$ such that 
	\begin{align}\label{1puz8}
		\log(1/r)\geq \lt(\^\chi_t-\eta \rt)n
	\end{align}
	whenever $\xi\in Y_2$ and $n=n(\xi,r)\geq n_2(\xi,\eta)$. As in the proof of \eqref{1puz3}, we will need a better estimate on $\log(1/r_n)$. Indeed, since the function 
	$u\longmapsto3\sqrt{u\log^2u}$
	belongs to the lower class, by \eqref{1puz5} and Theorem~\ref{t2puz3}, there exists a measurable set $Y_3\sub Y_2$ such that $\^\mu_t(Y_3)=1$ and 
	\begin{align}\label{2puz8}
		\log\frac{1}{r}\geq \^\chi_tn-3\sg_t(\xi)\sqrt{n\log^2n}
	\end{align}
if the integer $n\geq 1$ is large enough, say $n\geq n_3(\xi)\geq n_2(\xi,\eta)$. As a matter of fact we just derived \eqref{2puz8} assuming that $\sg_t^2(\xi)>0$. If however $\sg_t^2=0$, then by an appropriate theorem from \cite{mugdms} (or \cite{PUZ}) $\xi$ is cohomologous to a constant in the class of bounded H\"older continuous functions on $\cD_\cU^\infty$, and then the $2(\sg_t^2(\xi)+1)\sqrt{n\log^2n}$ term can be replaced by a constant (so, much better) in \eqref{2puz5}. Now again, as in the proof of formula \eqref{1puz3}, we consider an appropriate function $\th$. It is now defined by 
	$$
	\th(u):=\psi(u)\lt(1-\sqrt{1-\frac{3\sg_t(\zt)}{\^\chi_t}\sqrt{\frac{\log^2\lt((\^\chi_t-\eta)^{-1}u \rt)}{(\^\chi_t-\eta)^{-1}u}}} \rt)
	+(\^\chi_t-\eta)^{\frac{1}{4}}u^{-\frac{1}{4}}+(\^\chi_t-\eta)^{\frac{1}{8}}u^{-\frac{1}{8}}\psi^{\frac{3}{4}}(u).
	$$
	Since $\th>0$ and since $\psi$ is a slowly growing function, we have that
	$$
		\lim_{u\to\infty}\psi(u)\th(u)=0.
	$$
	Therefore, by Lemma~\ref{l6.1uifspuz}, there exists $\psi_-$, a function in the lower class such that
	\begin{align}\label{1puz9}
		\psi_-(u)\leq \psi((\^\chi_t-\eta)u)-\th((\^\chi_t-\eta)u).
	\end{align}
	Using \eqref{2puz6}, \eqref{1puz8}, and \eqref{2puz8}, it follows from Theorem~\ref{t2puz3} that 
	\begin{align}
		(\ka_t-t)S_n\(\zt-&\P(t)N_\zt\)(\xi)-\^\sg_t\^\chi_t^{-1/2}\psi(\log1/r)\sqrt{\log 1/r} \le
		\nonumber\\
		&\quad\leq \^\sg_t\sqrt{n}\psi_-(n)-\^\sg_t\^\chi_t^{-1/2}\psi((\^\chi_t-\eta)n)\^\chi_t^{1/2}\sqrt{n}\sqrt{1-\frac{3\sg_t(\xi)}{\^\chi_t}\sqrt{\frac{\log^2n}{n}}} \label{2puz9}
	\end{align}
	for $\^\mu_t$--a.e. $\xi\in\cD_\cU^\infty$ and all $n\geq 1$ large enough, say for all $\xi\in Y_4\sub Y_3$ with $\^\mu_t(Y_4)=1$ and all $n\geq n_4(\xi)\geq n_3(\xi)$. We also need to take care of the term $\ka_t\zt(\sg^{n+1}(\xi))$ appearing at the end of \eqref{2puz7}.  Since the function $|\zt|$ has all moments, in particular since 
	$$
		\int|\zt|^3 d\^\mu_t<+\infty,
	$$
	it follows from Tchebyshev's Inequality that 
	\begin{align*}
		\^\mu_t\Big(\Big\{\xi\in\cD_\cU^\infty:\ka_t|\zt(\xi)|\geq \^\sg_t\psi^{3/4}\((\^\chi_t-\eta)k\)k^{3/8} \big\}\Big)
		\leq\^\sg_t^{-3}\psi^{-1/4}((\^\chi_t-\eta)k)k^{-9/8}\int_{\cD_\cU^\infty}|\zt|^3 d\^\mu_t.
	\end{align*}
	Since the measure $\^\mu_t$ is $\sg$--invariant, we therefore obtain
	\begin{align*}
\sum_{k=1}^\infty \^\mu_t\Big(\Big\{\xi\in\cD_\cU^\infty:&\ka_t|\zt(\sg^{k+1}(\xi))|\geq \^\sg_t\psi^{3/4}\((\^\chi_t-\eta)k\)k^{3/8} \Big\}\Big)\le \\
		&\leq\^\sg_t^{-3}\sum_{k=1}^\infty\psi^{-1/4}\((\^\chi_t-\eta)k\)\^\mu_t(|\zt|^3)k^{-9/8}\\
		&\leq\^\sg_t^{-3}\psi^{-1/4}\((\^\chi_t-\eta)\)\^\mu_t(|\zt|^3)\sum_{k=1}^\infty k^{-9/8}\\
		&<+\infty.
	\end{align*}
	Therefore, in view of the Borel--Cantelli Lemma there exists a measurable set $Y_5\sub Y_4$ such that $\^\mu_t(Y_5)=1$ and for all $\xi\in Y_5$ there exists an integer $n_5(\xi)\geq n_4(\xi)$ such that 
	$$
		 \ka_t|\zt(\sg^{k+1}(\xi))|\leq\^\sg_t\psi^{3/4}((\^\chi_t-\eta)k)k^{3/8}
	$$
for all integers $k\ge n_5(\xi)$. 	Along with \eqref{2puz9} and \eqref{1puz9}, this yields
	\begin{align*}
		\((\ka_t-&t)S_n\zt-\P(t)N_\zt\)(\xi)-\^\sg_t\^\chi_t^{-1/2}\psi(\log(1/r))\sqrt{\log(1/r)}+\ka_t\zt|(\sg^{n+1}(\xi))|
		\\
		&\quad
		\leq \^\sg_t\sqrt{n}\psi_-(n)-\^\sg_t\^\chi_t^{-1/2}\psi((\^\chi_t-\eta)n)\^\chi_t^{1/2}\sqrt{n}\sqrt{1-\frac{3\sg_t(\xi)}{\^\chi_t}\sqrt{\frac{\log^2n}{n}}}+
		\\
		&\qquad\qquad\qquad\qquad
		+\^\sg_t\psi^{3/4}((\^\chi_t-\eta)n)n^{3/8}
		\\
		&\quad
		\leq  \^\sg_t\sqrt{n}\lt(\psi((\^\chi_t-\eta)n)-\th((\^\chi_t-\eta)n)-\psi((\^\chi_t-\eta)n)\sqrt{1-\frac{3\sg_t(\xi)}{\^\chi_t}\sqrt{\frac{\log^2n}{n}}} \rt)
		\\
		&\qquad\qquad\qquad\qquad
		+\^\sg_t\psi^{3/4}((\^\chi_t-\eta)n)n^{3/8}
		\\
		&\quad
		=\^\sg_t\psi^{3/4}((\^\chi_t-\eta)n)n^{3/8}-\^\sg_t\sqrt{n}n^{-1/4}-\^\sg_t\sqrt{n}\psi^{3/4}((\^\chi_t-\eta)n)n^{1/8}
		\\
		&\quad
		=-\^\sg_tn^{1/4}
	\end{align*}
	for all $\xi\in Y_5$ and $n=n(\xi,r)\geq n_5(\xi)$. Inserting this to \eqref{2puz7}, we get that
	\begin{align}
		\frac
		{m_t\lt([\^\xi\rvert_n] \times f_{\^\xi\rvert_n}^{-1}\lt(B_2\lt(f_{\^\xi\rvert_n}\(p_2\(\pi_\cU(\sg^n(\xi))\)\), 8r\lt|f_{\^\xi\rvert_n}'\(p_2(\pi_\cU(\xi))\) \rt| \rt)\rt)\rt)}
		{r^{\ka_t}\exp\lt(\^\sg_t\^\chi_t^{-1/2}\psi(\log(1/r_n))\sqrt{\log(1/r_n)}\rt)}
		\leq\exp\lt(-\^\sg_tn^{1/4}(\xi,r)\rt)\label{1puz10}
	\end{align}
	for every $\xi\in Y_5$ and every $r>0$ so small that $n(\xi,r)\geq n_5(\xi)$ (see also \eqref{5puz7}). Now, fix $\ep\in(0,1)$ and take $k\geq 1$ so large that 
	$$
		\^\mu_t\lt(Z_\ep':=Y_5\cap n_5^{-1}\lt(\{0,1,\dots,k \}\rt)\rt)\geq 1-\frac{\e}{2}
	$$
and
	\begin{align}\label{5puz13}
		\exp\lt(-\^\sg_tk^{1/4}\rt)\leq \frac{\ep}{\#(16K^2,a)},
	\end{align}
	where $a$ is given by the formula 
\begin{align}\label{6puz13}
		a:=32K^2R,
	\end{align}
with, we recall, $R>0$ comes from Definition \ref{d2nsii6} 
and $\#(16K^2,a)$ comes from Proposition~\ref{p1h2a.1}. 
By \eqref{5puz7} there exists $r(\ep)>0$ so small that 
	\begin{align}\label{3puz13}
		\^\mu_t\big(Z_\ep:=\{\xi\in Z_\ep': n(\xi,r(\ep))\geq k \} \big)\geq 1-\ep.
	\end{align}
	Fix $r\in(0,r(\ep)]$. Our goal now is to apply Proposition~\ref{p1h2a.1}. For every $x\in p_2(\pi_\cU(Z_\ep))$, we define 
	$$
		\cF_0(x,r):=\big\{\^\xi\rvert_{n(\xi,r)+1}:\xi\in Z_\ep\cap(p_2\circ\pi_\cU)^{-1}(B_2(x,r)) \big\}.		
	$$ 
	Furthermore, we take $\cF(x,r)$ to be a family of mutually incomparable elements of $\cF_0(x,r)$ such that
	\begin{align}\label{4puz13}
		\bigcup_{\tau\in\cF(x,r)}[\tau]=\bigcup_{\tau\in\cF_0(x,r)}[\tau].
	\end{align}
This is the largest family of mutually incomparable words in $\cF_0(x,r)$.
	For every $\tau\in\cF(x,r)$ fix $\tau^+\in Z_\ep\cap(p_2\circ\pi_\cU)^{-1}(B_2(x,r))$ such that 
	\begin{align}\label{2puz13}
		\^\tau^+\rvert_{|\tau|}=\tau.
	\end{align}
	Define 
	$$
		R_\tau:=16Kr\big|f_{\hat\tau}'(p_2(\pi_\cU(\tau^+)))\big|.
	$$
The condition \eqref{ess2} is obviously satisfied by the definition of $\cF(x,r)$. 
We now deal with condition \eqref{ess0}. Keeping $\tau\in\cF(x,r)$, it follows from $\frac14$--Koebe's Distortion Theorem that
$$
f_{\hat\tau}^{-1}\Big(B_2\(f_{\hat\tau}(p_2\circ\pi_\cU(\tau^+)), 4r\big|f_{\hat\tau}'(p_2(\pi_\cU(\tau^+)))\big|\)\Big)
\spt B_2(p_2\circ\pi_\cU(\tau^+),r).
$$
Hence,
\beq\label{120190426}
x\in f_{\hat\tau}^{-1}\Big(B_2\(f_{\hat\tau}(p_2\circ\pi_\cU(\tau^+)), 4r\big|f_{\hat\tau}'(p_2(\pi_\cU(\tau^+)))\big|\)\Big).
\eeq
Consequently,
$$
f_{\hat\tau}(x)\in B_2\(f_{\hat\tau}(p_2\circ\pi_\cU(\tau^+)), 4r\big|f_{\hat\tau}'(p_2(\pi_\cU(\tau^+)))\big|\)
=B_2\(f_{\hat\tau}(p_2\circ\pi_\cU(\tau^+)),(4K)^{-1}R_\tau\).
$$
But 
$$
f_{\hat\tau}(p_2\circ\pi_\cU(\tau^+))
\in f_{\tau_*}^{-1}(p_2\circ\pi_\cU(\tau^+))
\in f_{\tau_*}^{-1}(J(G)).
$$
Thus, 
$$
\dist\(f_{\hat\tau}(x),f_{\tau_*}^{-1}(J(G))\)\le \frac12R_\tau,
$$
and, consequently, \eqref{ess0} holds because of Remark~\ref{r120190418}. It follows from \eqref{120190511} that
$$
\begin{aligned}
\dist\(f_{\hat\tau}(x),\PCV(G)\)
&\ge \dist\(f_{\hat\tau}(p_2\circ\pi_\cU(\tau^+)),\PCV(G)\) 
-\big|f_{\hat\tau}(x)-f_{\hat\tau}(p_2\circ\pi_\cU(\tau^+))\big| \\
&\ge 64Kr\big|f_{\hat\tau}'(p_2(\pi_\cU(\tau^+)))\big|-4r\big|f_{\hat
  \tau}'(p_2(\pi_\cU(\tau^+)))\big| \\
&=60Kr\big|f_{\hat\tau}'(p_2(\pi_\cU(\tau^+)))\big| \\
&>2R_\tau.
\end{aligned} 
$$
Hence, a unique analytic branch $f_{\hat{\tau},x}^{-1}:
B_2(f_{\hat{\tau}}(x),2R_\tau)\lra\C$ of $f_{\hat{\tau}}^{-1}$ sending $f_{\hat{\tau}}(x)$ to $x$ exists. It also follows from \eqref{120190426} that
$$
K^{-1}
\le \frac{\big|f_{\hat\tau}'(x)\big|}{\big|f_{\hat\tau}'(p_2(\pi_\cU(\tau^+)))\big|}
\le K.
$$
Therefore,
$$
 \frac{1}{16K^2}
\le \frac{\big|f_{\hat\tau}'(x)\big|r}{R_\tau}
\le \frac{1}{16}.
$$
In conclusion, condition \eqref{ess1} of Definition~\ref{d1h2a.1} is satisfied with $M=16K^2$ and, by \eqref{4puz13}, with $a$ given by \eqref{6puz13}. So, $\cF(x,r)$ is $(16K^2,a,V)$--essential for $(x,r)$ with 
	$$
		V:=\bigcup_{\tau\in\cF(x,r)}[\tau].
	$$
In consequence, Proposition~\ref{p1h2a.1} applies, and, in particular, its item \eqref{p1h2a.1 item c} yields
	\begin{align}\label{1puz13}
		\#\cF(x,r)\leq\#(16K^2,a).
	\end{align} 
	We want to show that 
	\begin{align}
		&\pi_\cU\lt(Z_\ep\cap(p_2\circ\pi_\cU)^{-1}(B_2(x,r))\rt)\sbt 
		\nonumber\\
		&\qquad\qquad
\sub \bigcup_{\tau\in\cF(x,r)}[\hat\tau]\times f_{\hat\tau}^{-1}\Big(B_2\lt(p_2\(\pi_\cU(\sg^{|\hat\tau|}(\tau^+))\), 8r\lt|f_{\hat\tau}'(p_2(\pi_\cU(\tau^+))) \rt| \rt)\Big).\label{1puz14}
	\end{align}
	Indeed, let $\xi\in Z_\ep\cap(p_2\circ\pi_\cU)^{-1}(B_2(x,r))$. By \eqref{4puz13} there exists $\xi'\in Z_\ep\cap(p_2\circ\pi_\cU)^{-1}(B_2(x,r))$ such that 
	$$
		\xi'\rvert_{n(\xi',r)+1}\in\cF(x,r),
	$$
	\begin{align}\label{3puz14}
		n(\xi,r)\geq n(\xi',r),
	\end{align}
	and
	\begin{align}\label{2puz14}
		\xi\rvert_{n(\xi',r)+1}=\xi'\rvert_{n(\xi',r)+1}.
	\end{align}
	Denoting $n(\xi',r)$ by $n$ and $\xi'\rvert_{n(\xi',r)+1}$ by $\g$, it follows from the $\frac14$--Koebe's Distortion Theorem that 
	\begin{align*}
	%	&
f_{\hat\ga}^{-1}\lt(B_2\lt(p_2\(\pi_\cU(\sg^{|\hat\g|}(\g^+)\), 8r\lt|f_{\hat\g}'(p_2(\pi_\cU(\g^+))) \rt| \rt)\rt)
	%	\\
	%	&\qquad
		\bus B_2\(p_2(\pi_\cU(\g^+)),2r\)\ni p_2(\pi_\cU(\xi)).
	\end{align*}
	Since, also by \eqref{2puz14} and \eqref{3puz14},
	$$
		p_1(\pi_\cU(\xi))=\^\xi\in[\^\xi\rvert_{n(\xi,r)}]\sub[\xi'\rvert_{n(\xi',r)}]=[\hat\g],
	$$
	we conclude that \eqref{1puz14} holds. This formula, along with \eqref{1puz10}, \eqref{5puz13}, \eqref{3puz13}, and \eqref{1puz13}, gives
	$$
		\frac{m_t\lt(\pi_\cU\lt(Z_\ep\cap(p_2\circ\pi_\cU)^{-1}(B_2(x,r))\rt) \rt)}{r^{\ka_t}\exp\lt(\^\sg_t\^\chi_t^{-1/2}\psi\(\log(1/r)\)\sqrt{\log(1/r)}\rt)}\leq \ep.
	$$
	But 
	$$
		\pi_\cU\lt(Z_\ep\cap(p_2\circ\pi_\cU)^{-1}(B_2(x,r))\rt)=\pi_\cU(Z_\ep)\cap p_2^{-1}(B_2(x,r)),
	$$ 
	so 
	\begin{align}\label{4puz14}
		\frac{m_t\lt(\pi_\cU(Z_\ep)\cap p_2^{-1}(B_2(x,r))\rt)}
		{r^{\ka_t}\exp\lt(\^\sg_t\^\chi_t^{-1/2}\psi\(\log(1/r)\)\sqrt{\log(1/r)}\rt)}\leq \ep.
	\end{align}
	Since, also by \eqref{3puz13}, 
	$$
		\mu_t(\pi_\cU(Z_\ep))\geq\^\mu_t(Z_\ep)\geq 1-\ep,
	$$
	formula \eqref{2puz3} of Lemma~\ref{l1puz3} follows by setting $\^J_\ep:=\pi_\cU(Z_\ep)$. The proof is now complete.
\end{proof}	

\sp Let $(X,\vr)$ is a metric space. Given also a monotone increasing function $g:[0,+\infty)\lra(0,+\infty)$, let for every subset $A$ of $X$, $\H_{g}(A)$\index{$\H_{g}(A)$} be the corresponding (generalized) \textbf{Hausdorff measure}\index{Hausdorff measure}\index{Hausdorff measure!$\H_{g}(A)$} of $A$. We recall that it is defined as the infimum over all $\d>0$ of the numbers $\H_{g}^{(\d)}(A)$,\index{Hausdorff measure!$\H_{g}^{(\d)}(A)$}\index{$\H_{g}^{(\d)}(A)$} where $\H_{g}^{(\d)}(A)$ is defined as the infimum of all countable sums
$$
\sum_{k=1}^\infty g(\diam(A_k)),
$$
where $\{A_k\}_{k=1}^\infty$ are countable covers of $A$ with $\diam(A_k)\leq \d$ for each $k\geq 1$. It is worth to note that also
$$
\H_{g}(A)=\lim_{\d\downto 0}\H_{g}^{(\d)}(A)
$$
and that the function $\H_{g}$ restricted to the Borel subsets of $X$ is a measure with values in $[0,+\infty]$. 

\sp Given $\psi:[0,+\infty)\lra(0,+\infty)$ and $t\in\De_G^*$, let 
$$
\^\psi_t:(1,+\infty)\lra(0,+\infty)
$$
be given by the formula
$$
\^\psi_t(u):=u^{\HD(\mu_t\circ p_2^{-1})}\exp\lt(\^\sg_t\chi_{\mu_t}^{-1/2}\psi\(\log (1/u)\)\sqrt{\log(1/u)} \rt),
$$

The proof of the following, main result of this section, is now fairly standard.

\begin{thm}\label{t1puz15}
Let $G$ be a \NOSC-FNR \ rational semigroup. Assume that $t\in\De_G^*$ and $\^\sg_t>0$. If $\psi:[0,+\infty)\to(0,+\infty)$ is a slowly growing function, then 
	\begin{enumerate}
		\item[\mylabel{a}{t1puz15 item a}] If $\psi$ is in the upper class, then the measures $\mu_t\circ p_2^{-1}$ and $\H_{\^\psi_t}$ on $J(G)$ are mutually singular, 
		\item[\mylabel{b}{t1puz15 item b}] If $\psi$ is in the lower class, then $\mu_t\circ p_2^{-1}$ is absolutely continuous with respect to $\H_{\^\psi_t}$ on $\^\C$. Moreover, $\H_{\^\psi_t}(E)=+\infty$ whenever $E\sub J(G)$ is a Borel set such that $\mu_t\circ p_2^{-1}(E)>0$.
	\end{enumerate}
\end{thm}
\begin{proof}
	To prove item \eqref{t1puz15 item a}, first fix $\ep\in (0,1)$. By formula \eqref{1puz3} of Lemma~\ref{l1puz3} and by Egorov's Theorem for every integer $n\geq 1$ there exists a Borel set $E_n(\ep)\sub J(G)$ such that 
	\begin{align}\label{1puz15}
		\mu_t\circ p_2^{-1}(E_n(\ep))>1-\ep2^{-n}
	\end{align}
	and for every $x\in E_n(\ep)$ there exists a closed ball $B_n(x)$ contained in $X$ with $\diam(B_n(x))<1/n$ and 
	$$
		m_t\circ p_2^{-1}(B_n(x))\geq n\^\psi(\diam(B_n(x))).
	$$
	Let $b(2)$ be the constant from Besicovi\v c's Theorem corresponding to dimension 2. Applying this theorem to the cover $\{B_n(x): x\in E_n(\ep) \}$ of $E_n(\ep)$, we get a countable set $S\sub E_n(\ep)$ such that 
	$$
		\bigcup_{x\in S}B_n(x)\bus E_n(\ep)
	$$
	and $S$ can be represented as a disjoint union $S_1\cup S_2\cup\dots\cup S_{b(2)}$ such that for every $1\leq k\leq b(2)$, the family $\{B_n(x): x\in S_k \}$ consists of mutually disjoint sets. Then 
	\begin{align*}
\H_{\^\psi_t}(E_n(\ep),1/n)
&\leq \sum_{x\in S}\^\psi_t\(\diam(B_n(x))\)
\leq \frac{1}{n}\sum_{x\in S} m_t\circ p_2^{-1}(B_n(x))\\
&=\frac{1}{n}\sum_{k=1}^{b(2)}\sum_{x\in S_k}m_t\circ p_2^{-1}(B_n(x))
=\frac{1}{n}\sum_{k=1}^{b(2)}m_t\circ p_2^{-1}\lt(\bigcup_{x\in S_k}B_n(x)\rt)\\
&\leq \frac{b(2)}{n}.
	\end{align*}
	So, if 
	$$
		E_\ep:=\bigcap_{n=1}^\infty E_n(\ep),
	$$
	then 
	$$
		\H_{\^\psi_t}(E_n(\ep))=0 
		\quad\text{and}\quad
		\mu_t\circ p_2^{-1}(E_n(\ep))\geq 1-\ep.
	$$
	Hence, setting 
	$$
		E:=\bigcup_{\ell=1}^\infty E_{1/\ell},
	$$
	we have $\H_{\^\psi_t}(E)=0$ and $\mu_t\circ p_2^{-1}(E)=1$. Thus, item \eqref{t1puz15 item a} is proved. 
	
	Now to prove item \eqref{t1puz15 item b} put 
	$$
		\eta:=\frac{1}{2}m_t\circ p_2^{-1}(E)>0,
	$$
	as by Theorem~\ref{t4h65}, the measures $m_t\circ p_2^{-1}$ and $\mu_t\circ p_2^{-1}$ are equivalent. Fix $\ka>0$ arbitrarily. Using Theorem~\ref{t4h65} again, it follows from Lemma~\ref{l1puz3} that there exists $\ep\in(0,\ka)$ such that $m_t(\^J_\ep)\geq 1-\eta$. Then
	\begin{align}\label{1puz17}
		m_t(\^J_\ep\cap p_2^{-1}(E))\geq m_t\circ p_2^{-1}(E)-\eta=\eta.
	\end{align}
	Fix $\d\in(0,r(\ep))$, where $r(\ep)>0$ also comes from Lemma~\ref{l1puz3}. Consider an arbitrary set $S\sub E\cap p_2(\^J_\ep)$ and for every $x\in S$ an arbitrary radius $r(x)\in(0,\d]$ such that 
	$$
		\bigcup_{x\in S}B_2(x,r(x))\bus E\cap p_2(\^J_\ep).
	$$
	We then get from formula \eqref{2puz3} of Lemma~\ref{l1puz3}, the choice of $\ep$, and \eqref{1puz3} that 
	\begin{align*}
		\sum_{x\in S}\^\psi_t(r(x))
		&\geq \frac{1}{\ep}\sum_{x\in S}m_t\lt(\^J_\ep\cap p_2^{-1}\(B_2(x,r(x))\)\rt)
		\geq \frac{1}{\ep}m_t\lt(\^J_\ep\cap\lt(\bigcup_{x\in S} p_2^{-1}\(B_2(x,r(x))\)\rt)\rt)\\
		&\geq \frac{1}{\ep}m_t\lt(\^J_\ep\cap p_2^{-1}\lt(\bigcup_{x\in S} B_2\(x,r(x)\)\rt)\rt)
		\geq \frac{1}{\ep}m_t\lt(\^J_\ep\cap p_2^{-1}\lt(E\cap p_2(\^J_\ep)\rt)\rt)\\
		&\geq \frac{1}{\ep}m_t\lt(\^J_\ep\cap p_2^{-1}(E)\rt)
		\geq\frac{\eta}{\ka}.
	\end{align*} 
	Hence, 
	$$
		\H_{\^\psi_t}(E\cap p_2(\^J_\ep),\d)\geq A\frac{\eta}{\ka},
	$$ 
	where $A\in (0,+\infty)$ is some universal constant, see \cite{mattila} or \cite{PUZ} for example. Therefore, 
	$$
		\H_{\^\psi_t}(E)
		\geq \H_{\^\psi_t}(E\cap p_2(\^J_\ep))
		\geq A\frac{\eta}{\ka}.
	$$
Since $\ka>0$ is arbitrary, this gives that $\H_{\^\psi_t}(E)=+\infty$, finishing the proof of item \eqref{t1puz15 item b} and of the whole Theorem~\ref{t1puz15}.
\end{proof}
Good examples of lower and upper class functions are ones of the form
$$
	\ell_c(u)=c\sqrt{\log^2u}, \quad c>0.
$$
Indeed, if $0\leq c\leq 2$, then $\ell_c$ belongs to the upper class, and if $c>2$, then $\ell_c$ belongs to the lower class. Therefore, as an immediate consequence of Theorem~\ref{t1puz15}, we get the following.

\begin{cor}\label{c1puz18}
Let $G$ be a \NOSC-FNR \ rational semigroup. Assume that $t\in\De_G^*$ and $\^\sg_t>0$. We have the following. 
	\begin{enumerate}
		\item[\mylabel{a}{c1puz18 item a}] If $0\leq c\leq 2$, then $\mu_t\circ p_2^{-1}$ and $\H_{(\^\ell_c)_t}$ on $J(G)$ are mutually singular. In particular, the measures $\mu_t\circ p_2^{-1}$ and $\H_{t^{\HD(\mu_t\circ p_2^{-1})}}$ are mutually singular.
		\item[\mylabel{b}{c1puz18 item b}] If $c>2$, then $\mu_t\circ p_2^{-1}$ is absolutely continuous with respect to $\H_{(\^\ell_c)_t}$ on $J(G)$. Moreover, $\H_{(\^\ell_c)_t}(E)=+\infty$ whenever $E\sub J(G)$ is a Borel set such that $\mu_t\circ p_2^{-1}(E)>0$.
	\end{enumerate}
\end{cor}

\section{$\HD(J(G))$ versus Hausdorff Dimension of Fiber Julia Sets $J_\om$, $\om\in\Sg_u$}\label{section:fiber-global}
In this section our goal is to relate the global Hausdorff dimension $\HD(J(G))$ with the Hausdorff dimension of the fibers $\HD(J_\om)$. We will show that for our systems, i.e. \CF\ balanced \TNR \ rational semigroups of finite type satisfying the Nice Open Set Condition, 
$$
\HD(J(G))>\HD(J_\om)
$$ 
for every $\om\in\Sg_u$. If in addition our semigroup is expanding, then 
$$
\sup\{\HD(J_\om):\om\in\Sg_u\}<\HD(J(G)).
$$
The concept expanding rational semigroups is well rooted in this theory for at least two decades. We recall its definition here. \index{expanding}

\bdfn\label{d120220905}
A rational semigroups is called {\bf expanding (along fibers)}\index{expanding} if and only if there exists an integer $n\ge 1$ such that 
$$
\big|\(\tf^n\)'(\xi)\big|\ge 2
$$
for all $\xi\in J(\tilde f)$. 

Equivalently, there are two constants $C>0$ and $\l>1$ such that 
$$
\big|\(\tf^n\)'(\xi)\big|\ge c\l^n
$$
for all $\xi\in J(\tilde f)$ and all integers $n\ge 0$. 
\edfn

\fr For more information about expanding rational functions the reader is advised to consult relevant papers of the second author.

First we shall prove the following auxiliary result. 

\begin{lem}\label{l1fj5.1}
Let $G$ be a \NOSC-FNR \ rational semigroup generated by a $u$--tuple map $f=(f_1,\dots,f_u)\in\Rat^u$. If $V\sbt \Sg_u\times\hat\C$ is a non--empty neighborhood of $\Crit_*(\^f)$, then 
	$$
		\HD(p_2(K(V)))<h=\HD(J(G)).
	$$
\end{lem}
\begin{proof}
	Have $R_2>0$ coming from Lemma~\ref{l2sp1}. Since the set $K(V)$ is compact there exists a finite set $\Xi\sbt K(V)$ such that 
	$$
		\bigcup_{\xi\in\Xi}B(\xi,R_2)\spt K(V).
	$$
	Let 
	\begin{align}\label{1fj5.1}
		\un{\P}_V^\Xi(t):=\liminf_{n\to\infty}\frac{1}{n}\log\sum_{\xi\in\Xi}\sum_{x\in K(V)\cap\^f^{-n}(\xi)}\lt|(\^f^n)'(x)\rt|^{-t}.
	\end{align}\index{$\un{\P}_V^\Xi(t)$}\index{topological pressure!$\un{\P}_V^\Xi(t)$}
	Then by convexity, the function $\R\ni t\longmapsto\un{\P}_V^\Xi(t)\in\R$ is continuous, and it follows from Lemma~\ref{l1sp1} that there exists $t\in(0,h)$ such that 
	\begin{align}\label{2fj5.1}
		\un{\P}_V^\Xi(t)<0.
	\end{align}
By Lemma~\ref{l1sp1}, for every $\xi\in\Xi$, every integer $n\geq 1$, and every $x\in\^f^{-n}(\xi)\cap K(V)$ there exists a unique holomorphic branch $\^f_x^{-n}:B(\xi,2R_2)\lra\Sg_u\times\C$ which sends $\xi$ to $x$. It now follows from \eqref{1fj5.1} and \eqref{2fj5.1} that 
	$$
		\liminf_{n\to\infty}\frac{1}{n}\log\sum_{\xi\in\Xi}\sum_{x\in K(V)\cap\^f^{-n}(\xi)}\diam^t\(p_2(\^f_x^{-n}(B(\xi,R_2)))\)=\un{P}_V^\Xi(t)<0.
	$$
	Therefore,
	$$
		\liminf_{n\to\infty}\sum_{\xi\in\Xi}\sum_{x\in K(V)\cap\^f^{-n}(\xi)}\diam^t\(p_2(\^f_x^{-n}(B(\xi,R_2)))\)=0.
	$$
	Since also
	$$
		\bigcup_{\xi\in\Xi}\bigcup_{x\in K(V)\cap \^f^{-n}(\xi)}p_2\(\^f_x^{-n}(B(\xi,R_2))\)\spt K(V),
	$$
	we thus conclude that $\H_t(p_2(K(V)))=0$. Hence, 
	$$
		\HD(p_2(K(V)))\leq t<h,
	$$
	and the proof is complete.
\end{proof}

Now we shall prove the following theorem which is one of the main results of our manuscript, probably the top one. It solves a long standing problem about the size relation between global and fiberwise Julia sets of a rational semigroup.

\begin{thm}\label{t1fj5}
If $G$ is a \NOSC-FNR rational semigroup generated by a $u$--tuple map $f=(f_1,\dots,f_u)\in\Rat^u$, then 
	$$
		\HD(J_\om)<h=\HD(J(G))
	$$
	for every $\om\in\Sg_u$. If in addition, $G$ is expanding, then 
	$$
		\sup\big\{\HD(J_\om):\om\in\Sg_u \big\}<h=\HD(J(G)).
	$$
\end{thm}
\begin{proof}
	Fix $\cU=\{U_s\}_{s\in S}$, an arbitrary nice family for $f$. For every $\rho\in \Sg_u^*\cup \Sg_u$, let 
	$$
		[\rho]^\sim:=\{\om\in\cD_\cU^\infty:\^\om\rvert_{[\rho]}=\rho \}.
	$$
	Still for every integer $n\geq 0$ and also for every $\tau\in\Sg_u$, we have 
	\begin{align}\label{1fj5}
		J_\tau(\cU)
	:&=p_2(\pi_\cU([\tau]^\sim))
		\sub p_2\circ\pi_\cU\lt(\bigcup\lt\{[\om\rvert_n]:\om\in[\tau\rvert_n]^\sim\rt\}\rt) \\
		&=p_2\circ\pi_\cU\lt(\bigcup\lt\{[\om]:\om\in\cD_\cU^n \text{ and }\^\om\rvert_n=\tau\rvert_n\rt\}\rt)\nonumber\\
		&\sub \bigcup\lt\{p_2(\phi_\om(X_{t(\om)})):\om\in[\tau\rvert_n]^\sim\rvert_n \rt\}.
	\end{align}
	Now, fixing for every $s\in X$, $\xi_s\in\cD_\cU^\infty$ such that $\pi_\cU(\xi_s)\in U_s\sbt X_s$, we have 
	\begin{align}
		\sum_{\om\in[\tau\rvert_n]^\sim\rvert_n}\diam^h(p_2(\phi_\om(X_{t(\om)})))
		\comp\sum_{\om\in[\tau\rvert_n]^\sim\rvert_n} \lt|\phi_\om'(\pi_\cU(\xi_{t(\om)}))\rt|^h
		=\sum_{s\in S}\cL_h^n(\1_{[\tau\rvert_n]^\sim})(\xi_s),\label{1fj6}
	\end{align}
	where $\cL_h:C_b(\cD_\cU^\infty)\to C_b(\cD_\cU^\infty)$\index{Perron--Frobenius (transfer) operator!$\cL_h$} is the \textbf{Perron--Frobenius (transfer) operator}\index{Perron--Frobenius (transfer) operator}\index{$\cL_h$} corresponding to the potential 
	$$
	\zeta_{h,\P(h)}=\zeta_{h,0}:\cD_\cU^\infty\lra \R.
	$$
	This operator is given by the formula
	\begin{align}\label{2fj6}
		\cL_h(g)(\xi)=\sum_{\mathclap{\substack{e\in\cD_\cU\\ A_{e,\xi_1}(\cU)=1}}}e^{\zeta_{h,0}(e\xi)}g(e\xi)
		=\sum_{\mathclap{\substack{e\in\cD_\cU\\ A_{e,\xi_1}(\cU)=1}}}|\phi_e'(\pi_\cU(\xi))|^hg(e\xi).
	\end{align}
	We want to show that the functions $\cL_h(\1_{[\tau\rvert_n]^\sim})$, $n\ge 1$, converge to zero uniformly exponentially fast. For this we want to use the existence of a spectral gap for the operator $\cL_h$. For the operator $\cL_h$ acting on the Banach space of bounded H\"older continuous functions from $\cD_\cU^\infty$ to $\R$, this was proved in \cite{mugdms}. This is however not sufficient for us as the H\"older norms of the characteristic functions of cylinders increase to infinity (even exponentially fast) with the length of the cylinder. We therefore turn our attention to a more sophisticated Banach space and apply the results of \cite{PolUbook}. We proceed slightly more generally than is really needed for the sake of arguments of this section. Namely except restricting ourselves to the parameter $h$, we fix, and we deal with, any parameter 
$$
t\in \De_G^*.
$$
Given a function $g\in L^1(\tilde\mu_t)$ and an integer $m\ge 0$, we define the \textbf{oscillation function}
$\osc_m(g):\cD_\cU^\infty\lra [0,\infty)$ \index{oscillation function} by the following formula: \index{oscillation function!$\osc_m$}\index{$\osc_m$}
\beq\label{1fp61}
\osc_m(g)(\om):=\ess\sup\{|g(\a)-g(\b)|:\a,\b\in[\om|_m]\}
\eeq
and 
$$
\osc_0(g):= \ess\sup(g)-\ess\inf(g).
$$
We further define:
\beq\label{2fp61}
|g|_{t,\th}:=\sup_{m\ge 0}\{\th^{-m}\|\osc_m(g)\|_1\},
\eeq
where $\|\cdot\|$ denotes the $L^1$--norm with respect to the measure $\tilde\mu_t$. The announced, non--standard (it even depends on the dynamics -- via $\tilde\mu_t$) Banach space is defined as follows: \index{$\Ba_\th(t)$}
$$
\Ba_\th(t):=\big\{g\in L^1(\tilde\mu_t):|g|_\th<+\infty\big\}
$$
and we denote
\beq\label{3fp61}
\|g\|_{t,\th}:=\|g\|_1+|g|_{t,\th}.
\eeq
Of course $\Ba_\th(t)$ is a vector space and the function 
\beq\label{4fp61}
\Ba_\th(t)\ni g\longmapsto \|g\|_{t,\th}
\eeq
is a norm on $\Ba_\th(t)$. This is the non--standard Banach space we will be working with now. 
	
A direct observation shows that for every $\tau\in \Sg_u$ and every integer $n\ge 0$, we have:
	\beq
	\osc_k(\1_{[\tau\rvert_n]^\sim})(\om)= 
	\begin{cases}
		0 \,  &{\rm if}~~ k\geq n  \\
		0 \,  &{\rm if}~~ k<n \;\text{ and }\; \^\om\rvert_k\neq\tau\rvert_k \\
		1 \,  &{\rm if}~~ k<n \;\text{ and }\; \^\om\rvert_k=\tau\rvert_k.
	\end{cases}
	\eeq
	Therefore, 
		\beq
	\int_{\cD_\cU^\infty}\osc_k(\1_{[\tau\rvert_n]^\sim})(\om)d\^\mu_t(\om)= 
	\begin{cases}
		0 \,  &{\rm if}~~ k\geq n \\
		\^\mu_t([\tau\rvert_k]^\sim) \,  &{\rm if}~~ k<n .
	\end{cases}
	\eeq
	Thus
	\begin{align}\label{1fj7}
		\big|\1_{[\tau\rvert_n]^\sim}\big|_{t,\th}=\max\big\{\th^{-k}\^\mu_t([\tau\rvert_k]^\sim):0\leq k<n\big \}.
	\end{align}
	We want to show that this number is uniformly bounded above. For this, we need a good upper estimate on $\^\mu_t([\tau\rvert_k]^\sim)$. We will prove it now. 
	
	%\sp\fr {\bf Claim~1:}
	\sp\fr \textbf{Claim~}\myClaimlabel{\textbf{1}}{Claim 1 fj7}\textbf{:}
		We have that 
		$$
			\lim_{k\to\infty}\max\big\{\om\in\Sg_u^k:\^\mu_t([\om]^\sim)\big\}=0.
		$$
	\begin{proof}
		For every $\om\in\Sg_u$ we have $[\om]^\sim\sub\pi_\cU^{-1}(\{\om\}\times J_\om)$, whence, by Lemma~\ref{p1sl4}, we get that 
		\begin{align}\label{2fj7}
			\^\mu_t([\om]^\sim)
			\leq \^\mu_t\(\pi_\cU^{-1}(\{\om\}\times J_\om(\cU))\)
			=\mu_t\(\{\om\}\times J_\om(\cU)\)
			\leq\mu_t\(\{\om\}\times J_\om\).
		\end{align} 
		Now, if $\om$ is not eventually periodic, then the sets $\{\^f^n(\{\om\}\times J_\om) \}_{n=0}^\infty$ are mutually disjoint. Since also 
		$$
			\mu_t\lt(\^f^{n+1}\lt(\{\om\}\times J_\om\rt)\rt)\geq\mu_t\lt(\^f^n\lt(\{\om\}\times J_\om\rt)\rt)
		$$
		and since $\mu_t(J(\^f))=1<+\infty$, we conclude that 
		\begin{align}\label{1fj7.1}
			\mu_t(\{\om\}\times J_\om)=0.
		\end{align}
		If $\om$ is eventually periodic, then $\sg^q(\om)$ with some $q\geq 0$ is periodic, and if $\mu_t\lt(\{\om\}\times J_\om\rt)>0$, then by \eqref{1fj7.1}, 
		$$
			\mu_t\lt(\{\sg^q(\om)\}\times J_{\sg^q(\om)}\rt)=\mu_t\lt(\^f^q\lt(\{\om\}\times J_\om\rt)\rt)>0.
		$$
		So, if $p\geq 1$ is a period of $\sg^q(\om)$, then 
		$$
			\mu_t\lt(\bigcup_{j=0}^{p-1} \^f\lt(\{\sg^q(\om)\}\times J_{\sg^q(\om)}\rt) \rt)>0,
		$$
		and 
		$$
			\^f\lt(\bigcup_{j=0}^{p-1} \^f^j\lt(\{\sg^q(\om)\}\times J_{\sg^q(\om)} \rt)\rt)
			=\bigcup_{j=0}^{p-1} \^f^j\lt(\{\sg^q(\om)\}\times J_{\sg^q(\om)}\rt).
		$$
		Hence, the ergodicity of the measure $\mu_t$ yields
		\begin{align}\label{2fj7.1}
			\mu_t\lt(\bigcup_{j=0}^{p-1} \^f^j\lt(\{\sg^q(\om)\}\times J_{\sg^q(\om)}\rt)\rt)=1.
		\end{align}
		But,
		$$
			\bigcup_{j=0}^{p-1} \^f\lt(\{\sg^q(\om)\}\times J_{\sg^q(\om)}\rt)
		$$
		is a closed proper subset of $J(\^f)$. So, \eqref{2fj7.1} contradicts the fact that $\supp(\mu_t)=J(\^f)$, i.e. that $\^\mu_t$ is a positive on non--empty open subsets of $J(\^f)$. In consequence, 
		$$
			\mu_t\lt(\{\om\}\times J_\om\rt)=0.
		$$
		Along with \eqref{1fj7.1} and \eqref{2fj7} this implies that 
		\begin{align}\label{3fj7.1}
			\^\mu_t([\om]^\sim)=0. 
		\end{align}
		Now seeking a contradiction, suppose that Claim~\ref{Claim 1 fj7} fails, i.e. that 
		$$
			\eta:=\frac{1}{2}\limsup_{k\to\infty}\max\big\{\^\mu_t([\om]^\sim):\om\in\Sg_u^k \big\}\in (0,1/2).
		$$ 
		Then for every $k\geq 0$, the family 
		$$
			\cF_k:=\lt\{\om\in\Sg_u^k:\^\mu_t([\om]^\sim )>\eta \rt\}
		$$	
		is not empty and for every $\om\in\Sg_u^{k+1}$, $\om\rvert_k\in\Sg_u^k$. This means that the families $(\cF_k)_{k\geq 0}$ form a tree rooted at $\cF_0=\{\emptyset\}$. Since this tree is finitely branched, the number of branches outgoing from each vertex being (uniformly) bounded above by $u$, it follows from K\"onig's Lemma that there exists $\om\in\Sg_u$ such that $\om\rvert_k\in\cF_k$ for all $k\geq 0$. Hence, 
		$$
			\^\mu_t([\om\rvert_k]^\sim)>\eta
		$$
		for all $k\geq 0$. Thus, 
		$$
			\^\mu_t([\om]^\sim)=\lim_{k\to \infty}\^\mu_t([\om\rvert_k]^\sim)\geq \eta,
		$$
		contrary to \eqref{3fj7.1}. Thus, Claim~\ref{Claim 1 fj7} is proved.
	\end{proof}
	Now we shall prove a substantial strengthening of Claim~\ref{Claim 1 fj7}. 

	\sp\fr \textbf{Claim~}\myClaimlabel{\textbf{2}}{Claim 2 fj8}\textbf{:} There exist two constants $A_t>0$ and $\a_t\in(0,1)$ such that for every $k\geq 0$
	$$
	\max\big\{\^\mu_t([\om]^\sim):\om\in\Sg_u^k \big\}\leq A_t\a_t^k.
	$$
	\begin{proof}
		It directly follows from Theorem~\ref{t1tf3} \eqref{t1tf3 item a} that there exists a constant $B\geq 1$ such that 
		\begin{align}\label{2fj8}
			\^\mu_t([\g\xi])\leq B\^\mu_t([\g])\^\mu_t([\xi])
		\end{align}	
		for all $\g,\xi\in\cD_\cU^\infty$. By Claim~\ref{Claim 1 fj7} there exists $q\geq 1$ so large that 
		$$
			\b:=\max\big\{\^\mu_t([\om]^\sim):\om\in\Sg_u^q \big\}< B^{-1}.
		$$
		Put
		$$
			s:=B\beta<1.
		$$
		We shall show by induction that 
		\begin{align}\label{1fj8}
\b_k:=\max\big\{\^\mu_t([\om]^\sim):\om\in\Sg_u^{qk} \big\}\leq s^k
		\end{align}
		for every $k\geq 0$. Indeed for $k= 0$, we have 1=1. So, suppose that \eqref{1fj8} holds for some $k\geq 0$. Using \eqref{2fj8}, we then get for every $\om\in\Sg_u^{q(k+1)}$ that 
		\begin{align*}
			\^\mu_t([\om]^\sim)
			&=\^\mu_t\lt(\bigcup\lt\{\g\xi: \g\in[\om\rvert_{qk}]^\sim, \, \xi\in[\sg^{qk}(\om)]^\sim, \, \g\xi\in\cD_\cU^{q(k+1)} \rt\}\rt)\\
			&=\sum_{*}\^\mu_t([\g\xi])
			\leq B\sum_*\^\mu_t([\g])\mu_t([\xi])\\
			&\leq B\^\mu_t([\xi])\sum_{\g\in[\om\rvert_{qk}]^\sim}\^\mu_t([\g])\sum_{\xi\in[\sg^{qk}(\om)]^\sim}
			=B\^\mu_t([\om\rvert_{qk}]^\sim)\^\mu_t([\sg^{qk}(\om)]^\sim)\\
			&\leq Bs^k\b=s^{k+1},
		\end{align*}
		where $\displaystyle\sum_*$  indicates the summation over the set 
		$$
			\lt\{\g\xi: \g\in[\om\rvert_{qk}]^\sim, \, \xi\in[\sg^{qk}(\om)]^\sim, \, \g\xi\in\cD_\cU^{q(k+1)} \rt\}.
		$$
		So, $\b_{k+1}\leq s^{k+1}$, and formula \eqref{1fj8} is proved. This formula directly entails Claim~\ref{Claim 2 fj8}. 
	\end{proof}
As indicated in \cite{PolUbook}, for the results of this paper one can take $\th\in(0,1)$ as close to 1 as one wishes. We take an arbitrary
$$
	\th\in(\a_t,1).
$$
It then follows from \ref{1fj7} and Claim~\ref{Claim 2 fj8} that 
\begin{align}\label{1fj9}
	\lt|\1_{[\tau\rvert_{n}]^\sim}\rt|_{t,\th}\leq A_t
\end{align}
for every $\tau\in\Sg_u$ and every integer $n\geq 0$. Now we want to apply some results from \cite{PolUbook}; for this, strictly speaking, we need to consider the normalized operator
\begin{align}\label{2fj9}
	\cL_{h,0}:=\frac{1}{\rho_h}\cL_h\circ \rho_h.
\end{align}
It formally follows from Proposition 2.4.2 (Fundamental Perturbative Result) of \cite{PolUbook}, with the sequence $U_n=\emptyset$ for all $n\geq 1$, (in fact it is much easier than the full fledged Proposition 2.4.2) that
\begin{align}\label{3fj9}
	\cL_{h,0}^ng:=\^\mu_h(g)+\De_h^ng,
\end{align} 
for all $n\geq 0$ and all $g\in\cB_\th(h)$, where 
$$
	\|\De_h^ng \|_\infty\leq\|\De_h^ng\|_{h,\th}\leq C_h\ka_h^n\|g\|_{h,\th}
$$
for all $n\geq 0$ with some constants $C_h\in(0,+\infty)$ and $\ka_h\in(0,1)$. Fix $\tau\in\Sg_u$ arbitrarily. It directly follows from \eqref{1fj9}--\eqref{3fj9}, Theorem~\ref{t1tf3} \eqref{t1tf3 item e}, and Claim~\ref{Claim 2 fj8} that with some constant $D\in[1,+\infty)$, we have for every integer $n\geq 1$ that 
\begin{align}
	\|\cL_h^n(\1_{[\tau\rvert_n]^\sim})\|_\infty &\leq\^\mu_h([\tau\rvert_n]^\sim)+\|\De_h^n\1_{[\tau\rvert_n]^\sim}\|_\infty\leq A_h\a_h^n+C_h\ka_h^nA_h\nonumber\\
	&\leq A_h(1+C_h)\max(\a_h^n,\ka_h^n)\label{1fj10} .
\end{align}
We now want some generalizations of formulas \eqref{1fj6} and \eqref{2fj6}. 

\sp\fr Keep $\tau\in\Sg_u$ fixed. For every $t\in\De_G^*$, let 
$$
	\un{\P}_\tau(t):=\liminf_{n\to\infty}\frac{1}{n}\log\|\cL_t^n(\1_{[\tau\rvert_n]^\sim})\|_\infty.	
$$
Invoking the last assertion of Corollary~\ref{c1_2017_01_26} and making use of the convexity argument, we conclude that the function
$$
(h-\d,h)\ni t\longmapsto\un \P_\tau(t)\in\R
$$
is convex, thus continuous. Since also by \eqref{1fj10}, 
\begin{align}\label{3fj10}
\un \P_\tau(h)\leq \log\max(\a_h,\ka_h)<0,
\end{align}
it thus follows that there exists $\a\in(h-\d,h)$ such that 
\begin{align}\label{2fj10}
\un \P_\tau(\a)<0.
\end{align}
Since formula \eqref{1fj6} holds with $h$ replaced by any $t\geq 0$, in particular by $\alpha$, and $\cL_h$ replaced by $\cL_t$,\index{Perron--Frobenius (transfer) operator!$\cL_t$}\index{$\cL_t$} the Perron--Frobenius operator corresponding to the potential $\zt_{t,0}:\cD_\cU^\infty\lra\R$, we conclude, using \eqref{2fj10}, that
$$
	\liminf_{n\to\infty}\frac{1}{n}\log\sum_{\om\in[\tau\rvert_n]^\sim\rvert_n}\diam^\a(p_2(\phi_\om(X_{t(\om)})))\leq \un \P_\tau(\a)<0.
$$
Therefore, 
$$
\liminf_{n\to\infty}\sum_{\om\in[\tau\rvert_n]^\sim\rvert_n}\diam^\a\( p_2(\phi_\om(X_{t(\om)}))\)=0.
$$
Hence, invoking also \eqref{1fj5}, we conclude that $\H_\a(\phi_\om(X_{t(\om)}))=0$. Thus, 
\begin{align}\label{1fj11}
	\HD(J_\tau(\cU))\leq \a<h.
\end{align}
Now we pass to the second step of the proof. Since the map $\^f:J(\^f)\lra J(\^f)$ is topologically exact, there exists an integer $q\geq 1$ such that 
\begin{align}\label{2fj11}
	\^f^q(U_s)\bus J(\^f)
\end{align}
for all $s\in S$. The key technical part of the second step is the following. 

\sp\fr \textbf{Claim~}\myClaimlabel{\textbf{3}}{Claim 3 fj11}\textbf{:}
Keeping $\tau\in\Sg_u$, we have for every integer $n\geq q$ that
$$	
J_{\sg^n(\tau)}(\cU)\sub\bigcup_{\om\in\Sg_u^q}f_{\om\tau\rvert_{q+1}^n}(J_{\om\sg^q(\tau)}(\cU)).
$$
\begin{proof}
	Fix $s\in S$ and $n\geq q$. Take an arbitrary $\xi\in [\sg^n(\tau)]^\sim$. Then, by the definition of the system $\cS_\cU$, there exists $\rho=\^f_\rho^{-(n-q)}:B(X_{i(\tau_{n+1})},R)\lra\Sg_u\times\C$, a holomorphic inverse branch of $\^f^{n-q}$ such that 
	\begin{align}\label{3fj11}
		\^\rho=\tau\rvert_{q+1}^n.
	\end{align} 
	By \eqref{2fj11} there exists $z\in U_s$ such that 
	$$
		\^f^q(z)=\^f_\rho^{-(n-q)}(\pi_\cU(\xi)).	
	$$
By the definition of the system $\cS_\cU$ again, there exists 
$$
\g=\^f_\g^{-q}:\^f_\rho^{-(n-q)}\(B(X_{i(\tau_{n+1})},R)\)\lra\Sg_u\times \C,
$$
a holomorphic inverse branch of $\^f^q$ such that $\^f_\g^{-q}(\^f^q(z))=z$. But, as $z\in U_s$, we then have that 
	$$
		\^f_\g^{-q}\circ \^f_\rho^{-(n-q)}(X_{i(\tau_{n+1})})\cap U_s\neq\emptyset.
	$$
	But then, by the definition of nice sets and the system $\zt$ for the final time, $\g\rho\in\cD_\cU^*$. Hence, 
	$$
		\g\rho\xi\in\cD_\cU^\infty.
	$$
	In addition,
	\begin{align}\label{1fj13}
		\^\g\in\Sg_u^q.
	\end{align}
	Therefore, using \eqref{3fj11} and the choice of $\xi$, we obtain
	\begin{align*}
p_2\circ\pi_\cU(\xi)
&=p_2\circ\pi_\cU(\sg^{|\g\rho|}(\g\rho\xi)) \\
&=f_{\^\g\^\rho}(p_2\circ\pi_\cU(\g\rho\xi))\in f_{\^\g\^\rho}(p_2\circ\pi_\cU([\^\g\sg^q(\tau)]^\sim)) %\\
		=f_{\^\g\^\rho}(J_{\^\g\sg^q(\tau)}(\cU)).
	\end{align*}
	Invoking \eqref{3fj11} again, and also \eqref{1fj13}, we conclude that 
	$$
		p_2\circ\pi_\cU(\xi)\in\bigcup_{\om\in\Sg_u^q}f_{\om\tau\rvert_{q+1}^n}(J_{\om\sg^q(\tau)}(\cU)).
	$$ 
	Claim~\ref{Claim 3 fj11} is thus proved.
\end{proof}
Since the collection $\{\om\tau\rvert_{q+1}^n:\om\in\Sg_u^q \}$ is finite, as an immediate consequence of Claim~\ref{Claim 3 fj11} and \eqref{1fj11}, we get that
\begin{align}\label{2fj13}
	\sup\big\{\HD(J_{\sg^n(\tau)}(\cU)):n\geq 0 \big\}<h.
\end{align}
Now observe that each point $z\in U_s$ such that $\^f^n(z)\in U_s$ for infinitely many $n\geq 0$, belongs in fact to $J_{\cU}$. So, for each point  $x$ of 
$$
J_\tau\bs\bigcup_{n=0}^\infty f_{\tau\rvert_n}^{-1}\(J_{\sg^n(\tau)}(\cU)\)
$$
there are only finitely many integers $n\geq 0$ such that $\^f^n(\tau,x)\in U_s$. Therefore,
$$
J_\tau\bs\bigcup_{n=0}^\infty f_{\tau\rvert_n}^{-1}\(J_{\sg^n(\tau)}(\cU)\)\sub\bigcup_{k=0}^\infty f_{\tau\rvert_k}^{-1}\(p_2(K(U_s))\).
$$
Hence, combining Lemma~\ref{l1fj5.1} and \eqref{2fj13}, we conclude that 
$$
	\HD(J_\tau)<h.
$$
So, the first part of Theorem~\ref{t1fj5} is proved.

\medskip Now, we pass to the expanding case. Then the dynamical system $\^f:J(\^f)\lra J(\^f)$ admits (finite!) Markov partitions with arbitrarily small diameters. In the language of the present paper it means that there exists a nice family (in fact nice sets with arbitrarily small diameters) $\cU=\{U_s\}_{s\in S}$ such that 
$$
\bigcup_{s\in S}U_s\bus J(\^f),	
$$
the alphabet $\cD_\cU$ of the corresponding GDMS $\cS_\cU$ is finite, and all elements $\phi_e$ of $\cS_\cU$ are holomorphic inverse branches of generating maps $f_j$, $j=1,\dots, u$. Moreover, for every $\tau\in \Sg_u$, 
$$
J_\tau=J_\tau(\cU).
$$
Having $t\geq 0$, $\tau\in\Sg_u$, and $n\geq 1$, we put 
$$
\P_\tau(t,n):=\frac{1}{n}\log\|\cL_h^n(\1_{[\tau\rvert_n]^\sim})\|_\infty.
$$
But in the current expanding case, 
$$
	\| \cL_h^n(\1_{[\tau\rvert_n]^\sim})\|_\infty=\sum_{\xi\in[\tau\rvert_n]^\sim}|\phi_\xi'(x_\tau(n))|^t
$$
for some $x_\tau(n)\in J(\^f)$. So, then 
$$
\P_\tau'(t,n)=\frac{1}{n}\cdot\frac{\sum_{\xi\in[\tau\rvert_n]^\sim}|\phi_\xi'(x_\tau(n))|^t\log|\phi_\xi'(x_\tau(n))|^t}{\sum_{\xi\in[\tau\rvert_n]^\sim}|\phi_\xi'(x_\tau(n))|^t},
$$
and in our special case there exists $\chi>0$ such that 
$$
|\phi_\xi'(x)|\geq \exp(-\chi|\xi|)
$$
for every $\xi\in\cD_\cU^*$ and every $x\in X_{t(\xi)}$. This entails 
$$
\P_\tau'(t,n)\geq -\chi t.
$$
So, denoting $\HD(J_\tau)$ by $h_\tau$, we get for every $n\geq 1$ that 
$$
\P_\tau(h,n)-\P_\tau(h_\tau,n)\geq -\chi(h-h_\tau).
$$
Therefore, 
$$
h-h_\tau\geq\frac{\P_\tau(h_\tau,n)-\P_\tau(h,n)}{\chi}.
$$
Hence, noting also that $\un \P_\tau(h_\tau)\geq 0$, using \eqref{3fj10}, and taking the limit as $n\to\infty$, we obtain
$$
h-h_\tau\geq\frac{1}{\chi}\(\un\P_\tau(h_\tau)-\P_\tau(h)\)
\geq-\frac{\P_\tau(h)}{\chi}
\geq -\frac{1}{\chi}\log\max\{\a_h,\ka_h\}>0.
$$ 
The proof of the second part of Theorem~\ref{t1fj5} is complete, and we are done.
\end{proof}

	\section{Examples}\label{s:Ex}

In this section, we describe some examples of finely non--recurrent rational semigroups satisfying the Nice Open Set Condition, i.e. being \NOSC-FNR.

\begin{ex}[\cite{sush}, comp. \cite{hiroki2, hiroki4}] 
	\label{oscshex1}
	Let 
	$$
	f_{1}(z):=z^{2}+2 \quad\text{ and }\quad f_{2}(z):=z^{2}-2.
	$$
	Let $f=(f_1,f_2)$ and 
	$$
	G:=\langle f_{1},f_{2}\rangle.
	$$  
	In addition,
	let 
	$$
	U:=\{ z\in \C: |z|<2\} .
	$$
	Then, $G$ is semi--hyperbolic but not hyperbolic (so, not expanding) (\cite[Example 5.8]{hiroki2}). Since $\Crit_*(f)=\{0\}$ is a singleton, it follows from Observation~\ref{o20190529} that $G$ is a {\rm TNR} semigroup. Moreover, $G$ satisfies the Nice Open Set Condition with $U.$ Since 
	$$
	J(G)\subset f_{1}^{-1}(\ov{U})\cup f_{2}^{-1}(\ov{U})
	\subsetneqq \ov{U},
	$$
	\cite[Theorem 1.25]{hiroki4} implies that $J(G)$ is porous and $\HD(J(G))<2.$ It is also easy to see that the $u$--tuple map $f$ is {\rm C--F} balanced. It follows from Lemma~\ref{l320190325} that $G$ is of finite type. Thus, $G$ is a finely non--recurrent rational semigroups satisfying the Nice Open Set Condition, i.e. \NOSC-FNR. In addition, by Theorem~\ref{Theorem A} we have that
	$$
	h_f=\HD(J(G))=\PD(J(G))=\BD(J(G)).
	$$
	Furthermore (see \text{{\rm Figure~\ref{fig:z2+2z2-2-1}}}),
	$$
	f_{1}^{-1}(\ov{U})\cap f_{2}^{-1}(\ov{U})\neq \emptyset.
	$$
	\begin{figure}[htbp]
		\caption{The Julia set of
			$\langle f_{1},f_{2}\rangle$,
			where $f_{1}(z)=z^{2}+2,\, f_{2}(z)=z^{2}-2$.}
		\includegraphics[width=0.5\textwidth]{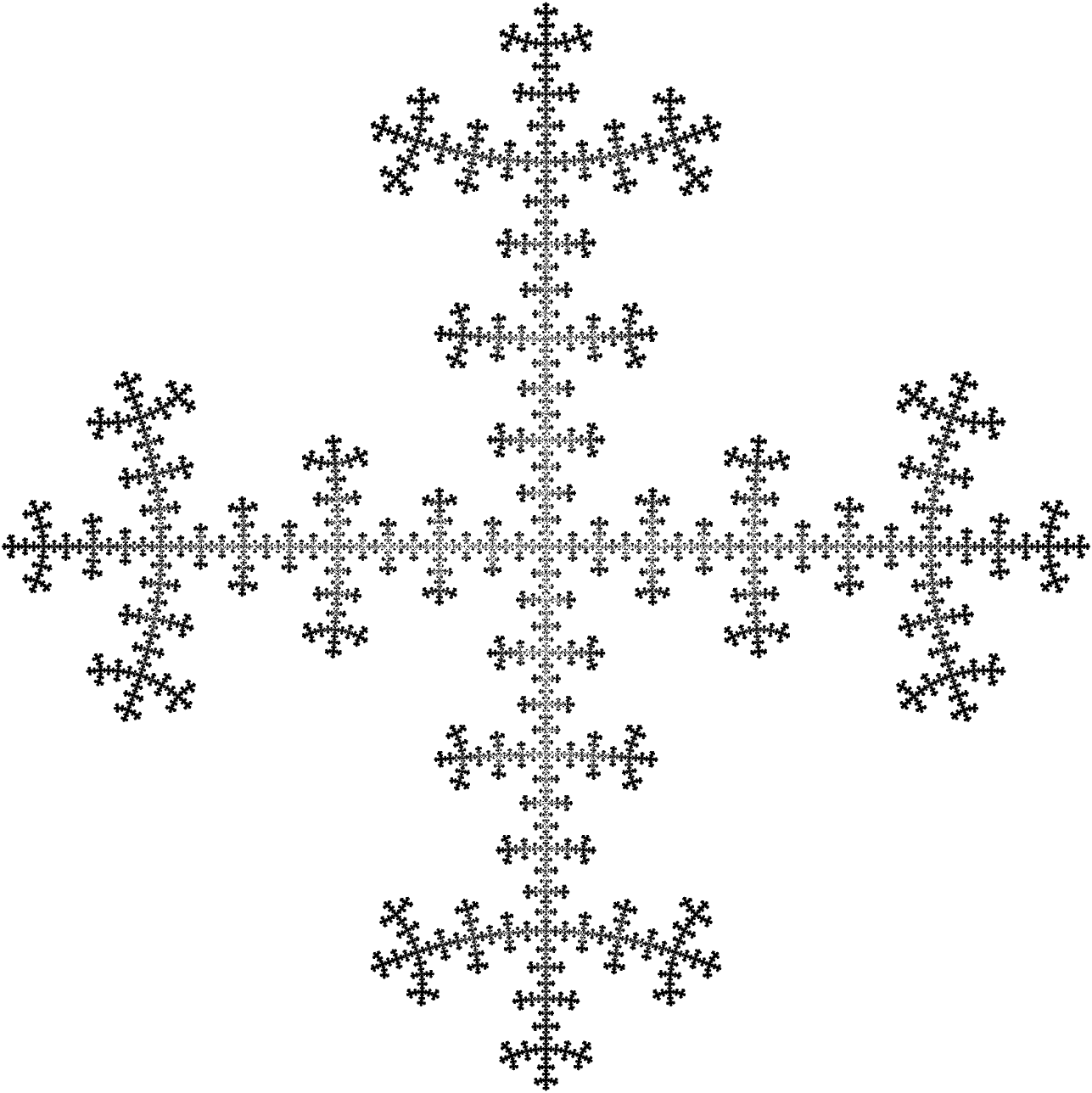}
		\label{fig:z2+2z2-2-1}
	\end{figure}
\end{ex}

\begin{ex}
	Let 
	$$
	f_{1}(z):=z^{2}+2, \quad f_{2}(z):=z^{2}-2, \quad\text{ and }\quad f_{3}(z):=z^{2}.
	$$
	Let $f=(f_1,f_2,f_3)$ and 
	$$
	G:=\langle f_{1},f_{2},f_3\rangle.
	$$  
	In addition,
	let 
	$$
	U:=\{ z\in \C: |z|<2\} .
	$$
	Using the same reasoning as in Example ~\ref{oscshex1}, we obtain the same results for $J(G)$ as in Example ~\ref{oscshex1}. Figure \ref{fig:z2+2z2-2-1z2} shows the Julia set $J(G)$. 
	\begin{figure}[htbp]
		\caption{The Julia set of
			$\langle f_{1},f_{2},f_3\rangle$,
			where $f_{1}(z)=z^{2}+2,\, f_{2}(z)=z^{2}-2$, and $f_3(z)=z^2$.}
		\includegraphics[width=0.5\textwidth]{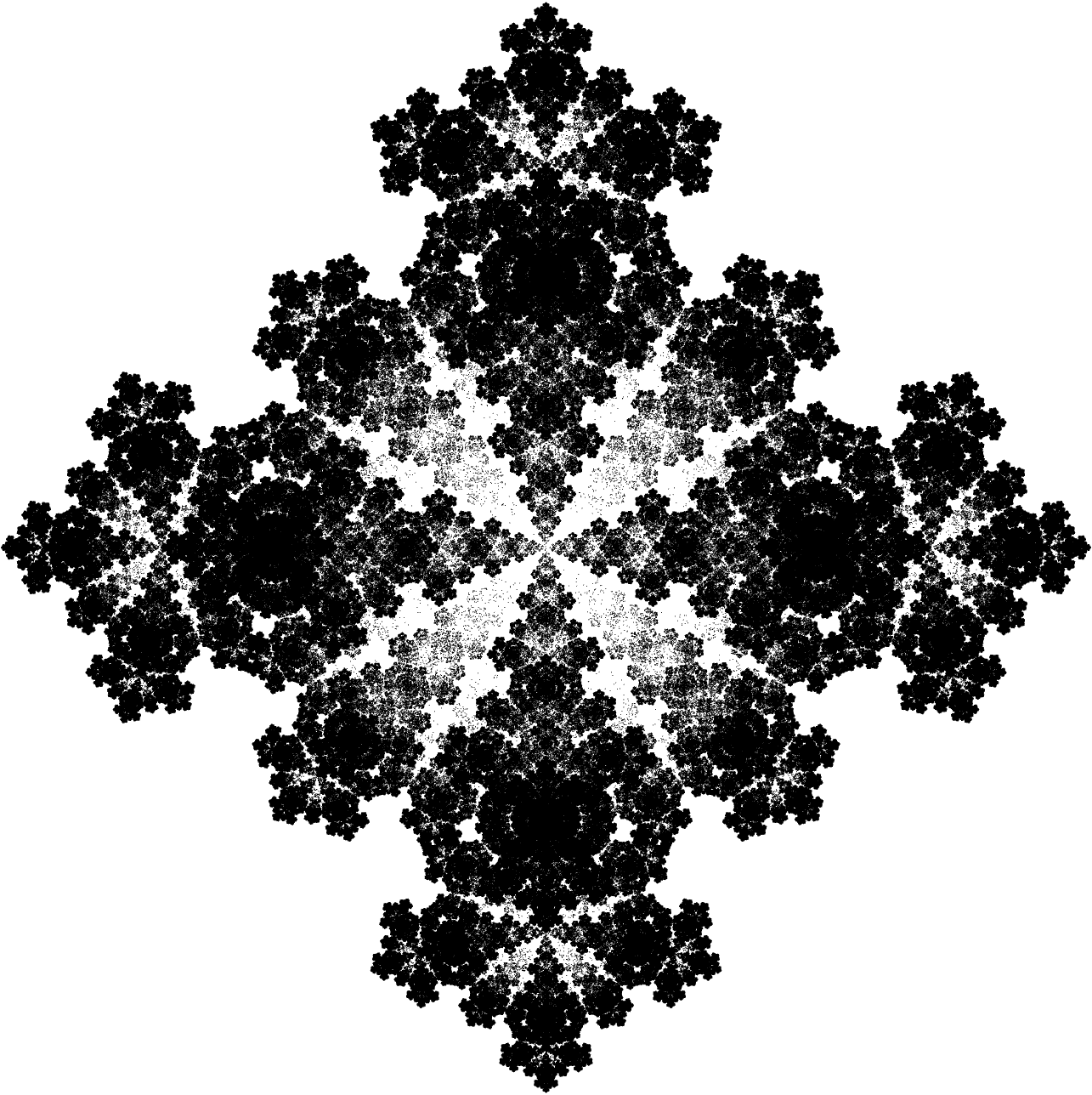}
		\label{fig:z2+2z2-2-1z2}
	\end{figure}
\end{ex}

\begin{ex} 
	A large class of examples is provided by the following.
\end{ex}

\begin{prop}[\cite{sush}, comp. \cite{SdpbpI, hiroki5}]\label{semihyposcexprop}
	Assume the following.
	
	\begin{itemize}
		\item Let $f_{1}$ be a semi--hyperbolic polynomial with $\deg (f_{1})$ $\geq 2$ such that $J(f_{1})$ is connected.
		
		\vspace{0.7mm}
		
		\item Let $K(f_{1})$ be the filled--in Julia set of $f_{1}$ and
		suppose that {\em Int}$(K(f_{1}))\ne\es$.
		
		\vspace{0.7mm}
		
		\item Let $b\in \mbox{{\em Int}}(K(f_{1}))$ be a point.
		
		\vspace{0.7mm}
		
		\item Let $d$ be a positive integer such that $d\geq 2.$ 
		
		\vspace{0.7mm}
		
		\item Suppose that $(\deg (f_{1}),d)\neq (2,2).$
	\end{itemize}
	
	\fr Then, there exists a real number $c>0$ such that
	for each $\l \in \{ \l\in \Bbb{C}: 0<|\l |<c\} $,
	setting 
	$$
	f_{\l }:=(f_{\l ,1},f_{\l ,2})=
	(f_{1},\l (z-b)^{d}+b ) \  \  {\rm and} \  \ G_{\l }:= \langle f_{1},f_{\l, 2}\rangle,
	$$
	we have the following.
	
	\begin{enumerate} 
		\item[\mylabel{1}{Example Prop item 1}] $G_{\l }$ is *semi--hyperbolic.
		
		\vspace{0.7mm}
		
		\item[\mylabel{2}{Example Prop item 2}] $f_{\l }$ satisfies the Nice Open Set Condition with an open set $U_{\l }$.
		
		\vspace{0.7mm}
		
		\item[\mylabel{3}{Example Prop item 3}] $J(G_{\l } )$ is porous.
		
		\vspace{0.7mm}
		
		\item[\mylabel{4}{Example Prop item 4}] $\HD(J(G_{\l }))=h_{f_{\l}}<2$.
		
		\vspace{0.7mm}
		
		\item[\mylabel{5}{Example Prop item 5}] $\PCV(G_{\l }) \setminus \{ \infty \}$ is bounded in $\Bbb{C}.$
		
		\vspace{0.7mm}
		
		\item[\mylabel{6}{Example Prop item 6}] $G_{\l }$ is {\rm C--F} balanced.
		
		\vspace{0.7mm}
		
		\item[\mylabel{7}{Example Prop item 7}]
		
		If, in addition, $f_1$ is totally non--recurrent {\rm TNR}, then so is $G_\l$, and it is of finite type, thus finely non--recurrent and, moreover, \NOSC-FNR.

	\end{enumerate}
\end{prop}
%\noindent{\sl Proof.}
\begin{proof}
	We will closely follow the exposition from \cite{sush}.
	Conjugating $f_{1}$ by a M\"{o}bius transformation,
	we may assume that $b=0$ and that the coefficient
	of the highest degree term of $f_{1}$ is equal to $1.$
	
	% For each $r>0$, we denote by $D(0,r)$ the Euclidean disc with
	%radius $r$ and center $0.$
	Let $r>0$ be a real
	number such that $\overline{B_2(0,r)}\subset \mbox{Int}K(f_{1}).$
	We set $d_{1}:=\deg (f_{1}).$
	Let $\alpha >0$ be a number.
	Since $d\geq 2$ and $(d,d_{1})\neq (2,2)$,
	it is easy to see that
	$$
	\lt(\frac{r}{\alpha }\rt)^{\frac{1}{d}}>
	2\left(2\lt(\frac{1}{\alpha }\rt)
	^{\frac{1}{d-1}}\right)^{\frac{1}{d_{1}}}
	$$
	if and only if
	\begin{equation}
		\label{Contproppfeq1}
		\log \alpha <
		\frac{d(d-1)d_{1}}{d+d_{1}-d_{1}d}
		\lt( \log 2-\frac{1}{d_{1}}\log \frac{1}{2}-\frac{1}{d}\log r\rt) .
	\end{equation}
	We set
	\begin{equation}
		\label{Contproppfeq2}
		c_{0}:=\exp \left(\frac{d(d-1)d_{1}}{d+d_{1}-d_{1}d}
		\Big( \log 2-\frac{1}{d_{1}}\log \frac{1}{2}-\frac{1}{d}\log r\Big) \right)
		\in (0,\infty ).
	\end{equation}
	Let $0<c<c_{0}$ be a number required to be sufficiently small later in the proof and let $\l \in \Bbb{C}$ be a number with $0<|\l |<c.$
	Put 
	$$
	f_{\l ,2}(z):=\l z^{d}.
	$$
	Then, we have
	$$
	K(f_{\l ,2})=\lt\{ z\in \Bbb{C}:|z|\leq \lt(\frac{1}{|\l|}\rt)^{\frac{1}{d-1}}\rt\} 
	$$ 
	and
	$$
	f_{\l,2}^{-1}\(\{ z\in \Bbb{C}: |z|=r\}\)=
	\lt\{ z\in \Bbb{C}:|z|=\lt(\frac{r}{|\l |}\rt)^{\frac{1}{d}}\rt\}.
	$$
	Let
	$$
	D_{\l }:=\overline{B_2\lt(0,2\lt(\frac{1}{|\l |}\rt)^{\frac{1}{d-1}}\rt)}.
	$$
	Since 
	$$
	f_{1}(z)=z^{d_{1}}(1+o(1)) \  \  {\rm as} \  \  z\longrightarrow \infty,
	$$
	it follows that if $0<c<c_0$ is small enough, then
	for every $\l \in \Bbb{C} $ with $0<|\l |<c$, we have that
	%$h^{-1}\left( \overline{D(0,2(\frac{1}{|a|})^{\frac{1}{d-1}}})\right) \subset
	$$
	f_{1}^{-1}(D_{\l })\subset
	\left\{ z\in \Bbb{C}:
	|z|\leq 2\left( 2\lt(\frac{1}{|\l |}\rt)^{\frac{1}{d-1}}\right)
	^{\frac{1}{d_{1}}}\right\}.
	$$
	This implies that
	\begin{equation}
		\label{Contproppfeq3}
		%h^{-1}\left( \overline{D(0,2(\frac{1}{|a|})^{\frac{1}{d-1}}})\right) \subset
		f_{1}^{-1}(D_{\l })\subset f_{\l ,2}^{-1}\(\{ z\in \Bbb{C}:|z|<r\}\).
	\end{equation}
	Hence, setting 
	$$
	U_{\l }:=\mbox{Int}(K(f_{\l ,2}))\setminus K(f_{1}),
	$$
	we will have
	$$
	f_{1}^{-1}(U_{\l })\cup f_{\l ,2}^{-1}(U_{\l })\subset U_{\l }
	\  \  \  {\rm and} \  \  \
	f_{1}^{-1}(U_{\l })\cap f_{\l, 2}^{-1}(U_{\l })=\emptyset.
	$$
	Furthermore, since $f_{1}$ is semi--hyperbolic, we get that
	$\oc \setminus K(f_{1})$ is a John domain by \cite{CJY}.
	Hence, $U_{\l }$ satisfies \eqref{osc3}. Therefore, $G_{\l }$ satisfies the Nice Open Set Condition witnessed by $U_{\l }$.
	
	We have 
	$$
	J(G_{\l })\subset \overline{U_{\l }}
	\subset K(f_{\l ,2})\setminus \mbox{Int}(K(f_{1})). 
	$$
	In particular, 
	$$
	\Int(K(f_{1}))\cup (\oc \setminus K(f_{\l ,2}))\subset
	F(G_{\l }).
	$$
	Furthermore, (\ref{Contproppfeq3}) implies that
	$f_{\l ,2}(K(f_{1}))\subset \mbox{Int}(K(f_{1}))$.
	Thus, we have 
	$$
	\PCV(G_{\l })\setminus \{ \infty \}
	= \bu _{g\in G_{\l }^{\ast }}
	g\(\Crit{\rm V}^{\ast }(f_{1})\cup \Crit{\rm V}^{\ast }(f_{\l, 2})\)
	\subset K(f_{1}),
	$$
	where $\Crit{\rm V}^{\ast }(\cdot )$ denotes the set of 
	all critical values in $\Bbb{C}.$ Hence,
	$\PCV(G_{\l })\setminus \{ \infty \} $ is bounded in $\Bbb{C}$.
	
	Since $f_{1}$ is semi--hyperbolic,
	there exist an $N\in \N $ and a $\delta _{1}>0$ such that
	for each $x\in J(f_{1})$ and for each $n\in \N $,
	$$
	\deg\(f_{1}^{n}:V\longrightarrow B_2(x,\delta _{1})\)\leq N
	$$ 
	for each connected component $V$ of $f_{1}^{-n}(B_2(x,\delta _{1}).$
	Also, $f_{\l, 2}^{-1}(J(f_{1}))\cap K(f_{1})=\emptyset $, and so 
	$$
	f_{\l ,2}^{-1}(J(f_{1}))\subset \oc \setminus \PCV(G_{\l }). 
	$$
	It follows from the above that there exists a $0<\delta _{2} <\delta _{1}$ such that for each $x\in J(f_{1})$ and each $g\in G_{\l }$,
	$$
	\deg \(g:V\longrightarrow B_2(x,\delta _{2})\)\leq N
	$$ 
	for each connected component $V$ of $g^{-1}(B_2(x,\delta _{2})).$
	Since $\PCV(G_{\l })\setminus \{ \infty \} \subset K(f_{1})$ again,
	we obtain that there exists a number $0<\delta _{3} <\delta _{2} $ such that
	for each $x\in J(G_{\l })$ and each $g\in G_{\l }$,
	$$
	\deg \(g:V\longrightarrow B_2(x,\delta _{3})\)\leq N
	$$ 
	for each connected component $V$ of $g^{-1}(B_2(x,\delta _{3})).$ Thus, $G_{\l }$ is semi--hyperbolic.
	
	Since $J(G_{\l })\subset f_{1}^{-1}(\ov{U_{\l }})\cup f_{\l ,2}^{-1}(\ov{U_{\l }})
	\subsetneqq \ov{U_{\l }}$, \cite{hiroki4} implies that $J(G_{\l })$ is porous and
	$$
	\HD (J(G_{\l }))<2.
	$$ 
	Also, by Theorem~\ref{Theorem A}, we have that
	$$
	h_{f_{\l }}=\HD(J(G_{\l })).
	$$
	Finally, it is easy to see that item \eqref{Example Prop item 7} holds, and we are done.
	%\qed
	
\end{proof}

Figure~\ref{fig:z2-1z2/4} shows the Julia set of semigroup generated by second iterate of each of the maps  $z\mapsto z^{2}-1$ and  $z\mapsto z^{2}/4$, which satisfies the hypotheses of Proposition~\ref{semihyposcexprop}.

Figure~\ref{fig:z2-1,z3/2} shows the Julia set of semigroup generated by the maps  $z\mapsto z^{2}-1$ and  $z\mapsto z^{3}/2$, which satisfies the hypotheses of Proposition~\ref{semihyposcexprop}.

Figure~\ref{fig:z2-1,iz2/4} shows the Julia set of semigroup generated by the maps  $z\mapsto z^{2}-1$ and  $z\mapsto iz^{4}$, which satisfies the hypotheses of Proposition~\ref{semihyposcexprop}. 

\begin{figure}
	\caption{The Julia set of
		$\langle f_{1}^{2},f_{2}^{2}\rangle$,
		where $f_{1}(z)=z^{2}-1, f_{2}(z)=z^{2}/4$.}
	\includegraphics[width=0.5\textwidth]{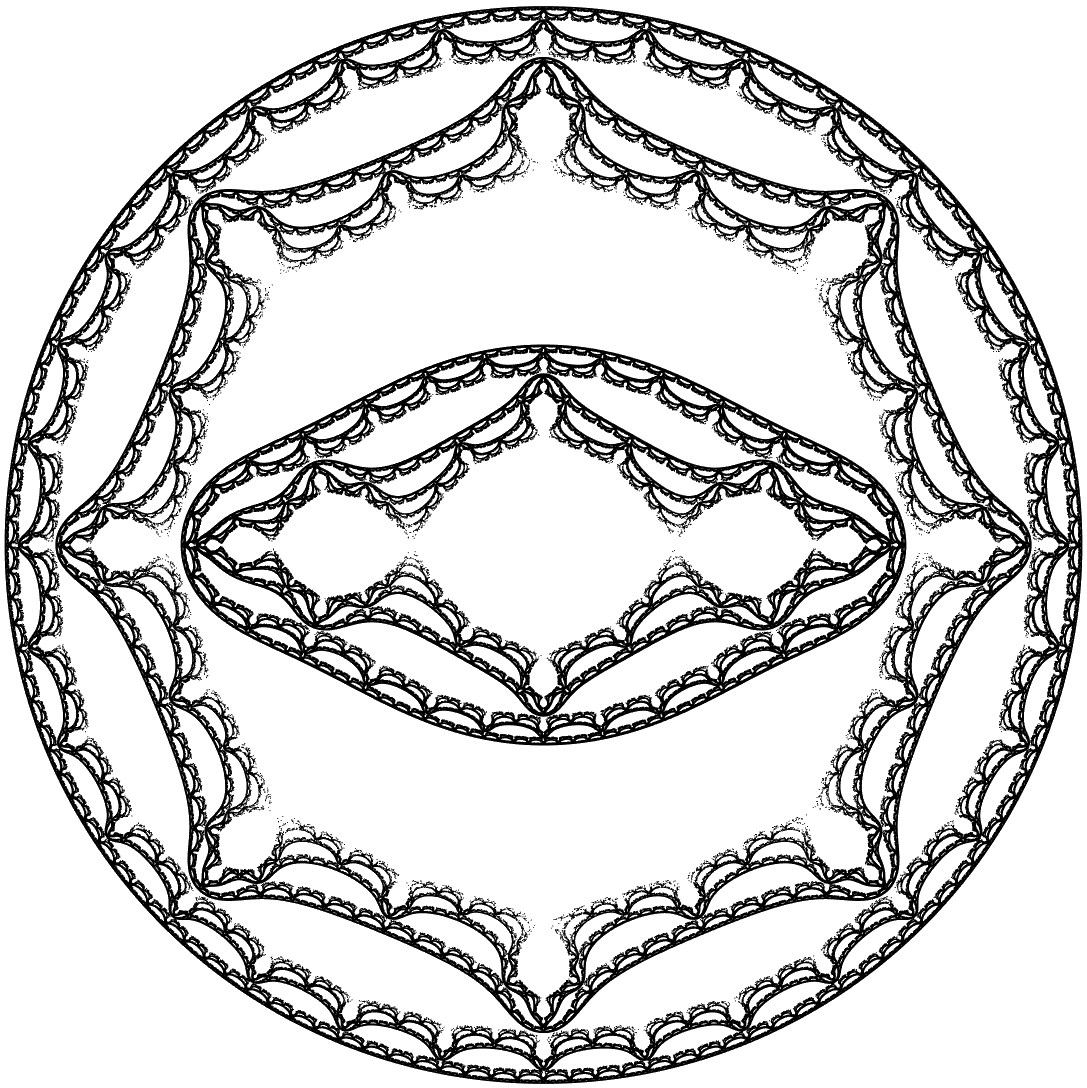}
	\label{fig:z2-1z2/4}
\end{figure}

\begin{figure}
	\caption{The Julia set of
		$\langle f_{1},f_{2}\rangle$,
		where $f_{1}(z)=z^{2}-1, f_{2}(z)=z^{3}/2$.}
	\includegraphics[width=0.5\textwidth]{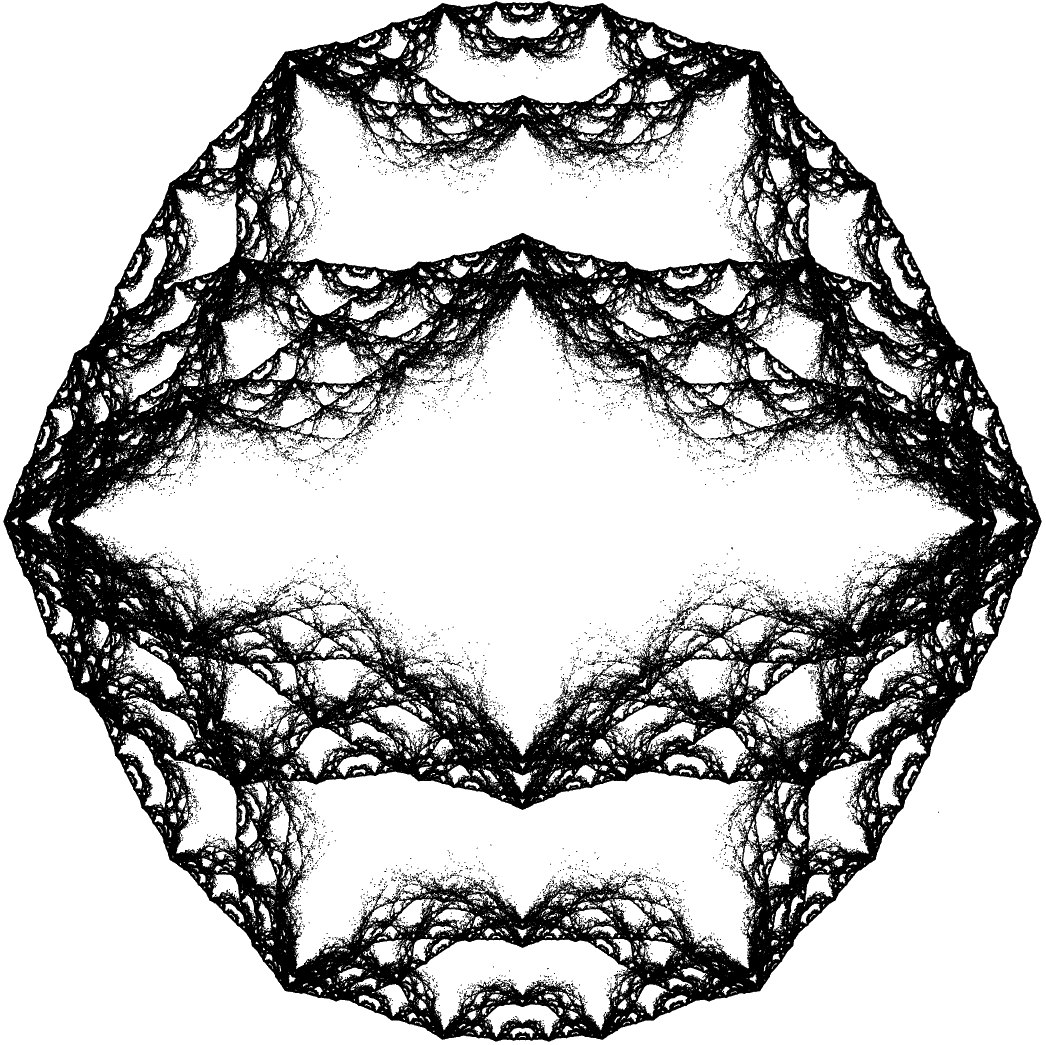}
	\label{fig:z2-1,z3/2}
\end{figure}

\begin{figure}
	\caption{The Julia set of
		$\langle f_{1},f_{2}\rangle$,
		where $f_{1}(z)=z^{2}-1, f_{2}(z)=iz^{4}$.}
	\includegraphics[width=0.5\textwidth]{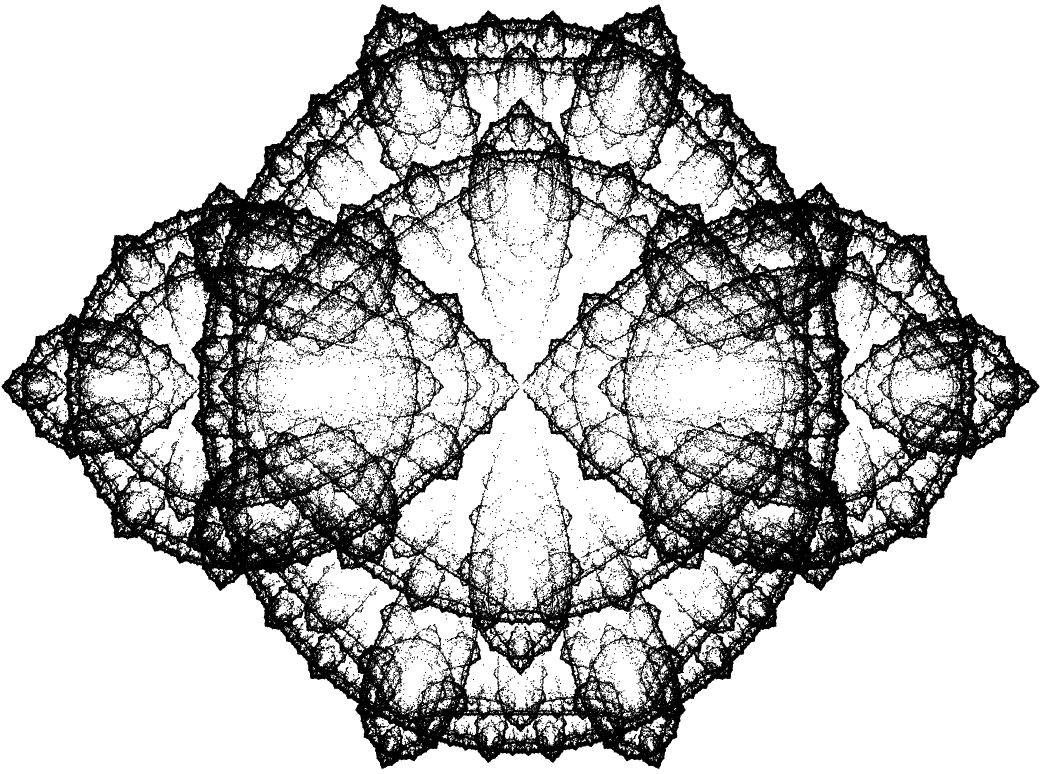}
	\label{fig:z2-1,iz2/4}
\end{figure}

\part{Appendices}

\appendix
\sp\section{Absolutely Continuous $\sg$--Finite Invariant Measures: Martens Method}

In this appendix we state again Theorem~\ref{t1h75} and provide its full proof. As we have already explained in the paragraph preceding Definition~\ref{d:mmmap}, we do this (provide the proof) for the sake of completeness, the importance of Theorem~\ref{t1h75} for our overall approach, and since there is no published proof of this theorem in such large generality (see \cite{KUbook} and \cite{MundayRoyUI}). 

\bthm\label{t1h75B} 
If $(X,\mathcal{B},m)$ is a probability space and $T:X\longrightarrow X$ is a Martens map with a Martens cover $\{ X_{j}\} _{j=0}^{\infty }$, 
then we have the following.

\sp\begin{itemize}
\item There exists a $\sigma$--finite $T$--invariant 
measure $\mu $ on $X$ equivalent to $m$. 

\sp\item In addition, $0<\mu (X_{j})<+\infty $ for each $j\geq 0$. 
\end{itemize}

\sp\fr The measure $\mu $ is constructed in the following way: 
Let $l_{B}:l^\infty\longrightarrow \R$ be a Banach limit. 
For each $A\in \mathcal{B}$, set 
$$
m_{n}(A):=\frac{\sum _{k=0}^{n}m(T^{-k}(A))}{\sum _{k=0}^{n}m(T^{-k}(X_{0}))}.
$$
If $A\in \mathcal{B}$ and $A\subset Y_{j}$ with some $j\geq 0$, 
then we obtain $(m_{n}(A))_{n=1}^{\infty }\in l^\infty.$ 
We set 
$$\
\mu (A):= l_{B}((m_{n}(A))_{n=1}^{\infty }).
$$ 
For a general measurable subset $A\subset X$, set 
$$\mu (A):=\sum _{j=0}^{\infty }\mu (A\cap Y_{j}).$$ 
In addition, if for a measurable subset $A\subset X$,  
the sequence $(m_{n}(A))_{n=1}^{\infty }$ is bounded, 
then we have the following formula: 
\begin{equation}
\label{eq:muequB}
\mu (A)
=l_{B}\((m_{n}(A))_{n=1}^{\infty }\)
-\lim _{l\rightarrow \infty }
l_{B}\lt(\lt(m_{n}\Big(A\cap \bigcup _{j=l}^{\infty }Y_{j}\Big)\rt)_{n=0}^{\infty}\rt).
\end{equation}
In particular, if $A\in\mathcal{B}$ is contained in a finite union 
of sets $X_j$, $j\geq0$, then
\[
\mu(A)=l_B\bigl((m_n(A))_{n=1}^\infty\bigr). 
\]
Furthermore, if the measure--preserving transformation $T:X\longrightarrow X$ is ergodic (equivalently with respect to  the measure $m$ or $\mu $), 
then the $T$--invariant measure $\mu $ is unique up to a multiplicative constant. 
\ethm

The proof of this theorem will consist of several lemmas. We start with the following.

\blem\label{l:dumeasure} 
Let $(Z,\mathcal{F})$ be a measurable space such that:
\begin{itemize}
\item[\mylabel{a}{l:dumeasure item a}] $\displaystyle Z=\bigcup_{j=0}^\infty Z_j$ for some mutually disjoint sets $Z_j\in\mathcal{F}$,

\,

\fr and

\, 
\item[\mylabel{b}{l:dumeasure item b}] $\nu_j$ is a finite measure on $Z_j$ for each $j\geq0$.
\end{itemize}
Then the set function $\nu:\mathcal{F}\to[0,\infty]$ defined by
\[
\nu(F):=\sum_{j=0}^\infty\nu_j(F\cap Z_j)
\] 
is a $\sigma$--finite measure on $Z$.
\elem

\begin{proof} Clearly, $\nu(\emptyset)=0$. 
Let $F\in\mathcal{F}$ and let $\{F_n\}_{n=1}^\infty$ be a partition 
of $F$ into sets in $\mathcal{F}$. Then
\begin{align*}
\nu(F)
=&\sum_{j=0}^\infty\nu_j(F\cap Z_j)
=\sum_{j=0}^\infty\nu_j\Bigl(\bigcup_{n=1}^\infty(F_n\cap Z_j)\Bigr) \\
=&\sum_{j=0}^\infty\sum_{n=1}^\infty\nu_j(F_n\cap Z_j)
=\sum_{n=1}^\infty\sum_{j=0}^\infty\nu_j(F_n\cap Z_j)\\
=&\sum_{n=1}^\infty\nu(F_n),
\end{align*}
where the order of summation could be changed since all terms
involved are non--negative. Thus, $\nu$ is a measure. Moreover, by definition, 
$Z=\bigcup_{j=0}^\infty Z_j$ and $\nu(Z_j)=\nu_j(Z_j)<\infty$ for all $j\geq0$. 
Therefore $\nu$ is $\sg$--finite.
%\qed 
\end{proof}

From this point on, all lemmas rely on the same main hypotheses as Theorem~\ref{t1h75}.

\blem\label{mn}
For all $n,j\geq0$ and all $A,B\in\mathcal{A}$ 
with $A\cup B\sbt X_j$, we have
\[
m_n(A)\,m(B)\leq K_j\,m(A)\,m_n(B).
\]
\elem

\begin{proof}
This follows directly from the definition of $m_n$ and condition~\eqref{d:mmmap item 4} of Definition~\ref{d:mmmap}.
%\qed 
\end{proof}

\blem\label{l:muxjfin} 
For every $j\geq0$, we have
$(m_n(X_j))_{n=1}^\infty\in l^\infty$ and $\mu(Y_j)\leq\mu(X_j)<\infty$.
\elem

\begin{proof} Fix $j\geq0$. By virtue of condition~\eqref{d:mmmap item 3} of Definition~\ref{d:mmmap}, there
exists $q\geq0$ such that $m(X_j\cap T^{-q}(X_0))>0$. By Lemma~\ref{mn} and the definition of $m_n$, 
for all $n\geq0$ we have that
\begin{eqnarray}
m_n(Y_j)
\leq m_n(X_j)
&\leq&K_j\frac{m(X_j)}{m\bigl(X_j\cap T^{-q}(X_0)\bigr)}m_n\bigl(X_j\cap T^{-q}(X_0)\bigr) \nonumber \\
&\leq&K_j\frac{m(X_j)}{m\bigl(X_j\cap T^{-q}(X_0)\bigr)}m_n(T^{-q}(X_0)) \nonumber \\
&=&K_j\frac{m(X_j)}{m\bigl(X_j\cap T^{-q}(X_0)\bigr)}
   \frac{\sum_{k=0}^{n+q}m(T^{-k}(X_0))}{\sum_{k=0}^nm(T^{-k}(X_0))} \nonumber \\
&=&K_j\frac{m(X_j)}{m\bigl(X_j\cap T^{-q}(X_0)\bigr)}
     \left[1+\frac{\sum_{k=n+1}^{n+q}m(T^{-k}(X_0))}{\sum_{k=0}^nm(T^{-k}(X_0))}\right] \nonumber \\
%&\leq&K_j\frac{m(X_j)}{m\bigl(X_j\cap T^{-q}(X_0)\bigr)} 
%     \left[1+\frac{q}{\sum_{k=0}^nm(T^{-k}(X_0))}\right] \nonumber \\ 
&\leq&K_j\frac{m(X_j)}{m\bigl(X_j\cap T^{-q}(X_0)\bigr)} 
     \left[1+\frac{q}{m(X_0)}\right]. \label{NCPIIp110}
\end{eqnarray}
Consequently, $(m_n(X_j))_{n=1}^\infty\in l^\infty$, and standard properties 
of a Banach limit yield that
\[
\mu(Y_j)
%:=l_B\bigl((m_n(Y_j))_{n=1}^\infty\bigr)
\leq K_j\frac{m(X_j)}{m\bigl(X_j\cap T^{-q}(X_0)\bigr)}\left[1+\frac{q}{m(X_0)}\right]
<\infty.
\]
Since $X_j=\bigcup_{i=0}^j Y_i$ and the $Y$s are mutually disjoint, we deduce that 
\[
\mu(Y_j)
\leq\sum_{i=0}^j\mu(X_j\cap Y_i)
=\sum_{i=0}^\infty\mu(X_j\cap Y_i) 
=:\mu(X_j)
\leq\sum_{i=0}^j\mu(Y_i)
<\infty.
\]
%
%Clearly, $0\leq(m_n(A))_{n=1}^\infty\leq(m_n(Y_i))_{n=1}^\infty\leq(m_n(X_i))_{n=1}^\infty$. 
%Since $(m_n(X_i))_{n=1}^\infty\in l^\infty$ according to Lemma~\ref{l:muxjfin}, 
%so is $(m_n(A))_{n=1}^\infty$. 
%
%\[
%\mu(X_j)
%:=\sum_{i=0}^\infty\mu(X_j\cap Y_i)
%=\sum_{i=0}^j\mu(X_j\cap Y_i)
%\geq\mu(Y_j)
%\]
%and
%\[
%\mu(X_j)
%%:=\sum_{i=0}^\infty\mu(X_j\cap Y_i)
%=\sum_{i=0}^j\mu(X_j\cap Y_i)
%\leq\sum_{i=0}^j\mu(X_j\cap Y_i)
%%=\sum_{i=0}^j l_B\bigl((m_n(X_j\cap Y_i))_{n=1}^\infty\bigr)
%%\leq\sum_{i=0}^j l_B\bigl((m_n(Y_i))_{n=1}^\infty\bigr)
%%\leq\sum_{i=0}^j K_i\frac{m(X_i)}{m(X_i\cap T^{-q(i)}(X_0))}
%<\infty.
%\]
%\qed 
\end{proof}

Now, for every $j\geq0$, set 
$$
\mu_j:=\mu|_{Y_j}.
$$

\blem\label{l:mumcomp} 
For every $j\geq0$ such that $\mu(Y_j)>0$ and
for every measurable set $A\sbt Y_j$, we have
\[
K_j^{-1}\frac{\mu(Y_j)}{m(Y_j)}m(A)
\leq\mu_j(A)
\leq K_j\frac{\mu(Y_j)}{m(Y_j)}m(A).
\]
\elem

\begin{proof} 
This follows from the definition of the measure $\mu$ 
and by setting $B=Y_j$ in Lemma~\ref{mn} and using the standard 
properties of a Banach limit. 
\end{proof}

\blem\label{l:mujcam} 
For each $j\geq0$, $\mu_j$ is a finite measure on $Y_j$.
\elem

\begin{proof} Let $j\geq0$. Assume without loss of
generality that $\mu_j(Y_j)>0.$ Let $A\sbt Y_j$ be a
measurable set and $(A_k)_{k=1}^\infty$ a countable
measurable partition of $A$. Using termwise operations
on sequences, for every $l\in\N$ we have
\begin{eqnarray*}%\label{eq:1rds103}
\left(\sum_{k=1}^\infty m_n(A_k)\right)_{n=1}^\infty - \sum_{k=1}^l\left(m_n(A_k)\right)_{n=1}^\infty 
&=&\left(\sum_{k=1}^\infty m_n(A_k)\right)_{n=1}^\infty - \left(\sum_{k=1}^l m_n(A_k)\right)_{n=1}^\infty \\ %\nonumber \\
&=&\left(\sum_{k=l+1}^\infty m_n(A_k)\right)_{n=1}^\infty.
\end{eqnarray*}
%\begin{equation}\label{eq:1rds103}
%\begin{aligned}
%&\left(\sum_{k=1}^\infty m_n(A_k)\right)_{n=1}^\infty - \sum_{k=1}^l\left(m_n(A_k)\right)_{n=1}^\infty = \\
%&\left(\sum_{k=1}^\infty m_n(A_k)\right)_{n=1}^\infty - \left(\sum_{k=1}^l m_n(A_k)\right)_{n=1}^\infty = \\
%&\left(\sum_{k=l+1}^\infty m_n(A_k)\right)_{n=1}^\infty.
%\end{aligned}
%\end{equation}
It therefore follows from Lemma~\ref{mn} (with $A=A_k$ and $B=Y_j$) that
\begin{eqnarray*} 
\left\|\left(\sum_{k=1}^\infty m_n(A_k)\right)_{n=1}^\infty - \sum_{k=1}^l\left(m_n(A_k)\right)_{n=1}^\infty\right\|_\infty 
&=&\left\|\left(\sum_{k=l+1}^\infty m_n(A_k)\right)_{n=1}^\infty\right\|_\infty \\
&\leq&\left\|\frac{K_j}{m(Y_j)}\left(m_n(Y_j)\sum_{k=l+1}^\infty m(A_k)\right)_{n=1}^\infty\right\|_\infty \\
&=&\frac{K_j}{m(Y_j)}\left\|\left(m_n(Y_j)\sum_{k=l+1}^\infty m(A_k)\right)_{n=1}^\infty\right\|_\infty.
\end{eqnarray*}
%\[
%\begin{aligned} 
%&\left\|\left(\sum_{k=1}^\infty m_n(A_k)\right)_{n=1}^\infty - \sum_{k=1}^l\left(m_n(A_k)\right)_{n=1}^\infty\right\|_\infty = \\
%&\left\|\left(\sum_{k=l+1}^\infty m_n(A_k)\right)_{n=1}^\infty\right\| \leq \\
%&\left\|\frac{K_j}{m(Y_j)}\left(m_n(Y_j)\sum_{k=l+1}^\infty m(A_k)\right)_{n=1}^\infty\right\|_\infty = \\
%&\frac{K_j}{m(Y_j)}\left\|\left(m_n(Y_j)\sum_{k=l+1}^\infty m(A_k)\right)_{n=1}^\infty\right\|_\infty.
%\end{aligned}
%\]
Since $(m_n(Y_j))_{n=1}^\infty\in l^\infty$ by Lemma~\ref{l:muxjfin} 
and since $\lim_{l\rightarrow\infty}\sum_{k=l+1}^\infty m(A_k)=0$, we 
conclude that 
\[
\lim_{l\rightarrow\infty}\left\|\left(\sum_{k=1}^\infty m_n(A_k)\right)_{n=1}^\infty - \sum_{k=1}^l(m_n(A_k))_{n=1}^\infty\right\|_\infty=0.
\] 
This means that 
\[
\left(\sum_{k=1}^\infty m_n(A_k)\right)_{n=1}^\infty=\sum_{k=1}^\infty(m_n(A_k))_{n=1}^\infty
\] 
in $l^\infty$. Hence, using the continuity of the Banach limit $l_B:l^\infty\rightarrow\R$, we get
\begin{eqnarray*}
\mu(A) 
&=&l_B\bigl((m_n(A))_{n=1}^\infty\bigr)
 = l_B\lt(\bigg(m_n\Bigl(\bigcup_{k=1}^\infty A_k\Bigr)\bigg)_{n=1}^\infty\rt)
 = l_B\lt(\Bigl(\sum_{k=1}^\infty m_n(A_k)\Bigr)_{n=1}^\infty\rt) \\
&=&\sum_{k=1}^\infty l_B\bigl((m_n(A_k))_{n=1}^\infty\bigr)
 = \sum_{k=1}^\infty\mu(A_k).
\end{eqnarray*}
%\begin{align*}
%\mu(A) 
%=&l_B((m_n(A))_{n=1}^\infty)
% =l_B((m_n(\bigcup_{k=1}^\infty A_k))_{n=1}^\infty)
% =l_B((\sum_{k=1}^\infty m_n(A_k))_{k=1}^\infty) \\
%=&\sum_{k=1}^\infty l_B((m_n(A_k))_{n=1}^\infty)
% =\sum_{k=1}^\infty\mu(A_k).
%\end{align*}
So $\mu_j$ is countably additive. Also, $\mu_j(\emptyset)=0$. Thus $\mu_j$ is a measure. 
By Lemma~\ref{l:muxjfin}, $\mu_j$ is finite.
%\qed 
\end{proof}

Combining Lemmas~\ref{l:dumeasure},~\ref{l:muxjfin},~\ref{l:mumcomp}, 
and~\ref{l:mujcam}, and condition~\eqref{d:mmmap item 2} of Definition~\ref{d:mmmap}, 
we get the following.

\blem\label{l:musfmeas} 
$\mu$ is a $\sigma$--finite measure on $X$
equivalent to $m$. Moreover, $\mu(Y_j)\leq\mu(X_j)<\infty$
and $\mu(X_j)>0$ for all $j\geq0$.
\elem

\blem\label{l:mueqpf} 
Formula~(\ref{eq:muequ}) holds.
\elem

\begin{proof} 
Fix $A\in\mathcal{A}$ such that $(m_n(A))_{n=1}^\infty\in l^\infty$. 
Then for every $l\in\N$ we have %that
\begin{align*}
l_B\lt((m_n(A))_{n=1}^\infty\rt)
=&l_B\lt(\sum_{j=0}^l(m_n(A\cap Y_j))_{n=1}^\infty\rt)
 +l_B\lt(\Bigl(m_n\Bigl(\bigcup _{j=l+1}^\infty A\cap Y_j\Bigr)\Bigr)_{n=1}^\infty\rt) \\
=&\sum_{j=0}^l l_B\Big((m_n(A\cap Y_j))_{n=1}^\infty\Big)
 +l_B\lt(\Bigl(m_n\Bigl(A\cap\bigcup_{j=l+1}^\infty Y_j\Bigr)\Bigr)_{n=1}^\infty\rt).
\end{align*}
Letting $l\rightarrow\infty$, we obtain that
\begin{align*}
l_B\bigl((m_n(A))_{n=1}^\infty\bigr)
=&\sum_{j=0}^\infty l_B\bigl((m_n(A\cap Y_j))_{n=1}^\infty\bigr)
 +\lim_{l\rightarrow\infty}l_B\Bigl(\Bigl(m_n\Bigl(A\cap\bigcup_{j=l+1}^\infty Y_j\Bigr)\Bigr)_{n=1}^\infty\Bigr) \\
=&\sum_{j=0}^\infty\mu(A\cap Y_j)
 +\lim_{l\rightarrow\infty}l_B\Bigl(\Bigl(m_n\Bigl(A\cap\bigcup_{j=l}^\infty Y_j\Bigr)\Bigr)_{n=1}^\infty\Bigr) \\
=&\mu(A)+\lim_{l\rightarrow\infty}l_B\Bigl(\Bigl(m_n\Bigl(A\cap\bigcup_{j=l}^\infty Y_j\Bigr)\Bigr)_{n=1}^\infty\Bigr).
\end{align*}
This establishes formula~(\ref{eq:muequ}). In particular, if $A\sbt\bigcup_{j=0}^k X_j$ for some $k\in\N$, then 
$A\cap\bigcup_{j=l}^\infty Y_j\sbt\bigl(\bigcup_{j=0}^k X_j\bigr)\cap\bigl(X\backslash\bigcup_{i<l}X_i\bigr)
=\emptyset$ for all $l>k$. In that case, the equation above reduces to
\[
l_B\bigl((m_n(A))_{n=1}^\infty\bigr)=\mu(A). 
\]
\end{proof}

%\begin{lemma}\label{l:mutinv} 
%The $\sigma$--finite measure $\mu$ is $T$--invariant.
%\end{lemma}
%
%\begin{proof} Let $i\geq0$ be such that $m(Y_i)>0$. Fix a
%measurable set $A\sbt Y_i$. It follows from~(\ref{NCPIIp110})
%that $(m_n(A))_{n=0}^\infty\in l^\infty$, and hence
%$(m_n(T^{-1}(A)))_{n=0}^\infty\in l^\infty$ too. In addition, it
%follows from ergodicity and conservativity of the measure $m$ and
%Corollary~\ref{Cor5.4a} that $\sum_{n=0}^\infty m(T^{-n}(X_0))=\infty$. 
%As $A\sbt Y_i$, it ensues from Lemma~\ref{l:mueqpf} that
%\[
%\mu(T^{-1}(A))
%\leq l_B((m_n(T^{-1}(A)))_{n=1}^\infty)
%=l_B((m_n(A))_{n=1}^\infty)
%=\mu(A).
%\]
%For an arbitrary $A\sbt X$, write $A=\bigcup_{j=0}^\infty A\cap Y_j$ and observe that
%\[
%\mu(T^{-1}(A))=\mu(\bigcup_{j=0}^\infty T^{-1}(A\cap Y_j))
%=\sum_{j=0}^\infty\mu(T^{-1}(A\cap Y_j))
%=\sum_{j=0}^\infty\mu(A\cap Y_j)
%=\mu(A).
%\]
%We are done.
%\qed \end{proof}

\blem\label{l:mutinv}
The $\sigma$--finite measure $\mu $ is $T$--invariant.
\elem

\begin{proof}
Let $i\geq0$ be such that $m(Y_i)>0$. Fix a measurable set $A\subset Y_i$.
By definition, $\mu(A)=l_B\bigl((m_n(A))_{n=1}^\infty\bigr)$.
Furthermore, for all $n\geq0$ notice that
\[
\bigl|m_n(T^{-1}(A))-m_n(A)\bigr|
=\frac{\bigl|m(T^{-(n+1)}(A))-m(A)\bigr|}{\sum_{k=0}^nm(T^{-k}(X_0))}
\leq\frac{1}{\sum_{k=0}^nm(T^{-k}(X_0))}.
\]
%\[
%m_n(A)-\frac{1}{m(X_0)}
%\leq
%m_n(A)+\frac{m(T^{-(n+1)}(A))-m(A)}{\sum_{k=0}^nm(T^{-k}(X_0))}
%=m_n(T^{-1}(A))
%\leq
%m_n(A)+\frac{1}{m(X_0)}
%\]
Thus, $(m_n(T^{-1}(A)))_{n=1}^\infty\in l^\infty$ because $(m_n(A))_{n=1}^\infty\in l^\infty$. Moreover,
by condition~\eqref{d:mmmap item 5} of Definition~\ref{d:mmmap}, it follows from the above and the standard properties of a Banach limit that 
$$
l_B\bigl((m_n(T^{-1}(A)))_{n=1}^\infty\bigr)=l_B\bigl((m_n(A))_{n=1}^\infty\bigr)=\mu(A).
$$  
Keep $A$ a measurable subset of $Y_i$. Fix $l\in\N$. We then have
\begin{align*}
m_n\Bigl(T^{-1}(A)\cap\bigcup_{j=l}^\infty Y_j\Bigr)
=&\frac{\sum_{k=0}^n m\bigl(T^{-k}\bigl(T^{-1}(A)\cap\bigcup_{j=l}^\infty Y_j\bigr)\bigr)}{\sum_{k=0}^n m(T^{-k}(X_0))} \\
&\leq\frac{\sum_{k=0}^n m\bigl(T^{-(k+1)}\bigl(A\cap T(\bigcup_{j=l}^\infty Y_j)\bigr)\bigr)}{\sum_{k=0}^n m(T^{-k}(X_0))} \\
&\leq m_{n+1}\Bigl(A\cap T\Bigl(\bigcup_{j=l}^\infty Y_j\Bigr)\Bigr)
      \cdot\frac{\sum_{k=0}^{n+1}m(T^{-k}(X_0))}{\sum_{k=0}^n m(T^{-k}(X_0))} \\
&\leq K_i\frac{m_{n+1}(Y_i)}{m(Y_i)}
      \cdot m\Bigl(A\cap T\Bigl(\bigcup_{j=l}^\infty Y_j\Bigr)\Bigr)
			\cdot\frac{\sum_{k=0}^{n+1}m(T^{-k}(X_0))}{\sum_{k=0}^n m(T^{-k}(X_0))},
\end{align*}
where the last inequality sign holds by Lemma~\ref{mn} since $A\sbt Y_i$. 
When $n\rightarrow\infty$, the last quotient on the right--hand side approaches $1$.
Therefore
\[
0
\leq l_B\lt(\Bigl(m_n\Bigl(T^{-1}(A)\cap\bigcup_{j=l}^\infty Y_j\Bigr)\Bigr)_{n=1}^\infty\rt)
\leq K_i\frac{\mu(Y_i)}{m(Y_i)}m\lt(T\Bigl(\bigcup_{j=l}^\infty Y_j\Bigr)\rt).
\]
Hence, by virtue of condition~\eqref{d:mmmap item 6} of Definition~\ref{d:mmmap},
\[
0
\leq\lim_{l\rightarrow\infty}l_B\lt(\bigg(m_n\Bigl(T^{-1}(A)\cap\bigcup_{j=l}^\infty Y_j\Bigr)\bigg)_{n=1}^\infty\rt)
\leq K_i\frac{\mu(Y_i)}{m(Y_i)}\lim_{l\rightarrow\infty}m\lt(T\Bigl(\bigcup_{j=l}^\infty Y_j\Bigr)\rt)=0.
\]
So
\[
\lim_{l\rightarrow\infty}l_B\lt(\bigg(m_n\Bigl(T^{-1}(A)\cap\bigcup_{j=l}^\infty Y_j\Bigr)\bigg)_{n=1}^\infty\rt)
=0.
\]
It thus follows from Lemma~\ref{l:mueqpf} that
\[
\mu(T^{-1}(A))
=l_B\bigl((m_n(T^{-1}(A)))_{n=1}^\infty\bigr)
=l_B\bigl((m_n(A))_{n=1}^\infty\bigr)
=\mu(A).
\]
For an arbitrary $A\in\mathcal{A}$, write $A=\bigcup_{j=0}^\infty A\cap Y_j$ and
observe that
\[
\mu(T^{-1}(A))
=\mu\Bigl(\bigcup_{j=0}^\infty T^{-1}(A\cap Y_j)\Bigr)
=\sum_{j=0}^\infty\mu\bigl(T^{-1}(A\cap Y_j)\bigr)
=\sum_{j=0}^\infty\mu(A\cap Y_j)
=\mu(A).
\]
We are done.
\end{proof}

\vspace{2mm}

\noindent{\bf Proof of Theorem~\ref{t1h75}:} Combining 
Lemmas~\ref{l:muxjfin},~\ref{l:musfmeas},~\ref{l:mueqpf}, and ~\ref{l:mutinv}, we obtain the full statement of Theorem~\ref{t1h75} except its last assertion. This last assertion however holds because of the following well known theorem:

\bthm\label{t1j77} 
Let $T:(X,\mathcal{A})\lra(X,\mathcal{A})$ be a measurable transformation and  
$m$ a $\sg$--finite quasi--$T$--invariant measure. If $T$ is ergodic and conservative with respect to $m$ then, up to a positive multiplicative constant, there exists at most one non--zero $\sg$--finite $T$--invariant measure $\mu$ which is absolutely continuous with respect to $m$.
\ethm
\qed

\brem\label{r5.21a} 
In the course of the proof of
Theorem~\ref{t1h75} we have shown that 
\[
0<\inf\{m_n(A):\,n\in\N\}\leq\sup\{m_n(A):\,n\in\N\}<+\infty 
\]
for all $j\geq0$ and all measurable sets $A\sbt X_j$ such that $m(A)>0$.
\erem

\section{Corrected Proofs of Lemma~7.9 and lemma~7.10 from \cite{sush}}

The formulations and proofs of Lemma~7.9 and Lemma~7.10 from \cite{sush} were not entirely correct. Although we do not rely on them in our current manuscript, we take now the opportunity to provide their corrected formulations and proofs based on the progress we have made in the current manuscript. For the convenience of the reader we formulate first Proposition~5.3 from \cite{sush}, which is one of the main ingredients in the proof of Lemma~\ref{l1h43} but has not been formulated in our current manuscript yet. We also use other results, definition, and notation from \cite{sush}, but since \cite{sush} is published and easily accessible we do not explain or copy them here. 

\bprop[Proposition 5.3 in \cite{sush}]\label{pu6.4h27}
Fix $\th\in(0,\min\{1,\g\})$. For all $(\tau,z)\in J(\tf)$ and  $r>0$
there exists a minimal integer $s=s(\th,(\tau,z),r)\ge 0$ with the following
properties (a) and (b).
\begin{itemize}
\item[\mylabel{a}{pu6.4h27 item a}] $|(\tf^s)'(\tau,z)|\ne 0$.

\,

\item[\mylabel{b}{pu6.4h27 item b}] Either $r|(\tf^s)'(\tau,z)|>\|\tf'\|_\infty^{-1}$ or there exists
$c\in\Crit(f_{\tau_{s+1}})$ such that
%$(c,\sg^s(\tau))\in J(\tf)$
$$
f_{\tau _{s+1}}(c)\in J(G) 
\  \  \  {\rm and} \  \  \
|f_{\tau|_s}(z)-c|\le \th r|f_{\tau|_s}'(z)|.
$$
In addition, for this $s$, we have

\,

\item[\mylabel{c}{pu6.4h27 item c}]
$\theta r|f_{\tau |_{s}}'(z)|\leq \theta <\gamma $ and
$$
\Comp\(z, f_{\tau|_s},(KA_{f}^2)^{-1}2^{-\#\Crit(f)}\th r|f_{\tau|_s}'(z)|\)
     \cap\Crit(f_{\tau|_s})=\es.
$$
\end{itemize}
\eprop

\blem[Lemma~7.9 in \cite{sush}]\label{l1h43}
Suppose that $\Ga$ is a closed subset of $J(G)$ such that $g(\Gamma )\cap J(G)\subset \Gamma $ for each $g\in G$, and that
$\^m$ is a Borel probability
nearly upper $t$--conformal measure on $J(\tf)$ respective to $\Ga$. 

Fix $i\in\{0,1,\ld,p\}$ and
suppose that for every critical point $c\in S_i(f)\cap\Ga$ the measure
$\^m|_{\Sg(c)\times\oc}\circ p_2^{-1}$%(B_2(^{-1}$ 
is upper $t$--estimable at $c$.

Then the measure $m$
is uniformly upper $t$--estimable at all points $z\in J_i(G)\cap\Ga$.
\elem

\noindent{\sl Proof.} Since $\Ga$ is a closed set and $\Crit(f)$ is finite, the number
$\De=\dist_\C(\Ga,\Crit(f)\sms\Ga)$ is positive
(if $\Crit(f)\setminus \Ga =\emptyset $ then we put $\De =\infty $).
Fix $\th\in(0,\min\{1,\g\})$ so small that
\beq\label{1h42}
\th\|\tilde{f}'\|_\infty^{-1}<\min \{ \De, \rho \} .
\eeq
Put
$$
\a=\th(KA_{f}^2)^{-1}2^{-\#\Crit(f)}.
$$
Let $z\in J_{i}(G)\cap \Ga .$
Fix $\tau\in\Sg_u$ such that $(\tau,z)\in J(\tf)$, i.e. $\tau\in p_1(J(\tf)\cap p_2^{-1}(z))$.
Assume $r\in(0,R_{f}]$ to be sufficiently small.
Let 
$$
s(\tau,r):=s(\th,(\tau,z),8\a^{-1}r)\ge 0
$$ 
be the integer produced in Proposition~\ref{pu6.4h27}.
Set 
$$
R_{\tau|_{s(\tau ,r)+1}}:=4r|f_{\tau|_{s(\tau ,r)}}'(z)|.
$$
It then follows from Proposition~\ref{pu6.4h27}
that the family
$$
\Fa(z,r)=\big\{\tau|_{s(\tau,r)+1}:\tau\in p_1(J(\tf)\cap p_2^{-1}(z))\big\}
$$
is $(4, \gamma, V)$-essential for the pair $(z,r)$, where 
$$
V:=\bu\big\{[\tau|_{s(\tau,r)+1}]:\tau\in p_1(J(\tf)\cap p_2^{-1}(z))\big\}.
$$
Keep 
$$
\tau\in p_1(J(\tf)\cap p_2^{-1}(z))
\  \  \   {\rm and} \  \  \
s=s(\tau,r).
$$
Suppose that the first alternative of \eqref{pu6.4h27 item b} in Proposition~\ref{pu6.4h27} holds. Then
$8\a^{-1}r|f_{\tau|_s}'(z)|>\|\tilde{f}'\|_\infty^{-1}$. So, using Koebe's Distortion Theorem, and assuming that $\theta $ is small enough, we
get from the nearly upper $t$--conformality of $\^m$ respective to $\Ga$ that
\beq\label{1h45}
\aligned
\^m\(\tilde f_{\tau|_s,z}^{-s}([\tau_{s+1}]\times B_2(f_{\tau|_s}(z),R_{\tau|_{s+1}}))\)
&\le \^m\(\tilde f_{\tau|_s,z}^{-s}(p_2^{-1}(B_2(f_{\tau|_s}(z),R_{\tau|_{s+1}})))\) \\
&\le K^t|f_{\tau|_s}'(z)|^{-t}\^mp_2^{-1}\(B_2(f_{\tau|_s}(z),R_{\tau|_{s+1}}))\)  \\
&\le K^t|f_{\tau|_s}'(z)|^{-t} \\
&\le (8K\a ^{-1}\| \tf '\|_\infty  )^tr^t.
\endaligned
\eeq
Now suppose that 
$$
8\alpha ^{-1}r| f_{\tau |_{s}}'(z)| \leq \| \tf '\|_\infty  ^{-1}.
$$ 
This implies that the second alternative of \eqref{pu6.4h27 item b} in Proposition~\ref{pu6.4h27} holds. Let $c\in\Crit(f_{\tau_{s+1}})$ come from item \eqref{pu6.4h27 item b}
of this proposition. In particular,
$$
f_{\tau_{s+1}}(c)\in J(G).
$$
Since $z\in J_i(G)$ (and $\theta \| \tf'\|_\infty ^{-1}<\rho $, where $\rho$ was defined in \cite{sush}), it
follows from the formula (16) in \cite{sush} and from Proposition~\ref{pu6.4h27} that $c\in S_i(f)$. Since
$8\a^{-1}r|f_{\tau|_s}'(z)|\le \|\tilde{f}'\|_\infty ^{-1}$, it follows from Proposition~\ref{pu6.4h27}\eqref{pu6.4h27 item b}
and (\ref{1h42}) that 
$$
|f_{\tau|_s}(z)-c|\le \th||\tilde{f}'||^{-1}<\De.
$$
Because of the definition of $z$ and $\tau$, we have that $f_{\tau|_s}(z)\in J(G)$. Since also $f_{\tau|_s}\in G$, we thus obtain that $f_{\tau|_s}(z)\in\Ga$. In conclusion, $c\in\Ga$. Hence,
making use of Proposition~\ref{pu6.4h27}\eqref{pu6.4h27 item b}, \eqref{pu6.4h27 item c}, as well as Koebe's Distortion Theorem,
nearly upper $t$--conformality of $\^m$, and our $t$-upper estimability assumption,
and assuming $\theta $ is small enough,
we get with some universal constant $C_1$ that
$$
\aligned
\^m\(\^f_{\tau|_s,z}^{-s}([\tau_{s+1}] &\times B_2(f_{\tau|_s}(z),R_{\tau|_{s+1}}))\) \\
&\le K^t|f_{\tau|_s}'(z)|^{-t}\^m\([\tau_{s+1}]\times B_2(f_{\tau|_s}(z),R_{\tau|_{s+1}})\) \\
&\le K^t|f_{\tau|_s}'(z)|^{-t}\^m|_{\Sg(c)\times\oc}\circ p_2^{-1}\(B_2(f_{\tau|_s}(z),R_{\tau|_{s+1}})\) \\
&\le K^t|f_{\tau|_s}'(z)|^{-t}\^m|_{\Sg(c)\times\oc}\circ p_2^{-1}
        \(B_2(c,R_{\tau|_{s+1}}+8\th\a^{-1}r|f_{\tau|_s}'(z)|)\) \\
&\le K^t|f_{\tau|_s}'(z)|^{-t}\^m|_{\Sg(c)\times\oc}\circ p_2^{-1}
        \(B_2(c,4(1+2\th\a^{-1})r|f_{\tau|_s}'(z)|)\) \\
&\le K^t|f_{\tau|_s}'(z)|^{-t}C_1\(4(1+2\th\a^{-1})r|f_{\tau|_s}'(z)|\)^t \\
&=   C_1(4K(1+2\th\a^{-1}))^tr^t.
\endaligned
$$
Combining this with (\ref{1h45}) and applying Proposition~\ref{p1h2a.1}, we get that
\begin{equation}
\label{eq:mbzrast1}
m(B_2(z,r))\le \#_{4,4}C_{1}\max\{8K\a ^{-1}\|\tf'\|_\infty , 4K(1+2\th\a^{-1})\} ^{t}r^t.
\end{equation}
We are done. \endpf

\sp

\blem[Lemma~7.10 in \cite{sush}]\label{l1h47}
There are two functions $(R,S)\mapsto R^{\ast }$ and $L\mapsto \hat{L}$ with the following
property.
\begin{itemize}
\item
Suppose that $\Ga$ is a closed subset of $J(G)$ such that
$$
g(\Gamma )\cap J(G)\subset \Gamma 
$$ 
for each $g\in G$, and that $\^m$ is a Borel probability
nearly upper $t$--conformal measure on $J(\tf)$ respective to $\Ga$
with nearly upper conformality radius $S.$ 

\,

\item Fix $i\in\{0,1,\ld,p\}$
and suppose that the measure $m|_{J_i(G)}$ is uniformly upper $t$--estimable at all points
$z\in J_i(G)\cap\Ga$ with the corresponding estimability constant $L$ and estimability radius $R$.

\,

\,

\item Then the measure $\^m|_{\Sg(c)\times\oc}\circ p_2^{-1}\big|_{J_i(G)}$ is $t$--upper
estimable, with upper estimability constant $\hat{L}$
and radius $R^{\ast }$ at every point
$c\in Cr_{i+1}(f)$
such that 
$$
\bigcup _{|\om |=l}f_{\om }(c_+)\cap J(G)\sbt\Ga.
$$
\end{itemize}
\elem

\noindent{\sl Proof.} Fix $c\in Cr_{i+1}(f)$ such that $\bigcup _{|\om |=l}f_{\om }(c_+)\sbt\Ga$ and also $j\in\{0,1,\ld,u\}$
such that $f_j'(c)=0$. Consider an arbitrary $\tau\in\Sg_u$ such that $\tau_1=j$ and
$(\tau,c)\in J(\tf)$. In view of Lemma~4.8 in \cite{sush}, we get that%\ref{lu5.9h17}
$$
f_{\tau|_{l+1}}(c)\in J_i(G)\cap\Ga.
$$
Let $R>0$ (sufficiently small) be the radius resulting from uniform $t$--upper estimability
at all points of $J_i(G)\cap\Ga$. Let $D_{\tau|_{l+1}}(c)$ be the connected component of
$f_{\tau|_{l+1}}^{-1}(B_2(f_{\tau|_{l+1}}(c),R))$ containing $c$. Set
$$
\nu_{\tau|_{l+1}}=\^m|_{[\tau|_{l+1}]\times\oc}\circ p_2^{-1}\big|_{D_{\tau|_{l+1}(c)}\cap J_i(G)\cap f_{\tau|_{l+1}}^{-1}(J(G))}.
$$
Applying nearly upper $t$--conformality of $\^m$ for every Borel set $A\sbt D_{\tau|_{l+1}}(c)
\sms \{c\}$ such that $f_{\tau|_{l+1}}|_A$ is injective we get that
$$
m(f_{\tau|_{l+1}}(A))
=\^m(\Sg_u\times f_{\tau|_{l+1}}(A)))
=\^m(\^f^{l+1}([\tau|_{l+1}]\times A))
\ge \int_A|f_{\tau|_{l+1}}'(x)|^td\nu_{\tau|_{l+1}}(x).
$$
It therefore follows from Lemma~2.10 %\ref{lku4.10} 
and item (c) of Definition~7.8 %\ref{d1h40} 
in \cite{sush} that the measure $\nu_{\tau|_{l+1}}$ is upper $t$-estimable at $c$ with upper estimability constant $L_{0}$ and radius $R_{0}$ independent of $\^m$ (but possibly $R_{0}$ depends on $(R,S)$ and
$L_{0}$ depends on $L$). Let
$$
\Fa=\big\{\tau|_{l+1}:(\tau,c)\in J(\tf) \  \text{ and } \ f_{\tau_1}'(c)=0\big\}.
$$
Let 
$$
D_c:=\bi_{\om\in\Fa}D_\om(c)\cap J_i(G)\cap f_{\tau|_{l+1}}^{-1}(J(G)).
$$
Since $\#\Fa\le u^{l+1}$ and since
$$
\^m|_{\Sg(c)\times\oc}\circ p_2^{-1}|_{D_{c}}=\sum_{\om\in\Fa}\nu_\om |_{D_{c}},
$$
we conclude that the measure $\^m|_{\Sg(c)\times\oc}\circ p_2^{-1}\big|_{J_i(G)}$ is $t$--upper
estimable at the point $c$ with upper estimability constant $\hat{L}$
and radius $R^{\ast }$ independent of $\^m$.
We are done. 
\endpf

\section{Definitions of Classes of Rational Semigroups Used \\ and Relations Between Them} \label{DCRS}

\bdfn[Definition~\ref{d120220905}]\label{d120220905}
A rational semigroups is called {\bf expanding (along fibers)} if and only if there exists an integer $n\ge 1$ such that 
$$
\big|\(\tf^n\)'(\xi)\big|\ge 2
$$
for all $\xi\in J(\tilde f)$. 

Equivalently, there are two constants $C>0$ and $\l>1$ such that 
$$
\big|\(\tf^n\)'(\xi)\big|\ge c\l^n
$$
for all $\xi\in J(\tilde f)$ and all integers $n\ge 0$. 
\edfn

\bdfn[Definition~]\label{d320220905}  
A rational semigroup $G$ is called \textbf{hyperbolic} if and only if 
$$
\PCV(G)\sbt F(G).
$$
\edfn

\bdfn[Definition~\ref{d1h3}]\label{d320220816}  
A rational semigroup $G$ is called \textbf{semi--hyperbolic} if and only if 
there exist an $N\in \N $ and a $\delta >0$ such that 
for each $x\in J(G)$ and each $g\in G$, 
$$\deg \(g:V\longrightarrow B_{s}(x,\delta )\)\leq N$$
for each connected component $V$ of $g^{-1}(B_{s}(x,\delta )).$ 
%In general, points enjoying this property are referred to as non--recurrent. 
\edfn

From now on all classes of rational semigroups considered in this appendix, including the two above, are assumed to satisfy the following condition.

\sp\noindent {\bf Fundamental Assumption:} If $G$ is a rational semigroup, then the following three conditions are assumed to hold.

\sp\begin{itemize}
\item There exists an element $g$ of $G$ such that $\deg (g)\geq 2.$

\sp\item Each element of Aut$(\oc )\cap G$ is loxodromic.  

\sp\item $F(G)\ne\es$. 
\end{itemize} 

\bdfn[Definition~\ref{d220220816}]\label{d420220816}
A rational semigroup $G$ is called \textbf{*semi--hyperbolic} if and only if it is semi--hyperbolic and, we repeat, it satisfies the Fundamental Assumption.
\edfn 

\bdfn[Definition~\ref{d1nsii5}]\label{d420220816}
We say that a finitely generated rational semigroup $G$ generated by a $u$--tuple map $f:=(f_1,\ld,f_u)\in \Rat^u$ and satisfying the Fundamental Assumption is \textbf{totally non--recurrent} (\TNR) if and only if

\begin{itemize}

\sp\item[\mylabel{a}{d1nsii5 item a}] For each $z\in J(G)$ there exists a neighborhood $U$ of $z$ in $\oc $ (in fact in $\C$)
such that for any sequence $\{ g_{n}\} _{n=1}^{\infty }$ in $G$, 
any domain $V$ in $\oc $ and any point 
$\zeta \in U$, the sequence $\{ g_{n}\} _{n=1}^{\infty }$ does not converge 
to $\zeta $ locally uniformly on $V$ 

\,

\fr and

\sp\item[\mylabel{b}{d1nsii5 item b}] 
$$
\Crit_*(f)\cap \PCV(G)=\es. 
$$
\end{itemize}
\edfn

\bobs[Observation~\ref{o1nsii6}]\label{o520220816}
Every \TNR \ rational semigroup is *semi--hyperbolic.
\eobs

\bdfn[Definition~\ref{d1_2016_06_19}]\label{d620220816}
A rational semigroup $G$ is called \textbf{\textup{C--F} balanced} if 
$$
D(G):=\dist\(J(G),\PCV(G)\cap F(G)\)>0.
$$
\edfn

\bdfn[Definition~\ref{d720220816}]\label{d820220816}
A finitely generated rational semigroup $G$ generated by a $u$--tuple map $f:=(f_1,\ld,f_u)\in \Rat^u$ is said to be of \textbf{finite type} if and only if the set $\Crit_*(\tf)$, i.e. the set of all critical points of $\tf$ lying in the Julia set $J(\^f)$, is finite.
\edfn

\bdfn[Definition~\ref{d920220816}]\label{d1020220816}
Any \CF\ balanced \TNR\ rational semigroup of finite type is called \textbf{finely non--recurrent}, abbreviated as \FNR. If in addition this group satisfies the Nice Open Set Condition, then it is referred to as \textbf{\NOSC-FNR}. 
\edfn

\bobs[Lemma~\ref{l320190325}]\label{d1120220816}
Any \TNR\ rational semigroup that satisfies the Nice Open Set Condition is of finite type, whence it is a \NOSC-FNR semigroup.
\eobs

We now summarize the inclusions holding between various classes of rational semigroups.

\bfact
$$
{\rm semi-hyperbolic}\supset {\rm ^*semi-hyperbolic}
\supset \TNR \supset \FNR \supset \NOSC-FNR
$$
and 
$$
\begin{aligned}
\FNR \supset  {\rm expanding}={\rm hyperbolic} &\supset  {\rm hyperbolic \ satisfying \ the \ Open \ Set \ Condition}
\\
&\supset {\rm hyperbolic \ satisfying \ the \ Nice \ Open \ Set \ Condition}
\end{aligned}
$$
\efact

\section{Open Problems} 

In this short section we formulate several problems which are somehow related to the content of our manuscript.

In the context of single rational functions the ones satisfying the Exponential Shrinking Property (ESP) can be also characterized (see \cite{PrzRiv07}) as Topological Collet--Eckmann (TCE) rational maps. The concept of Topological Collet--Eckmann maps easily extends to the theory of rational semigroups, especially when expressed in terms of the skew product map $\tf:J(\^f)\lra J(\^f)$. Our first problem is the following.

\bprob
Under which hypotheses Topological Collet--Eckmann rational semigroups coincide with those satisfying the Exponential Shrinking Property?
\eprob

Our second problem concerns this class of rational semigroups.

\bprob
Which results of our manuscript would extend to the class of ESP/TCE rational semigroups? Perhaps with some additional mild hypotheses, definitely milder than ours?
\eprob

One could be guided by our current manuscript, paper \cite{PrzRiv07}, and some other papers by Przytycki and Rivera--Letelier, but one would actually have to start the theory of the class of ESP/TCE from scratch. This would be a kind of titanic work but we conjecture that most of our results, perhaps some of them in a weaker or largely modified form, would go through.

\sp On a different note:

\bprob
Do there exist non--trivial, i.e. satisfying the standard hypotheses, rational semigroups for which some fiber Julia sets have the same Hausdorff dimension as the global Julia set? 
\eprob

\sp It is an immediate consequence of ergodicity that for any Borel probability shift--invariant ergodic measure $\mu$ on the symbolic space $\Sg_u$, $\mu$--almost all fiber Julia sets have the same Hausdorff dimension. But we ask the following.

\bprob
Do there exist non--trivial, i.e. satisfying the standard hypotheses, rational semigroups for which all the fiber Julia sets have the same Hausdorff dimension?
\eprob

\sp Having all the results, main and technical, of Section~\ref{section:MA},  Hausdorff Dimension of Invariant Measures $\mu_t$ and Multifractal Analysis of Lyapunov Exponents, we are very confident that the following is true.

\bcon
In the context of Section~\ref{section:MA} one can perform the multifractal analysis of the level sets of the function 
$$
J(G)\ni x\longmapsto d_{\mu_t\circ p_2^{-1}}(x) \in [0,2]
$$
analogously as done in \cite{MundayRoyUII}; see also the relevant references therein.
\econ

\sp Our last problem concerns Diophantine properties of the measures $\mu_t\circ p_2^{-1}$, $t\in\De_G^*$. More precisely, we feel confident to state the following. 

\bcon
All the measures $\mu_t\circ p_2^{-1}$, $t\in\De_G^*$, are quasi--decaying, in the sense of \cite{DFSU_1} and \cite{DFSU_2} for all \NOSC-FNR rational semigroups; consequently, these are extremal in the sense of Diophantine analysis.   
\econ

\

\

\

%===

\printindex
\clearpage
\thispagestyle{plain}
\section*{Acknowledgements}
We are very grateful and indebted to the referee of our manuscript whose valuable comments and suggestions greatly influenced the final exposition and improved the quality of our text.

\sp \thanks{\bf Research Support}: Research of the first author was supported by an ARC Discovery Project.
	Research of the second author supported in part by 
	JSPS Grant--in--Aid for Scientific Research (C) 21540216. 
	Research of the third author supported in part by 
	NSF Grant DMS 0700831 and Simons Grant 581668. \newline}
	
{\bf Addresses}:	
\thanks{\ \newline 
	\noindent
	Jason Atnip \newline
	School of Mathematics and Statistics, University of New South Wales, Sydney, NSW 2052, Australia \newline
	E-mail: j.atnip@unsw.edu.au \ \  
	Web: https://atnipmath.com\newline
	\ \newline 
	Hiroki Sumi\newline 
	Course of Mathematical Science
	Department of Human Coexistence,
	Graduate School of Human and Environmental Studies,
	Kyoto University, 
	Yoshida Nihonmatsu-cho, Sakyo-ku, Kyoto, 606-8501, Japan \newline 
	E-mail:  sumi@math.h.kyoto-u.ac.jp   \  \ 
	Web:http://www.math.h.kyoto-u.ac.jp/~sumi/  \newline
	\ \newline 
	Mariusz Urba\'nski\newline Department of Mathematics,
	University of North Texas, Denton, TX 76203-1430, USA\newline  
	E-mail: urbanski@unt.edu \ \
	Web: http://www.math.unt.edu/$\sim$urbanski/}

\end{document}